\documentclass[12pt,openright,twoside]{report}
\usepackage[utf8]{inputenc}
\usepackage{graphicx}
\graphicspath{{images}/}
\usepackage[export]{adjustbox}
\usepackage{caption}
\usepackage{fancyhdr}
\usepackage{subcaption}
\usepackage[a4paper,width=150mm,top=40mm,bottom=40mm]{geometry}
\usepackage[dvipsnames]{xcolor}
\usepackage{fancyhdr}
\usepackage{amsmath,amssymb}
\usepackage{amsthm}
\usepackage{setspace}
\usepackage{url}
\usepackage{bm}
\usepackage{scrextend}
\usepackage{amsfonts}
\usepackage{stmaryrd}
\addtokomafont{labelinglabel}{\ttfamily}
\usepackage{float}
\usepackage{tocloft}

\usepackage{svg}

\usepackage{pgfplots} 

\pgfplotsset{compat=newest} 
\pgfplotsset{plot coordinates/math parser=false} 
\newlength\figureheight 
\newlength\figurewidth

\usepackage[
sorting=nyt,
minbibnames=5,
maxbibnames=5,
giveninits=true,
isbn=false,
doi=false,
eprint=false
]{biblatex}
\usepackage[shortlabels]{enumitem}

\usepackage{hyperref}
\hypersetup{
    linktocpage=true,
    colorlinks=true,
    linkcolor=blue,
    citecolor=red,
    urlcolor=black,
    bookmarksopen=true,
    }

\newcommand{\numberset}{\mathbb}
\newcommand{\N}{\numberset{N}}
\newcommand{\Z}{\numberset{Z}}
\newcommand{\R}{\numberset{R}}
\newcommand{\C}{\numberset{C}}

\newcommand\numberthis{\addtocounter{equation}{1}\tag{\theequation}}
\newcommand\restr[2]{{
  \left.\kern-\nulldelimiterspace
  #1
  \vphantom{\big|}
  \right|_{#2}
  }}

\makeatletter
\newcommand{\fixed@sra}{$\vrule height 2\fontdimen22\textfont2 width 0pt\shortrightarrow$}
\newcommand{\shortarrow}[1]{
  \mathrel{\text{\rotatebox[origin=c]{\numexpr#1*45}{\fixed@sra}}}
}
\makeatother

\makeatletter
\newcommand{\rightleftarrowss}[2]{
  \mathrel{\mathop{
    \vcenter{\offinterlineskip\m@th
      \ialign{\hfil##\hfil\cr
        \hphantom{$\scriptstyle\mspace{8mu}{#1}\mspace{8mu}$}\cr
        \rightarrowfill\cr
        \vrule height0pt width 2em\cr
        \leftarrowfill\cr
        \hphantom{$\scriptstyle\mspace{8mu}{#2}\mspace{8mu}$}\cr
        \noalign{\kern-0.3ex}
      }
    }
  }\limits^{#1}_{#2}}
}
\makeatother

\newtheorem{theorem}{Theorem}[chapter]
\newtheorem{definition}[theorem]{Definition}
\newtheorem{proposition}[theorem]{Proposition}
\newtheorem{corollary}[theorem]{Corollary}
\newtheorem{lemma}[theorem]{Lemma}
\newtheorem{conjecture}[theorem]{Conjecture}
\newtheorem{remark}[theorem]{Remark}

\title{???}

\author{Nicola Galante}
\date{???}

\begin{document}

\begin{titlepage}

\newboolean{english}
\setboolean{english}{true}
\newboolean{corelatore}
\setboolean{corelatore}{true}

\thispagestyle{empty}
\space
\begin{center}
\textsc{\Large{Università degli studi di Pavia}\\
\normalsize{Dipartimento di Matematica Felice Casorati\\
Corso di Laurea Magistrale in Matematica}}
\end{center}
\[\]
\begin{center}
	\includegraphics[width=0.35\textwidth]{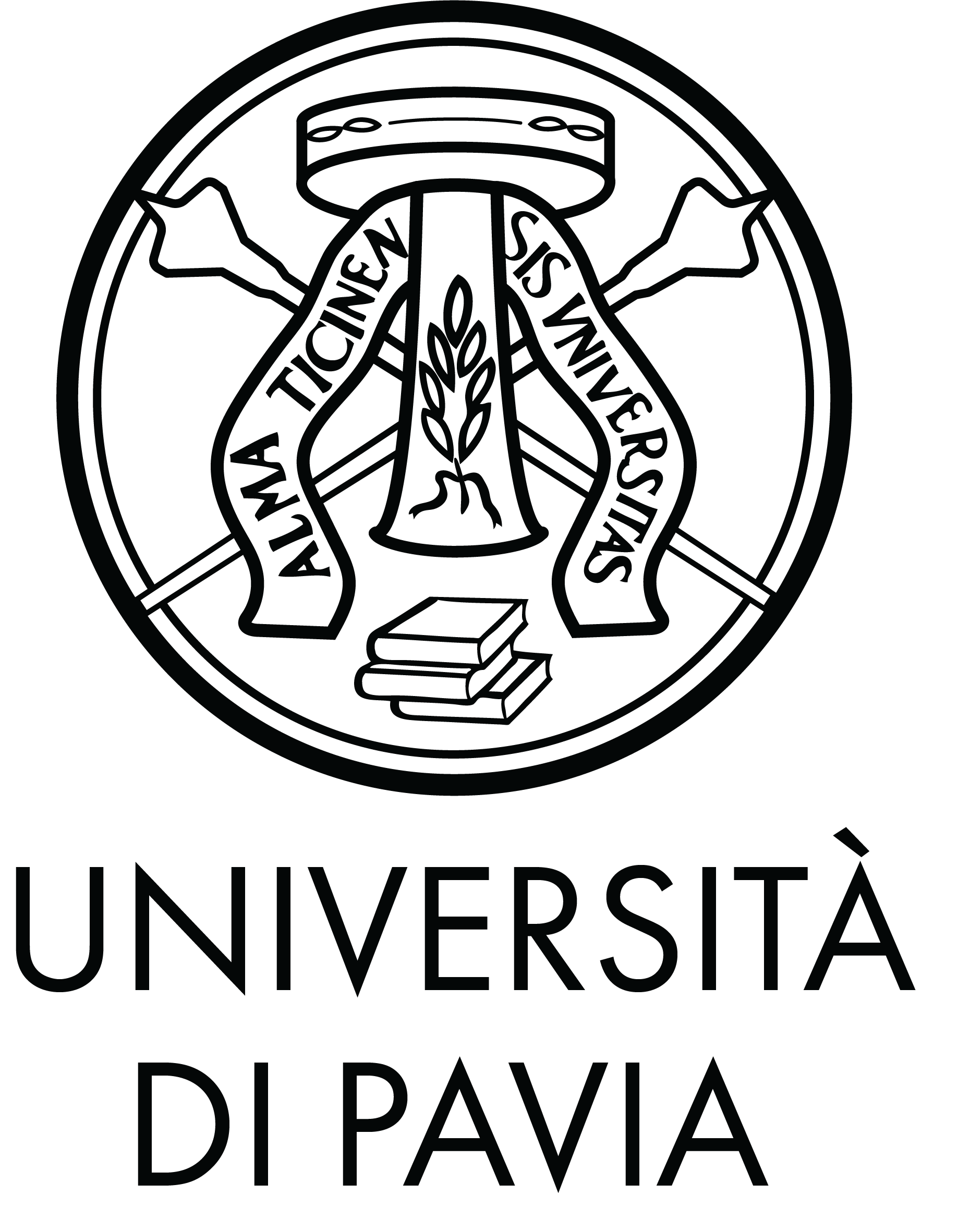}
\end{center}
\[\]
\space
\begin{center}
\ifthenelse{\boolean{english}}
{
\smallskip

\textbf{\Large{Evanescent Plane Wave Approximation of Helmholtz Solutions in Spherical Domains}}
}
{}
\end{center}
\singlespace
\[\]
\begin{center}
\textbf{Tesi di Laurea Magistrale in Matematica}
\end{center}
\[\]\[\]
\begin{flushleft}
\ifthenelse{\boolean{english}}{Relatore (Supervisor):\\}{Relatore:\\}
\textbf{Prof. Andrea Moiola}
\ifthenelse{\boolean{corelatore}}{
\ifthenelse{\boolean{english}}{\\Correlatore (Co-Supervisor):\\}{\\Correlatore:\\}
\textbf{Dr. Emile Parolin}
}
{}
\end{flushleft}
\[\]
\begin{flushright}
Tesi di Laurea di: \\
\textbf{Nicola Galante}\\
Matricola 493107
\end{flushright}
\[\]\[\]\[\]
\begin{center}
Anno Accademico 2021--2022
\end{center}

\end{titlepage}

\newpage\null\thispagestyle{empty}\newpage

\chapter*{Abstract}

The recent results presented in \cite{parolin-huybrechs-moiola} have led to significant developments in achieving stable approximations of Helmholtz solutions by plane wave superposition.
The study shows that the numerical instability and ill-conditioning inherent in plane wave-based Trefftz methods can be effectively overcome with regularization techniques, provided there exist accurate approximations in the form of expansions with bounded coefficients.
Whenever the target solution contains high Fourier modes, propagative plane waves fail to yield stable approximations due to the exponential growth of the expansion coefficients.
Conversely, evanescent plane waves, whose modal content covers high Fourier regimes, are able to provide both accurate and stable results.
The developed numerical approach, which involves constructing evanescent plane wave approximation sets by sampling the parametric domain according to a probability density function, results in substantial improvements when compared to conventional propagative plane wave schemes.

The following work extends this research to the three-dimensional setting, confirming the achieved results and introducing new ones.
By generalizing the 3D Jacobi--Anger identity to complex-valued directions, we show that any Helmholtz solution in a ball can be represented as a continuous superposition of evanescent plane waves.
This representation extends the classical Herglotz one and provides a relevant stability result that cannot be achieved with the use of propagative waves alone.
The proposed numerical recipes have been tailored for the 3D setting and extended with new sampling strategies involving extremal systems of points. These methods are tested by numerical experiments, showing the desired accuracy and bounded-coefficient stability, in line with the two-dimensional case.

\newpage\null\thispagestyle{empty}\newpage

\tableofcontents

\clearpage
\phantomsection
\addcontentsline{toc}{chapter}{Introduction}

\chapter*{Introduction}
\section*{Theoretical framework}

The homogeneous Helmholtz equation
\begin{equation}
\Delta u + \kappa^2 u=0,
\label{Helmholtz equation}
\end{equation}
is a fundamental mathematical model that arises in a wide range of scientific and engineering applications, including acoustics, electromagnetics, linear elasticity and quantum mechanics.
It is closely related to the scalar wave equation $\frac{1}{c^2}\frac{\partial^2 U}{\partial t^2}-\Delta U =0$, as it describes the spatial dependence of time-harmonic solutions $U(\mathbf{x},t)=\Re\{e^{-i\omega t}u(\mathbf{x})\}$, where the real parameter $\kappa=\omega/c>0$ is the wavenumber, while $c$ and $\omega$ are the wave speed and the angular frequency, respectively.

Contrary to the case where the wavenumber $\kappa$ is relatively small, and thus the Helmholtz equation can be regarded as a perturbation of the Laplace equation, approximating Helmholtz solutions becomes increasingly difficult and computationally expensive as soon as we enter the realm of high-frequency problems, namely when the wavelength $\lambda=2\pi/\kappa$ is much smaller than the characteristic length of the computational domain.
The main difficulty stems from the highly oscillatory nature of Helmholtz solutions, necessitating a large number of degrees of freedom to achieve high accuracy using piecewise polynomials.

Various numerical methods are presented in the literature to address this issue, among which the Trefftz methods stand out as particularly intriguing. These discretization schemes employ trial and test functions that are locally piecewise solutions of the differential equation being approximated -- in this instance, the Helmholtz equation -- and offer significant advantages, notably their relatively low computational cost. Like Finite Element Methods (FEMs), they provide a volume discretization, while like Boundary Element Methods (BEMs), they only require integration on lower-dimensional manifolds. Indeed, the Trefftz methods can often provide more accurate solutions with fewer degrees of freedom than standard numerical methods, resulting in significant computational savings, particularly for large-scale problems. 

Out of all Trefftz approximation spaces and their corresponding basis functions, propagative plane waves $e^{i\kappa \mathbf{d}\cdot \mathbf{x}}$, where $\mathbf{d} \in \R^n$ such that $\mathbf{d}\cdot\mathbf{d}=1$ represents the propagation direction, are noteworthy, since their simple exponential expression makes the Trefftz schemes implementation very cost-effective.
Specifically, the computation of integrals over any flat sub-manifold with a boundary that is piecewise flat or straight can be carried out in a closed form, and with an effort that is independent of $\kappa$ (see \cite[Sec.\ 4.1]{hiptmair-moiola-perugia1}).
Unfortunately, the linear systems spawned by propagative plane waves are susceptible to ill-conditioning when high-resolution trial spaces are employed. Indeed, the computation of the expansion is known to be numerically unstable due to the nearly linear dependence of propagative plane waves with similar propagation directions (see \cite[Sec.\ 4.3]{hiptmair-moiola-perugia1}). This is a purely numerical phenomenon that appears when using floating-point arithmetic. As a result, the convergence predicted by the approximation theory cannot be achieved and the accuracy of the numerical scheme stagnates.

\section*{Recent results in 2D}

The study presented in \cite{parolin-huybrechs-moiola} falls within the introduced framework, bringing new interesting developments in providing stable approximations of Helmholtz solutions by plane wave superposition.
The analysis remains restricted to the case of a circular domain, in an effort to provide explicit and precise theoretical results.
Recent progress in frame approximation theory (see \cite{huybrechs1,huybrechs2}) have shown that regularization techniques can effectively address the issue of ill-conditioning, given there exist accurate approximations in the form of expansions with bounded coefficients.
In this case, achieving accurate results in floating-point arithmetic requires not only examining the best approximation error but also studying the norm of the coefficients in the expansion, which is inherently dependent on the chosen representation. 

When employing linear combinations of propagative plane waves to represent Helmholtz solutions characterized by high Fourier modal contents, the resulting expansions invariably comprise exponentially large coefficients.
This evidence is captured by the proof exhibited in \cite[Th.\ 4.3]{parolin-huybrechs-moiola}, which shows that the instability of the approximation stems from the incapability of propagative plane waves to adequately represent the high Fourier modes of any smooth Helmholtz solution.

The key idea is then to enrich the propagative plane wave approximation set with other Helmholtz solutions, which allow to obtain accurate and bounded-coefficient approximations while retaining a simple and cheap implementation.
Evanescent plane waves seem to be the natural candidate for this task, since they conserve the form $e^{i\kappa \mathbf{d}\cdot \mathbf{x}}$, but are distinguished by a complex-valued direction $\mathbf{d} \in \C^n$, where $\mathbf{d}\cdot\mathbf{d}=1$ in order to satisfy the Helmholtz equation.
Hence, the main quest remains to identify a suitable set of evanescent plane waves and verify its effectiveness in ensuring both accuracy and stability.

The modal analysis, which is made possible due to the considered circular domain, indicates that evanescent plane waves can effectively approximate the high Fourier modes of the Helmholtz solutions. This is a crucial attribute that propagative plane waves lack.
Expectations are confirmed by the result provided in \cite[Th.\ 6.7]{parolin-huybrechs-moiola}, which establishes that every Helmholtz solution on the unit disk can be uniquely expressed as a continuous superposition of evanescent plane waves and, moreover, its corresponding density is bounded in a suitable norm. In this sense, this theorem can be considered a sort of stability result at the continuous level.
The primary emphasis is on the \textit{Herglotz transform} \cite[Def.\ 6.6]{parolin-huybrechs-moiola}, an integral operator which enables each solution of the Helmholtz equation to be expressed as the image of a unique function in a weighted $L^2$ space, named \textit{Herglotz density space} \cite[Eq.\ (6.5)]{parolin-huybrechs-moiola}. Moreover, the results in \cite[Cor.\ 6.13]{parolin-huybrechs-moiola} show that the Herglotz transform maps any point-evaluation functional of the Herglotz density space to an evanescent plane wave.
The ensuing equivalence between the Helmholtz solution approximation problem (by evanescent plane waves) and the approximation problem of the corresponding Herglotz density (by evaluation functionals) paves the way to stable discrete representations.

The numerical scheme described in \cite[Sec.\ 7]{parolin-huybrechs-moiola} adopts the ideas proposed in \cite{Cohen_Migliorati} in order to reconstruct the Herglotz densities starting from a finite number of sa-mples.
Despite the existing literature does not provide enough theoretical evidence to ensure the desired accurate and bounded-coefficient approximation properties, the continuous-level results and the numerical experiments presented endorse the validity of \cite[Conj.\ 7.1]{parolin-huybrechs-moiola}.
Consequently, similar approximation properties are expected to be shared by the evanescent plane wave sets in the space of Helmholtz solutions.

In the end, the numerical approach developed, which is based on circular geometries, is evaluated through various forms, exhibiting substantial enhancements if compared to conventional propagative plane wave schemes.

\section*{Extensions to 3D}

There are numerous potential paths for extending the results achieved in \cite{parolin-huybrechs-moiola}. For instance, the scope could be broadened to encompass more general geometries, thus allowing for the application of Trefftz schemes, or considering more sophisticated boundary value problems, involving time-harmonic Maxwell or elastic wave equations.
Nevertheless, the present work opts for a different but essential approach: the goal is to extend the results of \cite{parolin-huybrechs-moiola} to three-dimensional geometries.
Similarly, the analysis remains restricted to the case of spherical domains, such that explicit theoretical results can be exhibited through the use of modal analysis.

Thanks to the introduction of spherical waves,
the structure of the space of Helmholtz solutions can be readily extended to spherical domains.
The first challenge arises because, unlike the two-dimensional scenario, there is no obvious way of considering equally spaced points on the spherical surface.
This seems to be desirable for defining propagative plane wave approximation sets.
The extremal point systems discussed in \cite{sloan-womersley1,sloan-womersley2} provide a potential solution by possessing excellent geometric characteristics that lead to well-distributed points. Moreover, these point sets exhibit good integration properties when employed to establish an interpolatory integration rule. Thanks to the results presented in Corollary \hyperlink{Corollary 2.1}{2.5}, we are able to incorporate them into the sampling-based numerical method used to approximate Helmholtz solutions.
Furthermore, we exploit the extremal systems to provide numerical evidence of the negative results regarding the instability in approximating Helmholtz solutions using propagative plane waves, which are confirmed theoretically in Theorem \hyperlink{Theorem 3.1}{3.4}.

Another non-trivial aspect in the extension to the three-dimensional case concerns the introduction of evanescent plane waves in Definition \hyperlink{Definition 4.1}{4.1} and the related parametrization of the complex direction space $\{\mathbf{d} \in \C^3 : \mathbf{d} \cdot \mathbf{ d}=1\}$.
The idea behind their definition remains the same: given a complex-valued direction $\mathbf{d} \in \C^3$, the evanescent plane wave $e^{i\kappa \mathbf{d}\cdot \mathbf{x}}$ oscillate with an apparent wavenumber larger than $\kappa$ in the direction of propagation $\Re\{\mathbf{d}\}$ while decaying exponentially in the direction $\Im\{\mathbf{d}\}$.
However, unlike the two-dimensional case \cite[Sec.\ 5]{parolin-huybrechs-moiola}, the `parameter complexification' procedure -- namely the parametrization of the complex direction space obtained by complexifying the spherical coordinate angles -- turns out to be less suitable for the analysis of the Herglotz density space, due to difficulties in exhibiting an explicit Hilbert basis.
Therefore, we opted to define a complex-valued reference direction and then consider its rotations in space through the orthogonal matrices associated with the Euler angles.

This choice is closely tied to the Jacobi--Anger identity introduced in Theorem \hyperlink{Theorem 4.6}{4.7}. While this result is easily generalized to the complex case in two dimensions, due to the results in \cite{nist}, this step is not so trivial in the 3D setting. This requires extending the definition of spherical harmonics in order to include complex-valued directions, which is achieved by in turn extending the Ferrers functions (\ref{legendre polynomials}) to the associated Legendre polynomials (\ref{legendre2 polynomials}). Additionally, some fundamental algebraic properties such as the identities presented in \cite[Eq.\ (2.30)]{colton-kress} and \cite[Eq.\ (2.46)]{colton-kress} must also be generalized.

The modal analysis is then made possible thanks to the introduction of Wigner matrices.
The specific properties of these matrices (see \cite{devanathan,quantumtheory}) play a crucial role in confirming the effectiveness of evanescent plane waves in approximating high Fourier modes of the Helmholtz solutions. Moreover, they are pivotal in the construction of the Herglotz density space, and in defining and studying its Hilbert basis presented in Definition \hyperlink{Definition 5.1}{5.1}.

These results set the foundation for extending the Herglotz transform \cite[Def.\ 6.6]{parolin-huybrechs-moiola} to the 3D setting. Consequently, Theorem \hyperlink{Theorem 5.1}{5.9} states that any Helmholtz solution in the unit ball can be uniquely represented as a continuous superposition of evanescent plane waves. Furthermore, the corresponding density is bounded in a suitable norm.
The presented integral representation can be regarded as a generalization of the standard Herglotz representation (see \cite[Eq.\ (1.27)]{colton-kress}) and exhibits a robust stability result that is not valid when considering only propagative plane waves.
Thanks to these findings, all the results achieved in \cite[Sec.\ 6.2]{parolin-huybrechs-moiola} within the continuous-frame setting, as well as the reproducing kernel property of the Herglotz density space, are inherited, making way for discretization strategies.

The presented numerical recipe reflects the one proposed in \cite[Sec.\ 7]{parolin-huybrechs-moiola}, which is inspired by the optimal sampling procedure for weighted least-squares discussed in \cite{Cohen_Migliorati}.
The main objective is to generate a distribution of sampling nodes in the parametric domain that can be used to reconstruct the Herglotz density and, hence, the related Helmholtz solution.
An explicit knowledge of a Hilbert basis for the parametric space is critical to the success of these procedures.
Additionally, in Definition \hyperlink{Definition 6.4}{6.4} we consider simple variants that incorporate the extremal point systems.
However, to ensure the stability of the method, some numerical approximations are required, both with regard to the probability density functions and the normalization coefficients of the Herglotz density space basis, which are crucial ingredients for the method to be effective.

Although the proposed numerical techniques exhibit the desired accuracy and bounded-coefficient stability experimentally, as long as sufficient oversampling and regularization are used, they still lack a full proof and rely on the conjecture presented in \cite[Conj.\ 7.1]{parolin-huybrechs-moiola}.
In three dimensions, similar to the 2D case, the conjecture is supported by the continuous-level results and various numerical experiments.
In comparison to conventional propagative plane wave schemes, the described numerical method provides greater accuracy near singularities by approximating the high Fourier modes that inevitably arise.
Additionally, it seems to retain the quasi-optimality property (see \cite[Sec.\ 8.4]{parolin-huybrechs-moiola}), meaning that the degree of freedom budget required to approximate the first $N$ modes scales linearly with $N$ for a fixed level of accuracy.

Furthermore, similarly to \cite[Sec.\ 8.5]{parolin-huybrechs-moiola}, the developed evanescent plane wave approximation sets are tested on different geometries, such as cubes and tetrahedrons, despite being based on the analysis of the unit ball.
The results show excellent approximation properties, indicating the promising potential of the suggested numerical approach for plane wave approximations and Trefftz methods.

\section*{Outline of the thesis}

In Chapter \hyperlink{Chapter 1}{1}, we review established results regarding the structure of the Helmholtz solution space in spherical domains. We introduce the so-called spherical waves, showing that they form a Hilbert basis, and we study the asymptotic behavior of their $H^1$-normalization coefficients.

Chapter \hyperlink{Chapter 2}{2} introduces the key concept of stable approximation (see Definition \hyperlink{Definition 2.1}{2.1}) and presents a sampling-based numerical scheme for computing approximations of Helmholtz solutions. This simple method relies on regularized Singular Value Decomposition and oversampling. Definition \hyperlink{Extremal points system}{2.2} presents the extremal systems of points, and Section \hyperlink{Section 2.3}{2.3} explains how to compute them, highlighting their usefulness in defining propagative plane waves approximation sets and in constructing sampling point sets on the spherical surface to be used within the numerical scheme. The accuracy of the solutions provided by this scheme is proven in Corollary \hyperlink{2.1}{2.5}, namely, as long as the approximation set has the stable-approximation property and an appropriate set of sampling points (such as the extremal point systems) has been chosen, accurate solutions can be computed numerically.

Chapter \hyperlink{Chapter 3}{3} shows that, despite the use of regularization techniques, propagative plane waves cannot provide stable approximations in the unit ball due to the exponential growth of the expansion coefficients (see Theorem \hyperlink{Theorem 3.1}{3.4}). Furthermore, we show that the Herglotz density associated with spherical waves is not uniformly bounded, indicating that discretizing the related integral representation fails to produce discrete representations with bounded coefficients. Finally, the instability of propagative plane wave sets is confirmed through numerical experiments.

Chapter \hyperlink{Chapter 4}{4} introduces the evanescent plane waves (see Definition \hyperlink{Definition 4.1}{4.1}) and some essential components of their modal analysis, including the generalized Jacobi--Anger identity for complex-valued directions in Theorem \hyperlink{Theorem 4.6}{4.7} and the Wigner matrices in Definition \hyperlink{Definition 4.5}{4.7}. Unlike the propagative case, the modal content of evanescent plane waves is able to cover high Fourier regimes.

Chapter \hyperlink{Chapter 5}{5} presents the Herglotz density space, described in Definition \hyperlink{Definition 5.1}{5.1}, and shows its close link with the Helmholtz solution space through the Jacobi--Anger identity. This connection leads to the definition of an integral operator, the Herglotz transform (see Definition \hyperlink{Definition 5.8}{5.8}), which provides a means to represent any Helmholtz solution in the unit ball as a continuous superposition of evanescent plane waves, as detailed in Theorem \hyperlink{Theorem 5.1}{5.9}. This representation is a generalization of the classical Herglotz representation and provides a stable and robust result that is not achievable with only propagative plane waves. Theorem \hyperlink{Theorem 5.13}{5.13} proves that evanescent plane waves are a continuous frame for the space of Helmholtz solutions, while Proposition \hyperlink{Proposition 5.1}{5.14} shows that the Herglotz density space has the reproducing kernel property. These are crucial features that pave the way for the development of practical numerical methods.

In Chapter \hyperlink{Chapter 6}{6}, we discuss a method for achieving stable numerical approximations of Helmholtz solutions in the unit ball using evanescent plane waves. The method's core relies on generating a distribution of sampling nodes in the parametric domain according to a probability density function. We propose several sampling strategies, including simple variants that incorporate the extremal point systems.

Chapter \hyperlink{Chapter 7}{7} presents several numerical experiments supporting the use of evanescent plane waves to approximate Helmholtz solutions\footnote{The MATLAB code used to generate the numerical results of this paper is available at\\ \url{https://github.com/Nicola-Galante/evanescent-plane-wave-approximation}.}. The results show that the discussed method can achieve the desired accuracy and stability properties in both spherical geometries and other convex domains.

Finally, we present the conclusions of our work as well as opportunities for future research.

\chapter{Helmholtz equation in spherical geometry}
\hypertarget{Chapter 1}{In} this first chapter, taking into account many of the results known in the literature related to the solutions of the Helmholtz equation (\ref{Helmholtz equation}) in a spherical domain, we introduce the notion of spherical waves, henceforth indicated by $b_{\ell}^m$, and hence the space $\mathcal{B}$ generated by them. We then present some lemmas that show how $\mathcal{B}$, equipped with a suitable norm, is indeed a Hilbert space, of which the spherical waves constitute an orthonormal basis, and that this space coincides with the Helmholtz solution space. Lastly, the exponential growth of the normalization coefficients of the spherical waves is presented.

\section{Spherical waves}
\hypertarget{Section 1.1}{With} the intention of extending the work done in \cite{parolin-huybrechs-moiola} for the two-dimensional case, in this paper we consider three-dimensional geometries. In particular, we take into account only the simple case of spherical domains:
this approach allows for modal analysis through the separation of variables.
In fact, as it will become clear later on, spherical waves constitute an orthonormal basis and are bounded solutions of the Helmholtz equation (\ref{Helmholtz equation}) in the unit ball that are separable in spherical coordinates.
Therefore, up to rescaling of the wavenumber $\kappa$ and without loss of generality, we assume that the domain is the open unit ball, hereinafter denoted by $B_1:=\{\mathbf{x} \in \R^3: |\mathbf{x}| <1\}$.
We also introduce the notation $\mathbb{S}^2:=\{\mathbf{x} \in \R^3: |\mathbf{x}| =1\}$, whereas we use $\partial B_1$ to stress when the sphere is being used as the boundary of $B_1$ (in view of applications to domains other than $B_1$).
We also point out that spherical waves are widely used in many Trefftz schemes, see \cite{hiptmair-moiola-perugia1} and related references.

To begin our discussion, we first briefly get through some special functions that will be useful for defining spherical waves. Among the various conventions regarding the following definitions, we choose to rely on \cite{nist}. 
For notational convenience, we define the set of indices
\begin{equation*}
\mathcal{I}:=\{(\ell,m) \in \Z^2:  0 \leq |m| \leq \ell\}.
\end{equation*}
Conforming to \cite[Eq.\ (10.47.3)]{nist}, for every $\ell \geq 0$, the \textit{spherical Bessel function of the first kind} are solutions to the \textit{spherical Bessel equation}
\begin{equation}
r^2\frac{\textup{d}^2y}{\textup{d}r^2}+2r\frac{\textup{d}y}{\textup{d}r}+(r^2-\ell(\ell+1))y=0,
\label{spherical bessel equation}
\end{equation}
and are defined as
\begin{equation}
j_{\ell}(r):=\sqrt{\frac{\pi}{2r}}J_{\ell+\frac{1}{2}}(r),\,\,\,\,\,\,\,\,\,\,\,\,\,\,r>0,
\label{spherical bessel function}
\end{equation}
where $J_{\nu}(r)$ are the usual Bessel functions of the first kind (see \cite[Eq.\ (10.2.2)]{nist}). The spherical Bessel functions oscillate and decay as $r \rightarrow \infty$ and are bounded for any $r \geq 0$.

Following \cite[Eqs.\ (14.7.10) and (14.9.3)]{nist}, for every $(\ell,m) \in \mathcal{I}$, the \textit{Ferrers functions} (also known as \textit{Ferrers functions of the first kind} or as \textit{associated Legendre polynomials}) are solutions to the \textit{general Legendre equation}
\begin{equation}
\frac{\textup{d}}{\textup{d}x}\left(\left(1-x^2\right)\frac{\textup{d}y}{\textup{d}x} \right)+\left(\ell(\ell+1)-\frac{m^2}{1-x^2}\right)y=0,
\label{general legendre equation}
\end{equation}
and are defined as
\begin{equation}
\mathsf{P}_{\ell}^m(x):=\frac{(-1)^m}{2^{\ell}\ell!}(1-x^2)^{m/2}\frac{\textup{d}^{\ell+m}}{\textup{d}x^{\ell+m}}(x^2-1)^{\ell},\,\,\,\,\,\,\,\,\,\,\,\,\,\,|x|\leq 1,
\label{legendre polynomials}
\end{equation}
so that
\begin{equation}
\mathsf{P}_{\ell}^{-m}(x)=(-1)^m\frac{(\ell-m)!}{(\ell+m)!}\mathsf{P}_{\ell}^{m}(x),\,\,\,\,\,\,\,\,\,\,\,\,\,\,|x|\leq 1.
\label{negative legendre polynomials}
\end{equation}
In particular, $\mathsf{P}_{\ell}^m$ is called \textit{Ferrers function of degree $\ell$ and order $m$}, while, if $m=0$, it is simply called \textit{Legendre polynomial of degree $\ell$}. Among the numerous orthogonality relations known in the literature, for our discussion it is important to recall that the Ferrers functions are orthogonal for fixed order $m$ (see \cite[Eq.\ (14.17.6)]{nist}), namely
\begin{equation}
\int_{-1}^1\mathsf{P}_{\ell}^m(x)\mathsf{P}_q^m(x)\,\textup{d}x=\frac{\delta_{\ell, q}}{2\pi(\gamma_{\ell}^m)^2},
\label{P_l^m orthogonality}
\end{equation}
where
\begin{equation}
\gamma_{\ell}^m:=\left[\frac{2\ell+1}{4\pi} \frac{(\ell-m)!}{(\ell+m)!} \right]^{1/2}.
\label{gamma constant}
\end{equation}
\medskip

According to \cite[Eq.\ (14.30.1)]{nist}, for every $\mathbf{\hat{x}}=(\sin \theta_1 \cos \theta_2, \sin \theta_1 \sin \theta_2,\cos \theta_1) \in \mathbb{S}^2$, where $\theta_1 \in [0,\pi]$, $\theta_2 \in [0,2\pi)$, and for every $(\ell,m) \in \mathcal{I}$, the \textit{spherical harmonic function of degree $\ell$ and order $m$} is defined, with a little abuse of notation, as
\begin{equation}
Y_{\ell}^m(\mathbf{\hat{x}})=Y_{\ell}^m(\theta_1,\theta_2):=\gamma_{\ell}^m e^{im\theta_2}\mathsf{P}_{\ell}^m(\cos{\theta_1}),
\label{spherical harmonic}
\end{equation}
where the factor $\gamma_{\ell}^m$ acts as a normalization constant, i.e.\ it is such that $\|Y_{\ell}^m\|_{L^2(\mathbb{S}^2)}=1$.
For further details regarding these definitions, see for instance \cite{colton-kress,nedelec}. To avoid any confusion, note that we use the \textit{Condon-Shortley convention}, that is the phase factor of $(-1)^m$, in the definition of the Ferrers functions (\ref{legendre polynomials}) rather than in (\ref{gamma constant}).

The spherical harmonics $Y_{\ell}^m$ give rise to the \textit{solid harmonics} by extending from $\mathbb{S}^2$ to all $\R^3$ as a homogeneous polynomial of degree $\ell$, namely setting $R_{\ell}^{\,m}(\mathbf{x}):=|\mathbf{x}|^{\ell} Y_{\ell}^m \left(\mathbf{\hat{x}} \right)$, where $\mathbf{\hat{x}}:=\mathbf{x}/|\mathbf{x}|$, and $\{R_{\ell}^{\,m}\}_{(\ell,m) \in \mathcal{I}}$ turns out to be a basis of $\mathcal{H}_\ell$, the space of harmonic and homogeneous polynomials of degree $\ell$.

For every $\ell\geq0$, due to the orthogonality of the complex exponential family $\{\varphi \mapsto e^{im\varphi}\}_{|m|\leq \ell}$ in $L^2(0,2\pi)$ and denoting by $\mathcal{Y}_{\ell}$ the restriction to the unit sphere $\mathbb{S}^2$ of polynomials in $\mathcal{H}_{\ell}$,  $\{Y_{\ell}^m\}_{|m|\leq \ell}$ is an orthonormal basis of $\mathcal{Y}_{\ell}$ for the Hermitian product of $L^2(\mathbb{S}^2)$.

Moreover, thanks (\ref{P_l^m orthogonality}), the functions $\{Y_{\ell}^m\}_{(\ell,m) \in \mathcal{I}}$ constitute an orthonormal basis in $L^2(\mathbb{S}^2)$:
\begin{align*}
\left(Y_{\ell}^m,Y_{q}^n\right)_{L^2(\mathbb{S}^2)}&=\int_{0}^{2\pi}\int_0^\pi Y_{\ell}^m(\theta_1,\theta_2) \overline{Y_q^n(\theta_1,\theta_2)}\sin \theta_1 \textup{d}\theta_1 \textup{d}\theta_2 \\
&=\gamma_{\ell}^m \gamma_q^n\int_0^{2\pi}e^{i\theta_2(m-n)}\textup{d}\theta_2 \int_{-1}^{1}\mathsf{P}_{\ell}^m(x)\mathsf{P}_q^n(x)\textup{d}x\\
&=2\pi\delta_{m,n}\gamma_{\ell}^m\gamma_q^m\int_{-1}^{1}\mathsf{P}_{\ell}^m(x)\mathsf{P}_q^m(x)\textup{d}x=\delta_{\ell, q}\delta_{m,n}. \numberthis \label{Y_l^m orthogonality}
\end{align*}
Furthermore, for every $\ell\geq0$, $\mathcal{Y}_{\ell}$ is exactly the eigenspace of the \textit{Laplace--Beltrami operator} $\Delta_{\mathbb{S}^2}$ on the unit sphere associated with the eigenvalue $-\ell(\ell+1)$, i.e.\
\begin{equation}
\Delta_{\mathbb{S}^2} Y_{\ell}^m+\ell(\ell+1)Y_{\ell}^m=0,\,\,\,\,\,\,\,0 \leq |m|\leq \ell,
\label{Beltrami eigenvectors}
\end{equation}
where the eigenvalue $-\ell(\ell+1)$ has multiplicity $2\ell+1=\text{dim}(\mathcal{Y}_{\ell})$ \cite[Th.\ 2.4.1]{nedelec}. In this regard, we recall that, for any sufficiently smooth function $u:\mathbb{S}^2 \rightarrow \C$, the Laplace--Beltrami operator $\Delta_{\mathbb{S}^2}$ is defined by:
\begin{equation}
\Delta_{\mathbb{S}^2}u:=\frac{1}{\sin^2{\theta_1}}\frac{\partial^2 u}{\partial \theta_2^2}+\frac{1}{\sin{\theta_1}}\frac{\partial}{\partial \theta_1}\left(\sin{\theta_1}\frac{\partial u}{\partial \theta_1} \right).
\label{Beltrami definition}
\end{equation}

We are now ready to give the next definition. First of all, let us introduce the following  $\kappa$-dependent Hermitian product and associated norm: for any $u,v \in H^1(B_1)$,
\begin{equation}
\begin{split} \label{B norm}
(u,v)_{\mathcal{B}}&:=(u,v)_{L^2(B_1)}+\kappa^{-2}(\nabla u, \nabla v)_{L^2(B_1)^2},\\
\|u\|^2_{\mathcal{B}}&:=(u,u)_{\mathcal{B}}.  
\end{split}
\end{equation}
\medskip

\begin{definition}[Spherical waves]
We define, for any $(\ell,m) \in \mathcal{I}$
\begin{equation}
\begin{split} \label{b tilde definizione}
\tilde{b}_\ell^m(\mathbf{x})&:=j_\ell(\kappa |\mathbf{x}|)Y_\ell^m(\mathbf{\hat{x}}),\,\,\,\,\,\,\,\,\,\,\,\,\,\,\,\,\,\,\,\,\,\,\,\,  \forall \mathbf{x} \in B_1,\\
b_\ell^m&:=\beta_{\ell}\tilde{b}_\ell^m,\,\,\,\,\,\,\,\,\,\,\,\,\,\,\,\,\,\,\,\,\,\,\,\,\,\,\,\,\,\,\,\,\,\,\,\,\,\,\,\,\,\,\, \beta_\ell:=\|\tilde{b}_\ell^m\|^{-1}_\mathcal{B}, 
\end{split}
\end{equation}
where $\mathbf{\hat{x}}:=\mathbf{x}/|\mathbf{x}|$. Furthermore, we introduce the space
\begin{equation*}
\mathcal{B}:=\overline{\textup{span}\{b_\ell^m\}_{(\ell,m) \in \mathcal{I}}}^{\|\cdot\|_{\mathcal{B}}} \subsetneq H^{1}(B_1).
\end{equation*}
\end{definition}
Observe that, contrary to the two-dimensional case, the spherical waves $b_\ell^m \in \mathcal{B}$ depend on two different parameters $(\ell,m) \in \mathcal{I}$, but their norm $\beta_{\ell}$ is independent of the parameter $m$, as will be clear later on (see Lemma \hyperlink{Lemma 1.3}{1.4}). However, similarly to \cite{parolin-huybrechs-moiola}, we will refer to spherical waves with mode number $\ell < \kappa$ as \textit{propagative} modes (the `energy' of such modes is distributed in the bulk of the domain), for $\ell \gg \kappa$ the spherical waves are termed \textit{evanescent} (their `energy' is concentrated near the boundary of the domain) and lastly, in between, the waves such that $\ell \approx \kappa$ are called \textit{grazing} modes. Figure \ref{figure 1.1} shows the behavior of the functions $b_{\ell}^m$ on the boundary of the unit ball without the first octant $\partial \{B_1\setminus\{\mathbf{x}=(x,y,z) : x>0,y>0,z>0\}\}$ for different values of $\ell$.
\begin{figure}
\begin{subfigure}{.3\textwidth}
\centering
\includegraphics[trim=100 100 100 100,clip,width=4.5cm,height=4.5cm]{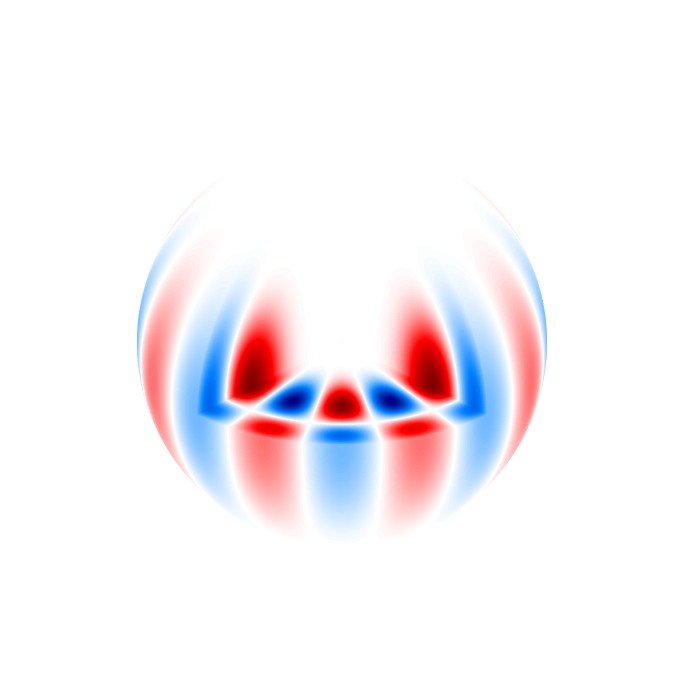}
\caption{Propagative $\ell=8$.}
\end{subfigure}\hfill
\begin{subfigure}{.3\textwidth}
\centering
\includegraphics[trim=100 100 100 100,clip,width=4.5cm,height=4.5cm]{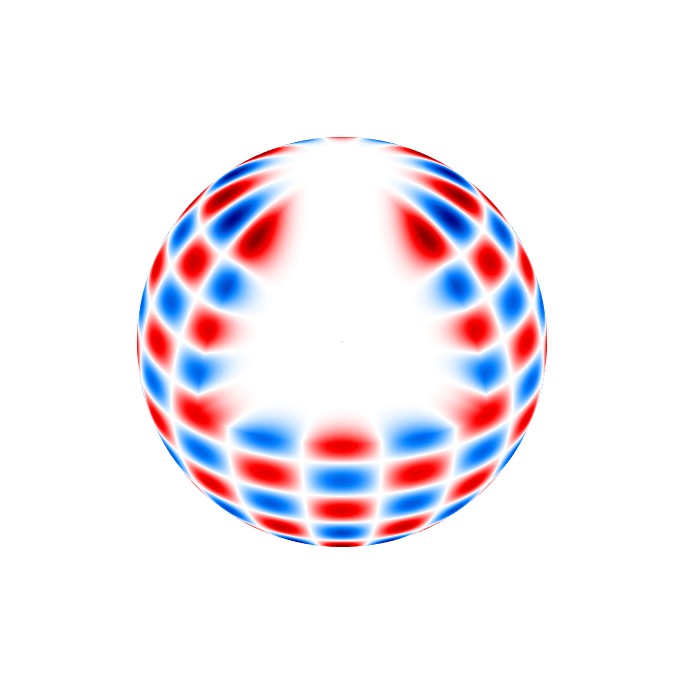}
\caption{Grazing $\ell=16$.}
\end{subfigure}\hfill
\begin{subfigure}{.3\textwidth}
\centering
\includegraphics[trim=100 100 100 100,clip,width=4.5cm,height=4.5cm]{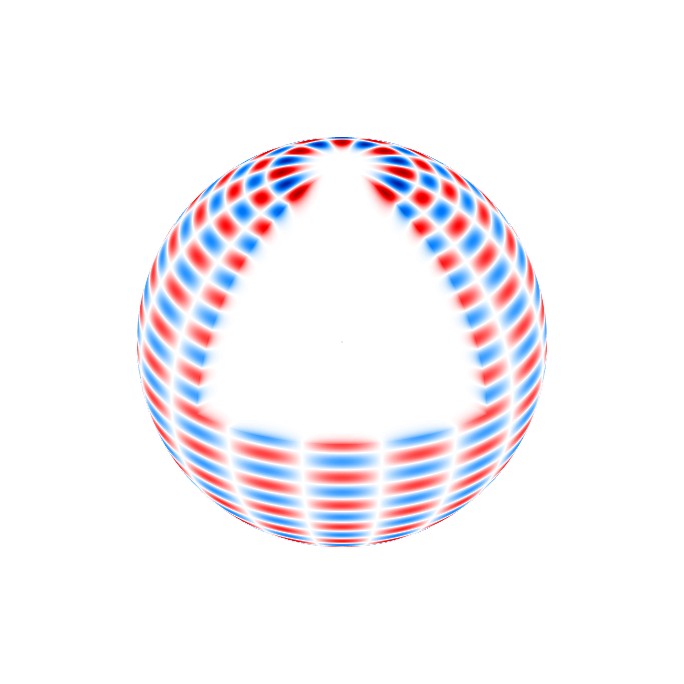}
\caption{Evanescent $\ell=32$.}
\end{subfigure}
\caption{Real part of some spherical waves $b_{\ell}^m$ on the boundary of the unit ball without the first octant $\partial \{B_1\setminus\{\mathbf{x}=(x,y,z) : x>0,y>0,z>0\}\}$ for $m=8$ and wavenumber $\kappa=16$. As the order $\ell$ increases, observe how the support of these functions becomes more and more concentrated near the boundary $\partial B_1$.}
\label{figure 1.1}
\end{figure}

Thanks to (\ref{Beltrami eigenvectors}) and (\ref{Beltrami definition}) is now straightforward to check that $b_{\ell}^m$ in (\ref{b tilde definizione}) is solution to the Helmholtz equation (\ref{Helmholtz equation}) for $(\ell,m) \in \mathcal{I}$: in fact, since $j_{\ell}$ is solution to (\ref{spherical bessel equation}), using spherical coordinates and due to the separation of variables in $\tilde{b}_{\ell}^m$, it follows
\begin{equation}
\begin{split} \label{blm is solution}
\Delta \tilde{b}_{\ell}^m +\kappa^2 \tilde{b}_{\ell}^m &=\frac{\kappa^2}{r^2}\frac{\textup{d}}{\textup{d}r}\left(r^2 \frac{\textup{d}j_{\ell}}{\textup{d}r}\right)Y_{\ell}^m+\frac{\kappa^2 j_{\ell}}{r^2}\Delta_{\mathbb{S}^2}Y_{\ell}^m+\kappa^2 j_{\ell} Y_{\ell}^m\\
&=\frac{\kappa^2 Y_{\ell}^m}{r^2}\left[\frac{\textup{d}}{\textup{d}r}\left(r^2 \frac{\textup{d} j_{\ell}}{\textup{d}r}\right)+(r^2-\ell(\ell+1))j_{\ell}\right]=0.
\end{split}
\end{equation}
We now present some lemmas, useful for setting up what follows, similarly to what is done in \cite{parolin-huybrechs-moiola}.
\begin{lemma}
\hypertarget{Lemma 1.1}{The} space $(\mathcal{B},\|\cdot\|_{\mathcal{B}})$ is a Hilbert space and the family $\{b_{\ell}^m\}_{(\ell,m) \in \mathcal{I}}$ is a Hilbert basis (i.e an orthonormal basis):
\begin{equation*}
(b_{\ell}^m,b_q^n)_{\mathcal{B}}=\delta_{\ell, q}\delta_{m,n},\,\,\,\,\,\,\,\,\,\,\,\,\,\,\forall\, (\ell,m),(q,n) \in \mathcal{I},
\end{equation*}
and
\begin{equation*}
u=\sum_{\ell=0}^{\infty}\sum_{m=-\ell}^{\ell}(u,b_{\ell}^m)_{\mathcal{B}}\,b_{\ell}^m,\,\,\,\,\,\,\,\,\,\,\,\,\,\,\forall u \in \mathcal{B}.
\end{equation*}
\end{lemma}
\begin{proof}
Thanks to how we defined the spherical waves in (\ref{b tilde definizione}), it is enough to prove that the family $\{\tilde{b}_{\ell}^m\}_{(\ell,m) \in \mathcal{I}}$ is orthogonal, which is a consequence of (\ref{Y_l^m orthogonality}). For $(\ell,m),(q,n) \in \mathcal{I}$ we have:
\begin{align*}
(\tilde{b}_{\ell}^m,\tilde{b}_q^n)_{L^2(B_1)}&=\int_0^1j_{\ell}(\kappa r)j_q(\kappa r)r^2\textup{d}r \int_{0}^{2\pi}\int_0^\pi Y_{\ell}^m(\theta_1,\theta_2) \overline{Y_q^n(\theta_1,\theta_2)}\sin \theta_1 \textup{d}\theta_1 \textup{d}\theta_2 \\
&=\int_0^1 j_{\ell}^2(\kappa r)r^2\textup{d}r\,\delta_{\ell, q}\delta_{m,n}, \numberthis \label{L2 norm tilde b}
\end{align*}
and, denoting by $\mathbf{n}$ the outward unit normal vector,
\begin{align*}
(\partial_{\mathbf{n}}\tilde{b}_{\ell}^m,\tilde{b}_q^n)_{L^2(\partial B_1)}&=\kappa j'_{\ell}(\kappa)j_q(\kappa)\int_{0}^{2\pi}\int_0^\pi Y_{\ell}^m(\theta_1,\theta_2) \overline{Y_q^n(\theta_1,\theta_2)}\sin \theta_1 \textup{d}\theta_1 \textup{d}\theta_2 \\
&=\kappa j'_{\ell}(\kappa)j_{\ell}(\kappa)\,\delta_{\ell, q}\delta_{m,n}. \numberthis \label{L2 boundary}
\end{align*}
The orthogonality with respect to the Hermitian product $(\cdot,\cdot)_{\mathcal{B}}$ can be easily seen from
\begin{align*}
(\nabla \tilde{b}_{\ell}^m,\nabla \tilde{b}_q^n)_{L^2(B_1)^2}&=(-\Delta \tilde{b}_{\ell}^m,\tilde{b}_q^n)_{L^2(B_1)}+(\partial_{\mathbf{n}}\tilde{b}_{\ell}^m,\tilde{b}_q^n)_{L^2(\partial B_1)}\\
&=\kappa^2(\tilde{b}_{\ell}^m,\tilde{b}_q^n)_{L^2(B_1)}+(\partial_{\mathbf{n}}\tilde{b}_{\ell}^m,\tilde{b}_q^n)_{L^2(\partial B_1)}. \numberthis \label{decomposition norm}
\end{align*}
\end{proof}
In particular, observe that, thanks to (\ref{decomposition norm}), assuming $u \in H^1(B_1)$ is a solution to the Helmholtz equation (\ref{Helmholtz equation}), for any $v \in H^1(B_1)$ we can rewrite the Hermitian product and associated norm in (\ref{B norm}) as
\begin{equation}
\begin{split} \label{B norm 2}
(u,v)_{\mathcal{B}}&:=2(u,v)_{L^2(B_1)}+\kappa^{-2}(\partial_{\mathbf{n}}u,v)_{L^2(\partial B_1)},\\
\|u\|^2_{\mathcal{B}}&:=2\|u\|^2_{L^2(B_1)}+\kappa^{-2}(\partial_{\mathbf{n}}u,u)_{L^2(\partial B_1)}.  
\end{split}
\end{equation}
The main reason for introducing spherical waves is the possibility to use them to expand any Helmholtz solution in $B_1$, as we show in the next lemma.
\begin{lemma}
$u \in H^1(B_1)$ satisfies the Helmholtz equation if and only if $u \in \mathcal{B}$.
\end{lemma}
\begin{proof}
The continuity of the Helmholtz operator $\mathcal{L}:H^1(B_1)\rightarrow H^{-1}(B_1)$ defined by
\begin{equation*}
\langle \mathcal{L}u,v\rangle_{H^{-1}\times H_0^1}:=(\nabla u,\nabla v)_{L^2(B_1)^2}-\kappa^2(u,v)_{L^2(B_1)},\,\,\,\,\,\,\,\forall u \in H^1(B_1),\,\forall v \in H_0^1(B_1),
\end{equation*}
implies that the kernel of $\mathcal{L}$ is a closed subspace of $H^1(B_1)$. Thanks to (\ref{blm is solution}), it is easily checked that $\mathcal{B}\subseteq \ker \mathcal{L}$.

Conversely, let $u \in H^1(B_1)$ satisfy (\ref{Helmholtz equation}) and set $g:=\partial_{\mathbf{n}}u-i\kappa u \in H^{-1/2}(\partial B_1)$. The Robin trace $g$ can be written as (see \cite[Sec.\ 2.5.1]{nedelec}):
\begin{equation*}
g(\theta_1,\theta_2)=\sum_{\ell=0}^{\infty}\sum_{m=-\ell}^{\ell}\hat{g}_{\ell}^m Y_{\ell}^m(\theta_1,\theta_2),\,\,\,\,\,\,\,\,\,\,\,\,\forall \theta_1 \in [0,\pi],\,\forall \theta_2 \in [0,2\pi),
\end{equation*}
\begin{equation*}
\text{where}\,\,\,\,\,\,\,\sum_{\ell=0}^{\infty}\sum_{m=-\ell}^{\ell}|\hat{g}_{\ell}^m|^2(\ell+1)^{-1}<\infty.
\end{equation*}
Let $L\geq 0$ and set $g_L:=\sum_{\ell=0}^L\sum_{m=-\ell}^{\ell}\hat{g}_{\ell}^m Y_{\ell}^m$. Then there exists a unique $u_L \in \text{span}\{b_{\ell}^m\}_{(\ell,m) \in \mathcal{I}\,:\,\ell\, \leq\, L}$ such that $g_L=\partial_{\mathbf{n}}u_L-i\kappa u_L$, namely
\begin{equation*}
u_L:=\sum_{\ell=0}^L\sum_{m=-\ell}^{\ell}\hat{g}_{\ell}^m\left[\kappa \gamma_{\ell}^m \beta_{\ell} \big(j'_{\ell}(\kappa)-ij_{\ell}(\kappa)\big)\right]^{-1}b_{\ell}^m,
\end{equation*}
where the term $j'_{\ell}(\kappa)-ij_{\ell}(\kappa)$ at the denominator is non-zero because of \cite[Eqs.\ (10.21.3) and (10.58.1)]{nist}. The well-posedness of the problem: find $v \in H^1(B_1)$ such that
\begin{equation*}
\Delta v + \kappa^2v=0,\,\,\,\text{in}\,B_1,\,\,\,\,\,\,\,\text{and}\,\,\,\,\,\,\,\partial_{\mathbf{n}}v-i \kappa v=h\,\,\,\text{on}\,\partial B_1,
\end{equation*}
for $h \in H^{-1/2}(\partial B_1)$ \cite[Prop.\ 8.1.3]{melenk}, implies that there exists a constant $C>0$, independent of $L$, such that $\|u-u_L\|_{\mathcal{B}}\leq C \|g-g_L\|_{H^{-1/2}(\partial B_1)}$. Letting $L$ tend to infinity, in the end we obtain that $u \in \mathcal{B}$ and therefore $\ker \mathcal{L} = \mathcal{B}$.
\end{proof}

\section{Asymptotics of normalization coefficients}

\hypertarget{Section 1.2}{The} coefficients $\beta_{\ell}$ grow super-exponentially with $\ell$ and independently of $m$ after a pre-asymptotic regime up to $\ell \approx \kappa$ (see Figure \ref{figure betal}). The precise asymptotic behavior is given by the following lemma.

\begin{lemma}
\hypertarget{Lemma 1.3}{We} have for all $(\ell,m) \in \mathcal{I}$
\begin{equation}
\beta_{\ell}=\left(\frac{\pi}{2 \kappa}\left[\left(1+\frac{\ell}{\kappa^2}\right)J^2_{\ell+\frac{1}{2}}(\kappa)-\left(J_{\ell-\frac{1}{2}}(\kappa)+\frac{1}{\kappa}J_{\ell+\frac{1}{2}}(\kappa) \right)J_{\ell+\frac{3}{2}}(\kappa) \right] \right)^{-1/2},
\label{betal expansion}
\end{equation}
therefore $\|\tilde{b}_{\ell}^m\|_{\mathcal{B}}$ is independent of the value of $m$ and furthermore
\begin{equation}
\beta_{\ell} \sim 2\sqrt{2}\kappa\left(\frac{2}{e \kappa}\right)^{\ell}\ell^{\ell+\frac{1}{2}},\,\,\,\,\,\,\,\,\,\,\,\,\,\,\,\text{as}\,\,\,\ell \rightarrow \infty.
\label{beta_l asymptotic}
\end{equation}
\end{lemma}
\begin{proof}
From (\ref{spherical bessel function}), (\ref{L2 norm tilde b}), (\ref{L2 boundary}) and using \cite[Eqs.\ (10.22.5) and (10.51.2)]{nist}, it follows:
\begin{align*}
\|\tilde{b}_{\ell}^m\|^2_{L^2(B_1)}&=\int_0^1j^2_{\ell}(\kappa r)r^2\textup{d}r=\frac{\pi}{2\kappa}\int_0^1J^2_{\ell+\frac{1}{2}}(\kappa r)r\textup{d}r\\
&=\frac{\pi}{4\kappa}\left(J^2_{\ell+\frac{1}{2}}(\kappa)-J_{\ell-\frac{1}{2}}(\kappa)J_{\ell+\frac{3}{2}}(\kappa) \right), \numberthis \label{L2}
\end{align*}
\begin{align*}
(\partial_{\mathbf{n}}\tilde{b}_{\ell}^m,\tilde{b}_{\ell}^m)_{L^2(\partial B_1)}&=\kappa j'_{\ell}(\kappa)j_{\ell}(\kappa)=\kappa \left(\frac{\ell}{\kappa}j_{\ell}(\kappa)-j_{\ell+1}(\kappa) \right)j_{\ell}(\kappa)\\
&=\ell j^2_{\ell}(\kappa)-\kappa j_{\ell}(\kappa)j_{\ell+1}(\kappa)\\
&=\frac{\pi}{2 \kappa} \left( \ell J^2_{\ell+\frac{1}{2}}(\kappa)-\kappa J_{\ell+\frac{1}{2}}(\kappa)J_{\ell+\frac{3}{2}}(\kappa)\right), \numberthis \label{partial}
\end{align*}
Therefore, (\ref{betal expansion}) follows directly from (\ref{B norm 2}) and  $\|\tilde{b}_{\ell}^m\|_{\mathcal{B}}$ is independent of the value of $m$. The proof of the asymptotic behavior consists in showing that we have:
\begin{align*}
&\|\tilde{b}_{\ell}^m\|_{L^2(\partial B_1)} \sim \frac{1}{2\sqrt{2}}\left(\frac{e \kappa}{2}\right)^{\ell}\ell^{-(\ell+1)},\\
&\|\tilde{b}_{\ell}^m\|_{L^2( B_1)} \sim \frac{1}{4}\left(\frac{e \kappa}{2}\right)^{\ell} \ell^{-\left(\ell+\frac{3}{2} \right)},\,\,\,\,\,\,\,\,\,\,\,\,\,\,\,\,\,\,\,\,\text{as}\,\,\,\ell \rightarrow \infty. \numberthis \label{asymptotic behavior}\\
&\|\tilde{b}_{\ell}^m\|_{\mathcal{B}} \sim \frac{1}{2\sqrt{2}\kappa}\left(\frac{e \kappa}{2}\right)^{\ell}\ell^{-\left(\ell+\frac{1}{2} \right)},
\end{align*}
Note that, even if it is not necessary for the purposes of the proof, we also study the behavior of the trace norm, since this result will be useful later on to prove the error bound (\ref{error bound 2}). From the definition (\ref{b tilde definizione}) of $\tilde{b}_{\ell}^m$, we immediately have $\|\tilde{b}_{\ell}^m\|^2_{L^2(\partial B_1)}=j^2_{\ell}(\kappa)$.
Thanks to (\ref{spherical bessel function}) and \cite[Eq.\ (10.19.1)]{nist}, namely
\begin{equation}
J_{\nu}(r) \sim \frac{1}{\sqrt{2\pi \nu}}\left(\frac{er}{2 \nu} \right)^{\nu},\,\,\,\,\,\,\,\,\,\,\,\,\,\,\text{as}\,\,\,\nu \rightarrow \infty,
\label{bessel asymptotics}
\end{equation}
we get
\begin{equation}
j_{\ell}(r)\sim\frac{1}{2\sqrt{r}}\left(\frac{er}{2}\right)^{\ell+\frac{1}{2}}\left(\ell+\frac{1}{2}\right)^{-(\ell+1)},\,\,\,\,\,\,\,\,\,\,\,\,\,\,\text{as}\,\,\,\ell \rightarrow \infty,
\label{spherical bessel asymptotics}
\end{equation}
and therefore as $\ell \rightarrow \infty$
\begin{equation}
\|\tilde{b}_{\ell}^m\|^2_{L^2(\partial B_1)} \sim \frac{1}{4 \kappa} \left(\frac{e \kappa}{2} \right)^{2\ell+1}\left(\ell+\frac{1}{2} \right)^{-2(\ell+1)}.
\label{b tilde boundary}
\end{equation}
Since, for every $x,y,z \in \R$, we have that
\begin{equation}
(\ell+x)^{y\ell+z}\sim \ell^{y\ell+z}\exp\bigg\{{xy+\frac{x(2z-xy)}{2\ell}}\bigg\}\sim \ell^{y\ell+z}e^{xy},\,\,\,\,\,\,\,\,\,\,\,\,\,\,\text{as}\,\,\,\ell \rightarrow \infty,
\label{ell asymptotics}
\end{equation}
then $\left(\ell+1/2 \right)^{-2(\ell+1)}$ is equivalent to ${\ell}^{-2(\ell+1)}e^{-1}$ at infinity; therefore, the first result in (\ref{asymptotic behavior}) follows directly from (\ref{b tilde boundary}) and (\ref{ell asymptotics}).

We now consider the $L^2(B_1)$ norm. From (\ref{L2}) and (\ref{bessel asymptotics}), we get as $\ell \rightarrow \infty$
\begin{equation*}
\|\tilde{b}_{\ell}^m\|^2_{L^2( B_1)} \sim \frac{1}{8 \kappa}\left(\frac{e \kappa}{2} \right)^{2\ell+1}\left(\ell+\frac{1}{2} \right)^{-2(\ell+1)}\Bigg[ 1- \frac{\left(\ell+\frac{1}{2} \right)^{2(\ell+1)}}{\left(\ell-\frac{1}{2} \right)^{\ell}\left(\ell+\frac{3}{2} \right)^{\ell+2}} \Bigg],
\end{equation*}
and, thanks to (\ref{ell asymptotics}), it is easily checked that the term inside the square brackets is equivalent to $\ell^{-1}$ at infinity.

We now consider the $\kappa$-weighted $H^1(B_1)$ norm (\ref{B norm 2}): we need to study the asymptotics of the boundary term. From (\ref{partial}) and (\ref{bessel asymptotics}), we get as $\ell \rightarrow \infty$
\begin{equation}
(\partial_{\mathbf{n}}\tilde{b}_{\ell}^m,\tilde{b}_{\ell}^m)_{L^2(\partial B_1)} \sim \frac{1}{4 \kappa}\left(\frac{e \kappa}{2} \right)^{2\ell+1}\left(\ell+\frac{1}{2} \right)^{-2(\ell+1)}\left[\ell- \frac{e \kappa^2}{2} \frac{\left(\ell+ \frac{1}{2} \right)^{\ell+1}}{\left(\ell+ \frac{3}{2} \right)^{\ell+2}}\right],
\label{www}
\end{equation}
and, thanks to (\ref{ell asymptotics}), it is readily checked that the second term inside the square brackets is dominated by the first one, since it is equivalent to $\kappa^2/2\ell$ at infinity. Thus, the dominant term in $\|\tilde{b}_{\ell}^m\|_{\mathcal{B}}$ in the limit $\ell \rightarrow \infty$ is the boundary term (\ref{www}).
\end{proof}

\begin{figure}
\centering
\includegraphics[width=10cm]{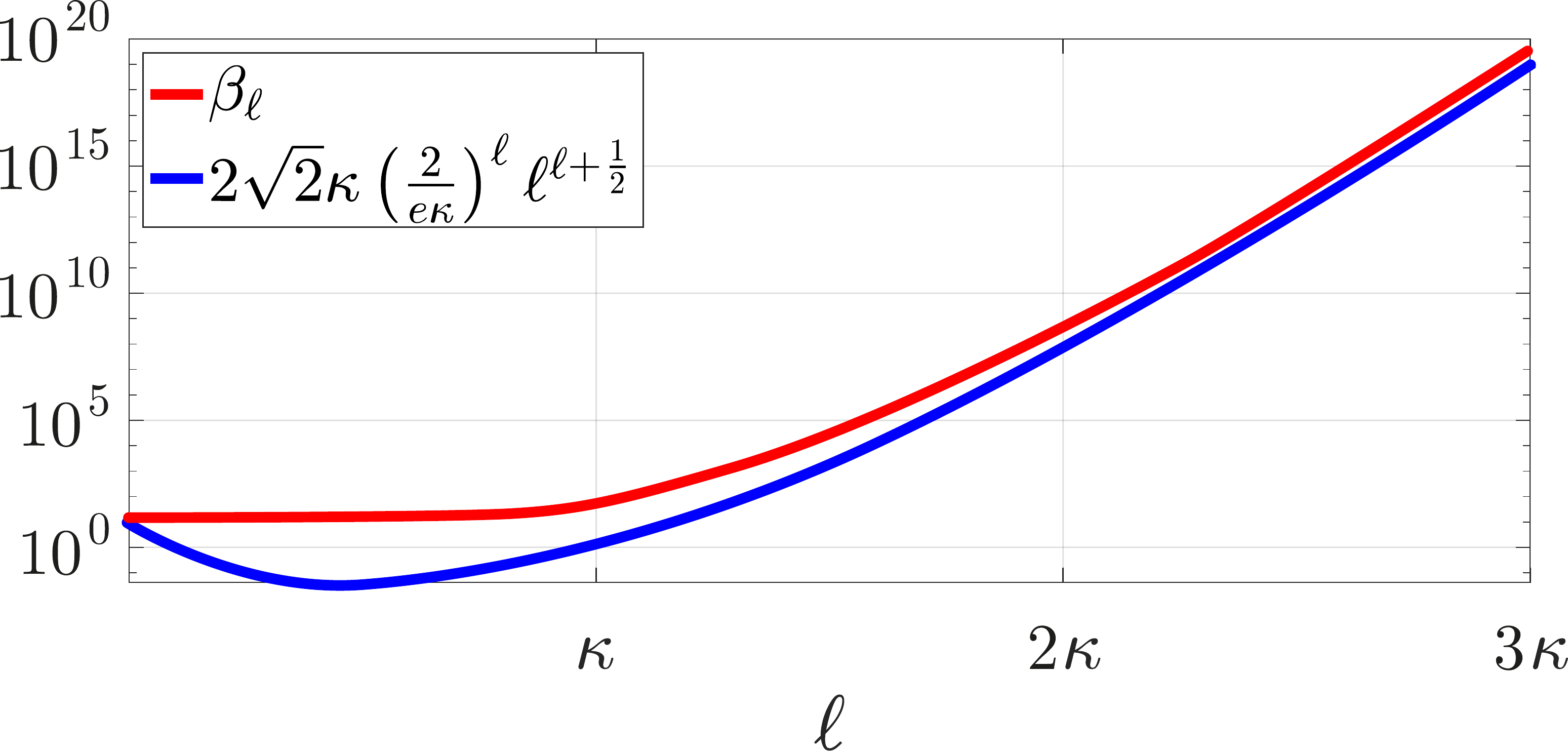}
\caption{Super-exponential growth of coefficients $\beta_{\ell}=\|\tilde{b}_{\ell}^m\|_{\mathcal{B}}^{-1}$. Note the pre-asymptotic regime up to $\ell \approx \kappa$. Wavenumber $\kappa=16$.}
\label{figure betal}
\end{figure}

\begin{remark}
We chose to normalize the spherical waves using the natural norm $\| \cdot \|_{\mathcal{B}}$, but obviously other choices are possible with some minor differences in the coefficients of the exponential growth of $\beta_{\ell}$ (see \textup{\cite[Remark 2.5]{parolin-huybrechs-moiola}} for the two-dimensional case).
\end{remark}

\chapter{Stable numerical approximation}
\hypertarget{Chapter 2}{In} this chapter, we briefly review and generalize to the 3D case the results obtained in \cite[Sec.\ 3]{parolin-huybrechs-moiola}, whose central aspect consists in the notion of stable approximation, crucial in computing numerical approximations in the form of series expansion.
The main underlying idea is to develop a numerical method that provides accurate approximations by using small coefficients in the expansion, thereby ensuring stability under the limitations of finite precision in computer arithmetic.
The proposed approach builds on the results in \cite{huybrechs1,huybrechs2}. We also describe a practical sampling-based numerical scheme to compute approximations of Helmholtz solutions that rely on regularized Singular Value Decomposition and oversampling.
We show that this procedure yields accurate solutions, provided that the approximation set has the stable approximation property and an appropriate choice of sampling points has been made.
In this regard, we introduce extremal point systems as a potential example. These sets of points have interesting geometrical and integration properties that make them useful throughout the paper for constructing plane wave approximation sets, as we will see in the following chapters.

\section{The notion of stable approximation}

\hypertarget{Section 2.1}{Let} us consider a \textit{sequence of approximation sets} in $\mathcal{B}$
\begin{equation}
\mathbf{\Phi}:=\{\mathbf{\Phi}_k\}_{k \in \N}\,\,\,\,\,\,\,\,\text{where}\,\,\,\,\,\,\,\,\mathbf{\Phi}_k:=\{\phi_{k,p}\}_p,\,\,\,\,\forall k \in \N,
\label{approximation sets}
\end{equation}
where, for each $k,p$, $\phi_{k,p} \in \mathcal{B}$ is a solution of the Helmholtz equation (\ref{Helmholtz equation}) in $B_1$ and $|\mathbf{\Phi}_k| < \infty$; these sets do not need to be nested. Associated to any set $\mathbf{\Phi}_k$ for some $k \in \N$, we define the following operator:
\begin{equation}
\mathcal{T}_{\mathbf{\Phi}_k}:\C^{|\mathbf{\Phi}_k|} \rightarrow \mathcal{B},\,\,\,\,\,\,\,\text{such that}\,\,\,\,\,\,\,\boldsymbol{\mu}=\{\mu_p\}_p \mapsto \sum_p \mu_p \phi_{k,p}.
\label{synthesis operator}
\end{equation}
In frame theory, this is often called \textit{synthesis operator}.

\begin{definition}[\hypertarget{Definition 2.1}{Stable approximation}]
The sequence $\mathbf{\Phi}$ of approximation sets \textup{(\ref{approximation sets})} is said to be a stable approximation for $\mathcal{B}$ if, for any tolerance $\eta >0$, there exist a stability exponent $\lambda \geq 0$ and a stability constant $C_{\textup{stb}} \geq 0$ such that
\begin{equation}
\forall u \in \mathcal{B}, \exists k \in \N, \boldsymbol{\mu} \in \C^{|\mathbf{\Phi}_k|}\,\,\,\,\,\,\,\text{such that}\,\,\,\,\,\,\, \begin{cases}
\|u-\mathcal{T}_{\mathbf{\Phi}_k}\boldsymbol{\mu}\|_{\mathcal{B}} \leq \eta \|u\|_{\mathcal{B}} & \text{and}\\
\|\boldsymbol{\mu}\|_{\ell^2} \leq C_{\textup{stb}}|\mathbf{\Phi}_k|^{\lambda}\|u\|_{\mathcal{B}}  & \text{.}
\end{cases}
\label{stable approximation}
\end{equation}
\end{definition}

A sequence of stable approximation sets allows for the representation of any Helmholtz solution as a finite expansion $\mathcal{T}_{\mathbf{\Phi}_k}\boldsymbol{\mu}$ with coefficients $\boldsymbol{\mu}$ having bounded $\ell^2$-norm  up to some algebraic growth. The stability exponent $\lambda$ determines the rate of increase of the coefficient norm, with a smaller $\lambda$ indicating a more stable sequence.
It is worth noting that using the $\ell^2$-norm in (\ref{stable approximation}) is not essential, as the growth of the coefficient norm can be measured using any discrete $\ell^p$-norm. This is possible due to the equivalence of these norms, as established by Hölder inequality, since we consider only finite dimensional sets.

Following \cite{parolin-huybrechs-moiola}, we will present two examples of approximation sets of the type (\ref{approximation sets}): propagative plane waves in (\ref{plane waves approximation set}) and evanescent plane waves in (\ref{evanescence sets}). As we will show in later chapters, these two choices have different stability properties. In Theorem \hyperlink{Theorem 3.1}{3.3}, we prove that propagative plane waves are unstable, whereas in Chapter \hyperlink{Chapter 7}{7} numerical experiments indicate that the evanescent plane wave sets are stable when built using the method described in Section \hyperlink{Section 6.1}{6.1} and in Section \hyperlink{Section 6.3}{6.3}.

\section{Regularized boundary sampling method}

\hypertarget{Section 2.2}{Let} us describe the method for computing the coefficients in practice.
We adopt for simplicity a sampling-type strategy, following \cite{huybrechs3} and in continuity with what was done in \cite{parolin-huybrechs-moiola}. Let us consider the Helmholtz problem with Dirichlet boundary conditions: find $u \in H^1(B_1)$ such that
\begin{equation*}
\Delta u+\kappa^2 u =0,\,\,\,\,\text{in}\,\,\,B_1,\,\,\,\,\,\,\,\,\,\,\text{and}\,\,\,\,\,\,\,\,\,\,\gamma u=g,\,\,\,\,\text{on}\,\,\,\partial B_1,
\end{equation*}
where $g \in H^{\frac{1}{2}}(\partial B_1)$ and $\gamma$ is the Dirichlet trace operator; this problem is known to be well-posed assuming that $\kappa^2$ is not an eigenvalue of the Dirichlet Laplacian.
In all of our numerical experiments, we aim at reconstructing a solution $u \in \mathcal{B}$ having access to its trace $\gamma u$ on the boundary. Thus, we assume for simplicity that $u \in \mathcal{B} \cap C^0(\overline{B_1})$, so as to allow us to consider point evaluations of the Dirichlet trace.

So let $u \in \mathcal{B} \cap C^0(\overline{B_1})$ be the target of our approximation problem. We look for a set of coefficients $\boldsymbol{\xi} \in \C^{|\mathbf{\Phi}_{k}|}$ for a given approximation set $\mathbf{\Phi}_{k}$ (introduced in (\ref{approximation sets})) such that $\mathcal{T}_{\mathbf{\Phi}_{k}}\boldsymbol{\xi} \approx u$. We also assume that, for every $p$, $\phi_{k,p} \in \mathcal{B} \cap C^0(\overline{B_1})$. It remains to understand how to choose the set of $S \geq |\mathbf{\Phi}_{k}|$ sampling points $\{\mathbf{x}_s\}_{s=1}^S$ on the unit sphere ($\mathbf{x}_s \in \partial B_1$ for $s=1,...,S$).
Observe that now, contrary to the two-dimensional case \cite[Eq.\ (3.6)]{parolin-huybrechs-moiola}, there is no obvious way to determine a set of equispaced points $\{\mathbf{x}_s\}_{s=1}^S$. This choice, as we will see in Section \hyperlink{Section 2.4}{2.4}, is basically aimed at ensuring the convergence of the cubature rule
\begin{equation}
\lim_{S \rightarrow \infty}\sum_{s=1}^S w_s v(\mathbf{x}_s)=\int_{\partial B_1}v(\mathbf{x})\textup{d}\sigma(\mathbf{x}),\,\,\,\,\,\,\,\,\,\,\forall v \in C^0(\partial B_1),
\label{Riemann sum}
\end{equation}
where $\mathbf{w}_S=(w_s)_s \in \R^S$ is a suitable vector of positive weights associated with the sampling point set $\{\mathbf{x}_s\}_{s=1}^S \subset \partial B_1$.
In this regard, in the following numerical experiments, as a particular choice, we consider the \textit{extremal systems of points} and associated weights (see \cite{marzo-cerda,reimer,sloan-womersley1}), which we describe in more detail in the next section.

Let us now introduce the matrix $A=(A_{s,p})_{s,p} \in \C^{S \times |\mathbf{\Phi}_{k}|}$ and the vector $\mathbf{b}=(b_s)_s \in \C^S$ such that
\begin{equation}
A_{s,p}:=w_s^{1/2}\phi_{k,p}(\mathbf{x}_s),\,\,\,\,\,\,\,\mathbf{b}_s:=w_s^{1/2}(\gamma u)(\mathbf{x}_s),\,\,\,\,\,\,\,\,\,\,\,1 \leq p \leq |\mathbf{\Phi}_{k}|,\,\,\,\ 1 \leq s \leq S.
\label{A matrix definition}
\end{equation}
The sampling method then consists in approximately solving the following, possibly overdetermined, linear system
\begin{equation}
A\boldsymbol{\xi}=\mathbf{b}.
\label{linear system}
\end{equation}
However, the matrix $A$ may often be ill-conditioned (see Section \hyperlink{Section 3.3}{3.3}) and this can lead to inaccurate solutions when using finite-precision arithmetic.
When the ill-conditioning of the matrix $A$ is caused only by the redundancy of the approximating functions, we can deem it as harmless.
In fact, even though this type of ill-conditioning results in non-uniqueness of the coefficients in an expansion, still all expansions may approximate the solution to similar accuracy.
If among those expansions there exist some with small coefficient norms, then it is possible to mitigate such ill-conditioning using regularization techniques. To obtain a solution even in the presence of ill-conditioning, we rely on the conjugation of oversampling and regularization techniques developed in \cite{huybrechs1,huybrechs2}.
The regularized solution procedure is divided into the following points:
\begin{itemize}
\item The first step is to perform the Singular Value Decomposition (SVD) on the matrix $A$, namely
\begin{equation*}
A=U\Sigma V^*.
\end{equation*}
Let us denote by $\sigma_p$ for $p=1,...,|\mathbf{\Phi}_k|$ the singular values of $A$, assumed to be sorted in descending order. For notational clarity, the largest singular value is renamed $\sigma_{\textup{max}}:=\sigma_1$.
\item Then, the regularization process consists in trimming the relatively small singular values by setting them to zero. A threshold parameter $\epsilon \in (0,1]$ is chosen, and the diagonal matrix $\Sigma$ is approximated by $\Sigma_{\epsilon}$ by replacing all $\sigma_m$ such that $\sigma_m < \epsilon \sigma_{\textup{max}}$ with zero. This leads to an approximate factorization of $A$, namely
\begin{equation}
A_{S,\epsilon}:=U\Sigma_{\epsilon}V^*.
\label{A Se}
\end{equation}
\item Lastly, an approximate solution to the linear system in (\ref{linear system}) is obtained by
\begin{equation}
\boldsymbol{\xi}_{S,\epsilon}:=A^{\dagger}_{S,\epsilon}\mathbf{b}=V\Sigma_{\epsilon}^{\dagger}U^*\mathbf{b}.
\label{xi Se solution}
\end{equation}
Here $\Sigma_{\epsilon}^{\dagger} \in \R^{|\mathbf{\Phi}_{k}| \times S}$ denotes the pseudo-inverse of the matrix $\Sigma_{\epsilon}$, namely the diagonal matrix defined by $(\Sigma_{\epsilon}^{\dagger})_{j,j}=(\Sigma_{j,j})^{-1}$ if $\Sigma_{j,j} \geq \epsilon \sigma_{\textup{max}}$ and $(\Sigma_{\epsilon}^{\dagger})_{j,j}=0$ otherwise. To compute $\boldsymbol{\xi}_{S,\epsilon}$ robustly, it is necessary to evaluate the right-hand-side of (\ref{xi Se solution}) from right to left, that is $\boldsymbol{\xi}_{S,\epsilon}=V(\Sigma_{\epsilon}^{\dagger}(U^* \mathbf{b}))$, to prevent small and large values on the diagonal of $\Sigma_{\epsilon}^{\dagger}$ from being mixed.
\end{itemize}

\section{Extremal system of points}

\hypertarget{Section 2.3}{We} first provide the definition of extremal systems of points following \cite{sloan-womersley1}.

\begin{definition}[\hypertarget{Extremal points system}{Extremal system}]
Let $S=(L+1)^2 \in \N$ for some $L \geq 0$. 
A set $\{\mathbf{x}_s\}_{s=1}^S \subset \partial B_1$ is said to be a fundamental system of points if $\Delta_L \neq 0$, where
\begin{equation}
\Delta_L(x_1,...,x_S):=\det \left(G_L \right),\,\,\,\,\,\,\,\,\,\,\,\,\,\,\text{with}\,\,\,\,\,\,\,\,\,\,\,\,\,\,G_L:=\left(Y_j(\mathbf{x}_s)\right)_{j,s} \in \C^{S \times S},
\label{G matrix}
\end{equation}
and $\{Y_j\}_{j=1}^S$ are the first $S$ spherical harmonics \textup{(\ref{spherical harmonic})}, up to degree $L$ and of any order.
A fundamental system is called extremal if it maximizes $|\Delta_L|$.
\end{definition}
The functions $\{Y_j\}_{j=1}^S$ constitute a basis for the spherical polynomial space $\mathbb{P}_L(\partial B_1)$, whose dimension is indeed $S=(L+1)^2$. More generally, fundamental systems are independent of the choice of basis in the interpolation matrix.

Extremal systems provide well-distributed points, have good integration properties if the points are used to determine an interpolatory integration rule, like in (\ref{Riemann sum}), and also have excellent geometrical properties. Furthermore, extremal systems are found to yield interpolatory cubatures rules with positive weights, at least up to $L=200$ (see \cite{womersley}).

The interest in such system of points stems from the association with Lagrange interpolation. Given a function $v \in C^0(\partial B_1)$, the unique polynomial $\Lambda_Lv \in \mathbb{P}_L(\partial B_1)$ that interpolates $v$ at the points of the fundamental system $\{\mathbf{x}_s\}_{s=1}^S$ can be written as
\begin{equation*}
\Lambda_Lv=\sum_{s=1}^Sv(\mathbf{x}_s)\ell_s,
\end{equation*}
where $\ell_s$ is the Lagrange polynomial associated with the $s$-th point $\mathbf{x}_s$, namely
\begin{equation*}
\ell_j \in \mathbb{P}_L(\partial B_1),\,\,\,\,\,\,\,\,\,\,\,\,\,\,\ell_j(\mathbf{x}_s)=\delta_{j,s},\,\,\,\,\,\,\,\,\,\,\,\,\,\,j,s=1,...,S.
\end{equation*}
The interpolatory cubature rule associated with the system of points $\{\mathbf{x}_s\}_{s=1}^S$ is
\begin{equation}
Q_L(v):=\int_{\partial B_1}(\Lambda_Lv)(\mathbf{x})\textup{d}\sigma(\mathbf{x})=\sum_{s=1}^Sw_sv(\mathbf{x}_s),
\label{cubature}
\end{equation}
\begin{equation}
\textup{where}\,\,\,\,\,\,\,\,\,\,\,\,\,\,w_s:=\int_{\partial B_1}\ell_s(\mathbf{x})\textup{d}\sigma(\mathbf{x}), \label{weigths}
\end{equation}
and it is such that the integral
\begin{equation}
I(v):=\int_{\partial B_1}v(\mathbf{x})\textup{d}\sigma(\mathbf{x})
\label{integral cubature}
\end{equation}
can be computed exactly for all polynomials in $\mathbb{P}_L(\partial B_1)$, namely $Q_L(p)=I(p)$ for every $p \in \mathbb{P}_L(\partial B_1)$.

According to \cite{sloan-womersley1}, the construction of the extremal points and related weights relies on the theory of Reproducing Kernel Hilbert Spaces; for a general reference see \cite{reproducing_kernels}.
The polynomial space $\mathbb{P}_L(\partial B_1) \subset L^2(\partial B_1)$ has the reproducing kernel property. For every $\mathbf{x},\mathbf{y} \in \partial B_1$, the reproducing kernel is given by
\begin{equation}
K(\mathbf{x},\mathbf{y})=K_{\mathbf{y}}(\mathbf{x})=\left(K_{\mathbf{y}},K_{\mathbf{x}}\right)_{L^2(\partial B_1)}=\sum_{\ell=0}^{L}\sum_{m=-\ell}^{\ell}\overline{Y_{\ell}^m(\mathbf{y})}Y_{\ell}^m(\mathbf{x}),
\label{reproducing kernel property 2}
\end{equation}
where $K_{\mathbf{x}} \in \mathbb{P}_L(\partial B_1)$ is the (unique) Riesz representation of the evaluation functional at $\mathbf{x} \in \partial B_1$, namely
\begin{equation}
v(\mathbf{x})=\left(v,K_{\mathbf{x}} \right)_{L^2(\partial B_1)},\,\,\,\,\,\,\,\,\,\,\,\,\forall v \in \mathbb{P}_L(\partial B_1).
\label{reproducing ker}
\end{equation}
Observe that, due to the addition theorem \cite[Eq.\ (2.30)]{colton-kress}, i.e.\
\begin{equation}
\sum_{m=-\ell}^{\ell}\overline{Y_{\ell}^m(\mathbf{y})}Y_{\ell}^m(\mathbf{x})=\frac{2\ell+1}{4\pi}\mathsf{P}_{\ell}(\mathbf{x}\cdot \mathbf{y}),\,\,\,\,\,\,\,\,\,\,\,\,\,\,\forall \mathbf{x},\mathbf{y} \in \partial B_1,
\label{addition theorem}
\end{equation}
the reproducing kernel $K$ is real and symmetric.
The reproducing kernel property implies that pointwise evaluation of elements of $\mathbb{P}_L(\partial B_1)$ on the unit sphere $\partial B_1$ is a continuous operation (see \cite[Def.\ 1.2]{reproducing_kernels}).

Given a fundamental system $\{\mathbf{x}_s\}_{s=1}^S$, the family $\{K_{\mathbf{x}_s}\}_{s=1}^S$ constitutes a basis for $\mathbb{P}_L(\partial B_1)$: in fact such polynomials are linearly independent since the Gram matrix $K_L \in \R^{S \times S}$ with elements
\begin{equation}
(K_L)_{j,s}:=\left(K_{\mathbf{x}_j},K_{\mathbf{x}_s}\right)_{L^2(\partial B_1)}=K(\mathbf{x}_j,\mathbf{x}_s)
\label{Gram matrix}
\end{equation}
is nonsingular if $\{\mathbf{x}_s\}_{s=1}^S$ is a fundamental system. It follows from (\ref{reproducing kernel property 2}) that the matrix $K_L$ can be written as $K_L=G^*_LG_L$, where $G_L$ is the interpolation matrix defined in (\ref{G matrix}).
Thanks to the addition theorem (\ref{addition theorem}), we have that $K(\mathbf{x},\mathbf{y})=\tilde{K}(\mathbf{x} \cdot \mathbf{y})$, where
\begin{equation}
\tilde{K}(t):=\frac{1}{4\pi}\sum_{\ell=1}^L(2\ell+1)\mathsf{P}_{\ell}(t),\,\,\,\,\,\,\,\,\,\,\,\,\,\,\forall t \in [-1,1].
\label{K tilde}
\end{equation}
Thus, we are able to easily calculate the matrix $K_L$ in (\ref{Gram matrix}) from (\ref{K tilde}) by upward recurrence of the Legendre polynomials (see \cite[Eq.\ (14.10.3)]{nist}). In particular, since $\mathsf{P}_{\ell}(1)=1$ for every degree $\ell \geq 0$, note that
\begin{equation*}
K(\mathbf{x},\mathbf{x})=\tilde{K}(1)=\frac{1}{4\pi}\sum_{\ell=0}^L(2\ell+1)=\frac{S}{4\pi},\,\,\,\,\,\,\,\,\,\,\,\,\,\,\forall \mathbf{x} \in \partial B_1,
\end{equation*}
therefore $K_L$ has equal diagonal entries.
\begin{figure}
\centering
\includegraphics[width=7.3cm]{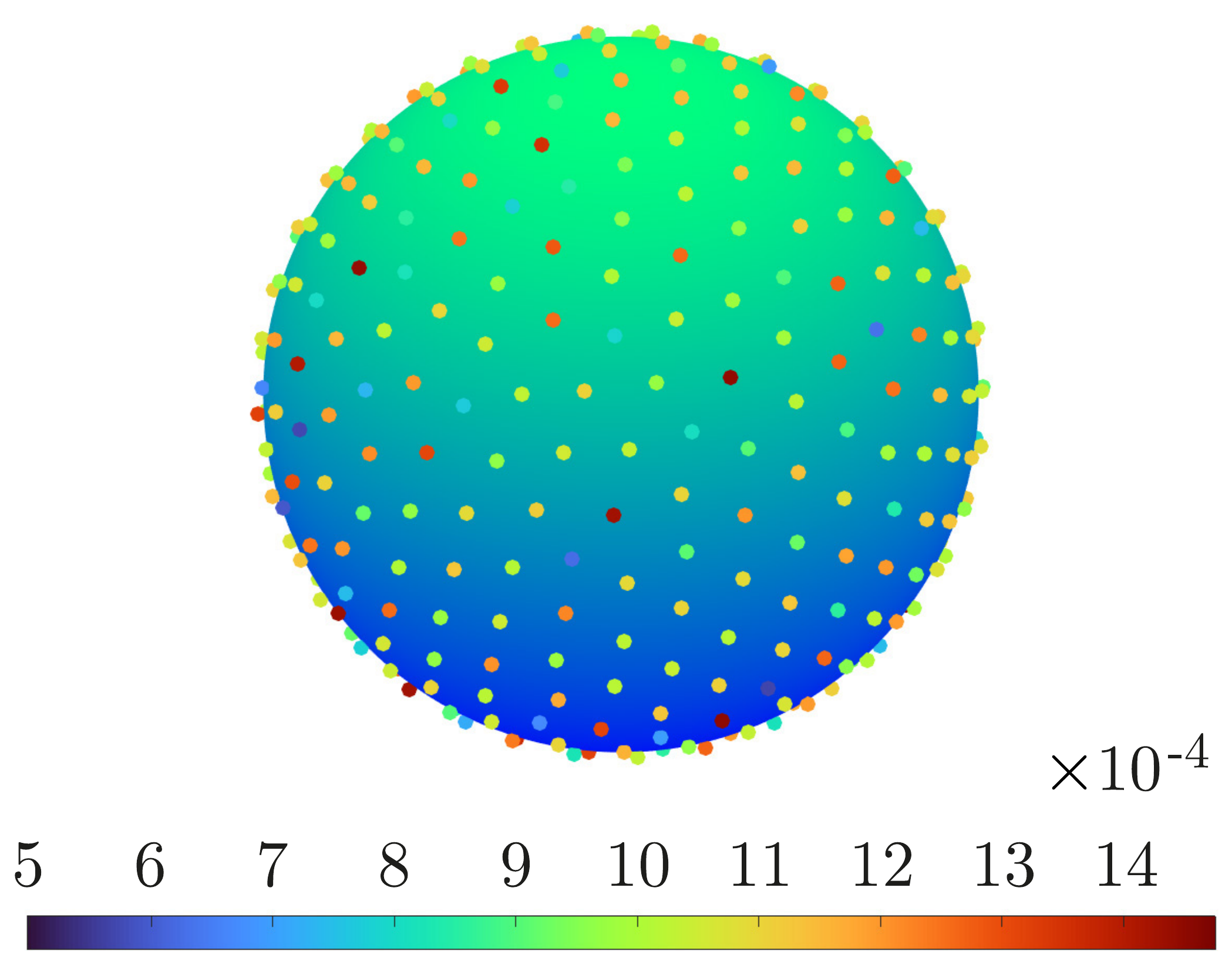}
\captionsetup{justification=centering}
\caption{Distribution of $S=400$ extremal points on the unit sphere $B_1$. Each point $\mathbf{x}_s$ is colored according to the value of the weight $w_s$ associated with it.}
\label{extremal ball}
\end{figure}

Clearly $K_L$ is positive semi-definite for any point set, and $\det(K_L)=\det(G_L)^2 \geq 0$, so an extremal system $\{\mathbf{x}_s\}_{s=1}^S$ can be obtained by maximizing the determinant of the symmetric positive definite matrix $K_L$. Both a restarted conjugate gradient algorithm, followed by a Newton method using a finite difference approximation to the Hessian, and a limited memory BFGS method were used in \cite{sloan-womersley1} to maximize $\log{\det(K_L)}$, so as to avoid overflow issues. This procedure only finds a local maximizer. Furthermore, since $K_L$ is rotationally invariant, the first point $\mathbf{x}_1$ is arbitrarily fixed at the north pole and, if $S>1$, the second point $\mathbf{x}_2$ is fixed on the prime meridian.

The weights $\mathbf{w}_S=(w_s)_s$ related to the extremal points $\{\mathbf{x}_s\}_{s=1}^S$ thus obtained, are defined in (\ref{weigths}). However, the practical computation of the weights proceeds differently. Due to the reproducing kernel property (\ref{reproducing ker}), the basis $\{K_{\mathbf{x}_s}\}_{s=1}^S$ satisfies
\begin{equation*}
\int_{\partial B_1}K_{\mathbf{x}_s}(\mathbf{x})\textup{d}\sigma(\mathbf{x})=1,\,\,\,\,\,\,\,\,\,\,\,\,\,\,s=1,...,S,
\end{equation*}
so the condition that the cubature (\ref{cubature}) is exact for all polynomials in $\mathbb{P}_L(\partial B_1)$ can be written as
\vspace{-2mm}
\begin{equation}
K_L \mathbf{w}_S=\mathbf{1},
\label{linear system cubature}
\end{equation}
where $\mathbf{w}_S=(w_s)_s$, $\mathbf{1}$ is the vector of ones in $\R^S$ and $K_L$ is defined in (\ref{Gram matrix}). In particular, observe that the cubature rule is exact for the constant polynomial $1 \in \mathbb{P}_L(\partial B_1)$, so that $\sum_{s=1}^Sw_s=4\pi$.
It is important to note that for other point systems, particularly when $L$ is large, solving for the weights using (\ref{linear system cubature}) can be difficult due to the ill-conditioning of the matrix $K_L$. However, it is a characteristic of extremal systems that the matrix $K_L$ is highly well-conditioned.

The error for the interpolatory cubature rule (\ref{cubature}) with respect to the exact integral (\ref{integral cubature}) can be bounded as follows: for any $p \in \mathbb{P}_L(\partial B_1)$,
\begin{equation*}
\left|I(v)-Q_L(v)\right|=\left|\left(I-Q_L\right)(v-p)\right| \leq \left(4\pi+\|\mathbf{w}_S\|_{\ell^1} \right)\|v-p\|_{\infty}
\end{equation*}
and therefore, assuming all weights $w_s$ are positive,
\begin{equation}
\left|I(v)-Q_L(v)\right|\leq 8\pi \inf_{p \in \mathbb{P}_L(\partial B_1)}\|v-p\|_{\infty}.
\label{error bound integral}
\end{equation}
Hence, if $\{\mathbf{x}_s\}_{s=1}^S \subset \partial B_1$ is an extremal system of points and $\mathbf{w}_S \in \R^S$ are its associated weights, thanks to (\ref{error bound integral}) and the Weierstrass--Stone theorem, the convergence (\ref{Riemann sum}) is guaranteed.

Some extremal points and associated weights are depicted in Figure \ref{extremal ball}; their tables of values are listed in \cite{womersley}. Other types of fundamental systems can be found in \cite{sloan-womersley2}.

\begin{remark}
\hypertarget{Remark 2.3}{Observe} that, in order to use the extremal system $\{\mathbf{x}_s\}_{s=1}^S$ as sampling point set on the unit sphere $\partial B_1$ in the linear system \textup{(\ref{linear system})}, then $S \in \N$ needs to be a perfect square.
\end{remark}

\section{Error estimates}

\hypertarget{Section 2.4}{By} using the regularization technique outlined in Section \hyperlink{Section 2.2}{2.2}, together with oversampling, i.e.\ $S$ larger than $|\mathbf{\Phi}_k|$, it is possible to achieve accurate approximations. This is provided that the set sequence is a stable approximation in the sense of Definition \hyperlink{Definition 2.1}{2.1} and (\ref{Riemann sum}) holds for the chosen sampling points and weights. This general statement is the main conclusion of \cite[Th.\ 5.3]{huybrechs1} and \cite[Th.\ 1.3 and 3.7]{huybrechs2}, and forms the basis of the investigation into stable approximation sets for the solutions of the Helmholtz equation.

More precisely, we have the following proposition  \cite[Prop.\ 3.2]{parolin-huybrechs-moiola}, which in turn builds on \cite[Th.\ 3.7]{huybrechs2} in the context of generalized sampling.
\begin{proposition}
\hypertarget{Proposition 2.1}{Let} $\gamma$ be the Dirichlet trace operator, $u \in \mathcal{B} \cap C^0(\overline{B_1})$ and $k \in \N$. Given some approximation set $\mathbf{\Phi}_k$ such that, for any $p$, $\phi_{k,p} \in \mathcal{B} \cap C^0(\overline{B_1})$, a set of sampling points $\{\mathbf{x}_s\}_{s=1}^S \subset \partial B_1$ together with a suitable choice of positive weights $\mathbf{w}_S \in \R^S$ such that \textup{(\ref{Riemann sum})} is satisfied and some regularization parameter $\epsilon \in (0,1]$, we consider the approximate solution of the linear system \textup{(\ref{linear system})}, namely $\boldsymbol{\xi}_{S,\epsilon} \in \C^{|\mathbf{\Phi}_k|}$ as defined in \textup{(\ref{xi Se solution})}. Then $\forall \boldsymbol{\mu} \in \C^{|\mathbf{\Phi}_k|}$, $\exists S_0 \in \N$ such that $\forall S \geq S_0$
\begin{equation}
\| \gamma(u-\mathcal{T}_{\mathbf{\Phi}_k}\boldsymbol{\xi}_{S,\epsilon})\|_{L^2(\partial B_1)} \leq 3\|\gamma(u-\mathcal{T}_{\mathbf{\Phi}_k}\boldsymbol{\mu})\|_{L^2(\partial B_1)} + \sqrt{2}\epsilon \sigma_{\textup{max}}\|\mathbf{w}_S\|^{1/2}_{\infty}\|\boldsymbol{\mu}\|_{\ell^2}.
\label{error bound 1}
\end{equation}
Assume moreover that $\kappa^2$ is not an eigenvalue of the Dirichlet Laplacian in $B_1$. Then there exists a constant $C_{\textup{err}} >0$ independent of $u$ and $\mathbf{\Phi}_k$ such that $\forall \boldsymbol{\mu} \in \C^{|\mathbf{\Phi}_k|}$, $\exists S_0 \in \N$ such that $\forall S \geq S_0$
\begin{equation}
\|u-\mathcal{T}_{\mathbf{\Phi}_k}\boldsymbol{\xi}_{S,\epsilon}\|_{L^2( B_1)} \leq C_{\textup{err}} \left(\|u-\mathcal{T}_{\mathbf{\Phi}_k}\boldsymbol{\mu}\|_{\mathcal{B}} + \epsilon \sigma_{\textup{max}}\|\mathbf{w}_S\|^{1/2}_{\infty}\|\boldsymbol{\mu}\|_{\ell^2} \right).
\label{error bound 2}
\end{equation}
\end{proposition}
\begin{proof}
Since $u$ and $\mathbf{\Phi}_k$ are assumed to be in $\mathcal{B} \cap C^0(\overline{B_1})$, the pointwise evaluations of their image by the Dirichlet trace operator $\gamma$ on the boundary $\partial B_1$ is well-defined.
Since (\ref{Riemann sum}) holds by hypothesis, for any $v \in \mathcal{B} \cap C^0(\overline{B_1})$, $\exists S_0 \in \N$ such that $\forall S \geq S_0$
\begin{equation}
\|\gamma v\|^2_{L^2(\partial B_1)} \leq 2\sum_{s=1}^Sw_s|(\gamma v)(\mathbf{x}_s)|^2 \leq 4\|\gamma v\|^2_{L^2(\partial B_1)},
\label{inequality riemann}
\end{equation}
where the constants $2$ and $4$ are arbitrary and can respectively be replaced by any pair of constants $C_1$ and $C_2$ such that $1<C_1<C_2$.
A similar argument is developed in \cite[Eq. (3.2)]{huybrechs2} (note that $A'=B'=1$ in the notations of \cite{huybrechs2}).
Moreover, observe that, thanks to (\ref{synthesis operator}) and (\ref{A matrix definition}), it follows
\begin{equation}
(A\boldsymbol{\mu})_s=w_s^{1/2}\gamma \left(\mathcal{T}_{\mathbf{\Phi}_k}\boldsymbol{\mu}\right)(\mathbf{x}_s),\,\,\,\,\,\,\,\,\,\,\,\,\,\,\forall \boldsymbol{\mu} \in \C^{|\mathbf{\Phi}_k|},
\label{Amu identity before}
\end{equation}
and therefore
\begin{equation}
\|A\boldsymbol{\mu}\|^2_{\ell^2}=\sum_{s=1}^Sw_s\left|\gamma \left(\mathcal{T}_{\mathbf{\Phi}_k}\boldsymbol{\mu}\right)(\mathbf{x}_s) \right|^2,\,\,\,\,\,\,\,\,\,\,\,\,\,\,\forall \boldsymbol{\mu} \in \C^{|\mathbf{\Phi}_k|}.
\label{Amu identity}
\end{equation}
Let $\boldsymbol{\mu} \in \C^{|\mathbf{\Phi}_k|}$. From (\ref{xi Se solution}) we have:
\begin{align*}
u-\mathcal{T}_{\mathbf{\Phi}_k}\boldsymbol{\xi}_{S,\epsilon}&=[u-\mathcal{T}_{\mathbf{\Phi}_k}\boldsymbol{\mu}]+[\mathcal{T}_{\mathbf{\Phi}_k}A^{\dagger}_{S,\epsilon}A \boldsymbol{\mu}-\mathcal{T}_{\mathbf{\Phi}_k}\boldsymbol{\xi}_{S,\epsilon}]+[\mathcal{T}_{\mathbf{\Phi}_k}\boldsymbol{\mu}-\mathcal{T}_{\mathbf{\Phi}_k}A^{\dagger}_{S,\epsilon}A \boldsymbol{\mu}]\\
&=[u-\mathcal{T}_{\mathbf{\Phi}_k}\boldsymbol{\mu}]+\mathcal{T}_{\mathbf{\Phi}_k}A^{\dagger}_{S,\epsilon}[A \boldsymbol{\mu}- \mathbf{b}]+\mathcal{T}_{\mathbf{\Phi}_k}[\text{Id}-A^{\dagger}_{S,\epsilon}A]\boldsymbol{\mu}. \numberthis \label{equation proposition}
\end{align*}
The proof proceeds by estimating the $L^2$ norm of the trace on $\partial B_1$ of each term.
The first term in (\ref{equation proposition}) readily appear in (\ref{error bound 1}), so we examine the second term. From (\ref{inequality riemann})--(\ref{Amu identity}), assuming that $S$ has been chosen sufficiently large, we can write:
\begin{align*}
\|\gamma(\mathcal{T}_{\mathbf{\Phi}_k} A^{\dagger}_{S,\epsilon}[A \boldsymbol{\mu} - \mathbf{b}])\|^2_{L^2(\partial B_1)}& \leq 2\sum_{s=1}^Sw_s|\gamma(\mathcal{T}_{\mathbf{\Phi}_k}A^{\dagger}_{S,\epsilon}[A \boldsymbol{\mu} - \mathbf{b}])(\mathbf{x}_s)|^2\\
&= 2 \|AA^{\dagger}_{S,\epsilon}[A\boldsymbol{\mu}-\mathbf{b}]\|^2_{\ell^2}.
\end{align*}
Furthermore, the regularization (\ref{A Se}) ensures that $\|AA^{\dagger}_{S,\epsilon}\| \leq 1$ and, using once more (\ref{inequality riemann})--(\ref{Amu identity before}) and provided that $S$ has been chosen sufficiently large, we have:
\begin{align*}
\|\gamma(\mathcal{T}_{\mathbf{\Phi}_k} A^{\dagger}_{S,\epsilon}[A \boldsymbol{\mu} - \mathbf{b}])\|^2_{L^2(\partial B_1)} &\leq 2 \|A\boldsymbol{\mu}-\mathbf{b}\|^2_{\ell^2}=2\sum_{s=1}^S w_s|\gamma(\mathcal{T}_{\mathbf{\Phi}_k}\boldsymbol{\mu} - u)(\mathbf{x}_s)|^2\\
&\leq 4\|\gamma(u-\mathcal{T}_{\mathbf{\Phi}_k} \boldsymbol{\mu})\|^2_{L^2(\partial B_1)}.
\end{align*}
We now examine the third term in (\ref{equation proposition}). Arguing as before, from (\ref{inequality riemann})--(\ref{Amu identity}), there exists $S$ sufficiently large such that
\begin{align*}
\|\gamma(\mathcal{T}_{\mathbf{\Phi}_k}[\text{Id}-A^{\dagger}_{S,\epsilon}A]\boldsymbol{\mu})\|^2_{L^2(\partial B_1)} &\leq 2\sum_{s=1}^S w_s|\gamma(\mathcal{T}_{\mathbf{\Phi}_k}[\text{Id}-A^{\dagger}_{S,\epsilon}A]\boldsymbol{\mu})(\mathbf{x}_s)|^2\\
& = 2 \|A[\text{Id}-A^{\dagger}_{S,\epsilon}A]\boldsymbol{\mu}\|^2_{\ell^2}.
\end{align*}
Regularization (\ref{A Se}) ensures that $\|A[\text{Id}-A^{\dagger}_{S,\epsilon}A]\| \leq \epsilon \sigma_{\textup{max}}\|\mathbf{w}_S\|^{1/2}_{\infty}$ so that
\begin{equation*}
\|\gamma(\mathcal{T}_{\mathbf{\Phi}_k}[\text{Id}-A^{\dagger}_{S,\epsilon}A]\boldsymbol{\mu})\|^2_{L^2(\partial B_1)} \leq 2 \epsilon^2 \sigma^2_{\textup{max}}\|\mathbf{w}_S\|_{\infty}\|\boldsymbol{\mu}\|^2_{\ell^2}.
\end{equation*}
Combining all estimates, we obtain (\ref{error bound 1}).

In order to show (\ref{error bound 2}), note first that the continuity of the trace operator $\gamma: \mathcal{B} \rightarrow L^2(\partial B_1)$ allows to write, for any $\boldsymbol{\mu} \in \C^{|\mathbf{\Phi}_k|}$:
\begin{equation}
\|\gamma(u-\mathcal{T}_{\mathbf{\Phi}_k} \boldsymbol{\mu})\|_{L^2(\partial B_1)} \leq \|\gamma\| \|u-\mathcal{T}_{\mathbf{\Phi}_k} \boldsymbol{\mu}\|_{\mathcal{B}}.
\label{continuity bound}
\end{equation}
It remains to bound the $L^2(B_1)$ norm of $u-\mathcal{T}_{\mathbf{\Phi}_k}\boldsymbol{\xi}_{S,\epsilon}$ by the $L^2(\partial B_1)$ norm of its trace. Let $\{\hat{e}_{\ell}^m\}_{(\ell,m) \in \mathcal{I}}$ the coefficients of $e:=u-\mathcal{T}_{\mathbf{\Phi}_k}\boldsymbol{\xi}_{S,\epsilon}$ in the Hilbert basis $\{b_{\ell}^m\}_{(\ell,m) \in \mathcal{I}}$. From the asymptotics (\ref{asymptotic behavior}), we have
\begin{equation*}
\|e\|^2_{L^2(B_1)}=\sum_{\ell=0}^{\infty}\frac{c_{\ell}^{(1)}}{1+{\ell}^2}\sum_{m=-\ell}^{\ell}|\hat{e}_{\ell}^m|^2\,\,\,\,\,\,\,\text{and}\,\,\,\,\,\,\,\|e\|^2_{L^2(\partial B_1)}=\sum_{\ell=0}^{\infty}\frac{c_{\ell}^{(2)}}{\sqrt{1+{\ell}^2}}\sum_{m=-\ell}^{\ell}|\hat{e}_{\ell}^m|^2,
\end{equation*}
where we introduced two sequences of strictly positive constants $\{c_{\ell}^{(i)}\}_{\ell \geq 0}$ for $i=1,2$, both bounded above and below and independent of $u-\mathcal{T}_{\mathbf{\Phi}_k}\boldsymbol{\xi}_{S,\epsilon}$. Note that the fact that $\{c_{\ell}^{(2)}\}_{\ell \geq 0}$ is bounded from below follows from the fact that $\kappa^2$ is not a Dirichlet eigenvalue. We derive (\ref{error bound 2}) from this, (\ref{continuity bound}) and (\ref{error bound 1}).
\end{proof}
Proposition \hyperlink{Proposition 2.1}{2.4} shows that the use of stable approximation sets as defined in Definition \hyperlink{Definition 2.1}{2.1}, in combination with appropriate sampling points and weights that satisfy (\ref{Riemann sum}), is a necessary condition for accurate reconstruction of Helmholtz solutions from samples on the sphere, provided that the number of sampling points $S$ is large enough and the regularization parameter $\epsilon$ is sufficiently small. More specifically, we have the following corollary.

\begin{corollary}
\hypertarget{Corollary 2.1}{Let} $\delta >0$. We assume to have a sequence of approximation sets $\{\mathbf{\Phi}_k\}_{k \in \N}$ that is stable in the sense of Definition \hyperlink{Definition 2.1}{\textup{2.1}} and a set of sampling points $\{\mathbf{x}_s\}_{s=1}^S \subset \partial B_1$ together with a positive weight vector $\mathbf{w}_S \in \R^S$ such that \textup{(\ref{Riemann sum})} is satisfied. Assume also that $\kappa^2$ is not a Dirichlet eigenvalue in $B_1$. Then, $\forall u \in \mathcal{B} \cap C^0(\overline{B_1})$, $\exists k \in \N$, $S_0 \in \N$ and $\epsilon_0 \in (0,1]$ such that $\forall S \geq S_0$ and $\epsilon \in (0,\epsilon_0]$
\begin{equation}
\|u-\mathcal{T}_{\mathbf{\Phi}_k}\boldsymbol{\xi}_{S,\epsilon}\|_{L^2(B_1)} \leq \delta \|u\|_{\mathcal{B}},
\label{bound corollary}
\end{equation}
where $\boldsymbol{\xi}_{S,\epsilon} \in \C^{|\mathbf{\Phi}_k|}$ is defined in \textup{(\ref{xi Se solution})}. Moreover, we can take the regularization parameter $\epsilon$ as large as
\begin{equation}
\epsilon_0=\frac{\delta}{2C_{\textup{err}}C_{\textup{stb}}\sigma_{\textup{max}}|\mathbf{\Phi}_k|^{\lambda}\|\mathbf{w}_S\|^{1/2}_{\infty}}.
\label{epsilon0}
\end{equation}
\end{corollary}
\begin{proof}
Let $\eta >0$ and $u \in \mathcal{B} \cap C^0(\overline{B_1})$. The stability assumption implies that $\exists k \in \N$ and $\boldsymbol{\mu} \in \C^{|\mathbf{\Phi}_k|}$ such that (\ref{stable approximation}) holds, namely
\begin{equation*}
\|u-\mathcal{T}_{\mathbf{\Phi}_k}\boldsymbol{\mu}\|_{\mathcal{B}} \leq \eta \|u\|_{\mathcal{B}}\,\,\,\,\,\,\,\,\,\,\text{and}\,\,\,\,\,\,\,\,\,\,\|\boldsymbol{\mu}\|_{\ell^2} \leq C_{\textup{stb}}|\mathbf{\Phi}_k|^{\lambda}\|u\|_{\mathcal{B}}.
\end{equation*}
Moreover, let $\epsilon \in (0,1]$. The previous proposition implies the existence of $S \in \N$ such that for this particular $\boldsymbol{\mu}$ we have:
\begin{align*}
\|u-\mathcal{T}_{\mathbf{\Phi}_k}\boldsymbol{\xi}_{S,\epsilon}\|_{L^2( B_1)} &\leq C_{\textup{err}} \left(\|u-\mathcal{T}_{\mathbf{\Phi}_k}\boldsymbol{\mu}\|_{\mathcal{B}} + \epsilon \sigma_{\textup{max}}\|\mathbf{w}_S\|^{1/2}_{\infty}\|\boldsymbol{\mu}\|_{\ell^2} \right)\\
& \leq C_{\textup{err}} \left(\eta + \epsilon \sigma_{\textup{max}}C_{\textup{stb}}|\mathbf{\Phi}_k|^{\lambda}\|\mathbf{w}_S\|^{1/2}_{\infty} \right)\|u\|_{\mathcal{B}}.
\end{align*}
Choosing $\eta \leq \frac{\delta}{2C_{\textup{err}}}$ and $\epsilon \leq \epsilon_0$ with $\epsilon_0$ given in (\ref{epsilon0}), the error estimate (\ref{bound corollary}) follows. 
\end{proof}

Note that Corollary \hyperlink{Corollary 2.1}{2.5} shows that the vector $\boldsymbol{\xi}_{S,\epsilon} \in \C^{|\mathbf{\Phi}_k|}$ -- which can be computed stably in floating-point arithmetic using the regularized SVD (\ref{xi Se solution}) -- provides an accurate approximation $\mathcal{T}_{\mathbf{\Phi}_k}\boldsymbol{\xi}_{S,\epsilon}$ of $u$. 
This is much stronger than saying that for any solution $u$ there exists a coefficient vector $\boldsymbol{\mu} \in \C^{|\mathbf{\Phi}_k|}$ such that $\mathcal{T}_{\mathbf{\Phi}_k}\boldsymbol{\mu}$ is an accurate approximation of $u$.
Therefore, the previous error bounds on $u-\mathcal{T}_{\mathbf{\Phi}_k}\boldsymbol{\xi}_{S,\epsilon}$ apply to the solution obtained by the sampling method when computed using computer arithmetic. This is in contrast with the classical theory for approximation by propagative plane waves, e.g.\ \cite{hiptmair-moiola-perugia4}, which provides rigorous best-approximation error bounds that are often not achievable numerically, because accurate approximations require large coefficients and cancellation, so results obtained using exact arithmetic are not consistent with those obtained using floating-point computation.

Lastly, to measure the error of the approximation, we introduce, in analogy with \cite{parolin-huybrechs-moiola}, the following relative residual
\begin{equation}
\mathcal{E}=\mathcal{E}(u,\mathbf{\Phi}_k,S,\epsilon):=\frac{\|A\boldsymbol{\xi}_{S,\epsilon}-\mathbf{b}\|_{\ell^2}}{\|\mathbf{b}\|_{\ell^2}},
\label{relative residual}
\end{equation}
where $\boldsymbol{\xi}_{S,\epsilon}$ is the solution (\ref{xi Se solution}) of the regularized system.
In fact, following the same reasoning as the proof of Proposition \hyperlink{Proposition 2.1}{2.4}, it can be shown that, for values of $S$ that are sufficiently large, the quantity $\mathcal{E}$ in (\ref{relative residual}) satisfies the inequality
\begin{equation*}
\|u-\mathcal{T}_{\mathbf{\Phi}_k}\boldsymbol{\xi}_{S,\epsilon}\|_{L^2( B_1)} \leq \tilde{C}\|u\|_{\mathcal{B}}\,\mathcal{E},
\end{equation*}
where $\tilde{C}$ is a constant independent of $u$, $\boldsymbol{\Phi}_k$ and $S$.

\chapter{Instability of propagative plane wave sets}
\hypertarget{Chapter 3}{In} this chapter, we show that the propagative plane waves -- similarly to what happens in the two-dimensional case in \cite[Sec.\ 4]{parolin-huybrechs-moiola} within the unit disk -- fail to yield stable approximations in the unit ball $B_1$.
However, approximations of Helmholtz solutions using propagative plane wave expansions are a key component of many Treffz schemes (see \cite{hiptmair-moiola-perugia1}).
We also introduce the classical notion of Herglotz function. All Herglotz functions are solutions of the Helmholtz equation (\ref{Helmholtz equation}), but not all solutions to the Helmholtz equation have such a representation. Furthermore we show that the density associated with the spherical waves is not uniformly bounded in $\ell$, which implies that the discretization of the related integral representation cannot yield approximate discrete representations with bounded coefficients. In the end, the instability of propagative plane wave sets is verified numerically.

\section{Propagative plane waves}

\hypertarget{Section 3.1}{We} will now introduce the concept of \textit{propagative plane wave}. The term `propagative' it is used here to distinguish the following definition from the notion of \textit{evanescent plane wave} that will be introduced later (see Definition \hyperlink{Definition 4.1}{4.1}).
\begin{definition}[\hypertarget{propagative plane wave}{Propagative plane wave}]
For any pair of angles $(\theta_1,\theta_2) \in [0,\pi] \times [0,2\pi)$, we let
\begin{equation}
\phi_{\mathbf{d}}(\mathbf{x}):=e^{i\kappa \mathbf{d}\cdot \mathbf{x}},\,\,\,\,\,\,\,\forall \mathbf{x} \in \R^3,
\label{propagative wave}
\end{equation}
where the propagation direction of the wave is given by
\begin{equation}
\mathbf{d}=\mathbf{d}(\theta_1, \theta_2):=(\sin \theta_1 \cos \theta_2, \sin \theta_1 \sin \theta_2,\cos \theta_1) \in \mathbb{S}^2 \subset \R^3.
\label{propagative direction}
\end{equation}
\end{definition}
It is immediate to check that any propagative plane wave satisfies the homogeneous Helmholtz equation (\ref{Helmholtz equation}) since $\mathbf{d} \cdot \mathbf{d}=1$.

In 3D, isotropic approximations are obtained by using almost-evenly distributed directions.
For some $P \in \N$, the \textit{propagative plane waves approximation set} is defined as
\begin{equation}
\mathbf{\Phi}_{P}:=\Biggl\{\frac{1}{\sqrt{P}}\, \phi_{\mathbf{d}_p}\Biggr\}_{p=1}^{P},
\label{plane waves approximation set}
\end{equation}
where $\{\mathbf{d}_p\}_{p=1}^{P} \subset \mathbb{S}^2$ is a nearly-uniform set of directions. In contrast to spherical waves, the approximation sets based on such propagative plane waves are in general not hierarchical.

\begin{remark}
\hypertarget{Remark 3.2}{Since} extremal systems presented in Definition \textup{\hyperlink{Extremal points system}{2.2}} provide well-distributed points and have excellent geometrical properties, in our numerical experiments we will use them to describe such a set of directions.
Similarly to what we said for the sampling point set, using extremal systems of points, the direction set $\{\mathbf{d}_p\}_{p=1}^P \subset \mathbb{S}^2$ is well-defined only if $P \in \N$ is a perfect square.
\end{remark}

Lastly, let us state an essential identity that is ubiquitous in the following analysis and provides a link between plane waves and spherical ones.
From the identity \cite[Eq. (2.46)]{colton-kress}, namely
\begin{equation}
e^{irt}=\sum_{\ell=0}^{\infty}i^{\ell}(2\ell+1)j_{\ell}(r)\mathsf{P}_{\ell}(t),\,\,\,\,\,\,\,\,\,\,\,\,\,\,\,\,\,\,\,\,\forall r \geq 0,\,\forall t \in [-1,1],
\label{pre jacobi-anger}
\end{equation}
and the addition theorem (\ref{addition theorem}), for any $\mathbf{x} \in B_1$ and $\mathbf{d} \in \mathbb{S}^2$ we deduce the \textit{Jacobi--Anger identity}:
\begin{equation}
\phi_{\mathbf{d}}(\mathbf{x})=e^{i\kappa \mathbf{d}\cdot \mathbf{x}}=4 \pi \sum_{\ell=0}^{\infty}i^{\ell}\sum_{m=-\ell}^{\ell} \overline{Y_{\ell}^m(\mathbf{d})}Y_{\ell}^m(\mathbf{\hat{x}})j_{\ell}(\kappa |\mathbf{x}|),
\label{jacobi-anger}
\end{equation}
where $\mathbf{\hat{x}}:=\mathbf{x}/|\mathbf{x}| \in \mathbb{S}^2$.

\vspace{-1.3mm}

\section{Herglotz representation}

\hypertarget{Section 3.2}{We} now recall the so-called \textit{Herglotz functions} in \cite[Eq.\ (3.43)]{colton-kress}, defined, for any $v \in L^2(\mathbb{S}^2)$, as
\begin{equation}
u_v(\mathbf{x}):=\int_{\mathbb{S}^2}v(\mathbf{d})\phi_{\mathbf{d}}(\mathbf{x})\textup{d}\sigma(\mathbf{d}),\,\,\,\,\,\,\,\forall \mathbf{x} \in \R^3.
\label{herglotz}
\end{equation}
Such an expression is termed \textit{Herglotz representation} and $v$ is called \textit{Herglotz density} of $u_v$. These functions $u_v \in C^{\infty}(\R^3)$ can be seen as a continuous superposition of propagative plane waves, weighted according to $v$, and are entire solutions of the Helmholtz equation.

In fact, since $v \in L^2(\mathbb{S}^2)$, we can rewrite it as an expansion into spherical harmonics, that is
\begin{equation}
v(\mathbf{d})=\sum_{\ell=0}^{\infty}\sum_{m=-\ell}^{\ell}\hat{v}_{\ell}^m Y_{\ell}^m(\mathbf{d}),\,\,\,\,\,\,\,\forall \mathbf{d} \in \mathbb{S}^2,
\label{v spherical harmonics}
\end{equation}
with $\{\hat{v}_{\ell}^m\}_{(\ell,m) \in \mathcal{I}} \in \ell^2(\mathcal{I})$, and, thanks to the identity (\ref{jacobi-anger}) and the orthogonality (\ref{Y_l^m orthogonality}), it follows
\begin{align*}
u_v(\mathbf{x}&)=\int_{\mathbb{S}^2}v(\mathbf{d})\phi_{\mathbf{d}}(\mathbf{x})\textup{d}\sigma(\mathbf{d})=\int_{\mathbb{S}^2}\sum_{\ell=0}^{\infty}\sum_{m=-\ell}^{\ell}\hat{v}_{\ell}^m Y_{\ell}^m(\mathbf{d})\phi_{\mathbf{d}}(\mathbf{x})\textup{d}\sigma(\mathbf{d})\\
&=\int_{\mathbb{S}^2}\sum_{\ell=0}^{\infty}\sum_{m=-\ell}^{\ell}\hat{v}_{\ell}^m Y_{\ell}^m(\mathbf{d})\left(4\pi \sum_{q=0}^{\infty}i^q\sum_{n=-q}^q\overline{Y_q^n(\mathbf{d})}\,\tilde{b}_q^n(\mathbf{x}) \right)\textup{d}\sigma(\mathbf{d})\\
&=\sum_{\ell=0}^{\infty}\sum_{m=-\ell}^{\ell}\left(\frac{4\pi i^{\ell} \hat{v}_{\ell}^m}{\beta_{\ell}} \right)b_{\ell}^m(\mathbf{x}). \numberthis \label{to the herglotz spherical waves}
\end{align*}
Due to the super-exponential growth of the coefficients $\{\beta_{\ell}\}_{\ell \geq 0}$ shown in Lemma \hyperlink{Lemma 1.3}{1.4}, we deduce that $u_v \in \mathcal{B}$.

Although spherical waves have a Herglotz representation, their Herglotz densities are not bounded uniformly with respect to the index $\ell$. For any $\ell \geq 0$, using once again the Jacobi-Anger identity (\ref{jacobi-anger}), we have
\begin{align*}
\int_{\mathbb{S}^2}Y_{\ell}^m(\mathbf{d})\phi_{\mathbf{d}}(\mathbf{x})\textup{d}\sigma(\mathbf{d})&=\int_{\mathbb{S}^2}Y_{\ell}^m(\mathbf{d})\left(4\pi \sum_{q=0}^{\infty}i^q\sum_{n=-q}^q\overline{Y_q^n(\mathbf{d})}\,\tilde{b}_q^n(\mathbf{x}) \right)\textup{d}\sigma(\mathbf{d})\\
&=4\pi i^{\ell} \tilde{b}_{\ell}^m(\mathbf{x}).
\end{align*}
Hence, we obtain the Herglotz representation of the spherical waves,
\begin{equation}
b_{\ell}^m(\mathbf{x})=\int_{\mathbb{S}^2}\left[\frac{\beta_{\ell}}{4\pi i^{\ell}}Y_{\ell}^m(\mathbf{d})\right]\phi_{\mathbf{d}}(\mathbf{x})\textup{d}\sigma(\mathbf{d}).
\label{herglotz spherical waves}
\end{equation}
Thanks to Lemma \hyperlink{Lemma 1.3}{1.4}, it is easy to see that the associated Herglotz density $\mathbf{d} \mapsto \beta_{\ell}(4 \pi)^{-1}i^{-\ell}Y_{\ell}^m(\mathbf{d})$ is not bounded uniformly with respect to the index $\ell$ in $L^2(\mathbb{S}^2)$.
As a consequence, the discretization of this exact integral representation does not produce approximate discrete representations with bounded coefficients, as we will establish next.

Furthermore, not all solutions of the Helmholtz equation can be written in the form (\ref{herglotz}) for some $v \in L^2(\mathbb{S}^2)$.
For any sequence $\{\hat{u}_{\ell}^m\}_{(\ell,m) \in \mathcal{I}} \in \ell^2(\mathcal{I})$, it follows that $u=\sum_{\ell=0}^{\infty}\sum_{m=-\ell}^{\ell}\hat{u}_{\ell}^m b_{\ell}^m \in \mathcal{B}$.
If moreover $u$ admits an Herglotz representation in the form (\ref{herglotz}), then the coefficients $\{\hat{v}_{\ell}^m\}_{(\ell,m) \in \mathcal{I}}$ (\ref{v spherical harmonics}) of the spherical harmonics expansion of the density $v$, thanks to (\ref{to the herglotz spherical waves}), satisfy the relation $\hat{v}_{\ell}^m=(4 \pi)^{-1}i^{-\ell}\beta_{\ell}\hat{u}_{\ell}^m$ for $(\ell,m) \in \mathcal{I}$. For $v$ to belong to $L^2(\mathbb{S}^2)$, these coefficients would need to belong to $\ell^2(\mathcal{I})$ and this is possible only if the coefficients $\{\hat{u}_{\ell}^m\}_{(\ell,m) \in \mathcal{I}}$ decay super-exponentially, to compensate for the growth of $\{\beta_{\ell}\}_{\ell \geq 0}$, again by Lemma \hyperlink{Lemma 1.3}{1.4}.
For instance, propagative plane waves are not Herglotz functions because their coefficients do not decay quickly enough, as can be easily seen from the Jacobi-Anger identity (\ref{jacobi-anger}): in fact, the suitable Herglotz density $v$ for propagative plane waves would have to be a generalized function, i.e.\ the Dirac distribution centered in $\mathbf{d}$.

\section{Propagative plane wave sets are unstable}

\hypertarget{Section 3.3}{We} will now look at a model approximation problem to exemplify the numerical issues posed by propagative plane wave expansions. Specifically, we will examine the case of approximating a spherical wave $b_{\ell}^m$ for some $(\ell,m) \in \mathcal{I}$ using the sequence of approximation sets of propagative plane waves defined in (\ref{plane waves approximation set}). It is shown that the two requirements in (\ref{stable approximation}), namely accurate approximation and small coefficients, cannot both be met at the same time. Therefore, it is not possible to achieve stable approximations using propagative plane waves.

\begin{lemma}
\hypertarget{Lemma 3.1}{Let} $(\ell,m) \in \mathcal{I}$, $0 < \eta \leq 1$ and $P \in \N$ be given. The approximation set $\mathbf{\Phi}_P$ of propagative plane waves defined in \textup{(\ref{plane waves approximation set})} is such that, for every $\boldsymbol{\mu} \in \C^{P}$,
\begin{equation}
\|b_{\ell}^m-\mathcal{T}_{\mathbf{\Phi}_P}\boldsymbol{\mu}\|_{\mathcal{B}} \leq \eta \|b_{\ell}^m\|_{\mathcal{B}}\,\,\, \Rightarrow\,\,\, \|\boldsymbol{\mu}\|_{\ell^2} \geq \left(1-\eta \right)\frac{\beta_{\ell}}{2\sqrt{\pi(2\ell+1)}}\|b_{\ell}^m\|_{\mathcal{B}}.
\label{plane wave instability}
\end{equation}
\end{lemma}
\begin{proof}
Let $\boldsymbol{\mu} \in \C^{P}$. Using the Jacobi--Anger identity (\ref{jacobi-anger}), we obtain:
\begin{equation}
\left(\mathcal{T}_{\mathbf{\Phi}_P} \boldsymbol{\mu} \right)(\mathbf{x})=\frac{4 \pi}{\sqrt{P}}\sum_{p=1}^{P}\mu_p\sum_{q=0}^{\infty}i^q \sum_{n=-q}^q\overline{Y_q^n(\mathbf{d}_p)}\,\tilde{b}_q^n(\mathbf{x})=\sum_{q=0}^{\infty}\sum_{n=-q}^q c_q^n \tilde{b}_q^n(\mathbf{x}),
\label{alpaka 2}
\end{equation}
where the coefficients
\begin{equation*}
c_q^n:=\frac{4\pi i^q}{\sqrt{P}}\sum_{p=1}^{P}\mu_p\overline{Y_q^n(\mathbf{d}_p)},\,\,\,\,\,\,\,\,\,\,\,\,\,\,\forall (q,n) \in \mathcal{I},
\end{equation*}
thanks to \cite[Eq. (2.4.106)]{nedelec}, satisfy
\begin{equation}
|c_q^n|=\frac{4 \pi}{\sqrt{P}} \left|\sum_{p=1}^{P}\mu_p\overline{Y_q^n(\mathbf{d}_p)} \right|\leq \frac{2\sqrt{\pi(2q+1)}}{\sqrt{P}}\sum_{p=1}^{P}|\mu_p| \leq 2\sqrt{\pi(2q+1)}\|\boldsymbol{\mu}\|_{\ell^2}.
\label{alpaka}
\end{equation}Therefore, due to (\ref{alpaka 2}), the approximation error is
\begin{equation*}
\|b_{\ell}^m-\mathcal{T}_{\mathbf{\Phi}_P}\boldsymbol{\mu}\|^2_{\mathcal{B}}=\sum_{q=0}^{\infty}\sum_{n=-q}^q\left|\delta_{\ell, q}\delta_{m,n}-c_q^n\beta^{-1}_q\right|^2.
\end{equation*}
To get the error $\|b_{\ell}^m-\mathcal{T}_{\mathbf{\Phi}_P}\boldsymbol{\mu}\|_{\mathcal{B}}$ below the tolerance $\eta>0$, we need at least
\begin{equation*}
\left|\delta_{\ell, q}\delta_{m,n}-c_q^n\beta^{-1}_{\ell}\right| \leq \eta,\,\,\,\,\,\,\,\,\,\,\,\,\,\,\forall (q,n) \in \mathcal{I}.
\end{equation*}
Thanks to (\ref{alpaka}), for $(q,n)=(\ell,m)$, this reads
\begin{equation*}
\eta \geq \left|1-c_{\ell}^m\beta^{-1}_{\ell}\right| \geq 1-|c_{\ell}^m|\beta^{-1}_{\ell} \geq 1-2\beta^{-1}_{\ell}\sqrt{\pi(2\ell+1)}\|\boldsymbol{\mu}\|_{\ell^2},
\end{equation*}
which can be written as (\ref{plane wave instability}), recalling that $\|b_{\ell}^m\|_{\mathcal{B}}=1$.
\end{proof}

The bound states that in order to accurately approximate the spherical waves $b_{\ell}^m$ in the form of propagative plane wave expansions $\mathcal{T}_{\mathbf{\Phi}_P}\boldsymbol{\mu}$ with a given accuracy (i.e.\ small $\eta >0$), the norms of the coefficients must increase at least like $\sim {\ell}^{\ell}$, i.e.\ super-exponentially fast in $\ell$ (see Lemma \hyperlink{Lemma 1.3}{1.4}). It clearly emerges that the properties of accuracy and stability are in this case mutually exclusive.
In analogy with \cite[Sec.\ 4.3]{parolin-huybrechs-moiola}, we condense this result into the following theorem.

\begin{theorem}
\hypertarget{Theorem 3.1}{The} sequence of approximation set $\{\mathbf{\Phi}_P\}_{P \in \N}$ consisting of propagative plane waves with almost-evenly distributed directions as defined in \textup{(\ref{plane waves approximation set})} is not a stable approximation for $\mathcal{B}$.
\end{theorem}
\begin{proof}
The previous lemma shows that the sequence of spherical waves $\{b_{\ell}^m\}_{(\ell,m) \in \mathcal{I}}$ can not be stably approximated by the sequence of approximation set $\{\mathbf{\Phi}_P\}_{P \in \N}$ in the sense of Definition \hyperlink{Definition 2.1}{2.1}.
Indeed, let $(\ell,m) \in \mathcal{I}$ and suppose that there exists $P \in \N$ and $\mu \in \C^{P}$ such that $\|b_{\ell}^m-\mathcal{T}_{\mathbf{\Phi}_P}\boldsymbol{\mu}\|_{\mathcal{B}} \leq \eta \|b_{\ell}^m\|_{\mathcal{B}}$ for some $1 > \eta >0$. Then
\begin{equation*}
\|\boldsymbol{\mu}\|_{\ell^2} \geq \left(1-\eta \right)\frac{\beta_{\ell}}{2\sqrt{\pi(2\ell+1)}}\|b_{\ell}^m\|_{\mathcal{B}},
\end{equation*}
which implies that $\|\boldsymbol{\mu}\|_{\ell^2}$ can not be bounded uniformly with respect to $\ell$ in virtue of Lemma \hyperlink{Lemma 1.3}{1.4}.
Since the stability condition (\ref{stable approximation}) is not met, we can conclude that the sequence of approximation sets $\{\mathbf{\Phi}_P\}_{P \in \N}$ is unstable according to the Definition \hyperlink{Definition 2.1}{2.1}.
\end{proof}

\section{Modal analysis}

\hypertarget{Section 3.5}{Another} perspective on the same issue is given by the Jacobi--Anger identity (\ref{jacobi-anger}), because it allows us to get a quantitative insight into the modal content of propagative plane waves. For any $\mathbf{x} \in B_1$ and $\mathbf{d}=\mathbf{d}(\theta_1,\theta_2) \in \mathbb{S}^2$ we have
\begin{align*}
\phi_{\mathbf{d}}(\mathbf{x})&= \sum_{\ell=0}^{\infty}\sum_{m=-\ell}^{\ell} 4 \pi i^{\ell} \overline{Y_{\ell}^m(\mathbf{d})}\,\tilde{b}_{\ell}^m(\mathbf{x})\\
&=\sum_{\ell=0}^{\infty}\sum_{m=-\ell}^{\ell}\left[4\pi i^{\ell}\beta_{\ell}^{-1} \gamma_{\ell}^m e^{-im\theta_2}\mathsf{P}_{\ell}^m(\cos \theta_1)\right] b_{\ell}^m(\mathbf{x}). \numberthis \label{expansion plane wave}
\end{align*}
Note that the moduli of the coefficients
\begin{equation} \hat{\phi}_{\ell}^m(\theta_1):=\left|\left(\phi_{\mathbf{d}},b_{\ell}^m\right)_{\mathcal{B}} \right|=\frac{4\pi}{\beta_{\ell}}\gamma_{\ell}^m\left| \mathsf{P}_{\ell}^m(\cos \theta_1)\right|
\label{propagative coefficients}
\end{equation}
in the expansion (\ref{expansion plane wave}) depend on $(\ell,m) \in \mathcal{I}$ and $\theta_1 \in [0,\pi]$. However, thanks to (\ref{gamma constant}) and (\ref{negative legendre polynomials}), $\gamma_{\ell}^{-m}|\mathsf{P}_{\ell}^{-m}|=\gamma_{\ell}^{m}|\mathsf{P}_{\ell}^{m}|$ and therefore $\hat{\phi}_{\ell}^{-m}(\theta_1)=\hat{\phi}_{\ell}^m(\theta_1)$ for every $(\ell,m) \in \mathcal{I}$. Furthermore, due to the parity of the Ferrers functions \cite[Eq.\ (14.7.17)]{nist}, when analyzing the properties of the moduli of the coefficients, it is enough to consider only the case $\theta_1 \in \left[0,\pi/2 \right]$.

\begin{figure}
\centering
\begin{tabular}{cccc}
$\hat{\phi}_{\ell}^m(\pi/2)$ & $\hat{\phi}_{\ell}^m(\pi/4)$ & $\hat{\phi}_{\ell}^m(\pi/64)$ & $\hat{\phi}_{\ell}^m(0)$\\
\includegraphics[width=.2\linewidth,valign=m]{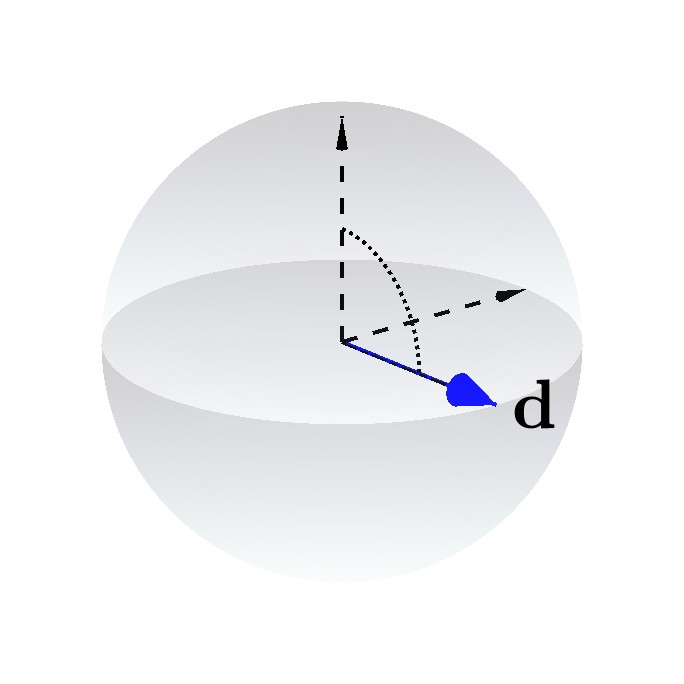} & \includegraphics[width=.2\linewidth,valign=m]{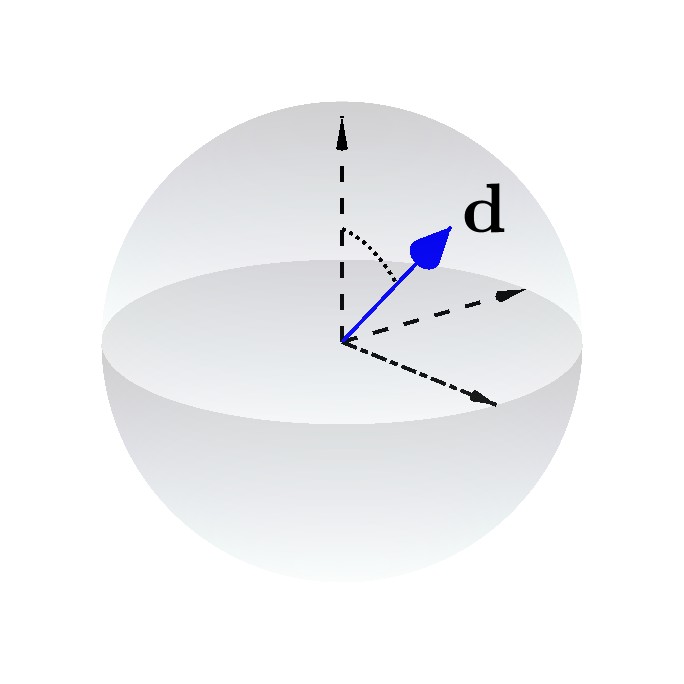} & \includegraphics[width=.2\linewidth,valign=m]{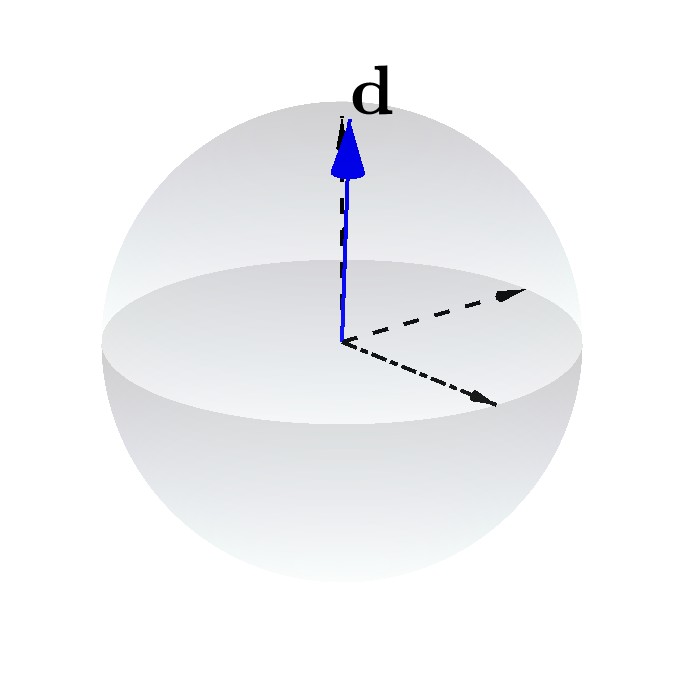} & \includegraphics[width=.2\linewidth,valign=m]{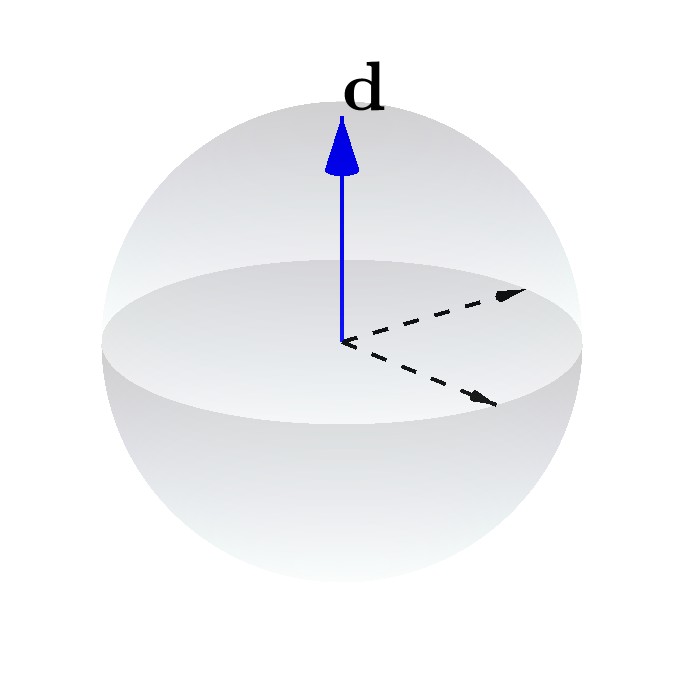}\\
\includegraphics[width=.2\linewidth,valign=m]{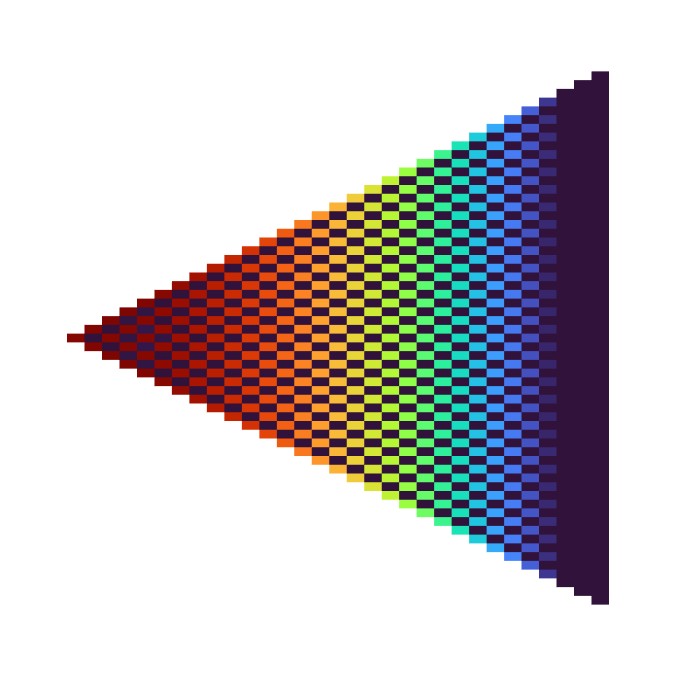} & \includegraphics[width=.2\linewidth,valign=m]{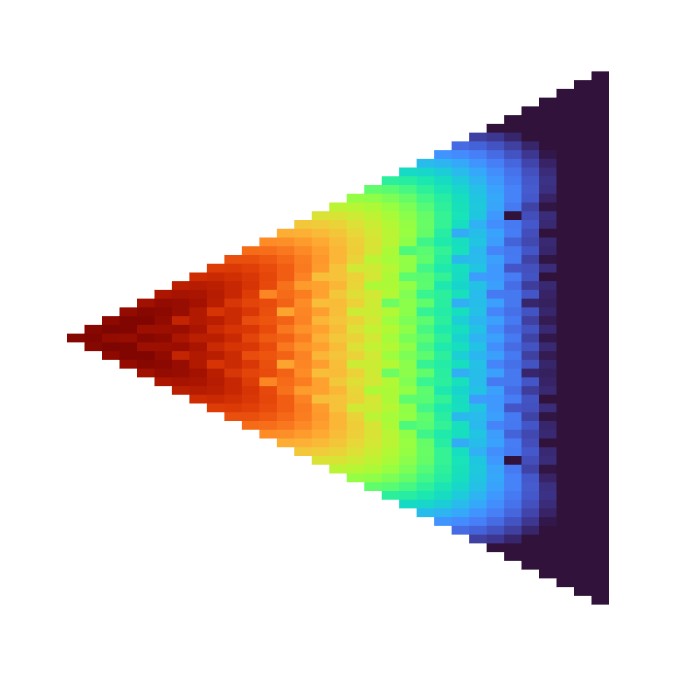} & \includegraphics[width=.2\linewidth,valign=m]{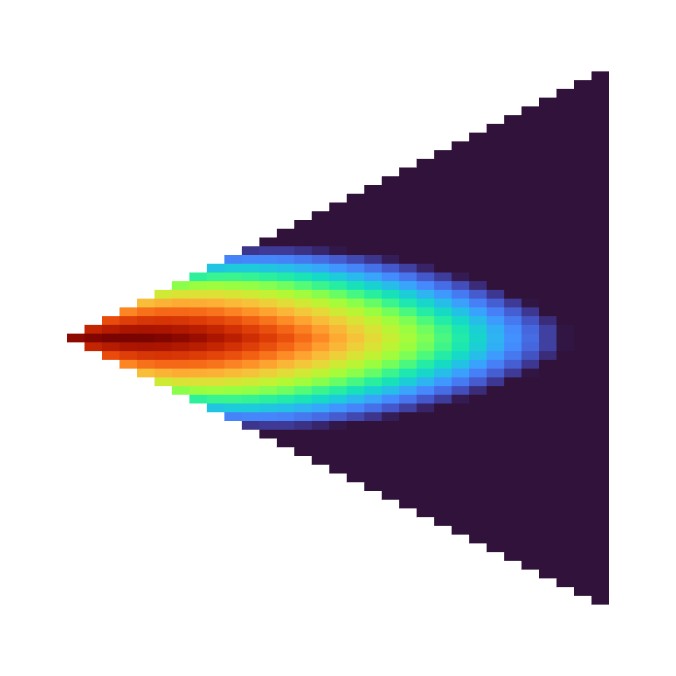} & \includegraphics[width=.2\linewidth,valign=m]{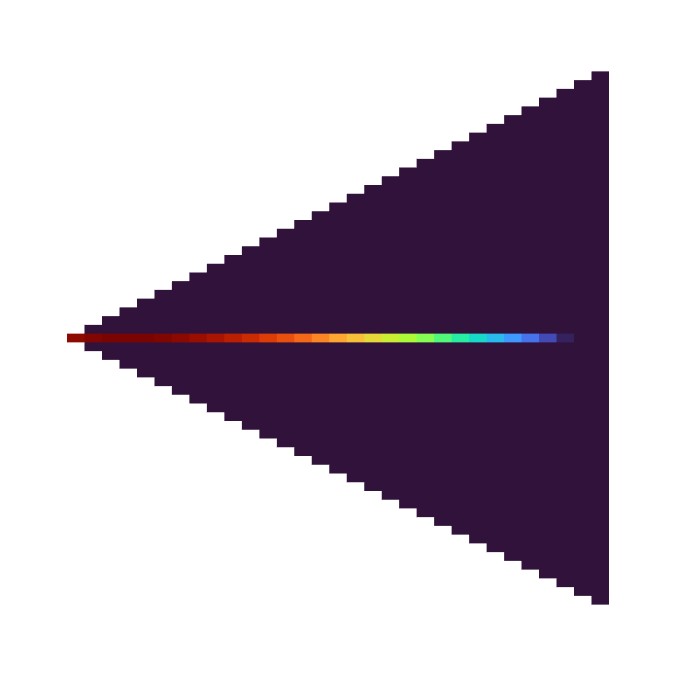}\\
\vspace{1mm}\\
\multicolumn{4}{c}{\includegraphics[width=0.85\linewidth]{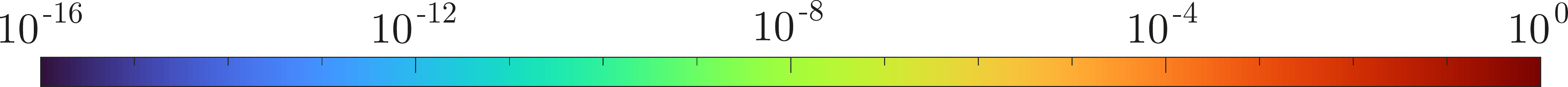}}\\
\end{tabular}
\caption{Modal analysis of the propagative plane waves: (above) representations of direction vectors $\mathbf{d}$ with fixed azimuthal angle $\theta_2=0$ and (below) related distributions of the coefficients $\hat{\phi}_{\ell}^m(\theta_1)$ in (\ref{propagative coefficients}) for different values of $\theta_1$. The index $\ell$ varies along the abscissa within the range $0 \leq \ell \leq 30$, while the index $m$ varies along the ordinate within the range $0 \leq |m| \leq \ell$ forming a triangle. Wavenumber $\kappa=6$.} \label{figure 3.1}
\end{figure}

It is evident that the modal distribution of each plane wave $\phi_{\mathbf{d}}$ depends on the vertical component of its propagation direction and hence on the zeros distribution of $\mathsf{P}_{\ell}^m$ \cite[Sec.\ 14.16.2]{nist}. For instance, $\hat{\phi}_{\ell}^m(\pi/2)=0$ every time $\ell+m$ is odd or, since $\mathsf{P}_{\ell}^{\ell}(\cos \theta_1)=(2\ell-1)!!\sin^{\ell} \theta_1$ \cite[Eq.\ (2.4.102)]{nedelec}, $\hat{\phi}_{\ell}^{\ell}(0)=0$ for every $\ell>0$. More generally, thanks to the definition of the Ferrers functions (\ref{legendre polynomials}), it follows that $\hat{\phi}_{\ell}^m(0)=0$ for every $m \neq 0$. Therefore, if $\theta_1 \approx 0$, then the closer $|m|$ gets to $\ell$ the closer $\hat{\phi}_{\ell}^m(\theta_1)$ gets to $0$. Some distributions of the coefficients (\ref{propagative coefficients}) are depicted in Figure \ref{figure 3.1}.

For instance, if we want to approximate a spherical wave $b_{\ell}^{\ell}$ with $\ell>0$, we will need many more propagative plane waves with `horizontal' rather than `vertical' directions. Conversely, for a propagative mode $b_{\ell}^0$ with an odd $\ell$, many more plane waves with `vertical' rather than `horizontal' directions will be needed.

However, what unites all plane waves (\ref{propagative wave}), regardless of their direction of propagation $\mathbf{d} \in \mathbb{S}^2$, is given by the fact that the coefficients $\hat{\phi}_{\ell}^m(\theta_1)$ decay super-exponentially fast in the evanescent-mode regime $\ell \gg \kappa$.
The direct result is that propagative plane waves are not suited for approximating Helmholtz solutions that have a high-$\ell$ Fourier modal content, as previously shown in Lemma \hyperlink{Lemma 3.1}{3.3}.
This follows directly from the images in Figure \ref{figure 3.1}.

\begin{figure}
\centering
\includegraphics[width=6.5cm]{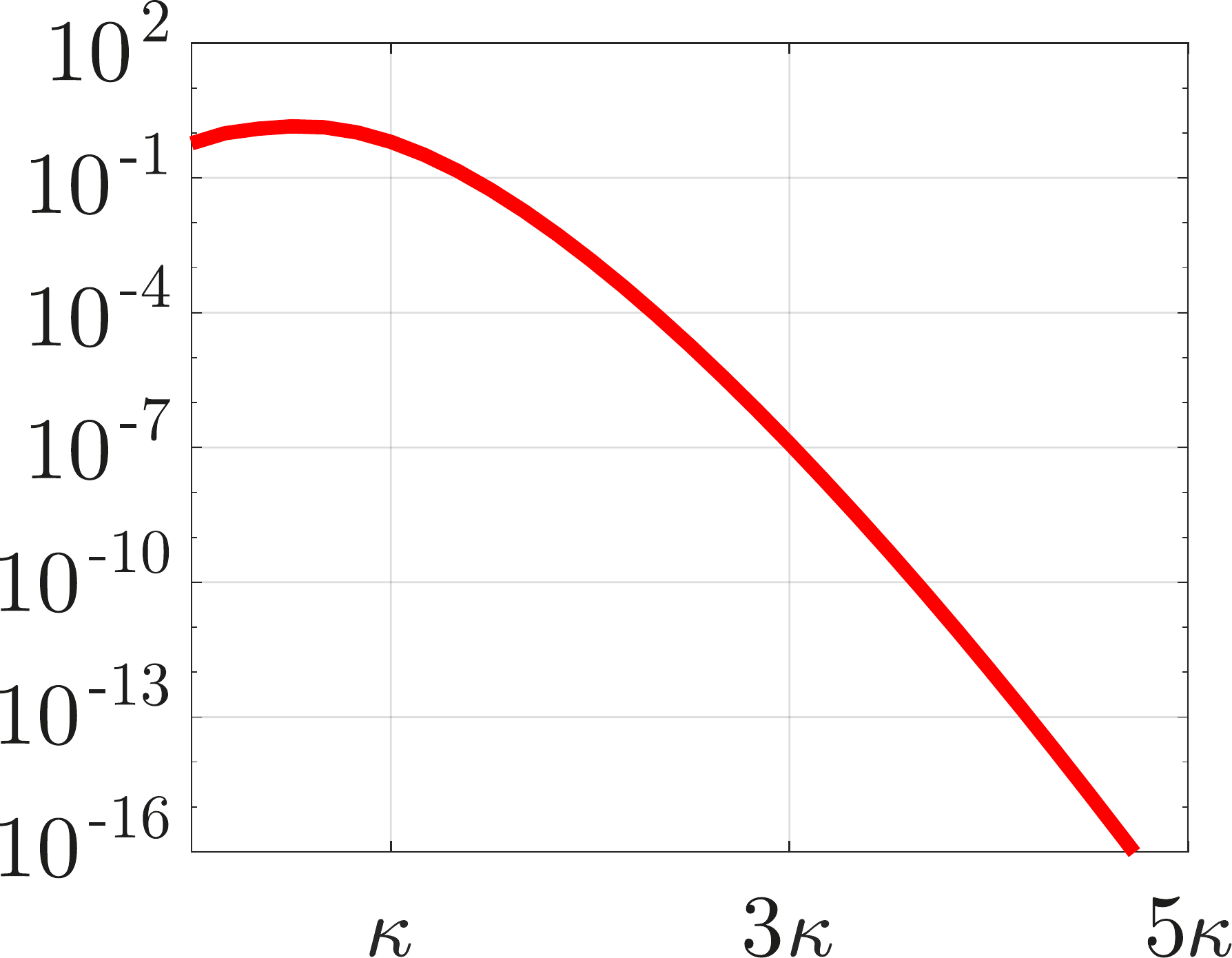}
\caption{Modal analysis of the propagative plane waves: distribution of the coefficients $\hat{\phi}_{\ell}$ in (\ref{l2 coefficients}).
This corresponds to taking the $\ell^2$-norms along vertical segments of the values in the triangles in Figure \ref{figure 3.1}.
Observe that these coefficients decay super-exponentially fast in the evanescent regime $\ell \gg \kappa$. Wavenumber $\kappa=6$.}
\label{figure 3.2}
\end{figure}

Furthermore, this can be seen more clearly by defining:
\vspace{-2mm}
\begin{equation}
\begin{split}
\tilde{b}_\ell[\mathbf{d}]&:=\sum_{m=-\ell}^{\ell}\left(\phi_{\mathbf{d}},b_{\ell}^m\right)_{\mathcal{B}}b_{\ell}^m,\,\,\,\,\,\,\,\,\,\,\,\,\,\,\,\,\,\,\,\,\,\,\,\,  \forall \ell \geq 0,\,\forall \mathbf{d} \in \mathbb{S}^2,\\
b_\ell[\mathbf{d}]&:=\hat{\phi}_{\ell}^{-1}\tilde{b}_\ell[\mathbf{d}],\,\,\,\,\,\,\,\,\,\,\,\,\,\,\,\,\,\,\,\,\,\,\,\,\,\,\,\,\,\,\,\,\,\,\,\,\,\,\,\,\,\,\,\,\,\,\,\,\,\, \hat{\phi}_{\ell}:=\big\|\tilde{b}_\ell[\mathbf{d}]\big\|_{\mathcal{B}}.
\end{split}
\label{dfg}
\end{equation}
In fact, thanks to (\ref{dfg}), we can write
\vspace{-1.5mm}
\begin{equation}
\phi_{\mathbf{d}}=\sum_{\ell=0}^{\infty}\sum_{m=-\ell}^{\ell}\left(\phi_{\mathbf{d}},b_{\ell}^m\right)_{\mathcal{B}}b_{\ell}^m=\sum_{\ell=0}^{\infty}\hat{\phi}_{\ell}\,b_{\ell}[\mathbf{d}],
\label{expansion2 propagative}
\end{equation}
\vspace{-1.5mm}
where $b_{\ell}[\mathbf{d}] \in \text{span}\{b_{\ell}^m\}_{m=-\ell}^{\ell}$ with $\|b_{\ell}[\mathbf{d}]\|_{\mathcal{B}}=1$ and
\begin{equation}
\hat{\phi}_{\ell}=\left(\sum_{m=-\ell}^{\ell}\left[\hat{\phi}_{\ell}^m(\theta_1)\right]^2\right)^{1/2}\!\!\!\!\!=\frac{4 \pi}{\beta_{\ell}}\left(\sum_{m=-\ell}^{\ell}\left|Y_{\ell}^m(\mathbf{d}) \right|^2\right)^{1/2}\!\!\!\!\!=\frac{2\sqrt{\pi(2\ell+1)}}{\beta_{\ell}}.
\label{l2 coefficients}
\end{equation}
Note that in (\ref{l2 coefficients}) the second equality holds due to (\ref{propagative coefficients}), while the last one thanks to \cite[Eq.\ (2.4.105)]{nedelec}.
Up to a multiplicative constant of $i^{\ell}$, the functions $b_{\ell}\left[\mathbf{d}\right]$ coincide with the spherical waves $b_{\ell}^0$ rotated according to the propagative plane wave direction $\mathbf{d} \in \mathbb{S}^2$. This can be readily checked thanks to the definitions (\ref{spherical harmonic}) and (\ref{b tilde definizione}) along with the identities (\ref{wigner property}) and (\ref{D-matrix0}), which involve the so-called \textit{Wigner matrices} (see Section \hyperlink{Section 4.3}{4.3}).
For instance, if $\mathbf{d} \in \mathbb{S}^2$ is the upward direction, from (\ref{expansion plane wave}) and (\ref{dfg}) it follows that $b_{\ell}[\mathbf{d}]=i^{\ell}b_{\ell}^0$.
The coefficients $\hat{\phi}_{\ell}$ in (\ref{expansion2 propagative}) are independent of $\mathbf{d}$ since $\{Y_{\ell}^m\}_{|m| \leq \ell}$ is an orthonormal basis of $\mathcal{Y}_{\ell}$ -- the space of spherical harmonic and homogeneous polynomials of degree $\ell$ -- and furthermore the sphere, the homogeneous polynomials and the Laplace operator are all invariant by rotation. The distribution of the coefficients $\hat{\phi}_{\ell}$ is depicted in Figure \ref{figure 3.2}: note the super-exponentially fast decay in the evanescent regime $\ell \gg \kappa$.

\section{Numerical experiments}

\hypertarget{Section 3.4}{The} instability result of Lemma \hyperlink{Lemma 3.1}{3.3} can be confirmed through numerical experiments. Let us examine again the problem of the approximation of the spherical wave $b_{\ell}^m$ for some $(\ell,m) \in \mathcal{I}$ by a propagative plane waves approximation set $\mathbf{\Phi}_P$ defined in (\ref{plane waves approximation set}).
As anticipated, in this section we use extremal systems of points (see Definition \hyperlink{Extremal points system}{2.2}) to describe both the propagation directions of the plane waves in $\mathbf{\Phi}_P$ and the sampling points $\{\mathbf{x}_s\}_{s=1}^S \subset \partial B_1$. The associated weights $\mathbf{w}_S \in \R^S$ are computed solving the linear system (\ref{linear system cubature}). Therefore, recalling Remark \hyperlink{Remark 2.3}{2.3} and Remark \hyperlink{Remark 3.2}{3.2}, the numerical results presented are based on using the smallest square integer greater than or equal to $P$ as the approximation set dimension.
In analogy with \cite[Sec.\ 4.4]{parolin-huybrechs-moiola}, we choose $S=\lceil \sqrt{2|\mathbf{\Phi}_P|}\rceil^2$.
The sampling matrix $A$ and the right-hand side $\mathbf{b}$ are defined according to (\ref{A matrix definition}), where $u=b_{\ell}^m$ for some $(\ell,m) \in \mathcal{I}$.

\begin{figure}
\centering
\includegraphics[width=8.5cm]{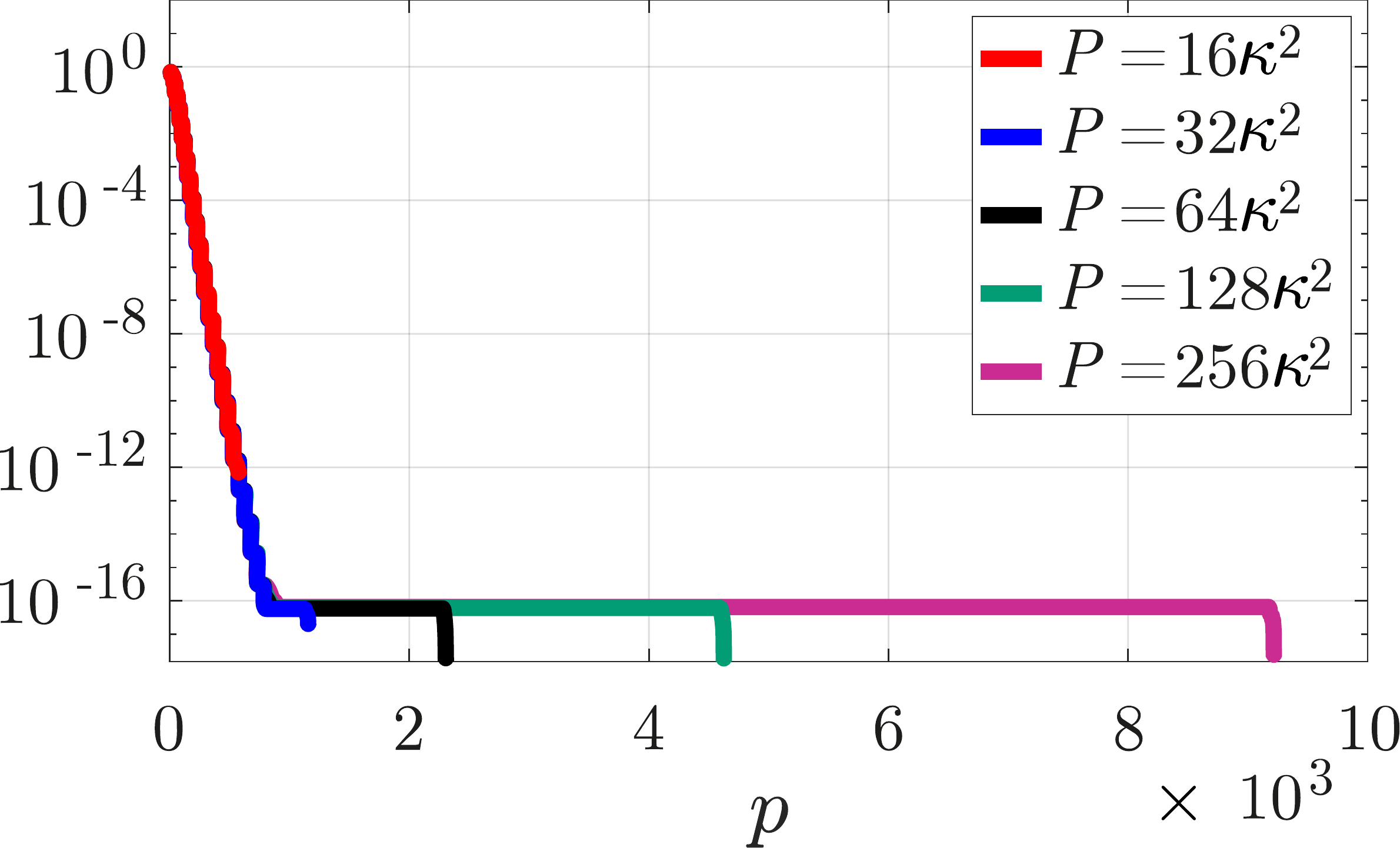}
\caption{Singular values $\{\sigma_p\}_{p}$ of the matrix $A$ using
propagative plane wave approximation sets (\ref{plane waves approximation set}).
Observe that the number of singular values above a fixed threshold does not increase when $P$ is raised. Wavenumber $\kappa=6$.}
\label{figure 3.3}
\end{figure}

The matrix $A$ is known to be ill-conditioned, as its condition number (the ratio of the largest singular value $\sigma_{\textup{max}}$ over the smallest one $\sigma_{\textup{min}}$) increases exponentially with the number of plane waves in the approximation set $\mathbf{\Phi}_P$, as can be inferred from Figure \ref{figure 3.3}. This fact is not unique to the sampling method, and can be observed in similar experiments in \cite[Sec.\ 4.3]{hiptmair-moiola-perugia1} for the mass matrix of a Galerkin formulation in a Cartesian geometry, for the case $S=|\mathbf{\Phi}_P|$. The least-squares formulation also has an even worse condition number. Therefore, in our subsequent numerical experiments, we will use the regularization technique outlined in Section \hyperlink{Section 2.2}{2.2} with threshold parameter $\epsilon=10^{-14}$.

\begin{figure}
\centering
\includegraphics[width=.82\linewidth]{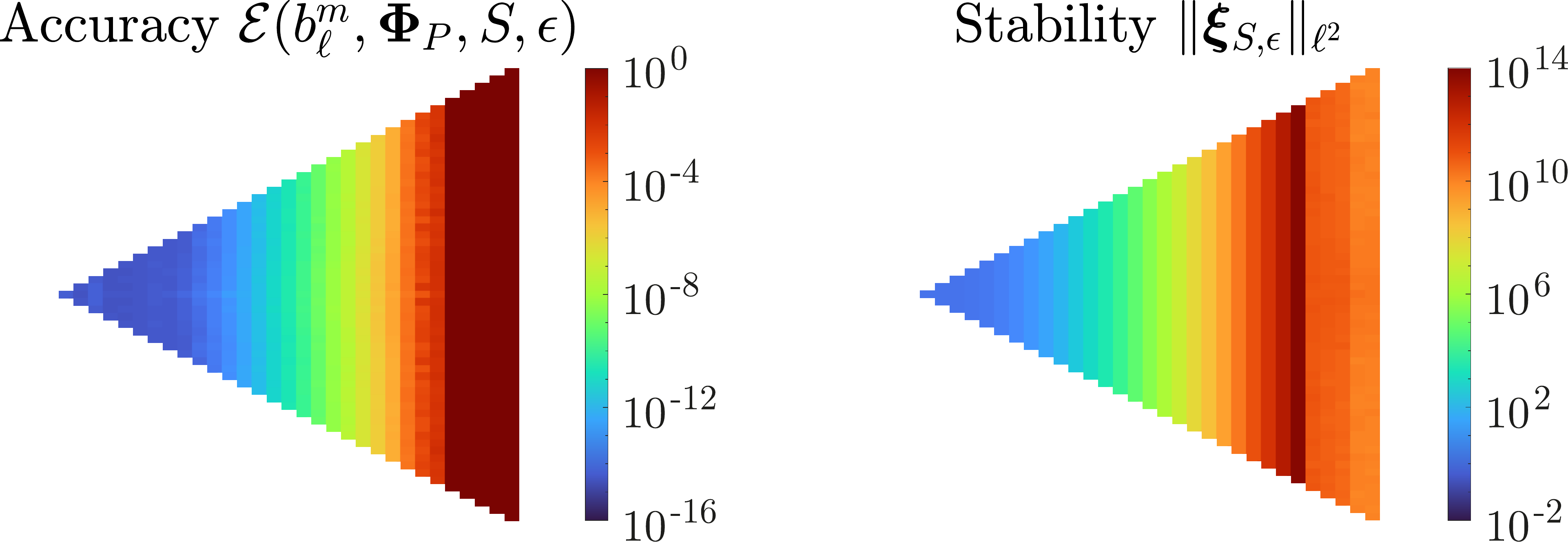}
\caption{Accuracy $\mathcal{E}$ as defined in (\ref{relative residual}) (left) and stability $\|\boldsymbol{\xi}_{S,\epsilon}\|_{\ell^2}$ (right) of the approximation of spherical waves $b_{\ell}^m$ by propagative plane waves. The index $\ell$ varies along the abscissa within the range $0 \leq \ell \leq 5\kappa$, while the index $m$ varies along the ordinate within the range $0 \leq |m| \leq \ell$ forming a triangle.
DOF budget $P=64\kappa^2$, wavenumber $\kappa=6$ and regularization parameter $\epsilon=10^{-14}$.} \label{figure triangle}
\end{figure}

As shown in Figure \ref{figure triangle} for the particular case where $P=64\kappa^2$, the mode number $m$ is irrelevant, since, for fixed $\ell$, the outcomes do not vary significantly as the order $m$ changes. Therefore, we will only examine the case where $m=0$ in the following.
Here the same matrix $A$ is used to approximate all the $b_{\ell}^m$’s for any $(\ell,m) \in \mathcal{I}$ up to $\ell = 5\kappa$.
On the left panel we report the relative residual $\mathcal{E}$ defined in (\ref{relative residual}) as a measure of the accuracy of the approximation. On the right panel we report the size of the coefficients, namely $\|\boldsymbol{\xi}_{S,\epsilon}\|_{\ell^2}$ as a measure of the stability of the approximation.
Other numerical results, with fixed $m=0$ and various choice of $P \in \N$, are reported with this same layout in Figure \ref{figure 3.4}.
We observe three regimes:
\begin{itemize}
\item For the propagative modes, which are those corresponding to spherical waves with mode number $\ell \leq \kappa$, the approximation is accurate ($\mathcal{E}<10^{-13}$) and the size of the coefficients is moderate ($\|\boldsymbol{\xi}_{S,\epsilon}\|_{\ell^2} < 10$).
\vspace{-0.1mm}
\item For mode numbers $\ell$ that are roughly larger than the wavenumber $\kappa$, the norms of the coefficients of the computed approximations blow up exponentially and the accuracy decreases proportionally.
\vspace{-0.1mm}
\item At a certain point (roughly between $\ell=4\kappa$ and $\ell=5\kappa$ in this specific numerical experiment), the exponential growth of the coefficients completely destroys the stability of the approximation and we are unable to approximate the target $b_{\ell}^0$ with any significant accuracy. When the relative error is of the order of $\mathcal{O}(1)$, the size of the coefficients reported is not meaningful. In fact, taking $\boldsymbol{\xi}_{S,\epsilon}$ identically zero would provide a similar error.
\end{itemize}
As in \cite[Sec.\ 4.4]{parolin-huybrechs-moiola}, even in three dimensions, increasing $P$ does not improve the accuracy beyond a certain point. In fact, Figure \ref{figure 3.3} shows that the $\epsilon$\textit{-rank} (the number of singular values larger than $\epsilon$) of the matrix $A$ does not increase as $P$ is raised.
Although increasing $P$ does not improve accuracy, it does not worsen the numerical instability any further. This is true despite the blow up with respect to $P$ of the condition number of the matrix $A$, that follows from Figure \ref{figure 3.3}.

As in Figure \ref{figure triangle}, also in Figure \ref{figure 3.4} for fixed $P$, the same matrix $A$ is used to approximate all the $b_{\ell}^0$’s for any mode number $\ell$ up to $\ell=5\kappa$ (i.e.\ to compute all markers of the same color). Even when the matrix $A$ is extremely ill-conditioned (for example, when $P=256 \kappa^2$ in the numerical experiments presented), we still get almost machine-precision accuracy for all propagative modes $\ell \leq \kappa$, while maintaining an error of order $\mathcal{O}(1)$ for evanescent modes with larger mode number $\ell \geq 4\kappa$. The simple regularization technique outlined in Section \hyperlink{Section 2.2}{2.2} enables us to obtain these results.

Any regularization technique can reduce but not eliminate the inherent instability of Trefftz methods that use propagative plane waves. Even with regularization, it remains impossible to achieve accurate approximation of evanescent modes within a given floating-point precision. Similarly to the two-dimensional case \cite{parolin-huybrechs-moiola}, the main goal of this paper is to create a discrete space of plane waves that enables stable approximation of all modes, for application in Trefftz schemes.
\begin{figure}
\centering
\includegraphics[width=\linewidth]{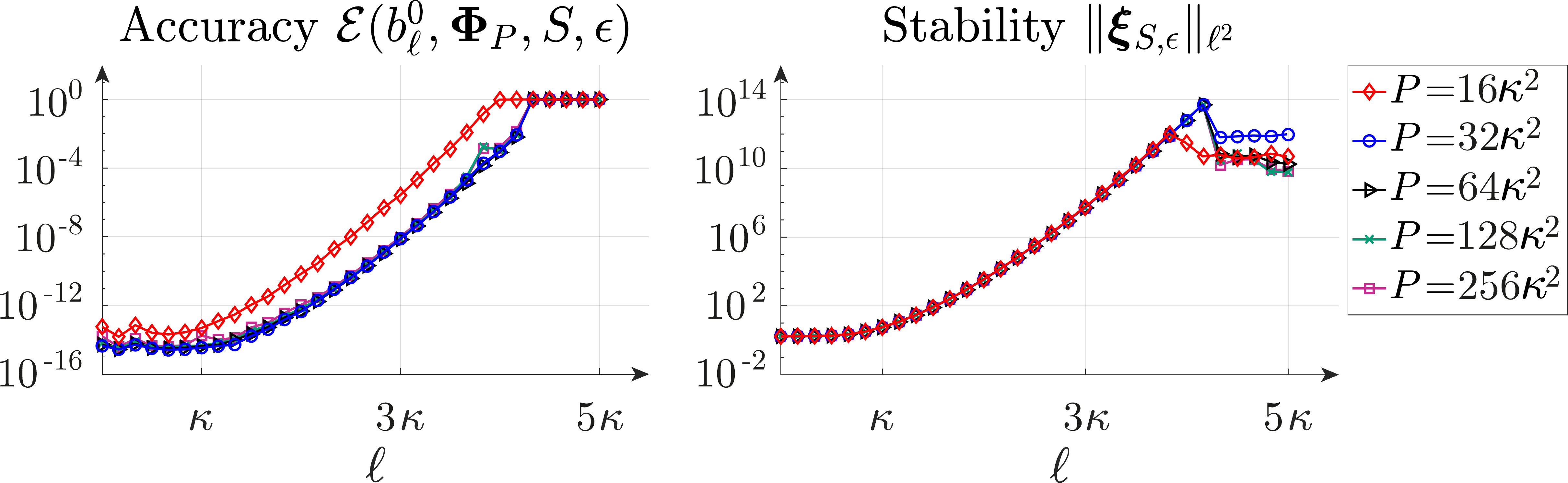}
\caption{Accuracy $\mathcal{E}$ as defined in (\ref{relative residual}) (left) and stability $\|\boldsymbol{\xi}_{S,\epsilon}\|_{\ell^2}$ (right) of the approximation of spherical waves $b_{\ell}^0$ by propagative plane waves. Wavenumber $\kappa=6$ and regularization parameter $\epsilon=10^{-14}$.}
\label{figure 3.4}
\end{figure}

\chapter{Evanescent plane waves}
\hypertarget{Chapter 4}{The} main purpose of this chapter is to present evanescent plane waves, which have a defining direction vector $\mathbf{d} \in \C^3$ instead of the propagative ones with $\mathbf{d} \in \R^3$ and to provide some understanding of why they are expected to have improved stability properties. Propagative and evanescent plane waves are sometimes referred to as homogeneous and inhomogeneous plane waves, respectively, since only the former have constant amplitude.
Evanescent plane waves oscillate with an apparent wavenumber larger than $\kappa$ in the direction of propagation, which is parallel to the vector $\Re\{\mathbf{d}\}$, and decay exponentially in the direction parallel to $\Im\{\mathbf{d}\}$. These two directions are mutually orthogonal.
We first provide the definition of evanescent plane wave by introducing the new \textit{evanescence parameters}: $\theta_3 \in [0,2\pi)$, which determines the decay direction, and $\zeta \in [0,+\infty)$, which regulates the magnitude of both the apparent wavenumber and the decay rate.
Then, in analogy with \cite[Sec.\ 5.2]{parolin-huybrechs-moiola}, we present the modal analysis of evanescent plane waves, generalizing the Jacobi--Anger identity (\ref{jacobi-anger}) to the complex field: what can be observed is that, in contrast to the propagative case, we can move the Fourier content of the plane waves to higher-mode regimes by adjusting the evanescence parameters $\theta_3$ and $\zeta$. As a result, we anticipate that evanescent plane waves are better suited for capturing the higher Fourier modes of Helmholtz solutions that are less regular, for instance in presence of close-by singularities.

\section{Definition}

\hypertarget{Section 4.1}{We} will now introduce the concept of \textit{evanescent plane wave}. It is easy to see that, in order to define a plane wave having the form $e^{i\kappa \mathbf{d}\cdot \mathbf{x}}$ and satisfying the Helmholtz equation (\ref{Helmholtz equation}), we need a direction vector $\mathbf{d} \in \C^3$ such that $\mathbf{d} \cdot \mathbf{d}=\sum_{i=1}^3\textup{d}_i^2=1$, i.e.\
\begin{equation}
\left|\Re\{\mathbf{d}\}\right|^2-\left|\Im\{\mathbf{d}\}\right|^2=1\,\,\,\,\,\,\,\,\,\,\,\,\,\,\text{and}\,\,\,\,\,\,\,\,\,\,\,\,\,\,\Re\{\mathbf{d}\} \cdot \Im\{\mathbf{d}\}=0.
\label{complex direction conditions}
\end{equation}
Our approach involves fixing a reference complex direction vector that satisfies the conditions (\ref{complex direction conditions}) and then taking all its possible rotations in space. Suppose, for instance, that its real and imaginary parts are non-negative and parallel to the $z$-axis and $x$-axis, respectively. Then the first equation in (\ref{complex direction conditions}) becomes $\Re\{\textup{d}_3\}^2-\Im\{\textup{d}_1\}^2=1$ and, defining $z:=\Re\{\textup{d}_3\}$, we get $\Im\{\textup{d}_1\}=\sqrt{z^2-1}$. Observe that, since $\Im\{\mathbf{d}\}$ is real, then we need $z \geq 1$. Therefore, for every $z \geq 1$, we define the \textit{reference upward complex direction vector} $\mathbf{d}_{\uparrow}(z)$ as
\begin{equation}
\mathbf{d}_{\uparrow}(z):=(i\sqrt{z^2-1},0,z) \in \C^3.
\label{reference complex direction}
\end{equation}
We are now ready to provide the definition of evanescent plane waves, along with a surjective parametrization of the complex-direction space $\{\mathbf{d} \in \C^3:\mathbf{d}\cdot \mathbf{d}=1\}$.
\begin{definition}[\hypertarget{Definition 4.1}{Evanescent plane wave parametrization}]
Let $\boldsymbol{\theta}:=(\theta_1,\theta_2,\theta_3) \in \Theta:=[0,\pi] \times [0,2\pi) \times [0,2\pi)$ be the Euler angles and $R_{\boldsymbol{\theta}}$ the associated rotation matrix defined according to the convention (z-y-z), namely $R_{\boldsymbol{\theta}}:=R_{z}(\theta_2)R_{y}(\theta_1)R_{z}(\theta_3)$, where
\begin{equation*}
R_{y}(\theta):=
\begin{bmatrix}
    \cos{(\theta)}       & 0 & \sin{(\theta)}\\
    0       & 1 & 0\\
    -\sin{(\theta)}       & 0 & \cos{(\theta)}
\end{bmatrix},\,\,\,\,\,\,\,
R_{z}(\theta):=
\begin{bmatrix}
    \cos{(\theta)}       & -\sin{(\theta)} & 0\\
    \sin{(\theta)}       & \cos{(\theta)} & 0\\
    0       & 0 & 1
\end{bmatrix}.
\end{equation*}
Furthermore, for every $z \geq 1$, let $\mathbf{d}_{\uparrow}(z)$ the reference upward complex direction vector \textup{(\ref{reference complex direction})}. For any $\mathbf{y}:=(\boldsymbol{\theta},\zeta) \in \Theta \times [0,+\infty)$, we let
\begin{equation}
\phi_{\mathbf{y}}(\mathbf{x}):=e^{i\kappa \mathbf{d}(\mathbf{y})\cdot \mathbf{x}},\,\,\,\,\,\,\,\,\,\,\,\,\,\,\forall \mathbf{x} \in \R^3,
\label{evanescent wave}
\end{equation}
where the complex-valued direction of the wave is given by
\begin{equation}
\mathbf{d}(\mathbf{y}):=R_{\boldsymbol{\theta}}\,\mathbf{d}_{\uparrow}\left(\frac{\zeta}{2\kappa}+1 \right) \in \C^3.
\label{complex direction}
\end{equation}
\end{definition}
\color{black} An explicit definition of the rotation matrix $R_{\boldsymbol{\theta}}$ is given by:
\begin{equation}
R_{\boldsymbol{\theta}}=
\begin{bmatrix}
    c_1c_2c_3-s_2s_3 & -c_1c_2s_3-s_2c_3 & s_1c_2\\
    c_1s_2c_3+c_2s_3 & -c_1s_2s_3+c_2c_3 & s_1s_2\\
    -s_1c_3 & s_1s_3 & c_1
\end{bmatrix},
\label{rotation matrix R}
\end{equation}
where we use the shorthand notation $c_i:=\cos{(\theta_i)}$ and $s_i:=\sin{(\theta_i)}$ for $i=1,2,3$. Therefore, for every $\mathbf{y} \in \Theta \times [0,+\infty)$, $\mathbf{d}(\mathbf{y})$ can be rewritten as
\begin{equation}
\mathbf{d}(\mathbf{y})=z\begin{pmatrix}
    s_1c_2\\
    s_1s_2\\
    c_1
\end{pmatrix}+i\sqrt{z^2-1}\begin{pmatrix}
    c_1c_2c_3-s_2s_3\\
    c_1s_2c_3+c_2s_3\\
    -s_1c_3
\end{pmatrix},
\label{complex direction 2}
\end{equation}
where $z=\zeta/2\kappa+1$.
It can be easily verified that the evanescent plane wave satisfies the homogeneous Helmholtz equation (\ref{Helmholtz equation}), as $\mathbf{d}(\mathbf{y}) \cdot \mathbf{d}(\mathbf{y})=1$ for any $\mathbf{y} \in \Theta \times [0,+\infty)$ by design, due to the fact that the rotation matrix $R_{\boldsymbol{\theta}}$ in (\ref{rotation matrix R}) is unitary.
In fact, $\mathbf{d}(\mathbf{y})$ satisfies the conditions (\ref{complex direction conditions}) and furthermore $\left|\Re\{\mathbf{d}(\mathbf{y})\}\right|=z$ and $\left|\Im\{\mathbf{d}(\mathbf{y})\}\right|=\sqrt{z^2-1}$ for every $\mathbf{y} \in \Theta \times [0,+\infty)$, where $z=\zeta/2\kappa+1$.

Observe that we chose to parameterize $\mathbf{d}(\mathbf{y})$ in (\ref{complex direction}) using $z=\zeta/2\kappa+1$ with $\zeta \in [0,+\infty)$. Although this choice may not be immediately apparent, it leads to simpler definitions and propositions in the subsequent analysis.

\begin{figure}
     \centering
     \begin{subfigure}[b]{0.3\textwidth}
         \centering
         \includegraphics[trim=125 125 5 5,clip,width=4.5cm,height=4.5cm]{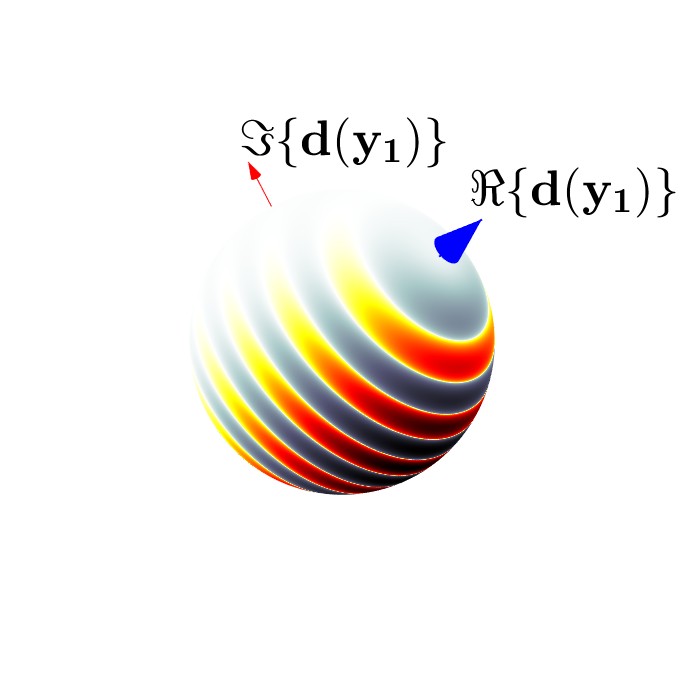}
     \end{subfigure}
     \begin{subfigure}[b]{0.3\textwidth}
         \centering
         \includegraphics[trim=125 125 5 5,clip,width=4.5cm,height=4.5cm]{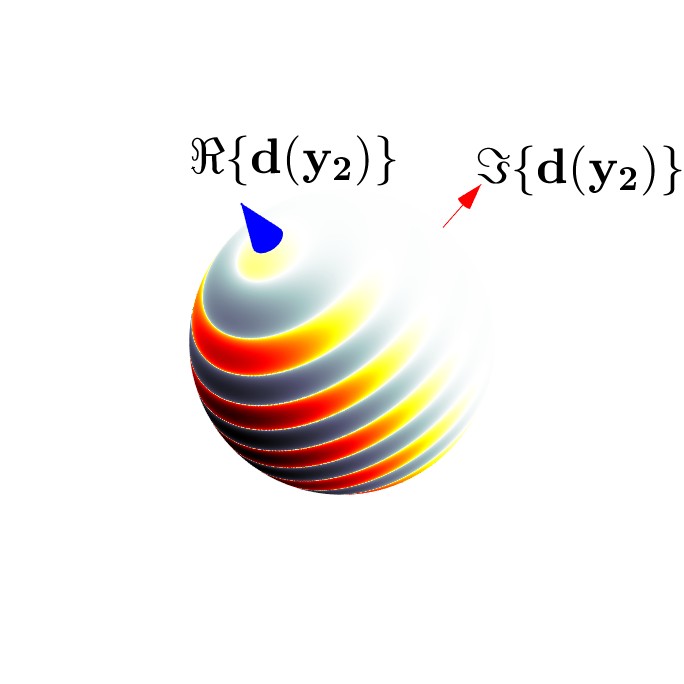}
     \end{subfigure}
     \begin{subfigure}[b]{0.3\textwidth}
         \centering
         \includegraphics[trim=120 120 5 5,clip,width=4.5cm,height=4.5cm]{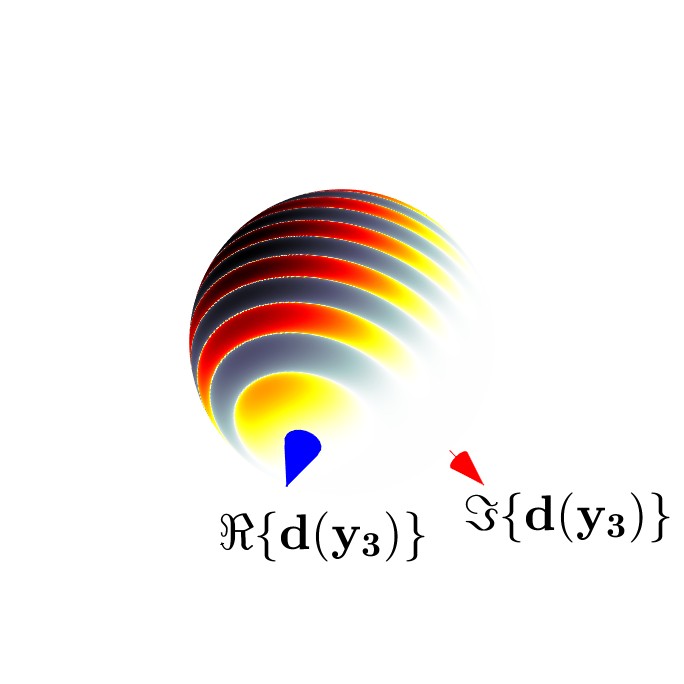}
     \end{subfigure}
\caption{Real part of three different evanescent plane waves $\mathbf{\phi}_{\mathbf{y}}$ restricted to $\partial B_1$.
The blue arrow indicates the direction of the real part $\Re\{\mathbf{d}(\mathbf{y})\}$, i.e.\ $\mathbf{d}(\theta_1,\theta_2)$ defined in (\ref{propagative direction}), while the red arrow indicates the direction of the imaginary part $\Im\{\mathbf{d}(\mathbf{y})\}$, i.e.\ $\mathbf{d}^{\bot}(\boldsymbol{\theta})$ that is the first column of $R_{\boldsymbol{\theta}}$ in (\ref{rotation matrix R}).
The size of the arrowheads is proportional to the norm of the vector and therefore dependent on $\zeta$ and $\kappa$. Starting from left to right, for $\mathbf{y}=(\boldsymbol{\theta},\zeta)$, we have respectively: $\mathbf{y}_1=(\pi/4,\pi/4,7\pi/8,2)$, $\mathbf{y}_2=(\pi/5,8\pi/5,\pi,4)$ and $\mathbf{y}_3=(5\pi/8,9\pi/5,\pi/5,8)$. Wavenumber $\kappa=16$.}
\label{figure 4.1}
\end{figure}

The choice of Euler angles and matrices in Definition \hyperlink{Definition 4.1}{4.1} is taken from \cite[Eqs. (4) and (5)]{pendleton}, with the difference that we change the signs of the angles to guarantee consistency with the notation adopted for the propagative waves. In fact, if we assume $\zeta=0$, for any $\boldsymbol{\theta} \in \Theta$, we recover the usual propagative plane wave of Definition \hyperlink{propagative plane wave}{3.1} with real direction $\mathbf{d}(\theta_1,\theta_2)$ in (\ref{propagative direction}): in this case the wave direction turns out to be independent of the new angular parameter $\theta_3$, since any rotation $R_z(\theta_3)$ around the vertical axis sends $\mathbf{d}_{\uparrow}(1)=(0,0,1)$ into itself.

Since the direction vector $\mathbf{d}(\mathbf{y})$ in (\ref{complex direction 2}) is complex, the wave behavior might be unclear. A more explicit expression of the evanescent plane wave in (\ref{evanescent wave}) is
\begin{equation*}
\phi_{\mathbf{y}}(\mathbf{x})=e^{i \left(\frac{\zeta}{2}+\kappa \right) \mathbf{d}(\theta_1,\theta_2) \cdot \mathbf{x}}e^{- \left( \zeta \left(\frac{\zeta}{4}+\kappa \right) \right)^{1/2} \mathbf{d}^{\bot}(\boldsymbol{\theta}) \cdot \mathbf{x}},
\end{equation*}
where $\mathbf{d}(\theta_1,\theta_2)$ is defined in (\ref{propagative direction}) and $\mathbf{d}^{\bot}(\boldsymbol{\theta})$ is the first column of the matrix $R_{\boldsymbol{\theta}}$ in (\ref{rotation matrix R}). We see from this formula that the wave oscillates with apparent wavenumber $\zeta/2+\kappa \geq \kappa$ in the propagation direction $\mathbf{d}(\theta_1,\theta_2)$, which is parallel to $\Re\{\mathbf{d}(\mathbf{y})\}$. In addition, the wave decays exponentially with rate $\left(\zeta(\zeta/4+\kappa) \right)^{1/2}$ in the direction $\mathbf{d}^{\bot}(\boldsymbol{\theta})$, which is parallel to $\Im\{\mathbf{d}(\mathbf{y})\}$ and thus orthogonal to $\mathbf{d}(\theta_1,\theta_2)$. Therefore, the decay direction is orthogonal to the propagation one and this is confirmed by the fact that $\mathbf{d}(\theta_1,\theta_2)$ coincides with the third column of $R_{\boldsymbol{\theta}}$ in (\ref{rotation matrix R}), which is unitary.
This justifies naming the new parameters $\theta_3 \in [0,2\pi)$ and $\zeta \in [0,+\infty)$, which control the imaginary part of the complex direction $\mathbf{d}(\mathbf{y})$ in (\ref{complex direction}), \textit{evanescence parameters}.
A representation of three different evanescent plane waves restricted to $\partial B_1$ is given in Figure \ref{figure 4.1}.

\begin{remark}
In order to define the evanescent plane waves, contrary to the two-dimensional case \textup{\cite[Sec.\ 5]{parolin-huybrechs-moiola}}, the `parameter complexification' procedure -- that is the parametrization of the complex-direction space $\{\mathbf{d} \in \C^3:\mathbf{d}\cdot \mathbf{d}=1\}$ obtained by complexifying the angles in \textup{(\ref{propagative direction})} -- turns out to be less suitable for the analysis of the space of Herglotz densities in 3D -- that we will introduce later in Chapter \textup{\hyperlink{Chapter 5}{5}} -- in particular in relation to the orthogonality and the asymptotic behavior of the basis $\{a_{\ell}^m\}_{(\ell,m) \in \mathcal{I}}$ (see Lemma \textup{\hyperlink{Lemma 5.2}{5.3}} and Lemma \textup{\hyperlink{Lemma 5.3}{5.4}}). For this reason, we chose to define a complex reference direction $\mathbf{d}_{\uparrow}(z)$ and then consider its rotations in space through the orthogonal matrix $R_{\boldsymbol{\theta}}$.
\end{remark}

\section{Complex-direction Jacobi--Anger identity}

\begin{figure}
\centering
\includegraphics[width=\linewidth]{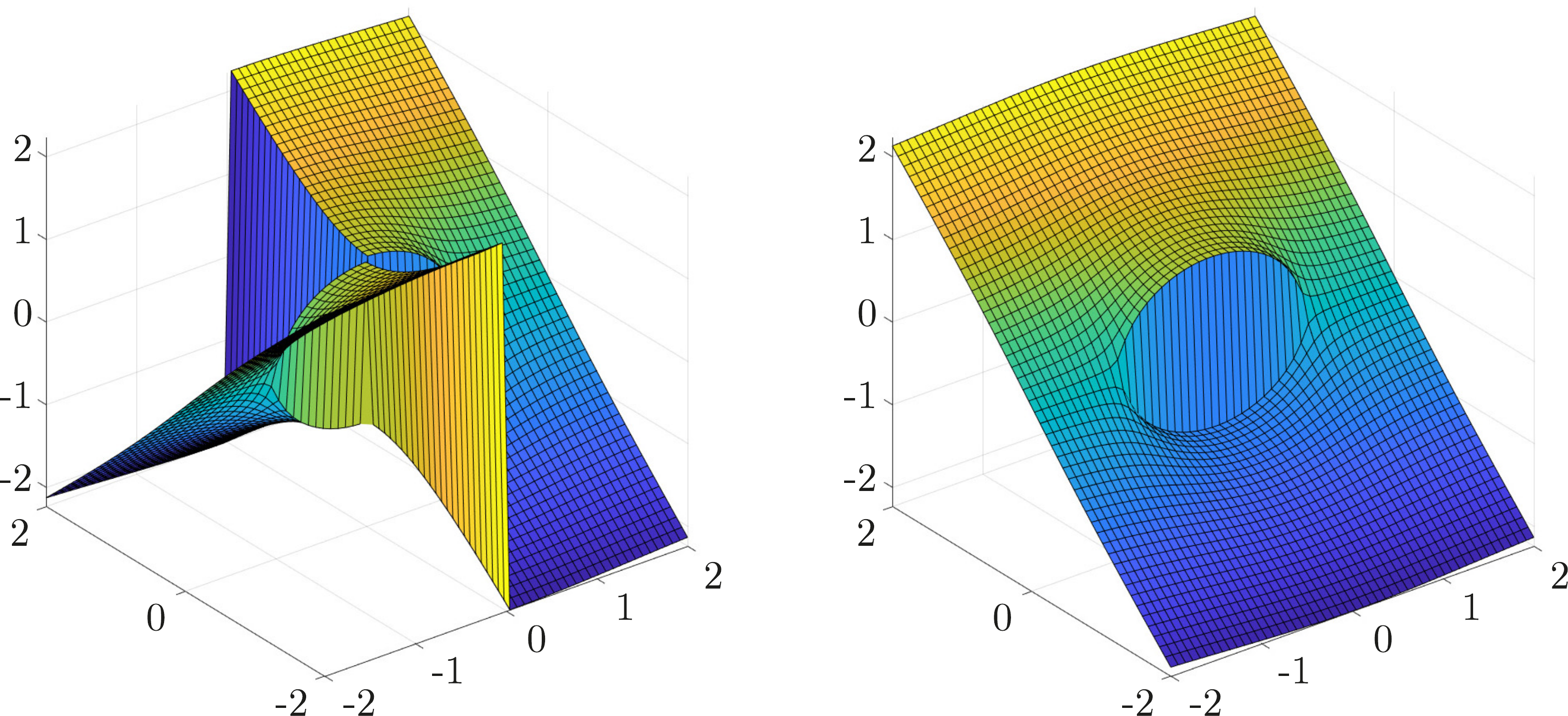}
\caption{Imaginary part of $(w^2-1)^{1/2}$: (left) considering the principal value, we have two branch cuts along $(-1,1)$ and $(-i\infty,+i\infty)$; (right) considering the definition in (\ref{convention}), we can eliminate the branch cut on the imaginary axis.}
\label{figure 4.2}
\end{figure}

\hypertarget{Section 4.2}{To} further study the evanescent plane waves, we require additional definitions for our analytic toolkit, specifically, we need to extend the Ferrers function to the complex domain (see (\ref{legendre2 polynomials})) and introduce the \textit{Wigner matrices} (see Definition \hyperlink{Definition 4.5}{4.8}).
According to \cite[Sec.\ 3.2, Eq.\ (6)]{erdelyi}, for any $m \in \Z$, throughout this section we use the convention
\begin{equation}
(w^2-1)^{m/2}:=\mathcal{P}\left[(w+1)^{m/2}\right]\,\mathcal{P}\left[(w-1)^{m/2}\right],\,\,\,\,\,\,\,\,\,\,\,\,\,\,\forall w \in \C,
\label{convention}
\end{equation}
where $\mathcal{P}[\,\,\cdot\,\,]$ indicates that the principal branch is chosen. Observe that, for every $w \in \C$ and odd $m \in \Z$, this is equivalent to say
\begin{equation*}
(w^2-1)^{m/2}:=\begin{cases}
-\mathcal{P}\left[(w^2-1)^{m/2}\right] & \text{if}\,\,\, \Re\{w\}<0\,\,\,\vee\,\,\,\left(\Re\{w\}=0\,\,\,\wedge\,\,\,\Im\{w\}<0 \right)\\
+\mathcal{P}\left[(w^2-1)^{m/2}\right]  & \text{if}\,\,\, \Re\{w\}>0\,\,\,\vee\,\,\,\left(\Re\{w\}=0\,\,\,\wedge\,\,\,\Im\{w\}\geq 0 \right)\\
\end{cases}.
\end{equation*}
Thanks to (\ref{convention}) we can get rid of the branch cut along the imaginary axis simply by mirroring the function values on the left-half of the complex plane in the right-half (with some corrections where $\Re\{w\}=0$); for the case $m=1$, see Figure \ref{figure 4.2}.

Following \cite[Eqs.\ (14.7.14) and (14.9.13)]{nist}, for every $(\ell,m) \in \mathcal{I}$, the \textit{associated Legendre polynomials} are solutions to the general Legendre equation (\ref{general legendre equation}) and are defined as
\begin{equation}
P_{\ell}^m(w):=\frac{1}{2^{\ell}\ell!}(w^2-1)^{m/2}\frac{\textup{d}^{\ell+m}}{\textup{d}w^{\ell+m}}(w^2-1)^{\ell},\,\,\,\,\,\,\,\,\,\,\,\,\,\,\forall w \in \C,
\label{legendre2 polynomials}
\end{equation}
so that
\begin{equation}
P_{\ell}^{-m}(w)=\frac{(\ell-m)!}{(\ell+m)!}P_{\ell}^{m}(w),\,\,\,\,\,\,\,\,\,\,\,\,\,\,\forall w \in \C.
\label{negative legendre2 polynomials}
\end{equation}
In particular, $P_{\ell}^m$ is called \textit{associated Legendre polynomial of degree $\ell$ and order $m$}.
For every $(\ell,m) \in \mathcal{I}$ such that $\frac{m}{2} \not\in \Z$, $P_{\ell}^m$ is a single-valued function on the complex plane with a branch cut along the interval $(-1,1)$, where it is continuous from above; otherwise, if $m$ is even, $P_{\ell}^m$ is a polynomial of degree $\ell$. From \cite[Eq.\ (14.23.1)]{nist}, it follows:
\begin{equation}
\lim_{\epsilon \shortarrow{7} 0}P_{\ell}^m(x\pm i\epsilon)=(\pm 1)^mP_{\ell}^m(x)=i^{\mp m}\mathsf{P}_{\ell}^m(x),\,\,\,\,\,\,\,\,\,\,\,\,\,\,\forall x \in (-1,1).
\label{on the cut}
\end{equation}
Observe that, if $m=0$, we simply obtain the \textit{Legendre polynomial of degree $\ell$}, which is defined on the entire complex plane, namely $P_{\ell}(w)=\mathsf{P}_{\ell}(w)$ for every $w \in \C$.

The modal analysis of evanescent plane waves, to which the next section is devoted, relies on the extension of the Jacobi--Anger identity (\ref{jacobi-anger}) to complex-valued directions $\mathbf{d}(\mathbf{y})$ in (\ref{complex direction}). 
First, we then need to extend the spherical harmonics (\ref{spherical harmonic}) to complex directions of the form $\mathbf{d}_{\uparrow}(z)$ in (\ref{reference complex direction}). With the introduction of the associated Legendre polynomials (\ref{legendre2 polynomials}), for every $(\ell,m) \in \mathcal{I}$, we can define:
\begin{equation*}
Y_{\ell}^m\left(\mathbf{d}_{\uparrow}(z)\right):=\gamma_{\ell}^m i^{-m}P_{\ell}^m(z),\,\,\,\,\,\,\,\,\,\,\,\,\,\,\forall z\geq 1,
\end{equation*}
where $\gamma_{\ell}^m$ was introduced in (\ref{gamma constant}).
This definition is supported by the following proposition, which generalizes the addition theorem (\ref{addition theorem}) when an upward complex direction vector is considered.

\begin{proposition}
\hypertarget{Proposition 4.3}{The} following identities hold for any $\ell \geq 0$, $\mathbf{x} \in \mathbb{S}^2$ and $z \geq 1$:
\begin{equation}
\sum_{m=-\ell}^{\ell}Y_{\ell}^m(\mathbf{d}_{\uparrow}(z))\overline{Y_{\ell}^m(\mathbf{x})}=\sum_{m=-\ell}^{\ell}Y_{\ell}^m(\mathbf{d}_{\uparrow}(z))Y_{\ell}^m(\mathbf{x})=\frac{2\ell+1}{4\pi}P_{\ell}(\mathbf{d}_{\uparrow}(z)\cdot \mathbf{x}).
\label{addition theorem 2}
\end{equation}
\end{proposition}
\begin{proof}
Let $\ell \geq 0$, $z \geq 1$ and $\mathbf{x}=(\sin{\theta}\cos{\varphi},\sin{\theta}\sin{\varphi},\cos{\theta})\in \mathbb{S}^2$, where $\theta \in [0,\pi]$ and $\varphi \in [0,2\pi)$.
Thanks to (\ref{on the cut}), it follows:
\begin{equation}
P_{\ell}(z)P_{\ell}(\cos \theta-i\epsilon)+2\sum_{m=1}^{\ell}\frac{(\ell-m)!}{(\ell+m)!}
(-1)^mP_{\ell}^m(z)P_{\ell}^m(\cos \theta-i\epsilon)\cos(m\varphi)
\label{first one}
\end{equation}
\small
\begin{equation*}
\Big\downarrow\,\,\,\epsilon \searrow 0
\end{equation*}
\normalsize
\begin{equation}
P_{\ell}(z)\mathsf{P}_{\ell}(\cos \theta)+2\sum_{m=1}^{\ell}\frac{(\ell-m)!}{(\ell+m)!}
i^{-m}P_{\ell}^m(z)\mathsf{P}_{\ell}^m(\cos \theta)\cos(m\varphi),\,\,\,
\label{first one 2}
\end{equation}
and moreover, according to (\ref{convention}), we have that
\begin{equation}
P_{\ell}\left(z(\cos\theta-i\epsilon)-\sqrt{z^2-1}\sqrt{(\cos \theta-i\epsilon)^2-1}\cos \varphi\right)
\label{second one}
\end{equation}
\small
\begin{equation*}
\Big\downarrow\,\,\,\epsilon \searrow 0
\end{equation*}
\normalsize
\begin{equation}
P_{\ell}\left(z\cos\theta+i\sqrt{z^2-1}\sin \theta \cos \varphi\right).
\label{second one 2}
\end{equation}
Due to \cite[Eqs.\ (14.7.16) and (14.28.1)]{nist}, the values of the expressions in (\ref{first one}) and (\ref{second one}) are the same, and therefore (\ref{first one 2}) and (\ref{second one 2}) also coincide. Hence, thanks to (\ref{negative legendre polynomials}), (\ref{gamma constant}), and (\ref{negative legendre2 polynomials}), we get
\begin{align*}
&\frac{4\pi}{2\ell+1}\!\sum_{m=-\ell}^{\ell}(\gamma_{\ell}^m)^2i^{-m}P_{\ell}^m(z)\mathsf{P}_{\ell}^m(\cos \theta)e^{\pm im\varphi}=\!\!\sum_{m=-\ell}^{\ell}\frac{(\ell-m)!}{(\ell+m)!}i^{-m}P_{\ell}^m(z)\mathsf{P}_{\ell}^m(\cos \theta)e^{\pm im\varphi}\\
&=\sum_{m=0}^{\ell}\frac{(\ell-m)!}{(\ell+m)!}i^{-m}P_{\ell}^m(z)\mathsf{P}_{\ell}^m(\cos \theta)e^{\pm im\varphi}+\!\!\sum_{m=-\ell}^{-1}\frac{(\ell+m)!}{(\ell-m)!}i^{m}P_{\ell}^{-m}(z)\mathsf{P}_{\ell}^{-m}(\cos \theta)e^{\pm im\varphi}\\
&=\sum_{m=0}^{\ell}\frac{(\ell-m)!}{(\ell+m)!}i^{-m}P_{\ell}^m(z)\mathsf{P}_{\ell}^m(\cos \theta)e^{\pm im\varphi}+\sum_{m=1}^{\ell}\frac{(\ell-m)!}{(\ell+m)!}i^{-m}P_{\ell}^{m}(z)\mathsf{P}_{\ell}^{m}(\cos \theta)e^{\mp im\varphi}\\
&=P_{\ell}(z)\mathsf{P}_{\ell}(\cos \theta)+2\sum_{m=1}^{\ell}\frac{(\ell-m)!}{(\ell+m)!}
i^{-m}P_{\ell}^m(z)\mathsf{P}_{\ell}^m(\cos \theta)\cos(m\varphi)\\
&=P_{\ell}\left(z\cos\theta+i\sqrt{z^2-1}\sin \theta \cos \varphi\right)=P_{\ell}\left(\mathbf{d}_{\uparrow}(z)\cdot \mathbf{x}\right)
\end{align*}
and (\ref{addition theorem 2}) follows.
\end{proof}
\begin{remark}
To be more precise, \textup{\cite[Eq.\ (14.28.1)]{nist}} states that the equality between \textup{(\ref{first one})} and \textup{(\ref{second one})} holds only if $\theta \in [0,\pi/2)$ and thus, following the previous proof, identities \textup{(\ref{addition theorem 2})} are proven only in this case.
This probably happens because \textup{\cite{nist}} does not adopt the convention \textup{(\ref{convention})} within the definition of the associated Legendre polynomials and therefore \textup{\cite[Eq.\ (14.28.1)]{nist}} is limited only to values with positive real part.
Nevertheless, since all terms in \textup{(\ref{addition theorem 2})} are analytic in $(0,\pi)$ as functions of $\theta$ (making explicit the dependence of $\mathbf{x}$ on $\theta$), these identities can be easily extend to this interval due to \textup{\cite[Th.\ 3.2.6]{ablowitz}}. Furthermore, they hold if $\mathbf{x}=(0,0,1)$, namely $\theta=\pi$: in fact $\mathsf{P}_{\ell}^m(-1)=(-1)^{\ell}\delta_{0,m}$ and, due to \textup{\cite[Eq.\ (14.7.17)]{nist}}, $P_{\ell}(-z)=(-1)^{\ell}P_{\ell}(z)$ for every $z\geq 1$.
\end{remark}
The previous result (\ref{addition theorem 2}) brings us close to deriving a Jacobi--Anger identity for the reference complex direction $\mathbf{d}_{\uparrow}(z)$, for $z\geq 1$. However, before we proceed, we must first extend the identity (\ref{pre jacobi-anger}) to complex values of $t$, which necessitates the use of the following lemma.
\begin{lemma}
Let $\ell \geq 0$. We have for every $0 \leq m \leq \ell$ and $w \in \C$
\begin{equation}
P_{\ell}^m(w)=\frac{(\ell+m)!}{2^{\ell}\ell!}\sum_{k=0}^{\ell-m}\binom{\ell}{k}\binom{\ell}{m+k}\left(w-1\right)^{\ell-\left(m/2+k\right)}\left(w+1\right)^{m/2+k}.
\label{sum legendre expansion}
\end{equation}
In particular, due to \textup{(\ref{negative legendre2 polynomials})}, $P_{\ell}^m(z)\geq 0$ for every $z\geq 1$ and $(\ell,m) \in \mathcal{I}$.
\end{lemma}
\begin{proof}
Through some calculation, we can see that
\begin{align*}
\frac{\textup{d}^{\ell+m}}{\textup{d}w^{\ell+m}}(w^2-1)^{\ell}&=\sum_{k=0}^{\ell+m}\binom{\ell+m}{k}\left(\frac{\textup{d}^k}{\textup{d}w^k}(w-1)^{\ell}\right)\left(\frac{\textup{d}^{\ell+m-k}}{\textup{d}w^{\ell+m-k}}(w+1)^{\ell}\right)\\
&=\sum_{k=m}^{\ell}\binom{\ell+m}{k}\left(\frac{\ell!}{(\ell-k)!}(w-1)^{\ell-k}\right)\left(\frac{\ell!}{(k-m)!}(w+1)^{k-m}\right)\\
&=\frac{(w-1)^{\ell}}{(w+1)^m}\sum_{k=m}^{\ell}\frac{(\ell+m)!}{k!(\ell+m-k)!}\frac{\ell!}{(\ell-k)!}\frac{\ell!}{(k-m)!}\left(\frac{w+1}{w-1}\right)^k\\
&=(w-1)^{\ell-m}\sum_{k=0}^{\ell-m}\frac{(\ell+m)!}{(m+k)!(\ell-k)!}\frac{\ell!}{(\ell-m-k)!}\frac{\ell!}{k!}\left(\frac{w+1}{w-1}\right)^k\\
&=(\ell+m)!(w-1)^{\ell-m}\sum_{k=0}^{\ell-m}\binom{\ell}{k}\binom{\ell}{m+k}\left(\frac{w+1}{w-1}\right)^k.
\end{align*}
Therefore, thanks to the definitions (\ref{convention}) and (\ref{legendre2 polynomials}), it follows
\begin{align*}
P_{\ell}^m(w)&=\frac{1}{2^{\ell}\ell!}(w^2-1)^{m/2}\frac{\textup{d}^{\ell+m}}{\textup{d}w^{\ell+m}}(w^2-1)^{\ell}\\
&=\frac{(\ell+m)!}{2^{\ell}\ell!}(w-1)^{\ell-m/2}(w+1)^{m/2}\sum_{k=0}^{\ell-m}\binom{\ell}{k}\binom{\ell}{m+k}\left(\frac{w+1}{w-1}\right)^k\\
&=\frac{(\ell+m)!}{2^{\ell}\ell!}\sum_{k=0}^{\ell-m}\binom{\ell}{k}\binom{\ell}{m+k}\left(w-1\right)^{\ell-\left(m/2+k\right)}\left(w+1\right)^{m/2+k}.
\end{align*}
\end{proof}

\begin{proposition}
\hypertarget{Proposition 4.6}{The} following identity holds for any $r \geq 0$ and $w \in \C$:
\begin{equation}
e^{irw}=\sum_{\ell=0}^{\infty}i^{\ell}(2\ell+1)j_{\ell}(r)P_{\ell}(w).
\label{pre jacobi anger 2}
\end{equation}
\end{proposition}
\begin{proof}
Let $R > 1$ and $B_R:=\{w \in \C: |w| <R\}$. We want to see that the right-hand side in (\ref{pre jacobi anger 2}) is well-defined for every $r \geq 0$ and $w \in B_R$. Due to (\ref{spherical bessel asymptotics}), it is enough to see that
\begin{equation}
\sum_{\ell=0}^{\infty}\left(\frac{er}{2\ell+1}\right)^{\ell}\left| P_{\ell}(w)\right|<\infty,\,\,\,\,\,\,\,\,\,\,\,\,\,\,\,\,\,\,\,\,\forall r \geq 0,\,\forall w \in B_R.
\label{other series}
\end{equation}
Thanks to (\ref{sum legendre expansion}) and the Vandermonde identity \cite[Eq.\ (1)]{Sokal}, it follows
\begin{equation*}
|P_{\ell}(w)|\leq \frac{1}{2^{\ell}}\sum_{k=0}^{\ell}\binom{\ell}{k}^2|w-1|^{\ell-k}|w+1|^k\leq\sum_{k=0}^{\ell}\binom{\ell}{k}^2\left(\frac{R+1}{2}\right)^{\ell}=\binom{2\ell}{\ell}\left(\frac{R+1}{2}\right)^{\ell},
\end{equation*}
and therefore, for every $r \geq 0$ and $w \in B_R$, the series (\ref{other series}) is dominated by
\begin{equation}
\sum_{\ell=0}^{\infty}c_{\ell},\,\,\,\,\,\,\,\,\,\,\,\,\,\,\text{where}\,\,\,\,\,\,\,\,\,\,\,\,\,\,c_{\ell}:=\binom{2\ell}{\ell}\left[\frac{er(R+1)}{4\ell+2}\right]^{\ell}.
\label{other other series}
\end{equation}
The series (\ref{other other series}) is convergent, as can be readily seen from the ratio test: in fact, thanks to (\ref{ell asymptotics}), we have
\begin{equation*}
\frac{c_{\ell+1}}{c_{\ell}}\sim \left(\frac{\ell+1/2}{\ell+3/2}\right)^{\ell+1}\frac{er(R+1)}{\ell+1}\sim\frac{r(R+1)}{\ell},\,\,\,\,\,\,\,\,\,\,\,\,\,\,\text{as}\,\,\,\ell \rightarrow \infty.
\end{equation*}
Therefore, the right-hand side in (\ref{pre jacobi anger 2}) is well-defined for every $r \geq 0$ and $w \in B_R$. The functions $w \mapsto e^{irw}$ and $w \mapsto P_{\ell}(w)$ are analytic on $B_R$ and, since identity (\ref{pre jacobi-anger}) holds, i.e.\ (\ref{pre jacobi anger 2}) with $w \in [-1,1]$, it follows that (\ref{pre jacobi anger 2}) also holds for every $r \geq 0$ and $w \in B_R$ due to \cite[Th. 3.2.6]{ablowitz}. Since $R >1$ is arbitrary, (\ref{pre jacobi anger 2}) is valid for every $w \in \C$.
\end{proof}

Finally, thanks to (\ref{addition theorem 2}) and (\ref{pre jacobi anger 2}), we can derive a Jacobi--Anger identity for the reference complex direction $\mathbf{d}_{\uparrow}(z)$ for every $z\geq 1$.
Once this result has been established, a corresponding Jacobi--Anger identity for all complex directions in the set $\{\mathbf{d} \in \C^3 : \mathbf{d} \cdot \mathbf{d}=1\}$ readily follows.

\begin{theorem}
\hypertarget{Theorem 4.6}{The} following identity holds for any $\mathbf{x} \in B_1$ and $z \geq 1$:
\begin{equation}
e^{i\kappa \mathbf{d}_{\uparrow}(z)\cdot \mathbf{x}}=4 \pi \sum_{\ell=0}^{\infty}i^{\ell}\sum_{m=-\ell}^{\ell} Y_{\ell}^m(\mathbf{d}_{\uparrow}(z))Y_{\ell}^m(\mathbf{\hat{x}})j_{\ell}(\kappa |\mathbf{x}|).
\label{complex jacobi-anger}
\end{equation}
Hence, for any $\mathbf{x} \in B_1$ and $\mathbf{y} \in \Theta \times [0,+\infty)$, it follows:
\begin{equation}
\phi_{\mathbf{y}}(\mathbf{x})=e^{i\kappa \mathbf{d}(\mathbf{y})\cdot \mathbf{x}}=4 \pi\sum_{\ell=0}^{\infty}i^{\ell} \sum_{{m}=-\ell}^{\ell}  Y_{\ell}^{m}\left(\mathbf{d}_{\uparrow}\left(\frac{\zeta}{2\kappa}+1\right)\right)Y_{\ell}^{m}\left(R^{-1}_{\boldsymbol{\theta}}\mathbf{\hat{x}}\right)j_{\ell}(\kappa |\mathbf{x}|).
\label{all complex jacobi-anger}
\end{equation}
\end{theorem}
\begin{proof}
The identity (\ref{complex jacobi-anger}) easily follows from (\ref{addition theorem 2}) and (\ref{pre jacobi anger 2}). In order to extend the Jacobi--Anger identity (\ref{jacobi-anger}) to complex-valued directions $\mathbf{d}(\mathbf{y})$ in (\ref{complex direction}) and thus prove the identity (\ref{all complex jacobi-anger}), we can note that:
\begin{align*}
\phi_{\mathbf{y}}(\mathbf{x})&=e^{i\kappa \mathbf{d}(\mathbf{y})\cdot \mathbf{x}}=e^{i\kappa R_{\boldsymbol{\theta}}\mathbf{d}_{\uparrow}\left(\frac{\zeta}{2\kappa}+1\right)\cdot \mathbf{x}}=e^{i\kappa \mathbf{d}_{\uparrow}\left(\frac{\zeta}{2\kappa}+1\right)\cdot R^{-1}_{\boldsymbol{\theta}}\mathbf{x}}\\ &=4 \pi\sum_{\ell=0}^{\infty}i^{\ell} \sum_{{m}=-\ell}^{\ell}  Y_{\ell}^{m}\left(\mathbf{d}_{\uparrow}\left(\frac{\zeta}{2\kappa}+1\right)\right)Y_{\ell}^{m}\left(R^{-1}_{\boldsymbol{\theta}}\mathbf{\hat{x}}\right)j_{\ell}(\kappa |\mathbf{x}|).
\end{align*}
\end{proof}
\vspace{-6mm}

\section{Wigner matrices}

\hypertarget{Section 4.3}{We} now present a short section dedicated to the \textit{Wigner matrices}. We propose the next definition accordingly to the notation of \cite[Eq.\ (1)]{feng} and \cite[Eq.\ (5.65)]{devanathan}.
\begin{definition}[Wigner matrices]
\hypertarget{Definition 4.5}{Let} $\boldsymbol{\theta} \in \Theta$ be the Euler angles and $\ell \geq 0$. The \textup{Wigner D-matrix} is the unitary matrix $D_{\ell}(\boldsymbol{\theta})=(D_{\ell}^{m,m'}(\boldsymbol{\theta}))_{m,m'} \in \C^{(2\ell+1) \times (2\ell +1)}$, where $|m|,|m'|\leq \ell$, whose elements are defined by
\begin{equation}
D_{\ell}^{m,m'}(\boldsymbol{\theta}):=e^{im'\theta_2}d_{\ell}^{\,m,m'}(\theta_1)e^{im\theta_3}.
\label{DD matrix}
\end{equation}
In turn, the matrix $d_{\ell}(\theta):=(d_{\ell}^{\,m,m'}(\theta))_{m,m'} \in \R^{(2\ell+1) \times (2\ell +1)}$, where $|m|,|m'|\leq \ell$, is called \textup{Wigner d-matrix} and its elements are
\begin{equation}
d_{\ell}^{\,m,m'}(\theta):=\sum_{k=k_{\textup{min}}}^{k_{\textup{max}}}w_{\ell,k}^{m,m'}\left(\cos{\frac{\theta}{2}}\right)^{2(\ell-k)+m'-m}\left(\sin{\frac{\theta}{2}}\right)^{2k+m-m'},
\label{d matrix}
\end{equation}
where
\begin{equation}
w_{\ell,k}^{m,m'}:=\frac{(-1)^k\left[(\ell+m)!(\ell-m)!(\ell+m')!(\ell-m')!\right]^{1/2}}{(\ell-m-k)!(\ell+m'-k)!(k+m-m')!\,k!}
\label{coefficienti Wigner}
\end{equation}
with $k_{\textup{min}}:=\max\{0,m'-m\}$ and $k_{\textup{max}}:=\max\{\ell-m,\ell+m'\}$.
\end{definition}
For other explicit definitions involving trigonometric functions see \cite[Sec.\ 4.3]{quantumtheory}. Sign differences in (\ref{DD matrix}) and (\ref{coefficienti Wigner}) with respect to the cited sources are due to the choice of the Euler angles. Observe that the Wigner d-matrix $d_{\ell}(\theta_1)$ is real: this is mainly due to the conventional choice (z-y-z) adopted in Definition \hyperlink{Definition 4.1}{4.1}.

Thanks to the Wigner D-matrix $D_{\ell}(\boldsymbol{\theta})$, any spherical harmonic of degree $\ell$ and order $m$, under the rotation $R_{\boldsymbol{\theta}}$, transforms into a linear combination of spherical harmonics of the same degree, in fact the expansion formula \cite[Sec.\ 4.1, Eq.\ (5)]{quantumtheory} holds, namely 
\begin{equation}
Y_{\ell}^{m}(\mathbf{d})=\sum_{m'=-\ell}^{\ell}\overline{D_{\ell}^{m,m'}(\boldsymbol{\theta})}Y_{\ell}^{m'}(R_{\boldsymbol{\theta}}\mathbf{d}),\,\,\,\,\,\,\,\,\,\,\,\,\,\,\,\,\forall \mathbf{d} \in \mathbb{S}^2,\,\forall (\ell,m) \in \mathcal{I}
\label{wigner property}
\end{equation}
(see also \cite[Eq.\ (33)]{pendleton} and \cite[Eq.\ (4.1)]{devanathan}). Finally, we have all the ingredients to proceed with the modal analysis of evanescent plane waves (\ref{evanescent wave}).

\begin{remark}
As pointed out in \textup{\cite{feng}}, the numerical computation of the Wigner's formula \textup{(\ref{d matrix})} is subject to intolerable numerical errors, because it appears as a sum of many large numbers with alternating signs. To avoid this problem, the authors of the previously cited article present a simple method by expanding the $d_{\ell}(\theta)$ matrix into a complex Fourier series and then calculate the Fourier coefficients by exactly diagonalizing the angular momentum operator $J_{y}$ in the eigenbasis of $J_z$ (for further details about these operators see \textup{\cite[Ch.\ 1]{devanathan})}. This leads to a procedure that is free from the numerical cancellation problem, since, due to the normalization of the eigenvectors of $J_y$, the norm of each Fourier coefficient is not larger than unity. More precisely, $J_y$ is first expressed as a $(2\ell+1)$-dimensional Hermitian matrix
\begin{equation*}
J_y=\frac{1}{2i}\begin{bmatrix}
    0 & -X_{-\ell+1} & & & &\\
    X_{\ell} & 0 & -X_{-\ell+2} & & &\\
    & X_{\ell-1} & 0 & \ddots & &\\
    & & \ddots & \ddots & \ddots &\\
    & & & \ddots & 0 & -X_{\ell}\\
    & & & & X_{-\ell+1} & 0\\
\end{bmatrix},
\end{equation*}
where $(J_y)_{m,m'}=\left(X_{-m'}\delta_{m,m'+1}-X_{m'}\delta_{m,m'-1}\right)/(2i)$, for $|m|,|m'|\leq \ell$, with the term $X_{m}:=[(\ell+m)(\ell-m+1)]^{1/2}$ satisfying $X_{\pm m}=X_{\mp m +1}$. Next, the Hermitian matrix $J_y$ is diagonalized in order to obtain all the eigenvectors $\{\mathbf{j}_m\}_{m=-\ell}^{\ell}$ and thus compute
\vspace{-2mm}
\begin{equation}
d_{\ell}^{\,m,m'}(\theta)=\sum_{\mu=-\ell}^{\ell}e^{i\mu\theta}j_{\mu,m}\overline{j_{\mu,m'}},
\label{d matrix computation}
\end{equation}
where $\mathbf{j_{\mu}}=(j_{\mu,m})_{m=-\ell}^{\ell}$.
As mentioned earlier, this method not only has the advantage of having all the coefficients $t_{\ell,\mu}^{m,m'}$ in \textup{(\ref{d matrix computation})} smaller than unity, but also the matrix $J_y$ is tridiagonal and Hermitian, and so it can be easily diagonalized. In the following numerical experiments we will use \textup{(\ref{d matrix computation})} to overcome any possible loss of precision due to Wigner's formula \textup{(\ref{d matrix})}, thus using the Fourier expansion \textup{(\ref{d matrix computation})}.
\end{remark}

\vspace{-4mm}
\section{Modal analysis}

\hypertarget{Section 4.4}{In} (\ref{all complex jacobi-anger}) we saw that is possible to extend the Jacobi--Anger identity (\ref{jacobi-anger}) to all the complex-valued directions in $\{\mathbf{d} \in \C^3 : \mathbf{d} \cdot \mathbf{d}=1\}$. Recovering the definitions proposed in the previous sections, we present the modal analysis of evanescent plane waves, trying to explain why we expect such waves to have better stability properties than the propagative ones.
To improve readability, we introduce the notation $\mathbf{D}^{m}_{\ell}(\boldsymbol{\theta})$, for $0 \leq |m| \leq \ell$, to indicate the columns of the Wigner D-matrix $D_{\ell}(\boldsymbol{\theta})$ and
\begin{equation}
\mathbf{P}_{\ell}(\zeta):=\left(\gamma_{\ell}^{m'}i^{m'}P_{\ell}^{m'}\left(\frac{\zeta}{2\kappa}+1\right)\right)_{m'=-\ell}^{\ell} \in \C^{2\ell+1}.
\label{vector P}
\end{equation}
The findings presented in this section are based on the following fundamental result.

\begin{proposition}
The following identity holds for any $\mathbf{x} \in B_1$ and $\mathbf{y} \in \Theta \times [0,\!+\infty)$:
\begin{equation}
\phi_{\mathbf{y}}(\mathbf{x})=\sum_{\ell=0}^{\infty}\sum_{{m}=-\ell}^{\ell}\left[4 \pi i^{\ell}\beta_{\ell}^{-1}\overline{\mathbf{D}^{m}_{\ell}(\boldsymbol{\theta}) \cdot \mathbf{P}_{\ell}(\zeta)} \right] b_{\ell}^m(\mathbf{x}),
\label{complex expansion}
\end{equation}
\end{proposition}
\begin{proof}
Thanks to the complex-direction Jacobi--Anger identity (\ref{all complex jacobi-anger}) and the expansion formula (\ref{wigner property}), we have
\begin{align*}
\phi_{\mathbf{y}}(\mathbf{x})&=\sum_{\ell=0}^{\infty}\sum_{{m'}=-\ell}^{\ell}4 \pi i^{\ell} Y_{\ell}^{m'}\left(\mathbf{d}_{\uparrow}\left(\frac{\zeta}{2\kappa}+1\right)\right)Y_{\ell}^{m'}\left(R^{-1}_{\boldsymbol{\theta}}\mathbf{\hat{x}}\right)j_{\ell}(\kappa |\mathbf{x}|)\\
&=\sum_{\ell=0}^{\infty}\sum_{{m'}=-\ell}^{\ell}4 \pi i^{\ell} Y_{\ell}^{m'}\left(\mathbf{d}_{\uparrow}\left(\frac{\zeta}{2\kappa}+1\right)\right)\left(\sum_{m=-\ell}^{\ell}\overline{D_{\ell}^{m',m}(\boldsymbol{\theta})}Y_{\ell}^{m}(\mathbf{\hat{x}}) \right)j_{\ell}(\kappa |\mathbf{x}|)\\
&=\sum_{\ell=0}^{\infty}\sum_{{m}=-\ell}^{\ell}\left[\sum_{{m'}=-\ell}^{\ell}4 \pi i^{\ell-m'}\beta_{\ell}^{-1}\gamma_{\ell}^{m'}P_{\ell}^{m'}\left(\frac{\zeta}{2\kappa}+1\right)\overline{D_{\ell}^{m',m}(\boldsymbol{\theta})} \right] b_{\ell}^m(\mathbf{x}).
\end{align*}
\end{proof}

It is worth noting that, thanks to definition (\ref{DD matrix}), \cite[Eq.\ (35)]{pendleton} and the Wigner d-matrix symmetry properties \cite[Sec.\ 4.4, Eq.\ (1)]{quantumtheory}, it holds
\begin{equation}
D_{\ell}^{0,m}(\boldsymbol{\theta})=\sqrt{\frac{(\ell-m)!}{(\ell+m)!}}\,e^{im\theta_2}\mathsf{P}_{\ell}^m(\cos \theta_1),\,\,\,\,\,\,\,\,\,\forall \boldsymbol{\theta} \in \Theta,\,\forall (\ell,m) \in \mathcal{I},
\label{D-matrix0}
\end{equation}
and moreover $P_{\ell}^{m}(1)=\delta_{0,m}$ due to (\ref{legendre2 polynomials}). Hence, assuming $\zeta=0$ in (\ref{complex expansion}), we recover the Jacobi--Anger expansion for propagative plane waves in (\ref{expansion plane wave}) for any $\boldsymbol{\theta} \in \Theta$.
The moduli of the coefficients
\begin{equation}
\hat{\phi}_{\ell}^m(\theta_1,\theta_3,\zeta)\!:=\!\left|\left(\phi_{\mathbf{y}},b_{\ell}^m\right)_{\mathcal{B}}\right|\!=\!\frac{4\pi}{\beta_{\ell}}\left|\sum_{m'=-\ell}^{\ell}\!\!\gamma_{\ell}^{m'}i^{-m'}\!d_{\ell}^{\,m',m}(\theta_1)e^{-im'\theta_3}P_{\ell}^{m'}\!\!\left(\frac{\zeta}{2\kappa}+1\right)\right|
\label{evanescent coefficients}
\end{equation}
in the modal expansion (\ref{complex expansion}) depend on $\theta_1 \in [0,\pi]$, $\theta_3 \in [0,2\pi)$ and $(\ell,m) \in \mathcal{I}$. Observe that if $\zeta=0$, then, thanks to (\ref{D-matrix0}), the coefficients (\ref{evanescent coefficients}) coincide with the propagative ones in (\ref{propagative coefficients}) and therefore are independent of $\theta_3$.

Some distributions of the coefficients (\ref{evanescent coefficients}) are depicted in Figure \ref{figure 4.4}.
By adjusting the evanescence parameters $\theta_3$ and $\zeta$, the Fourier modal content of the plane waves can be shifted to higher regimes for any $(\ell,m) \in \mathcal{I}$, as can be seen in comparison to Figure \ref{figure 3.1}: in fact, by varying $\zeta$ we are able to reach higher degrees, i.e larger values of $\ell$, while by varying $\theta_3$ we range over the different orders $m$.

To better see this result, in analogy with what was done in (\ref{dfg}), we define:
\begin{equation}
\begin{split}
\tilde{b}_\ell[\mathbf{y}]&:=\sum_{m=-\ell}^{\ell}\left(\phi_{\mathbf{y}},b_{\ell}^m\right)_{\mathcal{B}}b_{\ell}^m,\,\,\,\,\,\,\,\,\,\,\,\,\,\,\,\,\,\,\,\,\,\,\,\,  \forall \ell \geq 0,\,\forall \mathbf{y} \in \Theta \times [0,+\infty),\\
b_\ell[\mathbf{y}]&:=\hat{\phi}_{\ell}^{-1}(\zeta)\tilde{b}_\ell[\mathbf{y}],\,\,\,\,\,\,\,\,\,\,\,\,\,\,\,\,\,\,\,\,\,\,\,\,\,\,\,\,\,\,\,\,\,\,\,\,\,\,\,\,\,\,\,\,\,\,\,\,\,\,\,\, \hat{\phi}_{\ell}(\zeta):=\big\|\tilde{b}_\ell[\mathbf{y}]\big\|_{\mathcal{B}}.
\end{split}
\label{yuy}
\end{equation}

\begin{figure}[H]
\centering
\begin{tabular}{cccc}
$\hat{\phi}_{\ell}^m\left(\frac{\pi}{2},\frac{\pi}{4},30\right)$ & $\hat{\phi}_{\ell}^m\left(\frac{\pi}{2},\frac{\pi}{2},30\right)$ & $\hat{\phi}_{\ell}^m\left(\frac{\pi}{2},\frac{3\pi}{2},30\right)$ & $\hat{\phi}_{\ell}^m\left(\frac{\pi}{2},\frac{7\pi}{4},30\right)$\\
\includegraphics[width=.2\linewidth,valign=m]{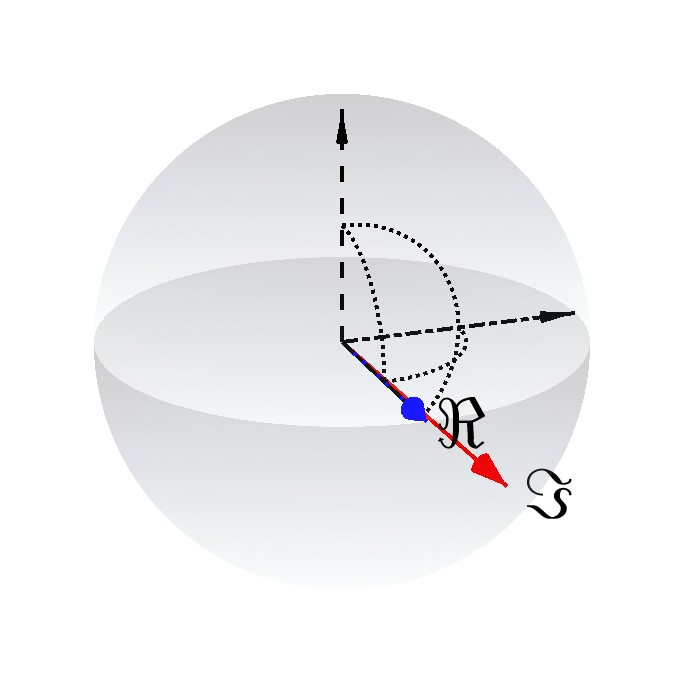} & \includegraphics[width=.2\linewidth,valign=m]{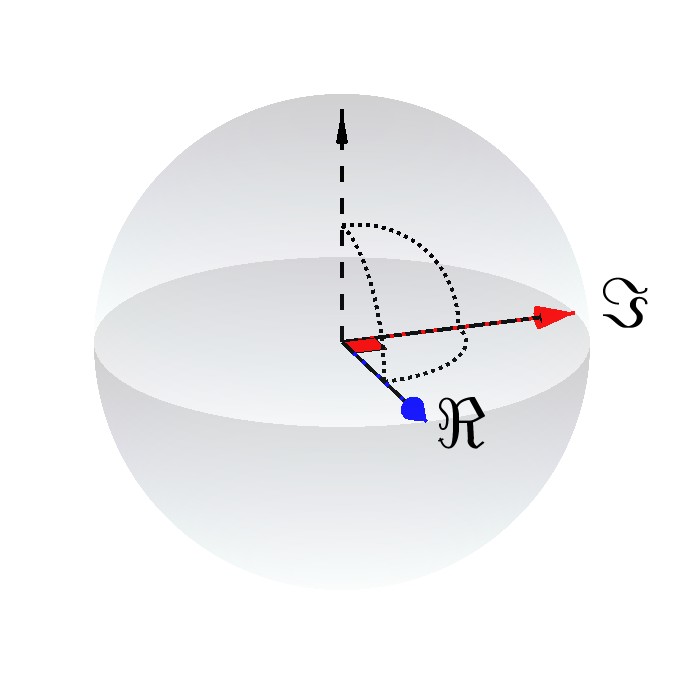} & \includegraphics[width=.2\linewidth,valign=m]{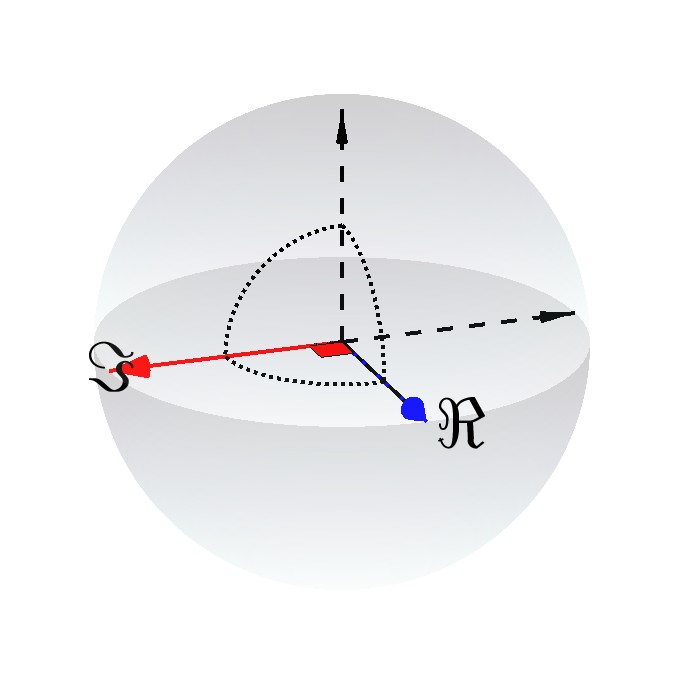} & \includegraphics[width=.2\linewidth,valign=m]{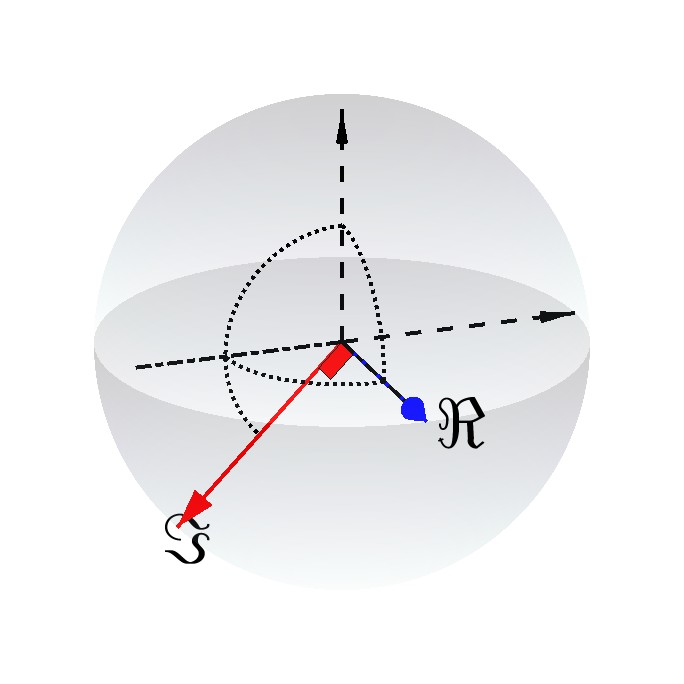}\\
\includegraphics[width=.2\linewidth,valign=m]{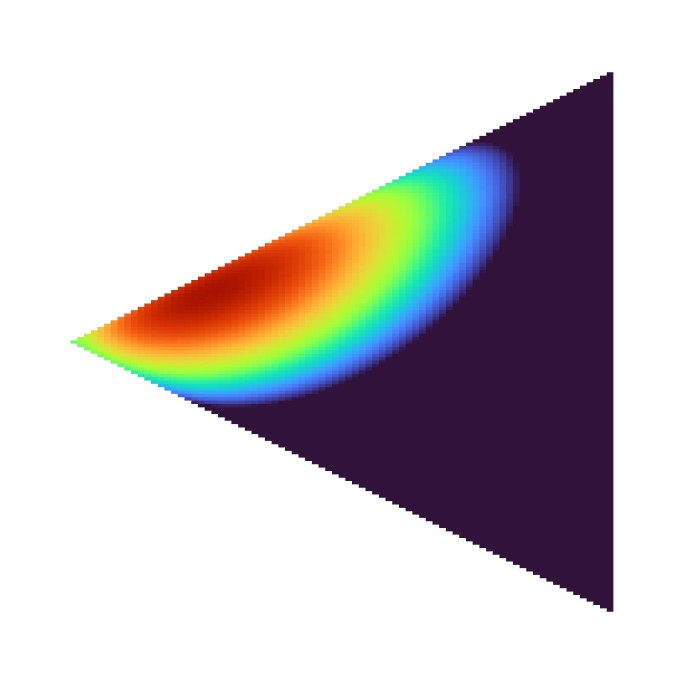} & \includegraphics[width=.2\linewidth,valign=m]{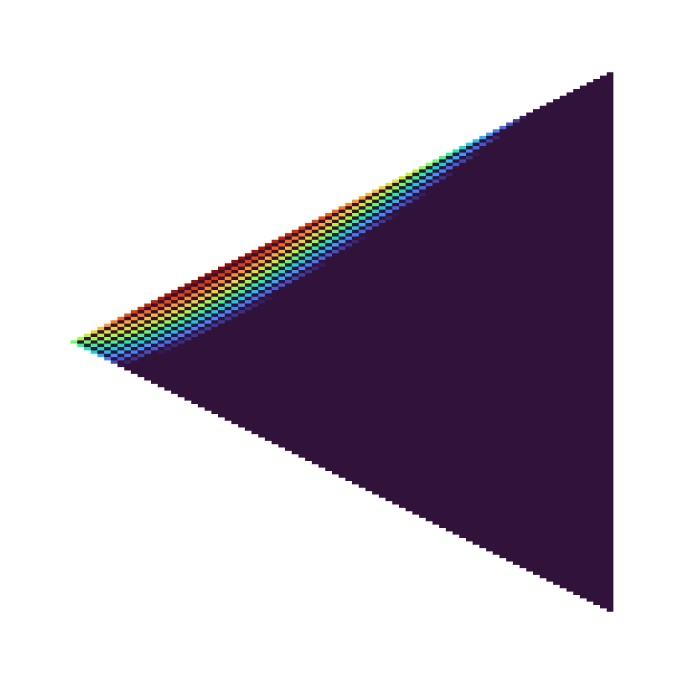} & \includegraphics[width=.2\linewidth,valign=m]{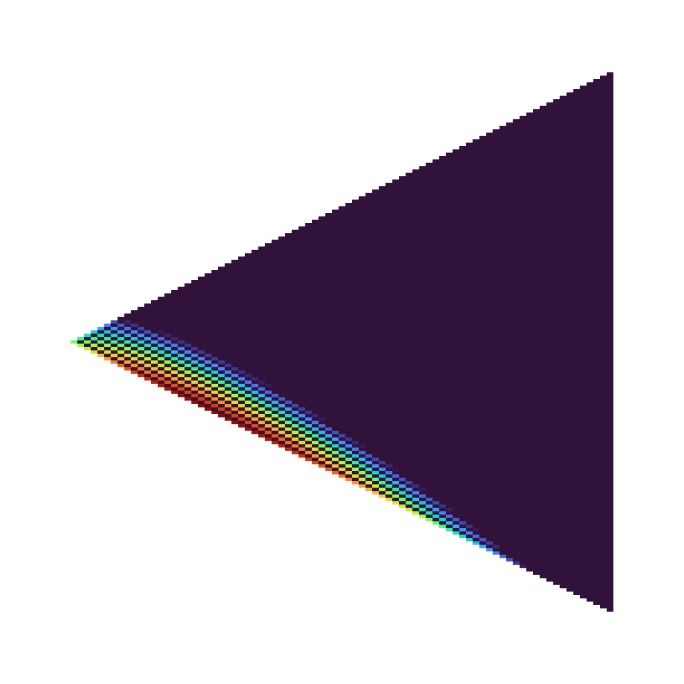} & \includegraphics[width=.2\linewidth,valign=m]{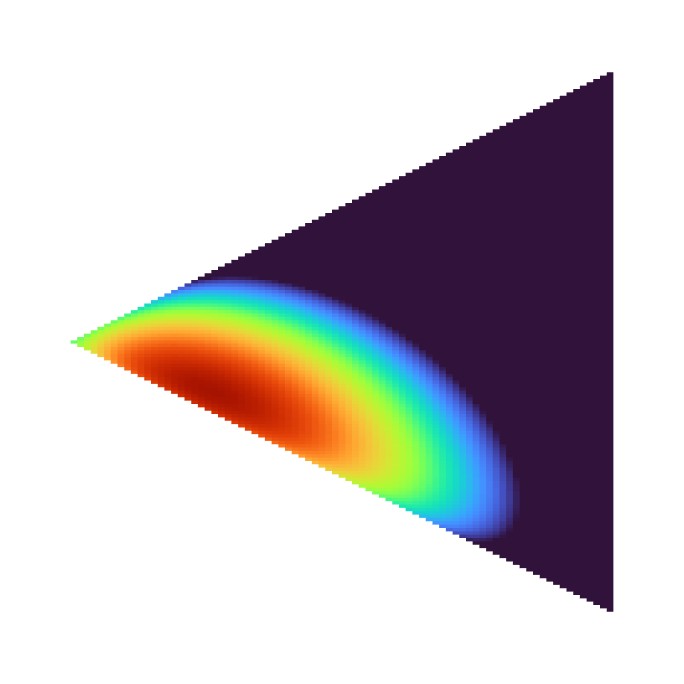}\\
\rule{0pt}{3.3ex} 
$\hat{\phi}_{\ell}^m\left(\frac{\pi}{4},\frac{\pi}{4},30\right)$ & $\hat{\phi}_{\ell}^m\left(\frac{\pi}{4},\frac{\pi}{2},30\right)$ & $\hat{\phi}_{\ell}^m\left(\frac{\pi}{4},\frac{3\pi}{2},30\right)$ & $\hat{\phi}_{\ell}^m\left(\frac{\pi}{4},\frac{7\pi}{4},30\right)$\\
\includegraphics[width=.2\linewidth,valign=m]{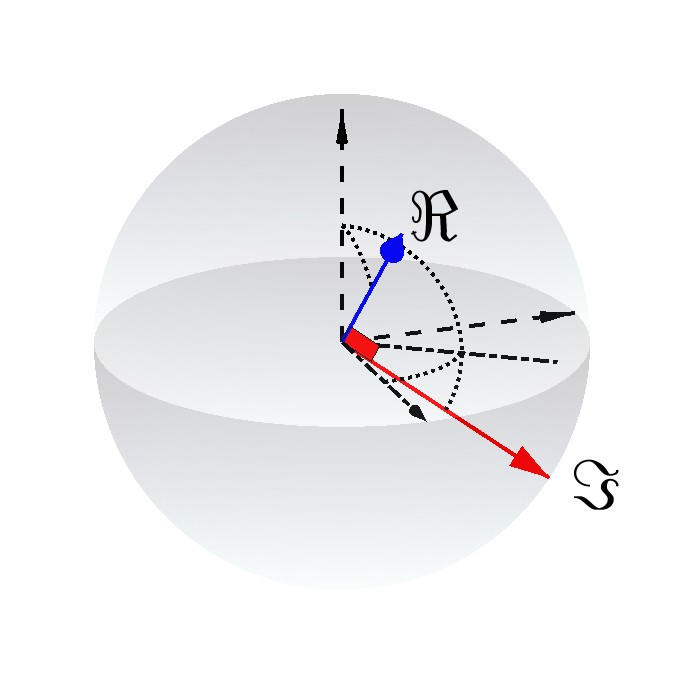} & \includegraphics[width=.2\linewidth,valign=m]{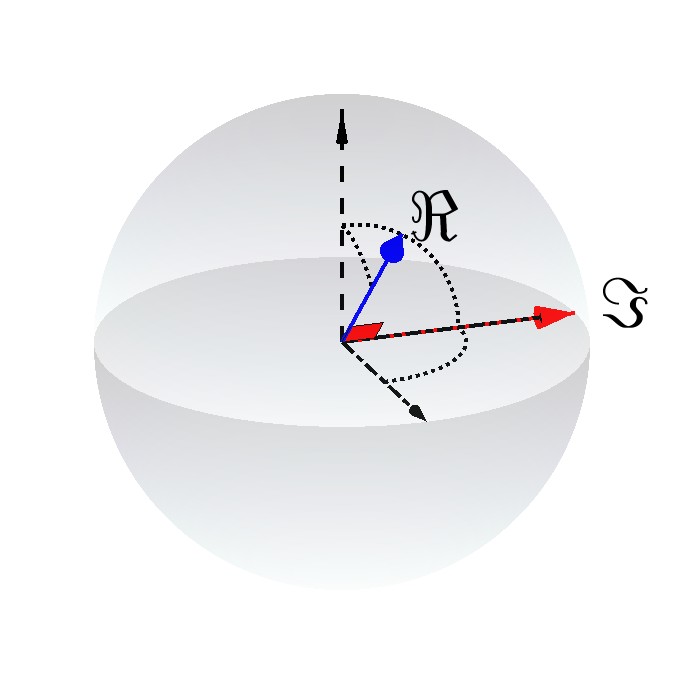} & \includegraphics[width=.2\linewidth,valign=m]{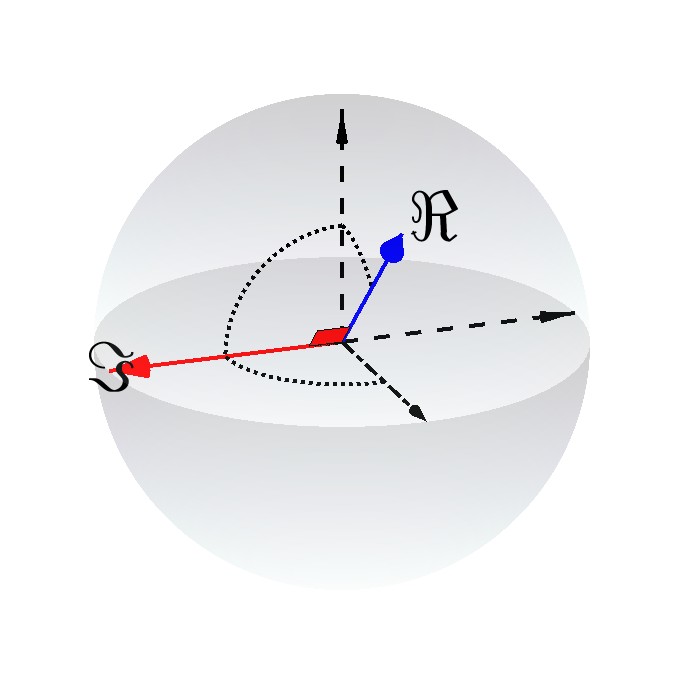} & \includegraphics[width=.2\linewidth,valign=m]{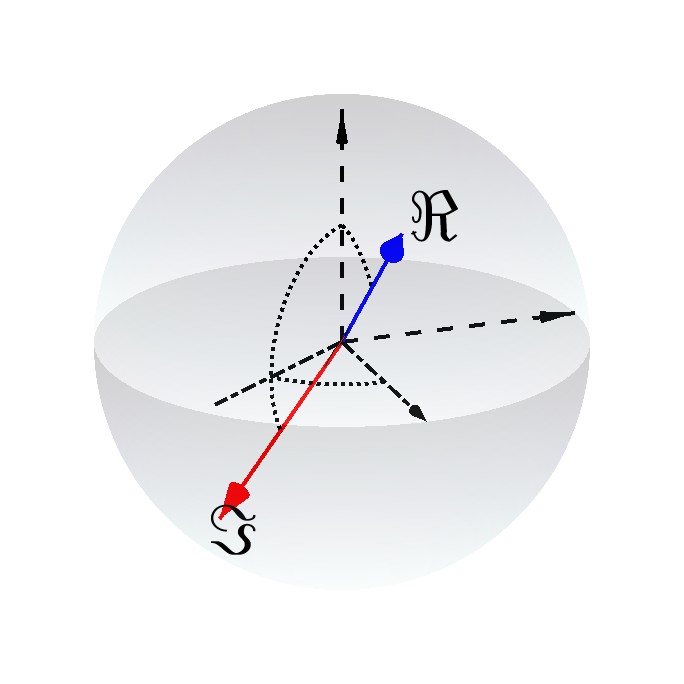}\\
\includegraphics[width=.2\linewidth,valign=m]{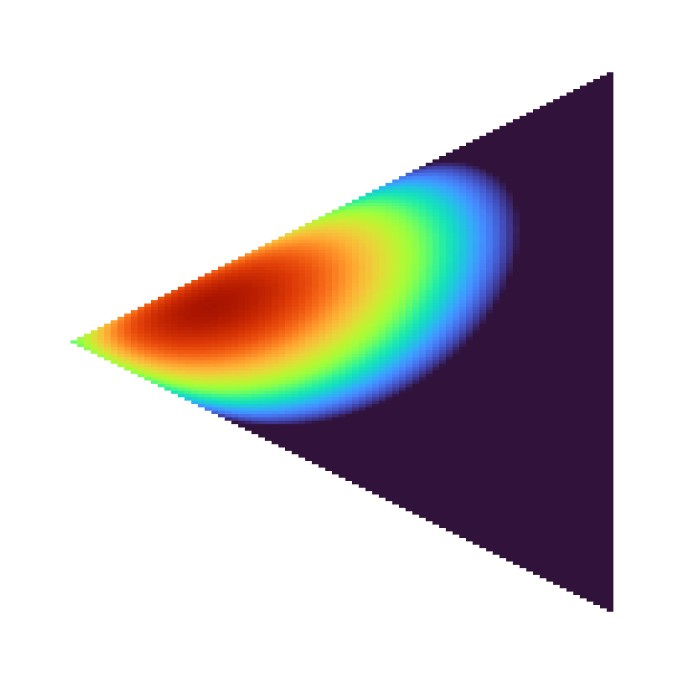} & \includegraphics[width=.2\linewidth,valign=m]{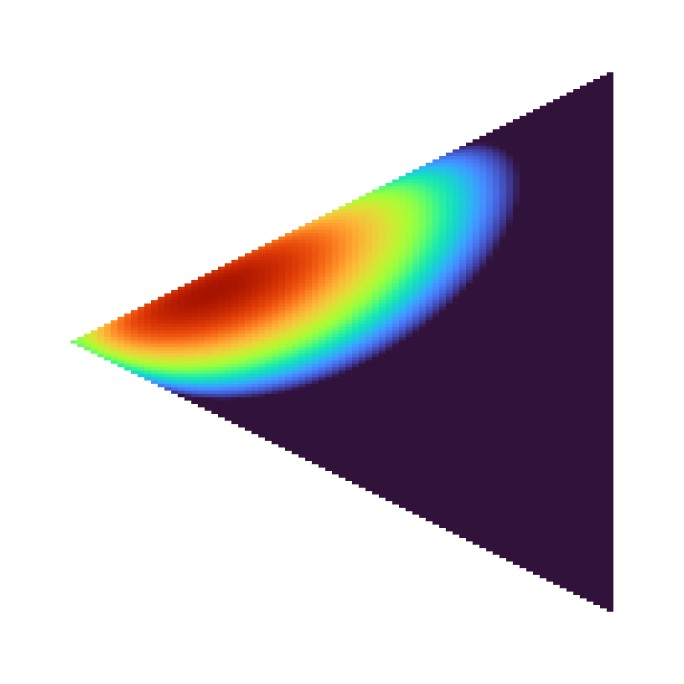} & \includegraphics[width=.2\linewidth,valign=m]{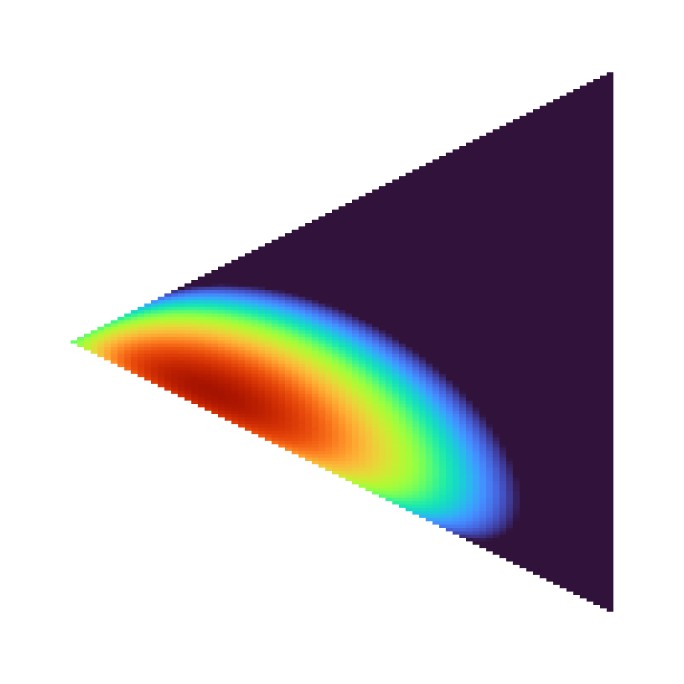} & \includegraphics[width=.2\linewidth,valign=m]{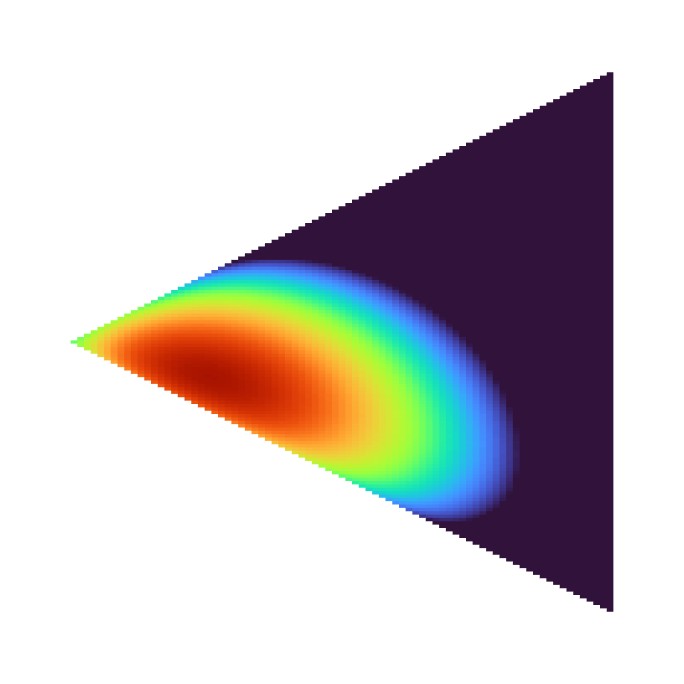}\\
\rule{0pt}{3.3ex}
$\hat{\phi}_{\ell}^m\left(\frac{\pi}{2},\frac{\pi}{4},60\right)$ & $\hat{\phi}_{\ell}^m\left(\frac{\pi}{2},\frac{\pi}{2},60\right)$ & $\hat{\phi}_{\ell}^m\left(\frac{\pi}{2},\frac{3\pi}{2},60\right)$ & $\hat{\phi}_{\ell}^m\left(\frac{\pi}{2},\frac{7\pi}{4},60\right)$\\
\includegraphics[width=.2\linewidth,valign=m]{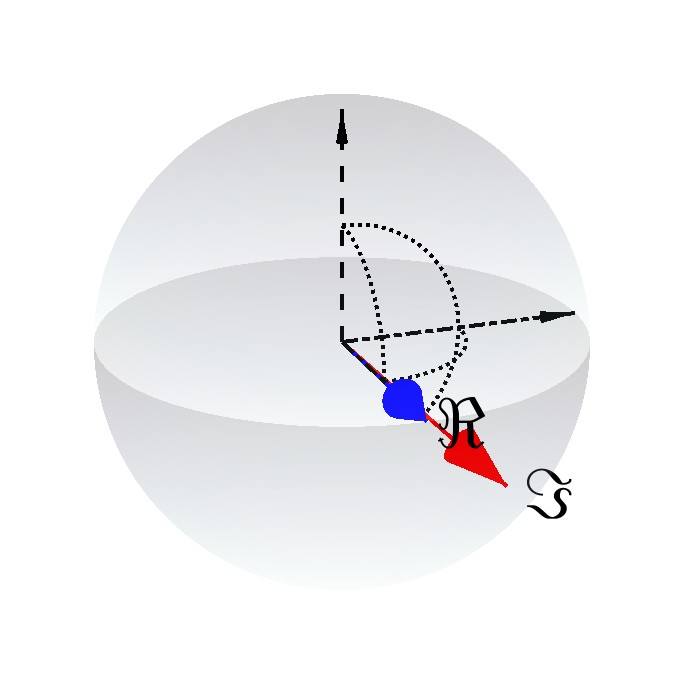} & \includegraphics[width=.2\linewidth,valign=m]{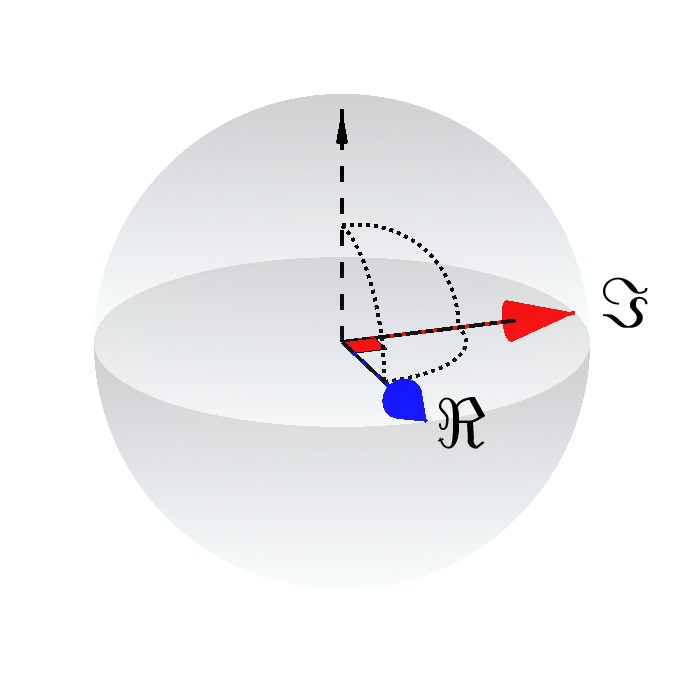} & \includegraphics[width=.2\linewidth,valign=m]{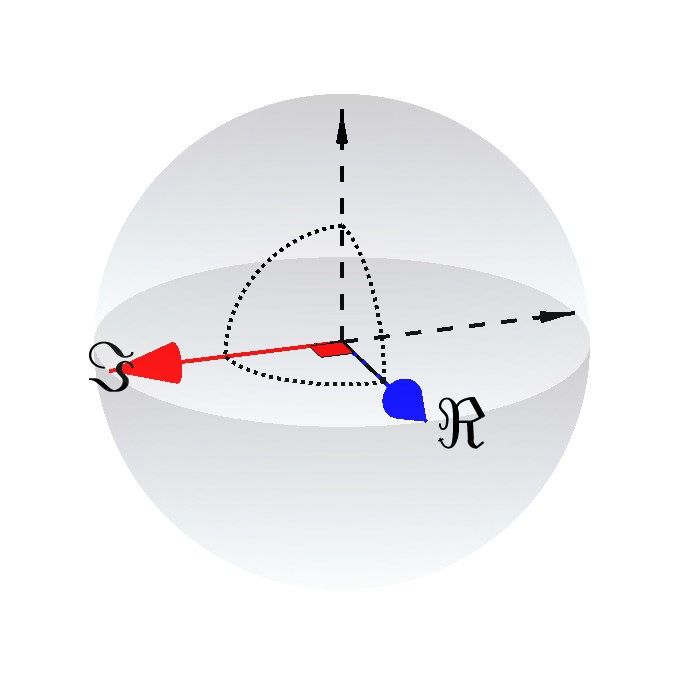} & \includegraphics[width=.2\linewidth,valign=m]{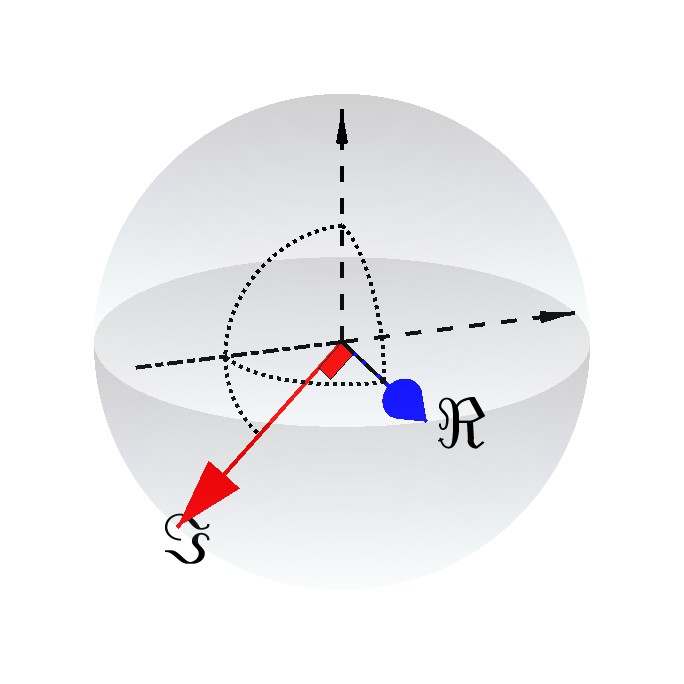}\\
\includegraphics[width=.2\linewidth,valign=m]{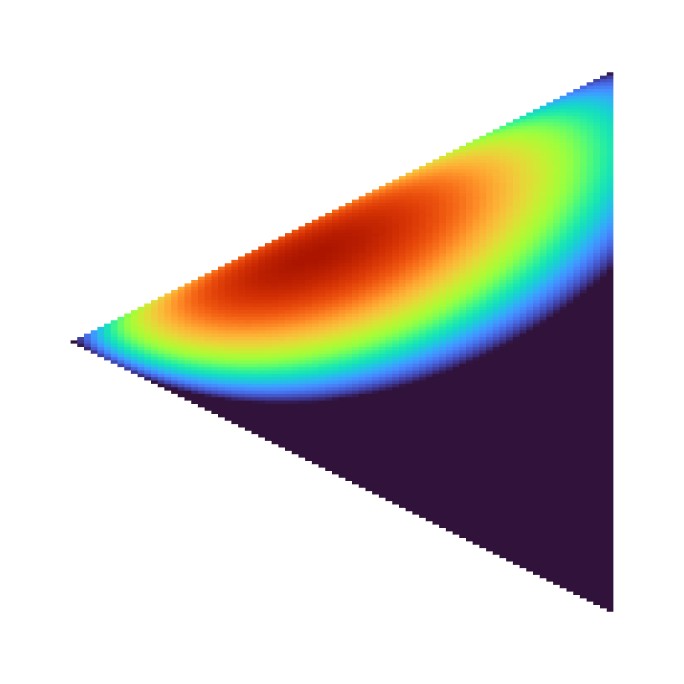} & \includegraphics[width=.2\linewidth,valign=m]{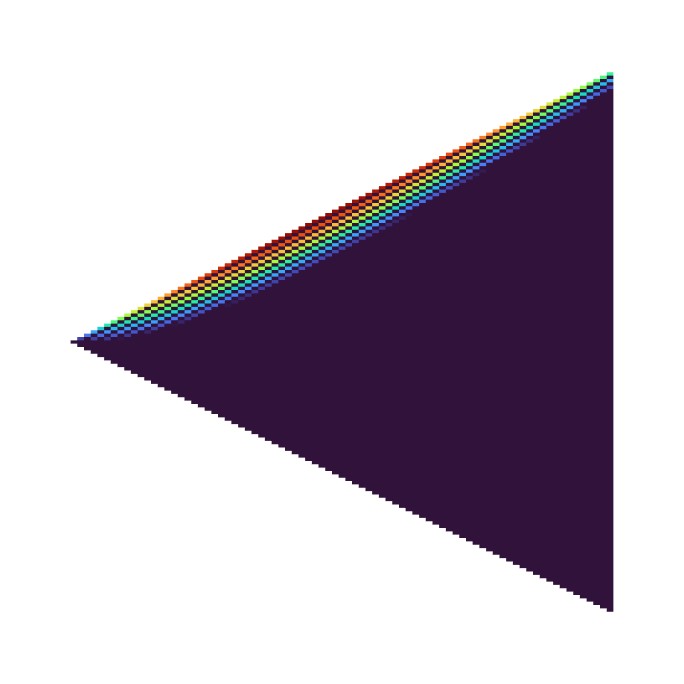} & \includegraphics[width=.2\linewidth,valign=m]{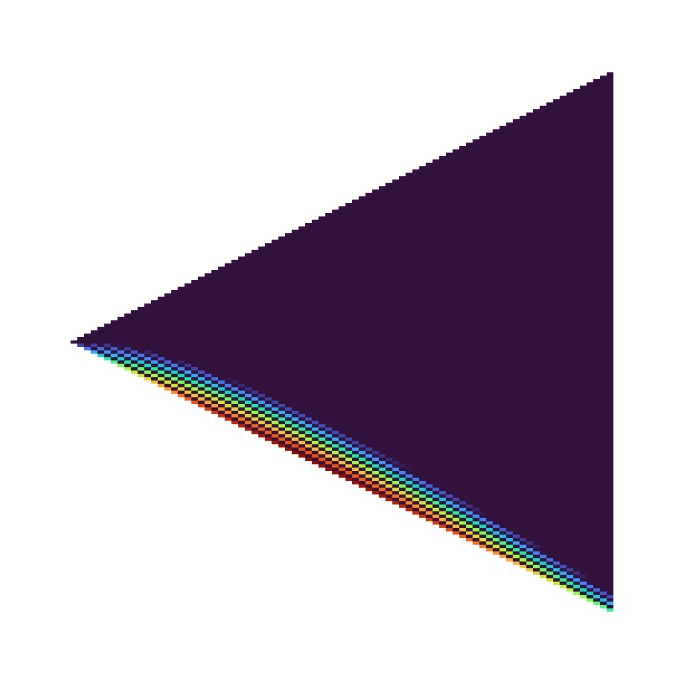} & \includegraphics[width=.2\linewidth,valign=m]{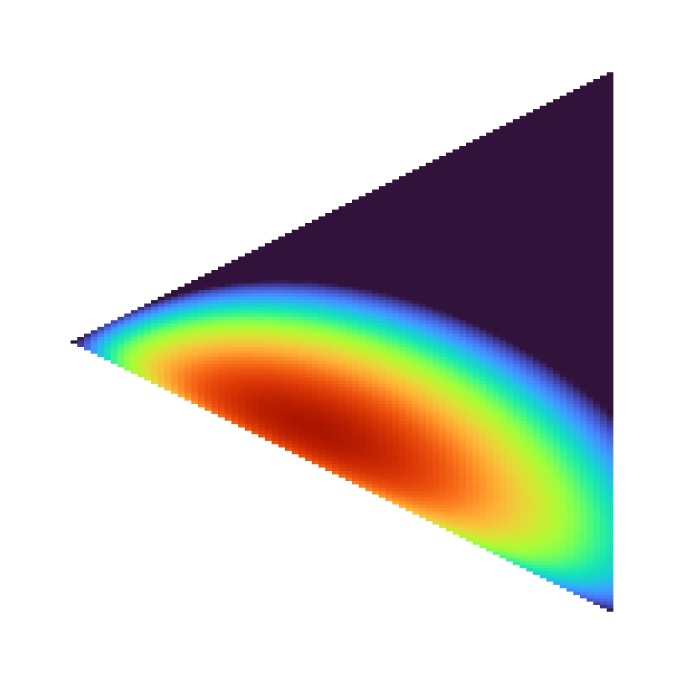}\\
\vspace{1mm}\\
\multicolumn{4}{c}{\includegraphics[width=0.85\linewidth]{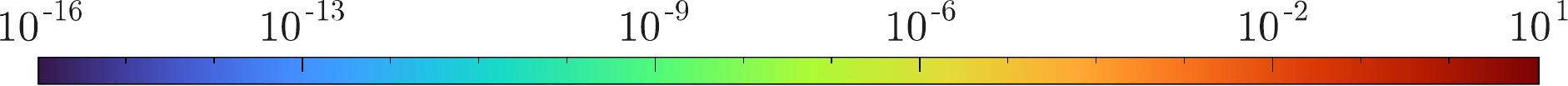}}\\
\end{tabular}
\end{figure}
\begin{figure}[H]
\centering
\begin{tabular}{cccc}
$\hat{\phi}_{\ell}^m\left(\frac{\pi}{4},\frac{\pi}{4},60\right)$ & $\hat{\phi}_{\ell}^m\left(\frac{\pi}{4},\frac{\pi}{2},60\right)$ & $\hat{\phi}_{\ell}^m\left(\frac{\pi}{4},\frac{3\pi}{2},60\right)$ & $\hat{\phi}_{\ell}^m\left(\frac{\pi}{4},\frac{7\pi}{4},60\right)$\\
\includegraphics[width=.2\linewidth,valign=m]{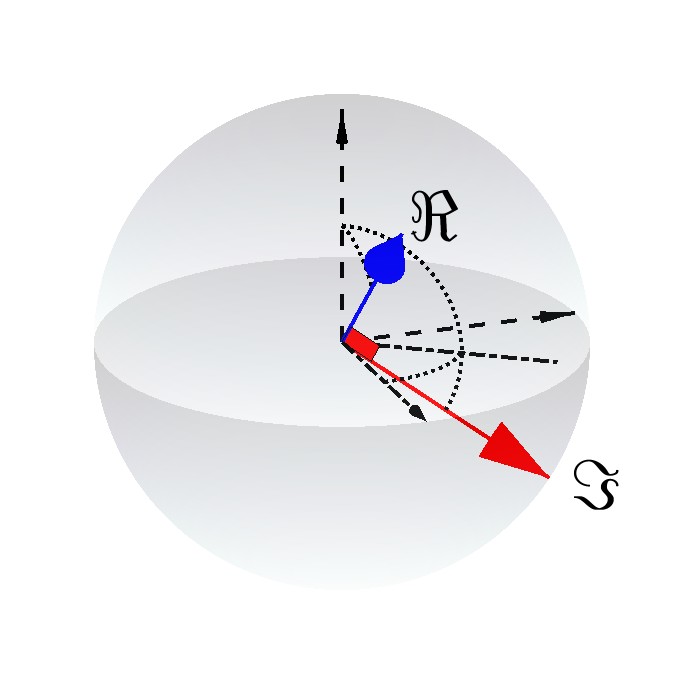} & \includegraphics[width=.2\linewidth,valign=m]{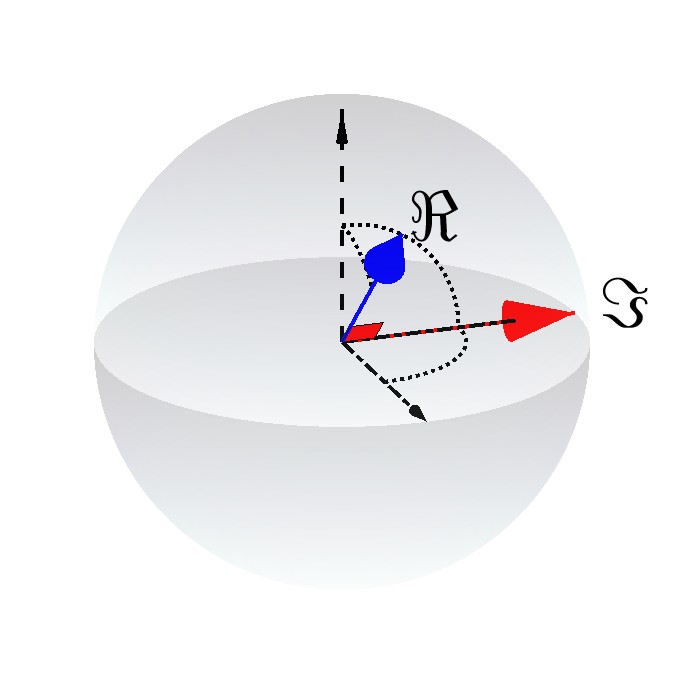} & \includegraphics[width=.2\linewidth,valign=m]{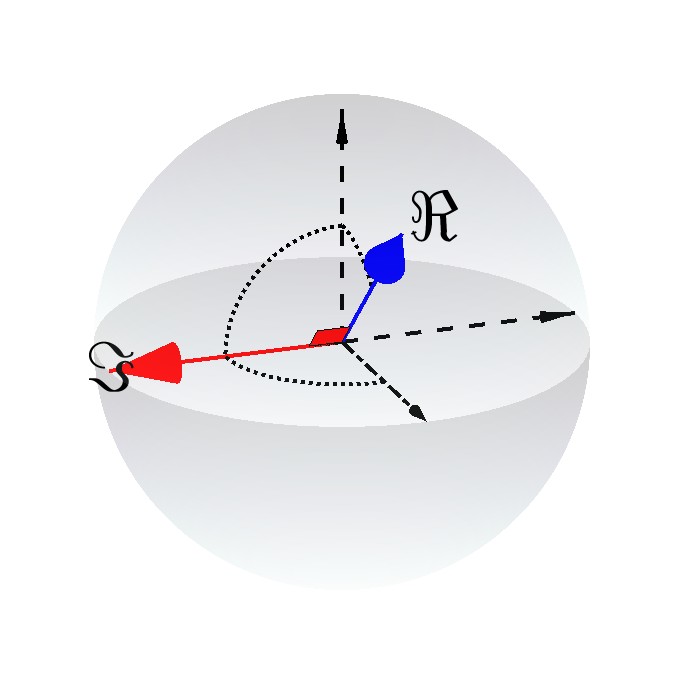} & \includegraphics[width=.2\linewidth,valign=m]{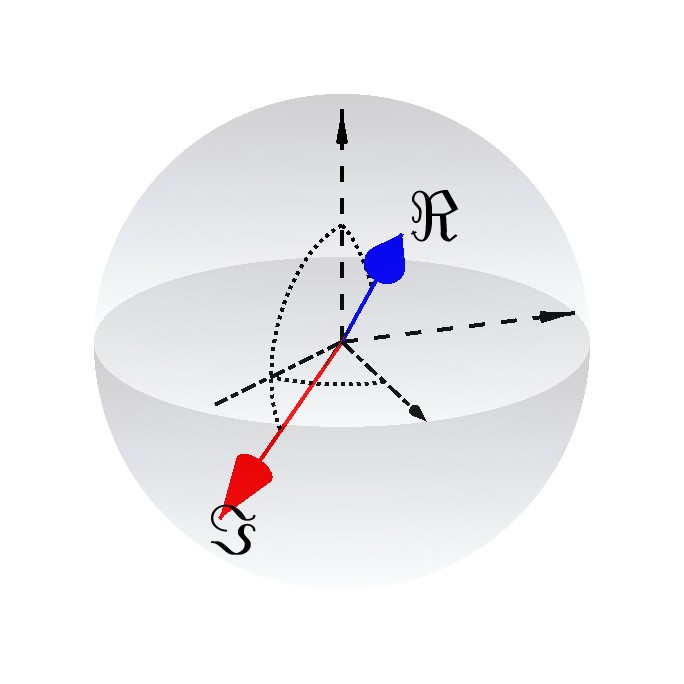}\\
\includegraphics[width=.2\linewidth,valign=m]{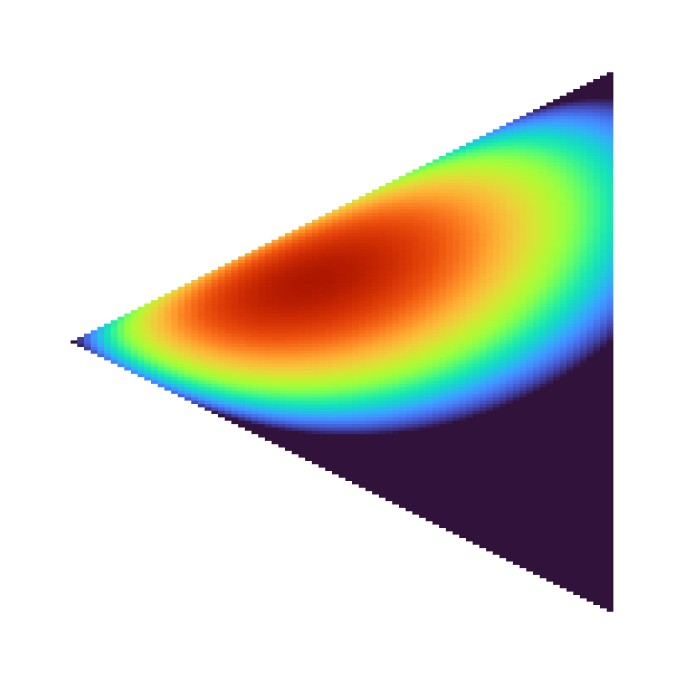} & \includegraphics[width=.2\linewidth,valign=m]{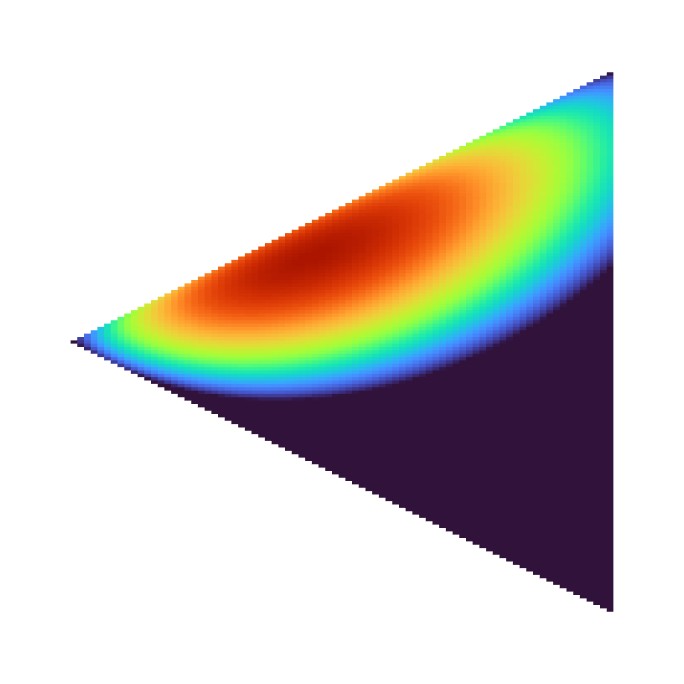} & \includegraphics[width=.2\linewidth,valign=m]{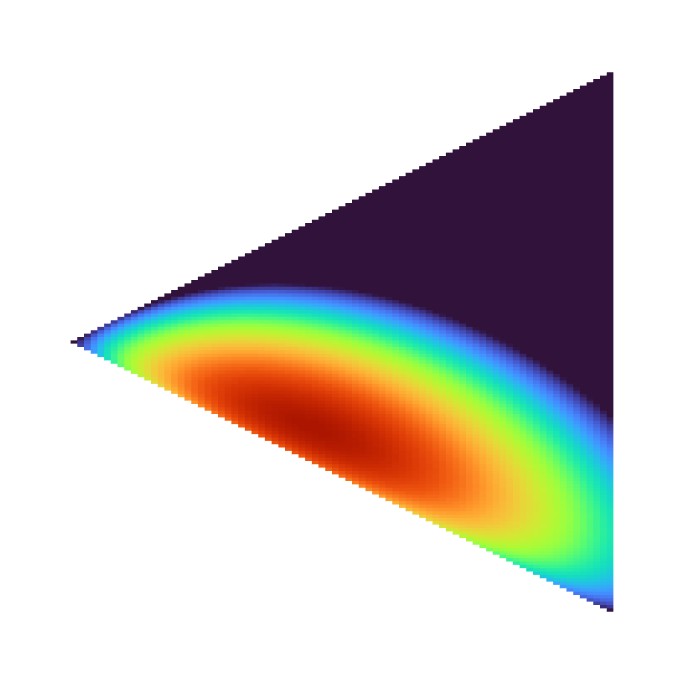} & \includegraphics[width=.2\linewidth,valign=m]{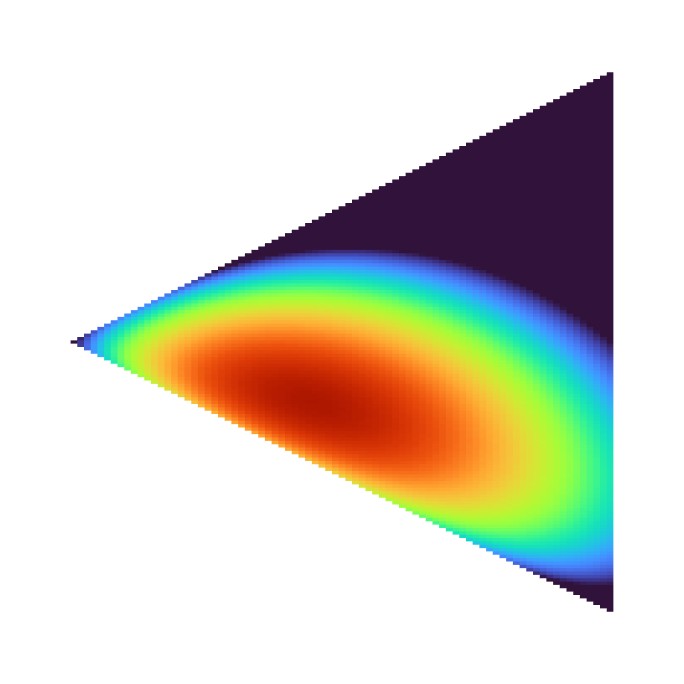}\\
\rule{0pt}{3.3ex} 
$\hat{\phi}_{\ell}^m\left(\frac{\pi}{2},\frac{\pi}{4},120\right)$ & $\hat{\phi}_{\ell}^m\left(\frac{\pi}{2},\frac{\pi}{2},120\right)$ & $\hat{\phi}_{\ell}^m\left(\frac{\pi}{2},\frac{3\pi}{2},120\right)$ & $\hat{\phi}_{\ell}^m\left(\frac{\pi}{2},\frac{7\pi}{4},120\right)$\\
\includegraphics[width=.2\linewidth,valign=m]{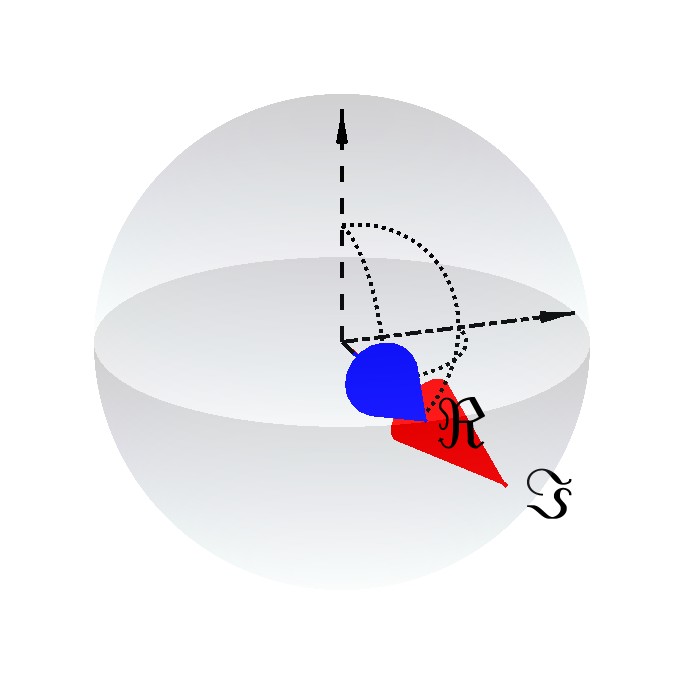} & \includegraphics[width=.2\linewidth,valign=m]{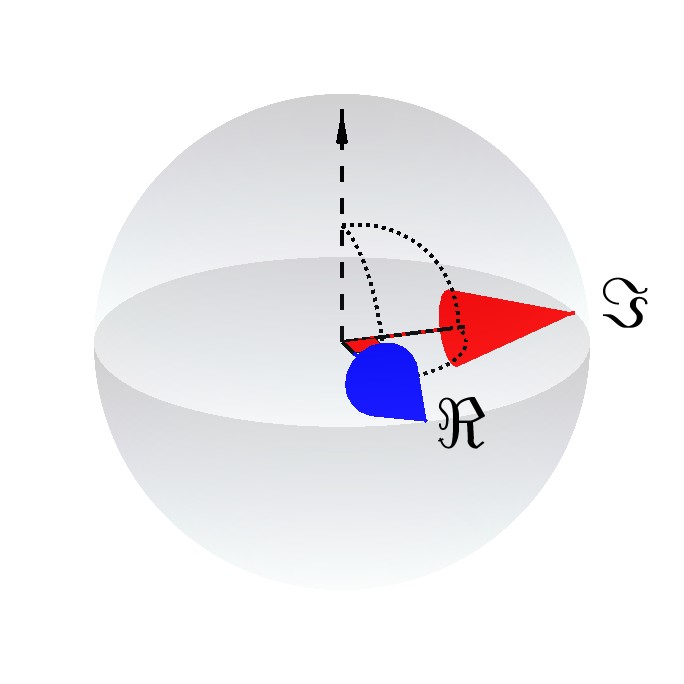} & \includegraphics[width=.2\linewidth,valign=m]{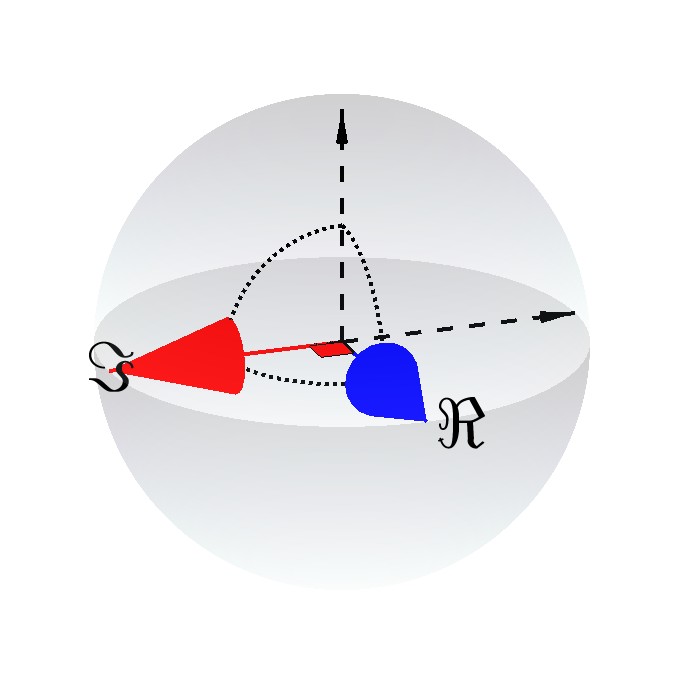} & \includegraphics[width=.2\linewidth,valign=m]{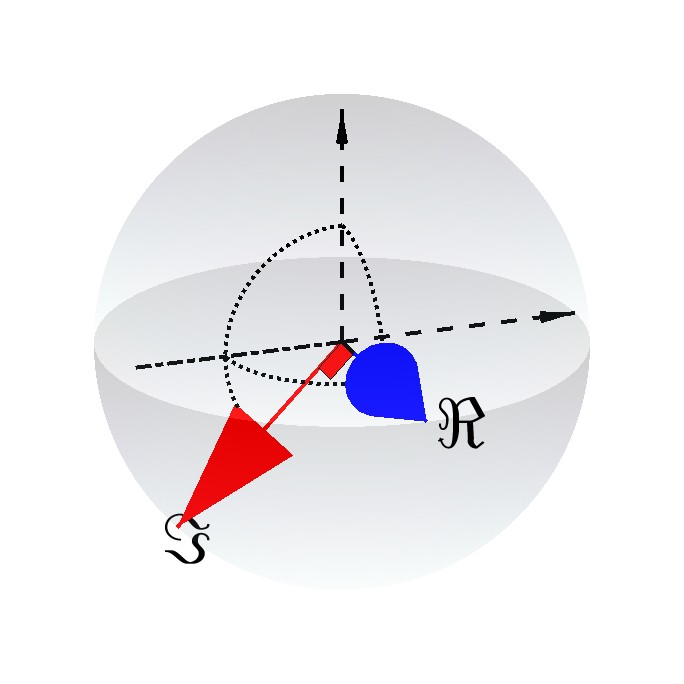}\\
\includegraphics[width=.2\linewidth,valign=m]{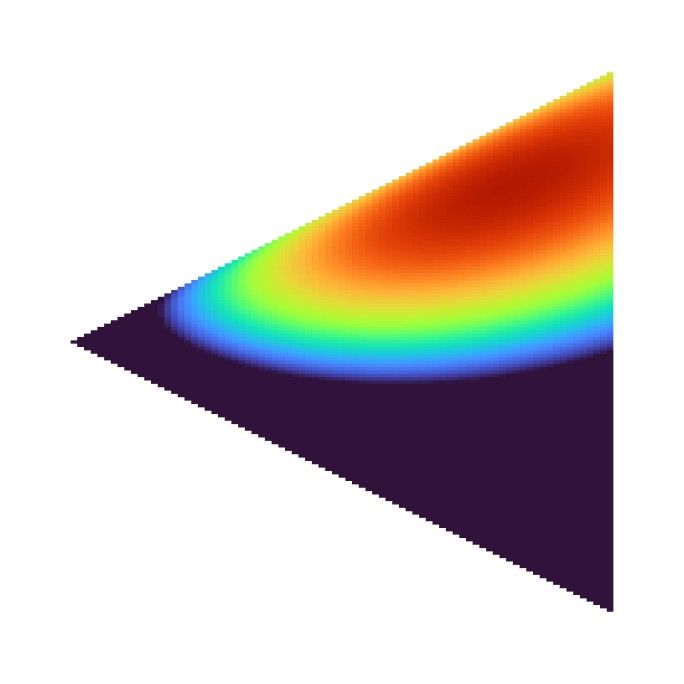} & \includegraphics[width=.2\linewidth,valign=m]{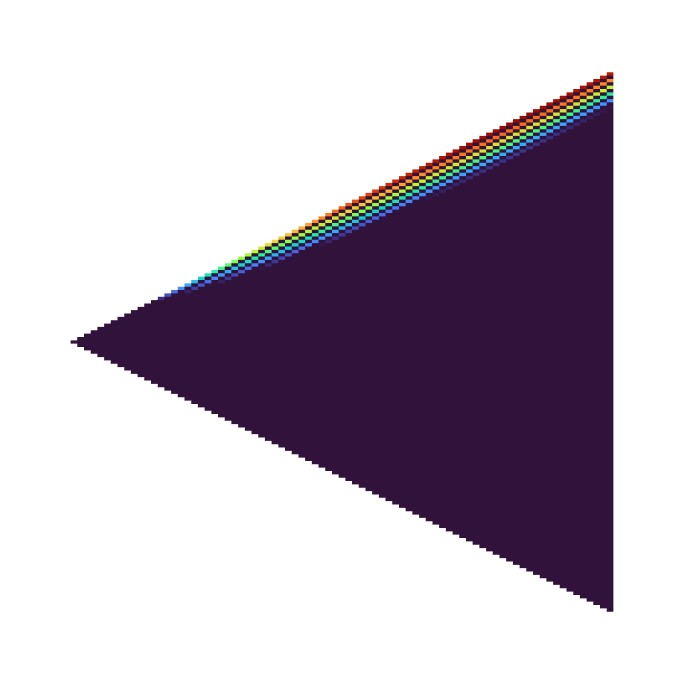} & \includegraphics[width=.2\linewidth,valign=m]{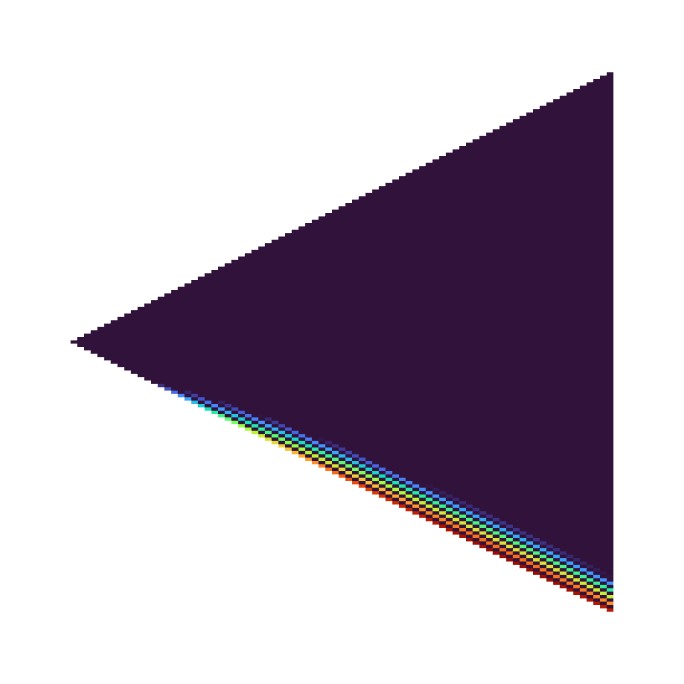} & \includegraphics[width=.2\linewidth,valign=m]{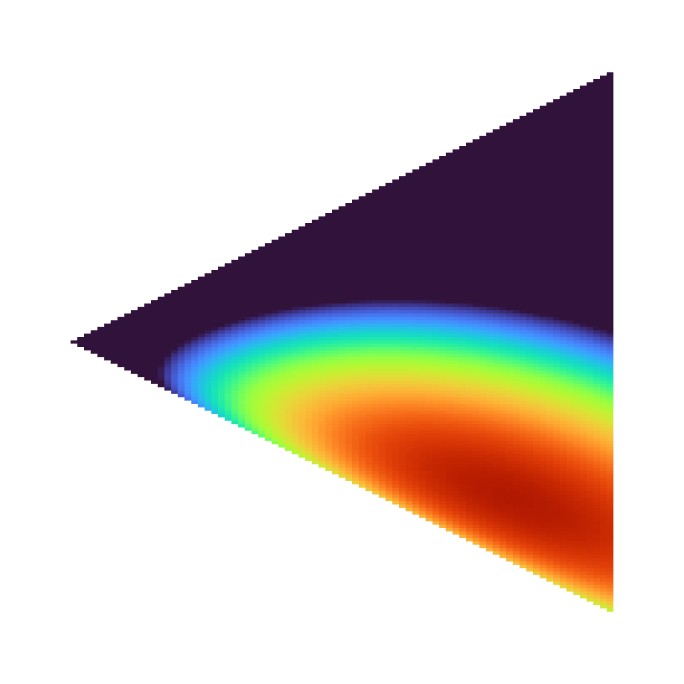}\\
\rule{0pt}{3.3ex}
$\hat{\phi}_{\ell}^m\left(\frac{\pi}{4},\frac{\pi}{4},120\right)$ & $\hat{\phi}_{\ell}^m\left(\frac{\pi}{4},\frac{\pi}{2},120\right)$ & $\hat{\phi}_{\ell}^m\left(\frac{\pi}{4},\frac{3\pi}{2},120\right)$ & $\hat{\phi}_{\ell}^m\left(\frac{\pi}{4},\frac{7\pi}{4},120\right)$\\
\includegraphics[width=.2\linewidth,valign=m]{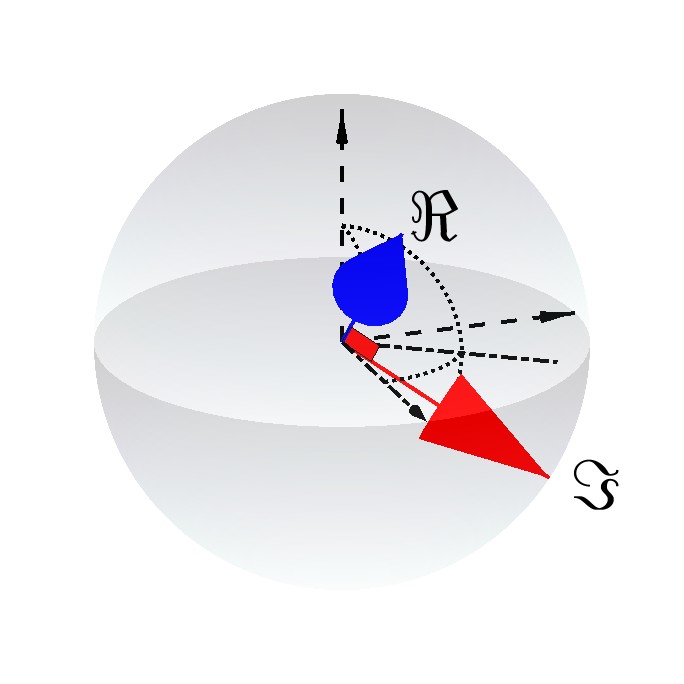} & \includegraphics[width=.2\linewidth,valign=m]{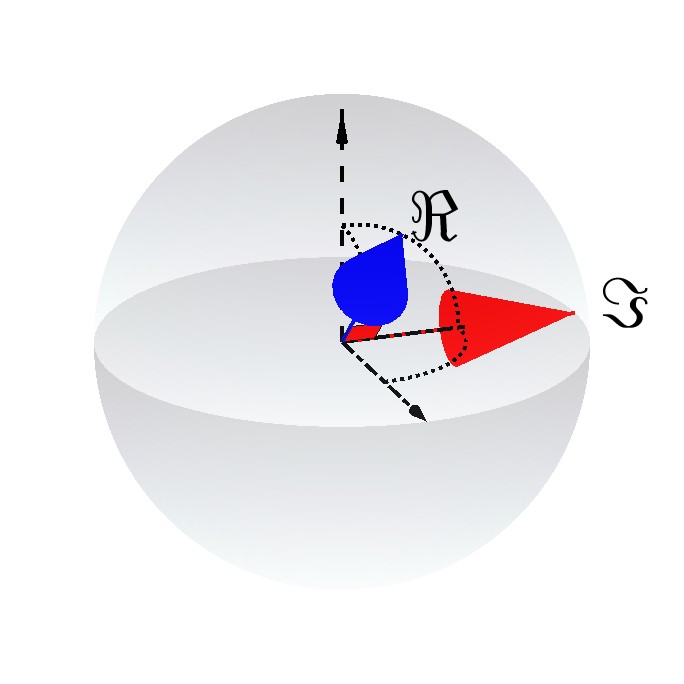} & \includegraphics[width=.2\linewidth,valign=m]{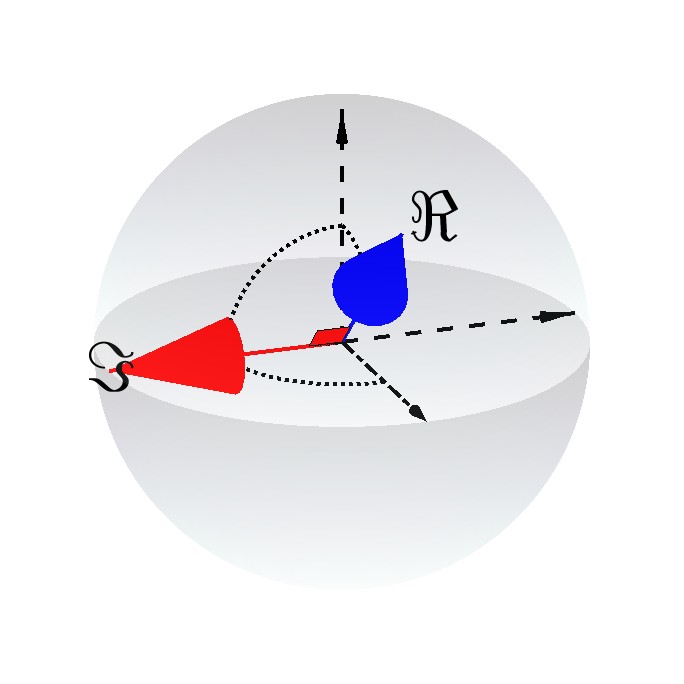} & \includegraphics[width=.2\linewidth,valign=m]{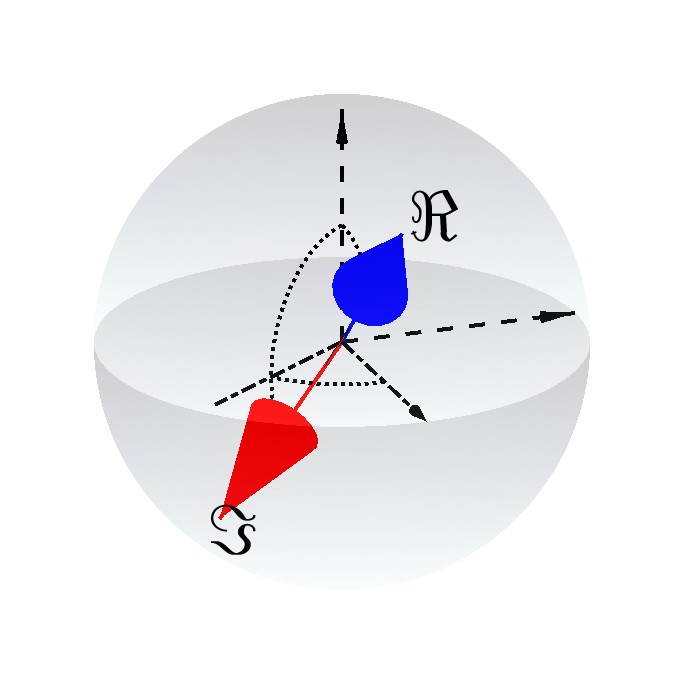}\\
\includegraphics[width=.2\linewidth,valign=m]{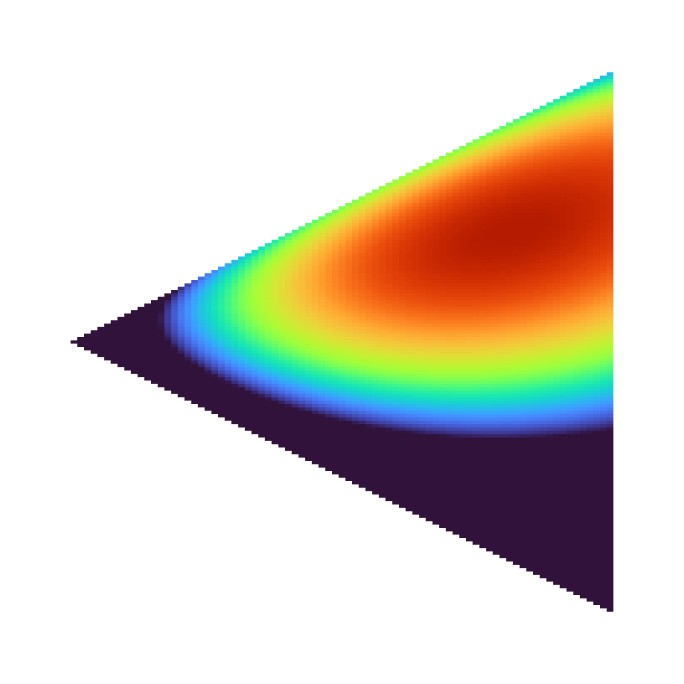} & \includegraphics[width=.2\linewidth,valign=m]{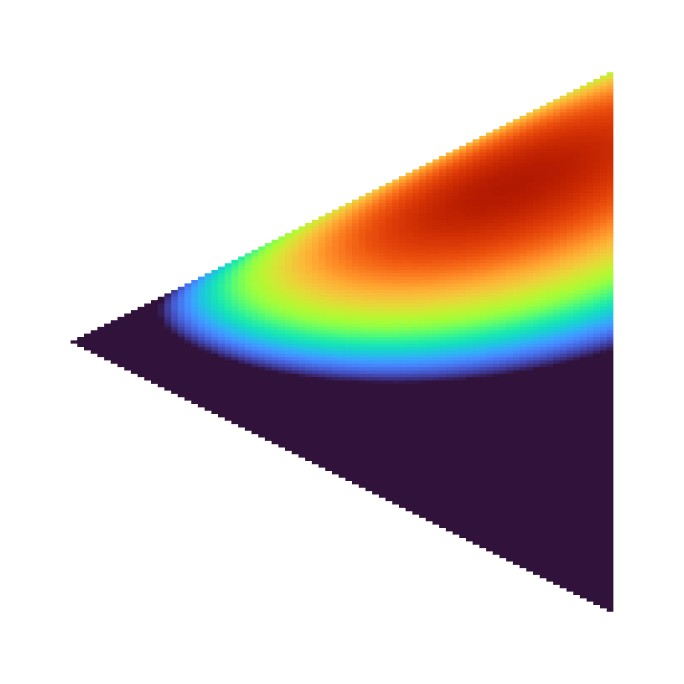} & \includegraphics[width=.2\linewidth,valign=m]{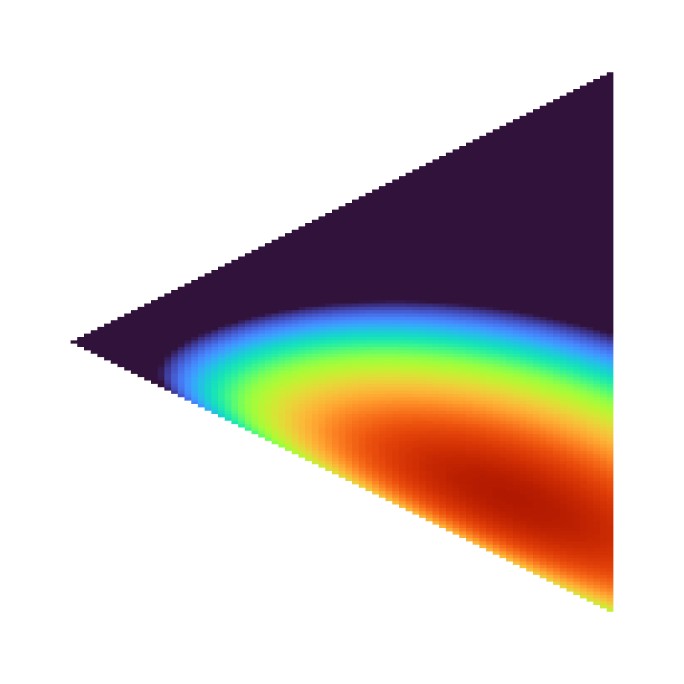} & \includegraphics[width=.2\linewidth,valign=m]{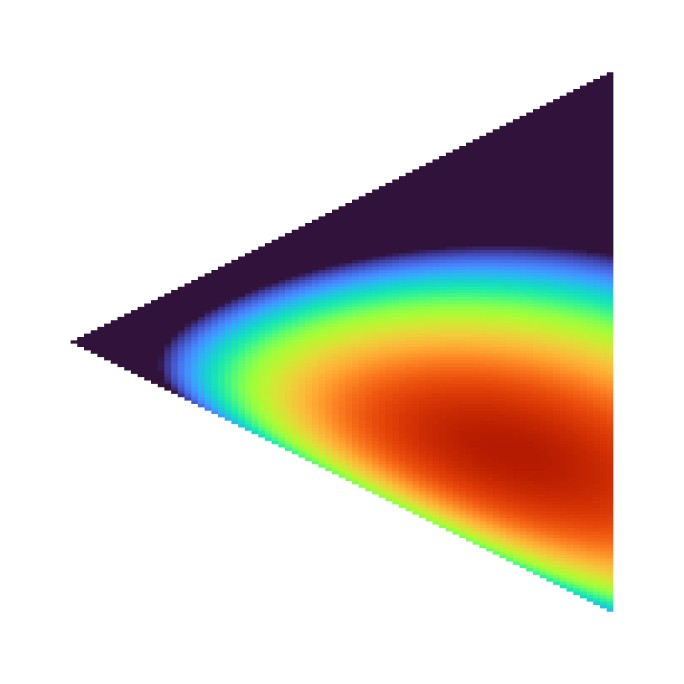}\\
\vspace{1mm}\\
\multicolumn{4}{c}{\includegraphics[width=0.85\linewidth]{imagespdf/colorbar2.pdf}}\\
\end{tabular}
\end{figure}
\begin{figure}[H]
\centering
\begin{tabular}{cccc}
$\hat{\phi}_{\ell}^m\left(0,0,30\right)$ & $\hat{\phi}_{\ell}^m\left(0,0,90\right)$ & $\hat{\phi}_{\ell}^m\left(0,0,180\right)$ & $\hat{\phi}_{\ell}^m\left(0,0,240\right)$\\
\includegraphics[width=.2\linewidth,valign=m]{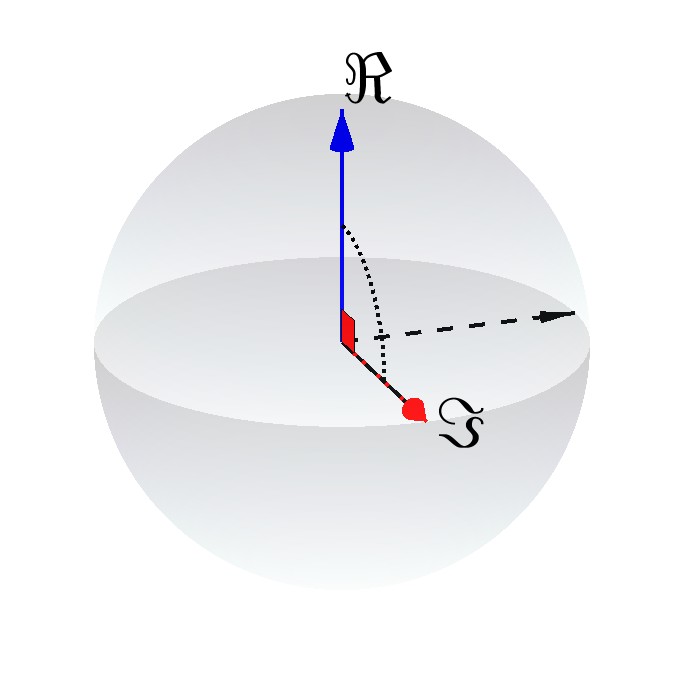} & \includegraphics[width=.2\linewidth,valign=m]{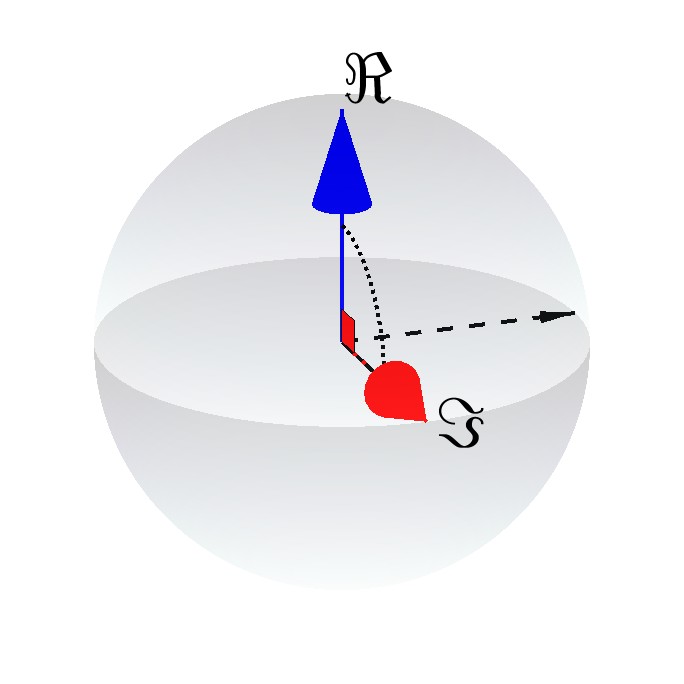} & \includegraphics[width=.2\linewidth,valign=m]{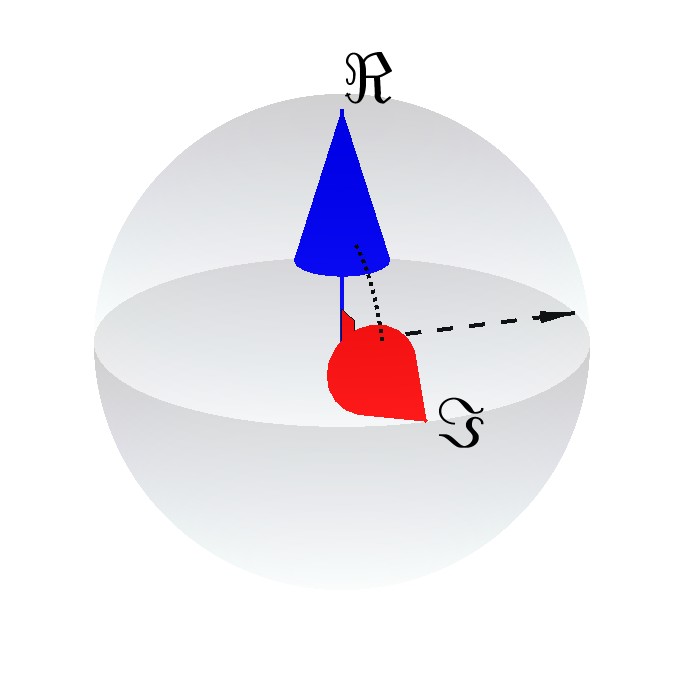} & \includegraphics[width=.2\linewidth,valign=m]{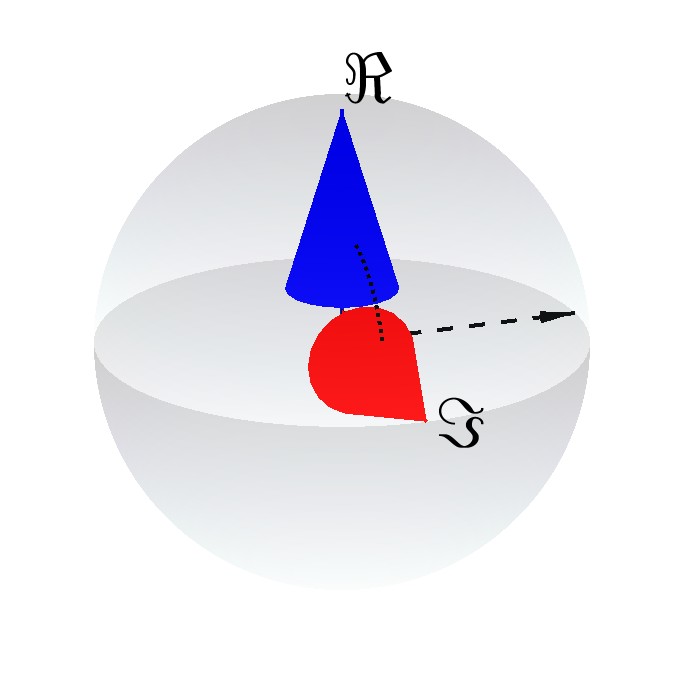}\\
\includegraphics[width=.2\linewidth,valign=m]{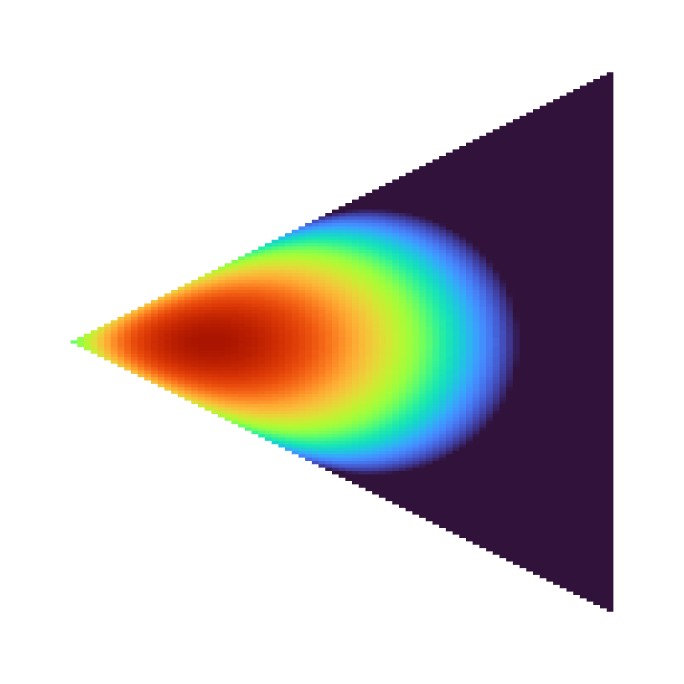} & \includegraphics[width=.2\linewidth,valign=m]{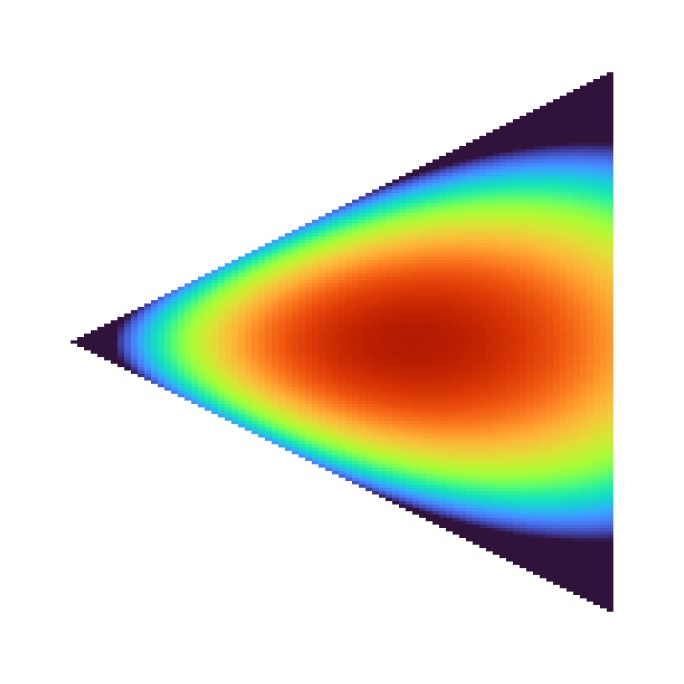} & \includegraphics[width=.2\linewidth,valign=m]{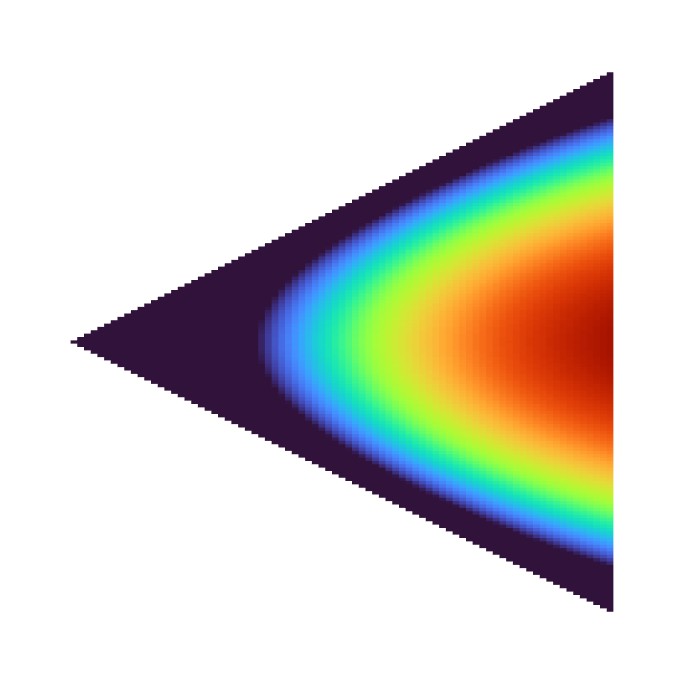} & \includegraphics[width=.2\linewidth,valign=m]{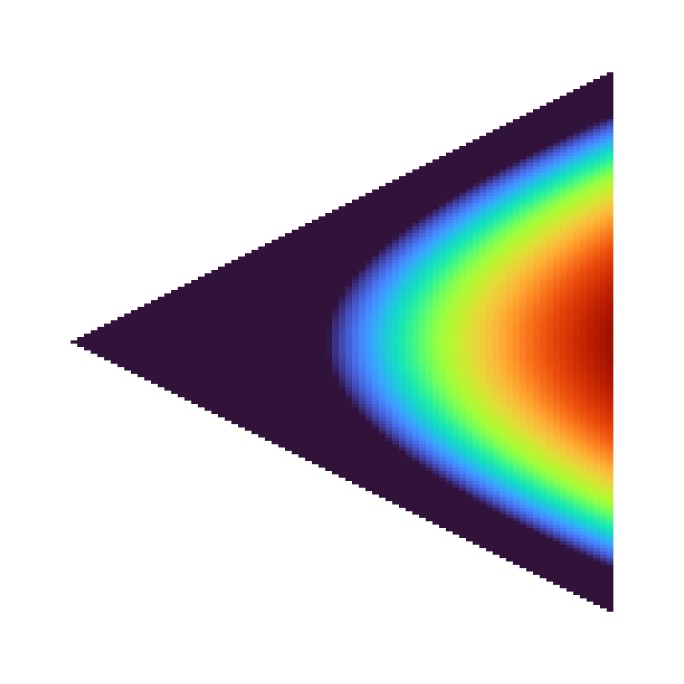}\\
\vspace{1mm}\\
\multicolumn{4}{c}{\includegraphics[width=0.85\linewidth]{imagespdf/colorbar2.pdf}}\\
\end{tabular}
\caption{Modal analysis of the evanescent plane waves. For each wave: (above) representations of both real and imaginary components (denoted respectively by $\Re$ and $\Im$ in figure to simplify the notation) of direction vectors $\mathbf{d}(\mathbf{y})$ with fixed azimuthal angle $\theta_2=0$. The size of the arrowheads is proportional to the norm of the vector and therefore dependent on $\zeta$ and $\kappa$; (below) related distributions of the coefficients $\hat{\phi}_{\ell}^m(\theta_1,\theta_3,\zeta)$ in (\ref{evanescent coefficients}) for different values of $\theta_1$, $\theta_3$ and $\zeta$. The index $\ell$ varies along the abscissa within the range $0 \leq \ell \leq 80$, while the index $m$ varies along the ordinate within the range $0 \leq |m| \leq \ell$ forming a triangle. We have conveniently normalized the coefficients according to a normalization factor (depending only on $\zeta$) which is described in the Chapter \protect\hyperlink{Chapter 6}{6}, namely the square root of $\mu_N$ in (\ref{rho density}), and computed using the approximation in (\ref{alphal approximation}). Wavenumber $\kappa=6$.} \label{figure 4.4}
\end{figure}
\vspace{-3mm}

In fact, thanks to (\ref{yuy}), we can write
\vspace{-1mm}
\begin{equation}
\phi_{\mathbf{y}}=\sum_{\ell=0}^{\infty}\sum_{m=-\ell}^{\ell}\left(\phi_{\mathbf{y}},b_{\ell}^m\right)_{\mathcal{B}}b_{\ell}^m=\sum_{\ell=0}^{\infty}\hat{\phi}_{\ell}(\zeta)\,b_{\ell}[\mathbf{y}],
\label{expansion2 evanescent}
\end{equation}
where $b_{\ell}[\mathbf{y}] \in \text{span}\{b_{\ell}^m\}_{m=-\ell}^{\ell}$ with $\|b_{\ell}[\mathbf{y}]\|_{\mathcal{B}}=1$ and
\vspace{-1mm}
\begin{equation}
\hat{\phi}_{\ell}(\zeta)=\left(\sum_{m=-\ell}^{\ell}\left[\hat{\phi}_{\ell}^m(\theta_1,\theta_3,\zeta)\right]^2\right)^{1/2}\!\!\!\!\!=\,\left(\sum_{m=-\ell}^{\ell}\left|\left(\phi_{\mathbf{y}},b_{\ell}^m\right)_{\mathcal{B}}\right|^2\right)^{1/2}\!\!\!\!=\,\frac{4\pi}{\beta_{\ell}}\left|\mathbf{P}_{\ell}(\zeta)\right|.
\label{evanescent l2 coefficients}
\end{equation}
The last equality in (\ref{evanescent l2 coefficients}) holds due to (\ref{evanescent coefficients}) and the unitarity condition \cite[Sec.\ 4.1, Eq.\ (6)]{quantumtheory}. Moreover, note that, setting $\zeta=0$, we fall back into the propagative case (\ref{l2 coefficients}). The distribution of the coefficients $\hat{\phi}_{\ell}(\zeta)$ in (\ref{evanescent l2 coefficients}) is depicted in Figure \ref{figure 4.3} for different values of the evanescence parameter $\zeta$.

\begin{figure}
\centering
\includegraphics[width=8.8cm]{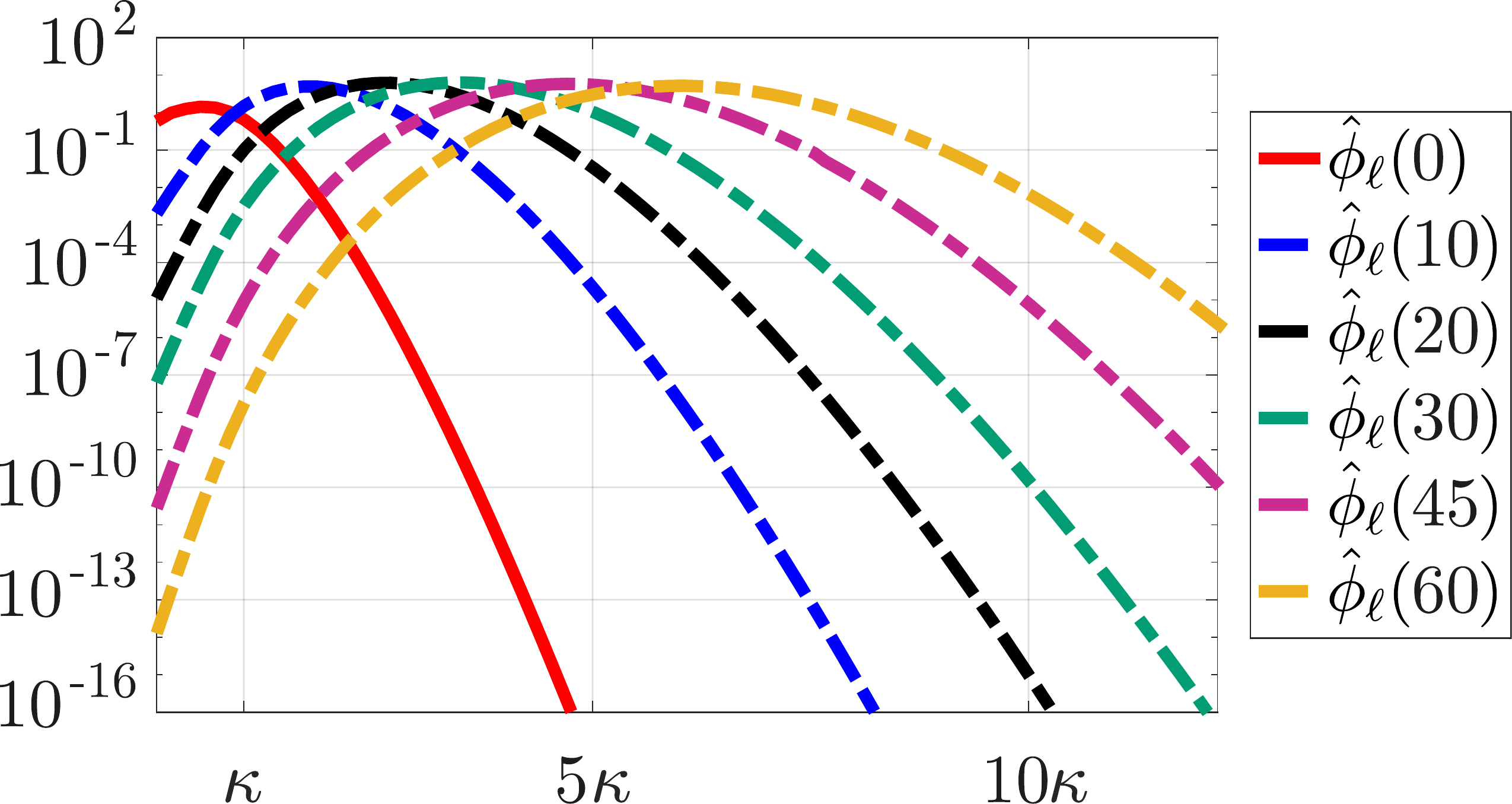}
\caption{Modal analysis of the evanescent plane waves: distribution of the coefficients $\hat{\phi}_{\ell}(\zeta)$ in (\ref{evanescent l2 coefficients}) for different values of the evanescence parameter $\zeta$.
For each $\zeta$, this corresponds to taking the $\ell^2$-norms along the vertical segments of the values in the triangles describing the distribution of the coefficients $\hat{\phi}_{\ell}^m(\theta_1,\theta_3,\zeta)$ in (\ref{evanescent coefficients}), such as those in Figure \ref{figure 4.1}.
Observe that by increasing the value of $\zeta$ it is possible to cover higher-$\ell$ Fourier modes, contrary to what happens in the propagative case of Figure \ref{figure 3.1}. We have conveniently normalized the coefficients according to a normalization factor (depending only on $\zeta$) which is described in Chapter \protect\hyperlink{Chapter 6}{6}, namely the square root of $\mu_N$ in (\ref{rho density}), and computed using the approximation in (\ref{alphal approximation}). Wavenumber $\kappa=6$.}
\label{figure 4.3}
\end{figure}

Lastly, let us now examine the symmetry properties of the coefficients in (\ref{evanescent coefficients}) in more detail.
Similarly to the propagative case, we can again limit ourselves to considering $\theta_1 \in [0,\pi/2]$: in fact, thanks to (\ref{negative legendre2 polynomials}) and the Wigner d-matrix symmetry property $d^{\,m,m'}_{\ell}(\theta)=(-1)^{\ell+m'}d^{\,-m,m'}_{\ell}(\pi-\theta)$ in \cite[Sec.\ 4.4, Eq.\ (1)]{quantumtheory}, we have
\begin{align*}
\hat{\phi}_{\ell}^m(\theta_1,\theta_3,\zeta)&=4\pi\beta_{\ell}^{-1}\left|\textstyle \sum_{m'=-\ell}^{\ell}\gamma_{\ell}^{m'}i^{m'}d_{\ell}^{\,m',m}(\theta_1)e^{im'\theta_3}P_{\ell}^{m'}\left(\zeta/2\kappa+1\right)\right|\\
&=4\pi\beta_{\ell}^{-1}\left|\textstyle \sum_{m'=-\ell}^{\ell}\gamma_{\ell}^{-m'}i^{m'}d_{\ell}^{\,-m',m}(\pi-\theta_1)e^{im'\theta_3}P_{\ell}^{-m'}\left(\zeta/2\kappa+1\right)\right|\\
&=4\pi\beta_{\ell}^{-1}\left|\textstyle \sum_{m'=-\ell}^{\ell}\gamma_{\ell}^{m'}i^{-m'}d_{\ell}^{\,m',m}(\pi-\theta_1)e^{-im'\theta_3}P_{\ell}^{m'}\left(\zeta/2\kappa+1\right)\right|\\
&=\hat{\phi}_{\ell}^m(\pi-\theta_1,\theta_3,\zeta). \numberthis \label{sym1}
\end{align*}
Furthermore, we can assume $\theta_3 \in [0,\pi/2] \cup [3\pi/2,2\pi)$ (or analogously in $[\pi/2,3\pi/2]$):
\vspace{-1.5mm}
\begin{align*}
\hat{\phi}_{\ell}^m(\theta_1,\theta_3,\zeta)&=4\pi\beta_{\ell}^{-1}\left|\textstyle \sum_{m'=-\ell}^{\ell}\gamma_{\ell}^{m'}i^{m'}d_{\ell}^{\,m',m}(\theta_1)e^{im'\theta_3}P_{\ell}^{m'}\left(\zeta/2\kappa+1\right)\right|\\
&=4\pi\beta_{\ell}^{-1}\left|\textstyle \sum_{m'=-\ell}^{\ell}\gamma_{\ell}^{m'}i^{-m'}d_{\ell}^{\,m',m}(\theta_1)e^{-im'(\pi-\theta_3)}P_{\ell}^{m'}\left(\zeta/2\kappa+1\right)\right|\\
&=\hat{\phi}_{\ell}^m(\theta_1,\pi-\theta_3,\zeta). \numberthis \label{sym2}
\end{align*}
The symmetric behavior of the coefficients by varying the value of $\theta_3$ is also relevant. Using the property $d_{\ell}^{\,m,m'}\!(\theta)\!=\!(-1)^{m'-m}d_{\ell}^{\,-m,-m'}\!(\theta)$ in \cite[Sec.\ 4.4, Eq.\ (1)]{quantumtheory}, it follows:
\vspace{-10mm}
\begin{align*}
\hat{\phi}_{\ell}^m(\theta_1,\pi-\theta_3,\zeta)&=4\pi\beta_{\ell}^{-1}\left|\textstyle \sum_{m'=-\ell}^{\ell}\gamma_{\ell}^{m'}i^{-m'}d_{\ell}^{\,m',m}(\theta_1)e^{-im'(\pi-\theta_3)}P_{\ell}^{m'}\left(\zeta/2\kappa+1\right)\right|\\
&=4\pi\beta_{\ell}^{-1}\left|\textstyle \sum_{m'=-\ell}^{\ell}\gamma_{\ell}^{-m'}i^{-m'}d_{\ell}^{\,-m',-m}(\theta_1)e^{im'\theta_3}P_{\ell}^{-m'}\left(\zeta/2\kappa+1\right)\right|\\
&=4\pi\beta_{\ell}^{-1}\left|\textstyle \sum_{m'=-\ell}^{\ell}\gamma_{\ell}^{m'}i^{m'}d_{\ell}^{\,m',-m}(\theta_1)e^{-im'\theta_3}P_{\ell}^{m'}\left(\zeta/2\kappa+1\right)\right|\\
&=4\pi\beta_{\ell}^{-1}\left|\textstyle \sum_{m'=-\ell}^{\ell}\gamma_{\ell}^{m'}i^{-m'}d_{\ell}^{\,m',-m}(\theta_1)e^{-im'(\pi+\theta_3)}P_{\ell}^{m'}\left(\zeta/2\kappa+1\right)\right|\\
&=\hat{\phi}_{\ell}^{-m}(\theta_1,\pi+\theta_3,\zeta). \numberthis \label{sym3}
\end{align*}
Moreover, thanks to the identity $d_{\ell}^{\,m,m'}(0)=\delta_{m,m'}$ (that follows directly from (\ref{d matrix})), it is possible to observe that $\hat{\phi}_{\ell}^{m}(0,\theta_3,\zeta)$ is actually independent of the value of $\theta_3$:
\begin{equation*}
\hat{\phi}_{\ell}^{m}(0,\theta_3,\zeta)=4\pi\beta_{\ell}^{-1}\left|\gamma_{\ell}^{m}e^{-im\theta_3}P_{\ell}^{m}\left(\zeta/2\kappa+1\right)\right|=4\pi\beta_{\ell}^{-1}\gamma_{\ell}^{m}P_{\ell}^{m}\left(\zeta/2\kappa+1\right).
\end{equation*}
In addition, the sets of coefficients of the form $\hat{\phi}_{\ell}^{m}(\pi/2,\pi/2,\zeta)$ and $\hat{\phi}_{\ell}^{m}(\pi/2,3\pi/2,\zeta)$ depict the same checkerboard-pattern as in Figure \ref{figure 3.1} for the case $\theta_1=\pi/2$. For instance, setting $\theta_1=\theta_3=\pi/2$ and denoting the addends in the sum of (\ref{evanescent coefficients}) with $c_{\ell}^{m',m}(\zeta)$, we have
\begin{align*}
c_{\ell}^{-m',m}(\zeta)&=\gamma_{\ell}^{-m'}i^{m'}d_{\ell}^{\,-m',m}\left(\pi/2\right)e^{i m'\pi/2}P_{\ell}^{-m'}\left(\zeta/2\kappa+1\right)\\
&=(-1)^{\ell+m+m'}\gamma_{\ell}^{-m'}d_{\ell}^{\,m',m}\left(\pi/2\right)P_{\ell}^{-m'}\left(\zeta/2\kappa+1\right)\\
&=(-1)^{\ell+m+m'}\gamma_{\ell}^{m'}d_{\ell}^{\,m',m}\left(\pi/2\right)P_{\ell}^{m'}\left(\zeta/2\kappa+1\right)\\
&=(-1)^{\ell+m}\gamma_{\ell}^{m'}i^{-m'}d_{\ell}^{\,m',m}\left(\pi/2\right) e^{-im' \pi/2}P_{\ell}^{m'}\left(\zeta/2\kappa+1\right)\\
&=(-1)^{\ell+m}c_{\ell}^{m',m}(\zeta),
\end{align*}
where we used the Wigner d-matrix property $d^{\,m,m'}_{\ell}(\pi/2)=(-1)^{\ell+m'}d^{\,-m,m'}_{\ell}(\pi/2)$ in \cite[Sec. 4.4, Eq. (1)]{quantumtheory}.
Therefore, if $\ell+m$ is odd, we have that $c_{\ell}^{-m',m}(\zeta)=-c_{\ell}^{m',m}(\zeta)$ and hence, due to (\ref{D-matrix0}), it follows
\begin{align*}
\hat{\phi}_{\ell}^{m}(\pi/2,\pi/2,\zeta)&=4\pi\beta_{\ell}^{-1}\left|\textstyle \sum_{m'=-\ell}^{\ell}c_{\ell}^{m',m}(\zeta)\right|=4\pi\beta_{\ell}^{-1}\left|c_{\ell}^{0,m}(\zeta)\right|\\
&=4\pi\beta_{\ell}^{-1}\gamma_{\ell}^{0}P_{\ell}\left(\zeta/2\kappa+1\right)\left|d_{\ell}^{\,0,m}\left(\pi/2\right)\right|\\
&=4\pi\beta_{\ell}^{-1}\gamma_{\ell}^{m}P_{\ell}\left(\zeta/2\kappa+1\right)\left|\mathsf{P}_{\ell}^m(0)\right|=0,\,\,\,\,\,\,\,\,\,\,\,\,\,\,\,\,\,\,\,\,\,\,\,\,\,\,\,\,\text{if }\ell+m\text{ is odd}.
\end{align*}
The same can be seen for $\hat{\phi}_{\ell}^{m}(\pi/2,3\pi/2,\zeta)$. It is easy to believe that, playing with the symmetry properties of the Wigner d-matrix and with those related to the Ferrers functions (\ref{legendre polynomials}) or the associated Legendre polynomials (\ref{legendre2 polynomials}) as in the previous formulae, many other relations of this type can be deduced.

To conclude, our findings suggest that evanescent plane waves are able to accurately capture the high Fourier modes of Helmholtz solutions that have less regularity. However, selecting the correct values for the evanescence parameters $\theta_3$ and $\zeta$ to create approximation spaces of a reasonable size remains a significant challenge. This issue will be the main focus of the rest of the paper.

\chapter{Herglotz transform}
\hypertarget{Chapter 5}{In} this chapter, following \cite[Sec.\ 6]{parolin-huybrechs-moiola}, we introduce a family of functions $\{a_{\ell}^m\}_{(\ell,m) \in \mathcal{I}}$ defined on the parametric domain $\Theta \times [0,+\infty)$ and consequentially the space $\mathcal{A}$ generated by them. We present some lemmas that show how $\mathcal{A}$, provided with a suitable weighted norm, is indeed a Hilbert space, of which the functions $a_{\ell}^m$ constitute an orthonormal basis. The asymptotic behavior of their related normalization coefficients $\alpha_{\ell}$ -- defined as in (\ref{b tilde definizione}) -- turns out to be reciprocal to the one of $\beta_{\ell}$, i.e.\ $\alpha_{\ell} \sim \beta_{\ell}^{-1}$ as $\ell$ goes to infinity, which allows us to introduce the notion of \textit{Herglotz transform} $T$ between the spaces $\mathcal{A}$ and the space $\mathcal{B}$ of Helmholtz solutions in the unit ball $B_1$. The integral representation of $T$ can be seen as a generalization of the Herglotz classical one (\ref{herglotz}).
This suggests us to call $\mathcal{A}$ the \textit{space of Herglotz densities}.
Furthermore, we prove that the operator $T$ is bounded and invertible.

Hence, any Helmholtz solution in the unit ball $B_1$ can be uniquely represented as a continuous superposition of evanescent plane waves and, moreover, its corresponding density is bounded in a suitable norm, i.e.\ the $\mathcal{A}$ norm.

This result indicates that evanescent plane waves are a continuous frame for the space of Helmholtz solutions.
This property lays the foundation for achieving stable and accurate discrete expansions.
As a significant implication, the space of Herglotz densities has the reproducing kernel property, meaning that point-evaluation functionals are continuous. Therefore, due to Riesz theorem, any point-evaluation functional can be identified with an element of $\mathcal{A}$, and thus mapped into an evanescent plane wave through the Herglotz transform.

\section{Space of Herglotz densities}

\hypertarget{Section 5.1}{To} shorten notations, we denote in the following the parametric domain as
\begin{equation*}
Y:=\Theta \times [0,+\infty).
\end{equation*}
We introduce a weighted $L^2$ space defined on $Y$. The weight function is
\begin{equation}
w(\mathbf{y})=w(\theta_1,\zeta):=\sin{(\theta_1)}\,\zeta^{1/2}e^{-\zeta},\,\,\,\,\,\,\,\,\,\,\forall \mathbf{y} \in Y.
\label{weight}
\end{equation}
\begin{figure}[H]
\centering
\begin{tabular}{c|cccccc}
\, & \, $\zeta=10^{\text{-}3}$ & \, $\zeta=10^{\text{-}2}$ & \, $\zeta=10^{\text{-}1}$ & \, $\zeta=10^{0}$ & \, $\zeta=10^{1}$ & \, $\zeta=10^{2}$\\
\vspace{-3mm}\\
\hline
\vspace{-3mm}\\
$m=0$ & \includegraphics[trim=30 30 30 30,clip,width=.118\linewidth,valign=m]{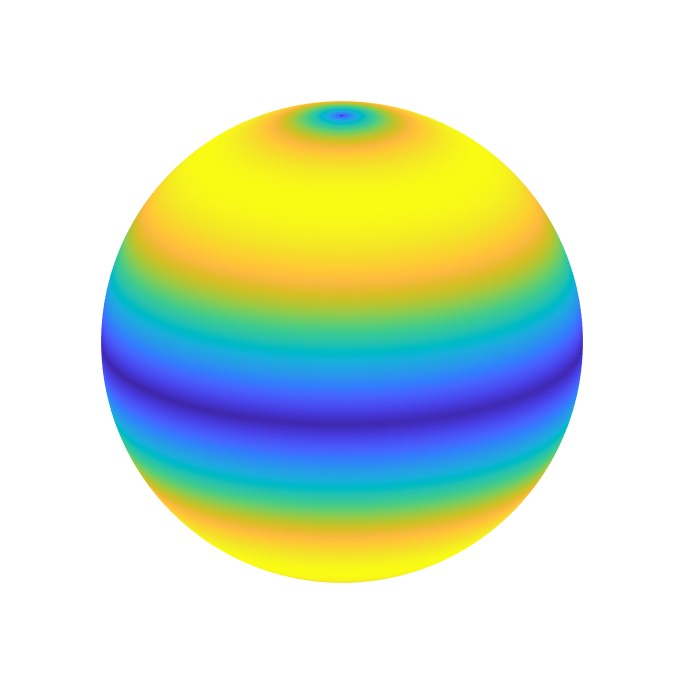} & \includegraphics[trim=30 30 30 30,clip,width=.118\linewidth,valign=m]{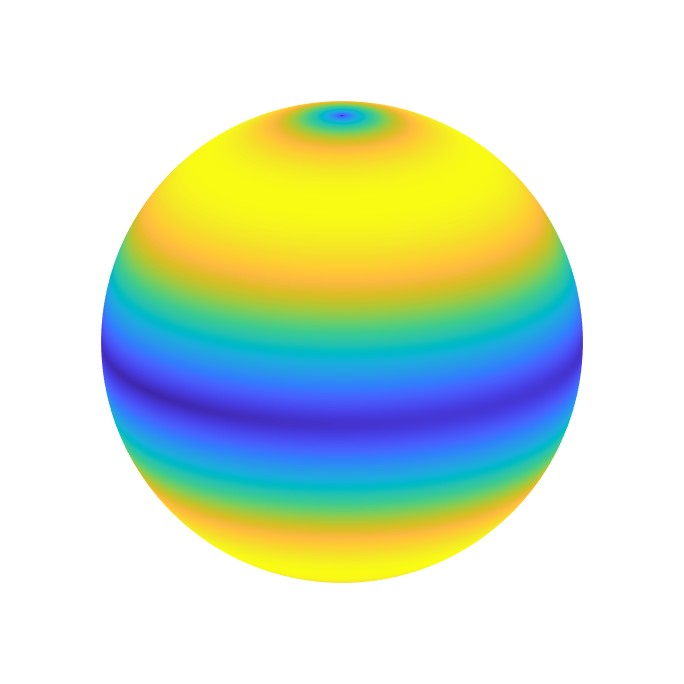} & \includegraphics[trim=30 30 30 30,clip,width=.118\linewidth,valign=m]{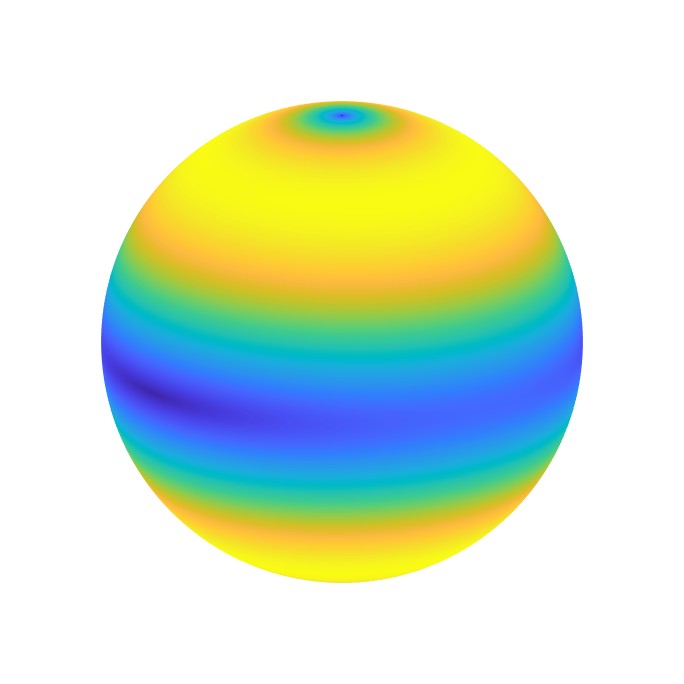} & \includegraphics[trim=30 30 30 30,clip,width=.118\linewidth,valign=m]{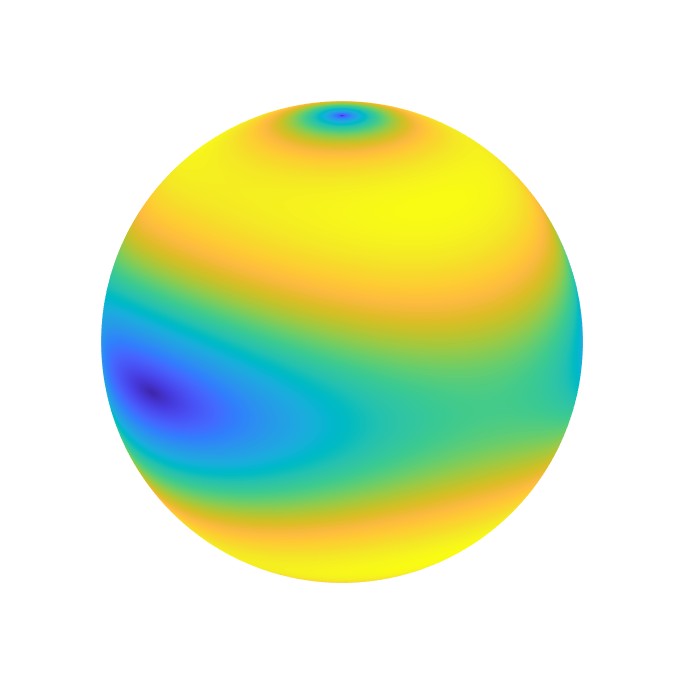} & \includegraphics[trim=30 30 30 30,clip,width=.118\linewidth,valign=m]{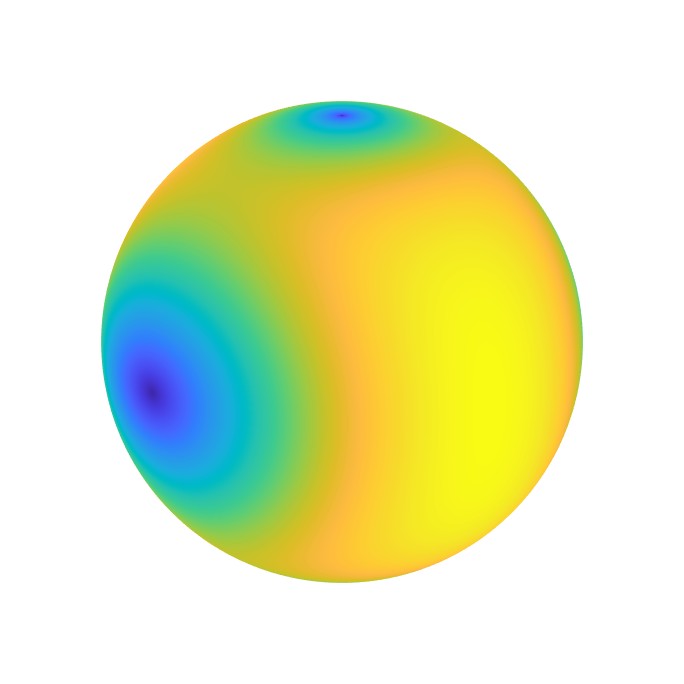} & \includegraphics[trim=30 30 30 30,clip,width=.118\linewidth,valign=m]{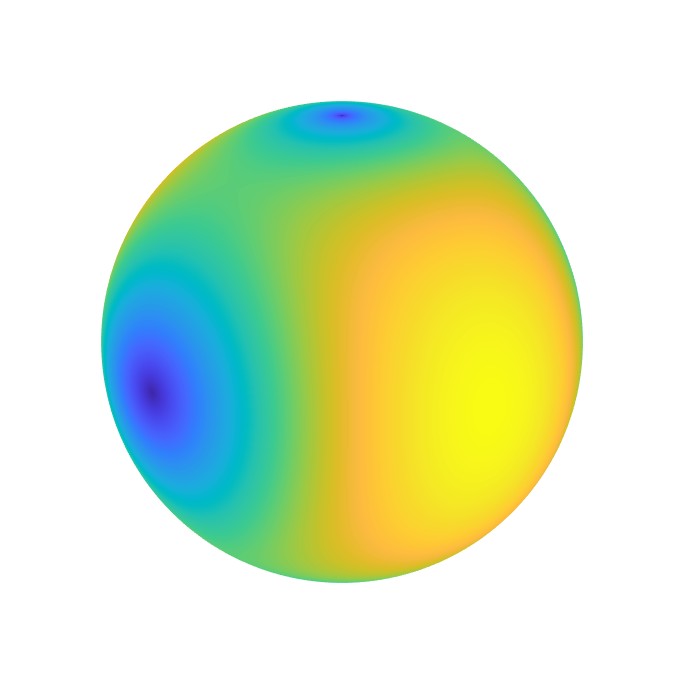}\\
$m=1$ & \includegraphics[trim=30 30 30 30,clip,width=.118\linewidth,valign=m]{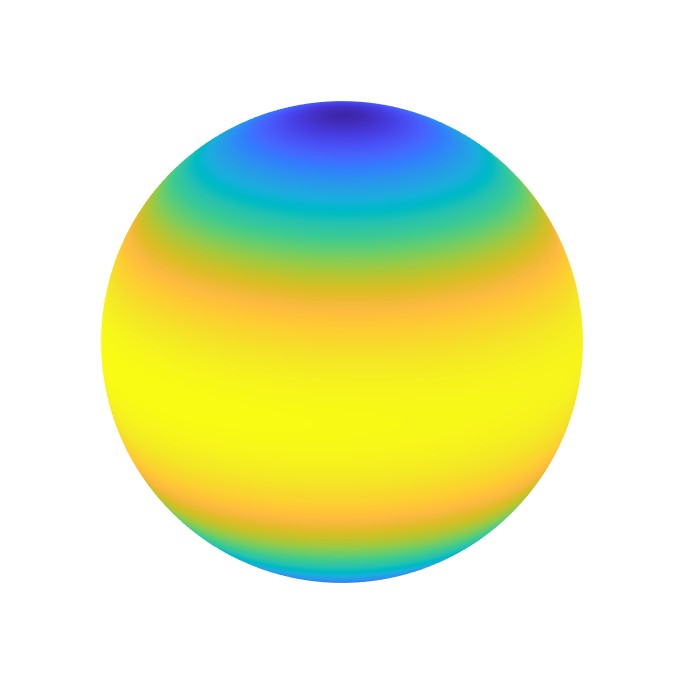} & \includegraphics[trim=30 30 30 30,clip,width=.118\linewidth,valign=m]{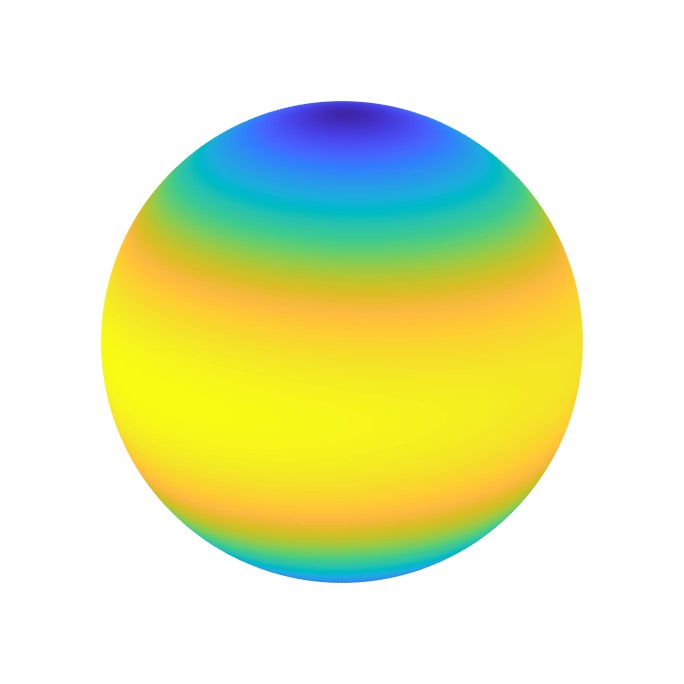} & \includegraphics[trim=30 30 30 30,clip,width=.118\linewidth,valign=m]{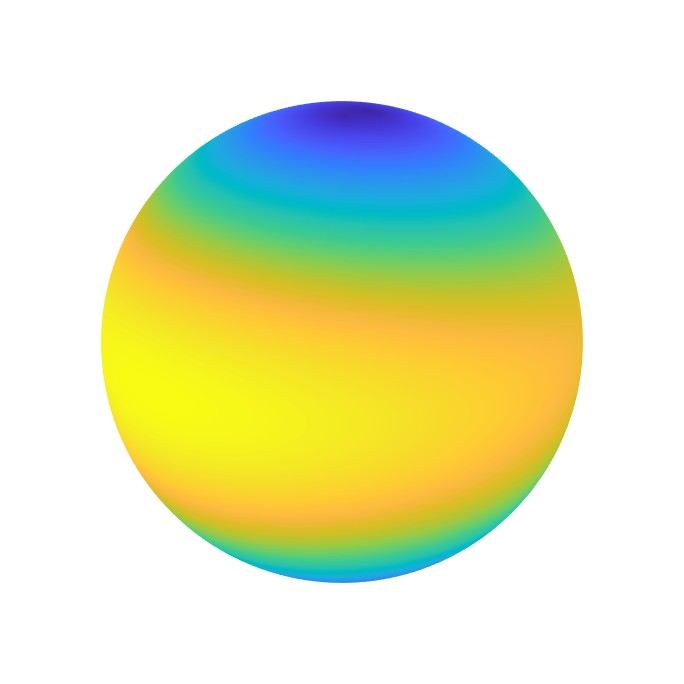} & \includegraphics[trim=30 30 30 30,clip,width=.118\linewidth,valign=m]{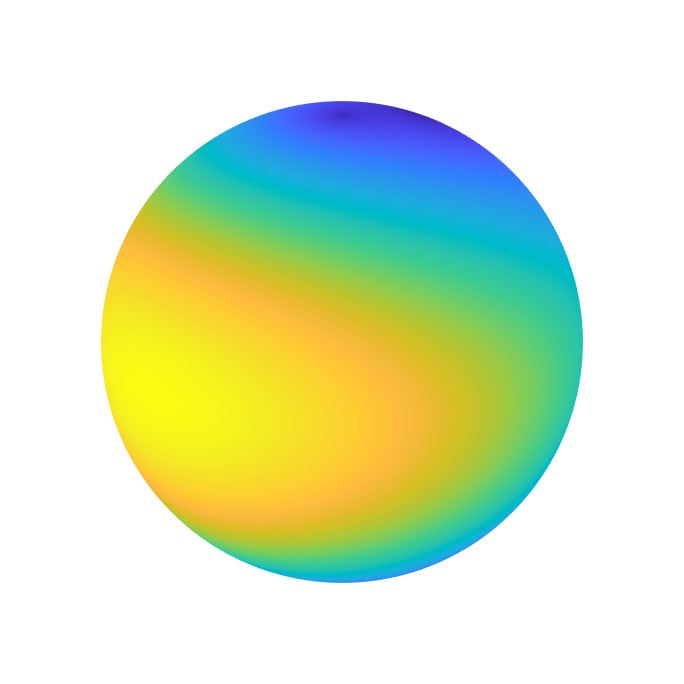} & \includegraphics[trim=30 30 30 30,clip,width=.118\linewidth,valign=m]{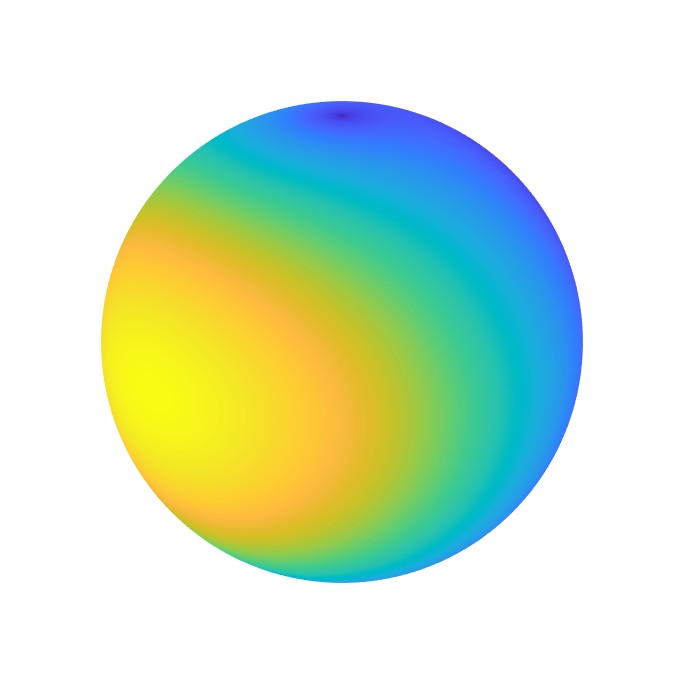} & \includegraphics[trim=30 30 30 30,clip,width=.118\linewidth,valign=m]{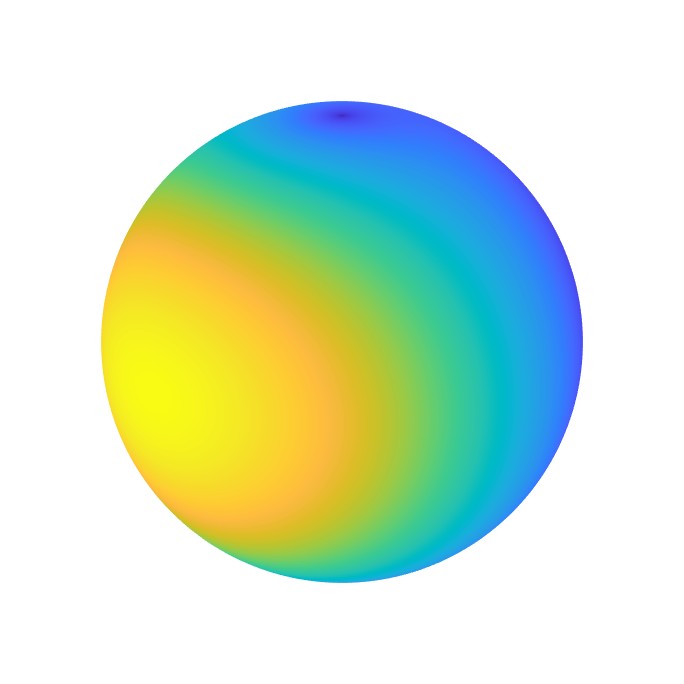}\\
\vspace{-3mm}\\
$\max$ & \footnotesize{$8\times10^{\text{-}3}$} & \footnotesize{$1.2\times10^{\text{-}2}$} & \footnotesize{$2.5\times10^{\text{-}2}$} & \footnotesize{$3.3\times10^{\text{-}2}$} & \footnotesize{$1.5\times10^{\text{-}3}$} & \footnotesize{$4\times10^{\text{-}22}$}\\
\vspace{-3mm}\\
\hline
\vspace{-3mm}\\
$m=0$ & \includegraphics[trim=30 30 30 30,clip,width=.118\linewidth,valign=m]{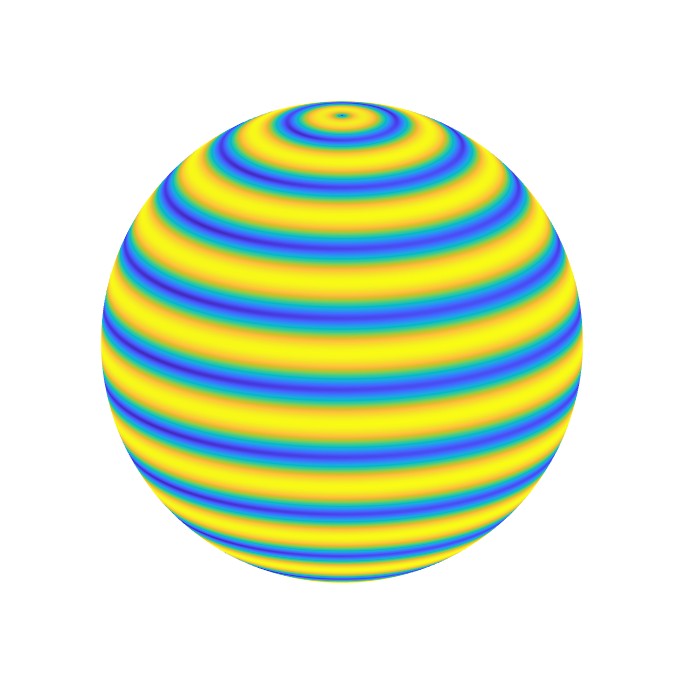} & \includegraphics[trim=30 30 30 30,clip,width=.118\linewidth,valign=m]{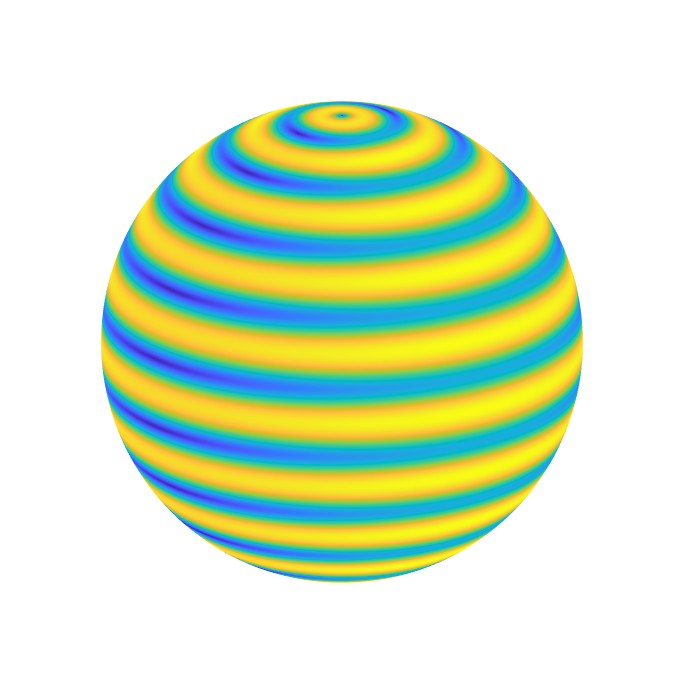} & \includegraphics[trim=30 30 30 30,clip,width=.118\linewidth,valign=m]{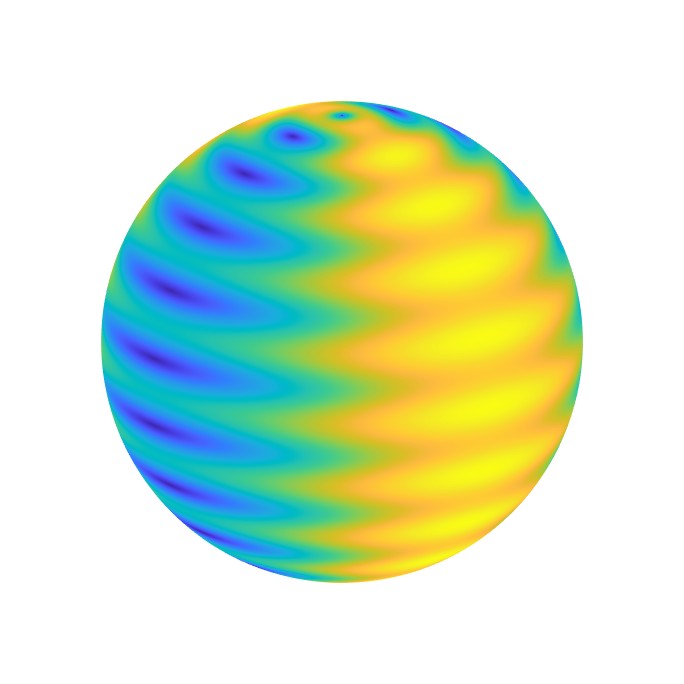} & \includegraphics[trim=30 30 30 30,clip,width=.118\linewidth,valign=m]{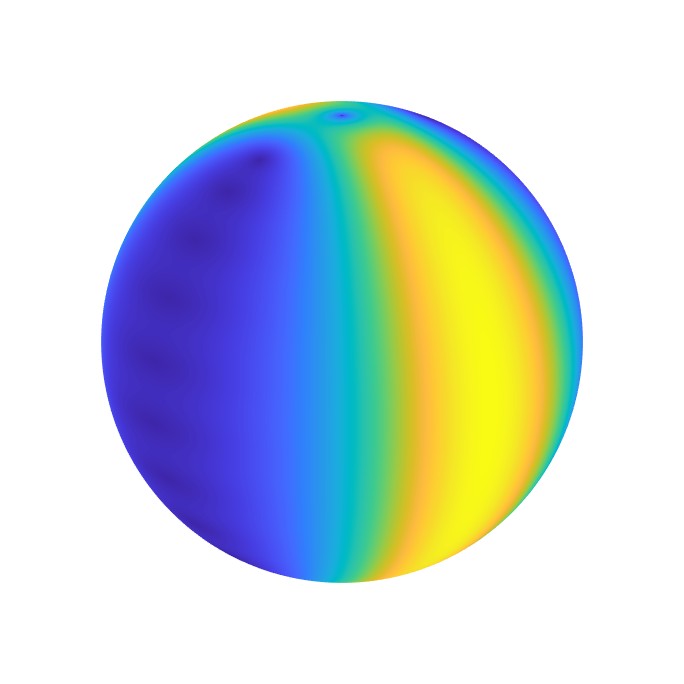} & \includegraphics[trim=30 30 30 30,clip,width=.118\linewidth,valign=m]{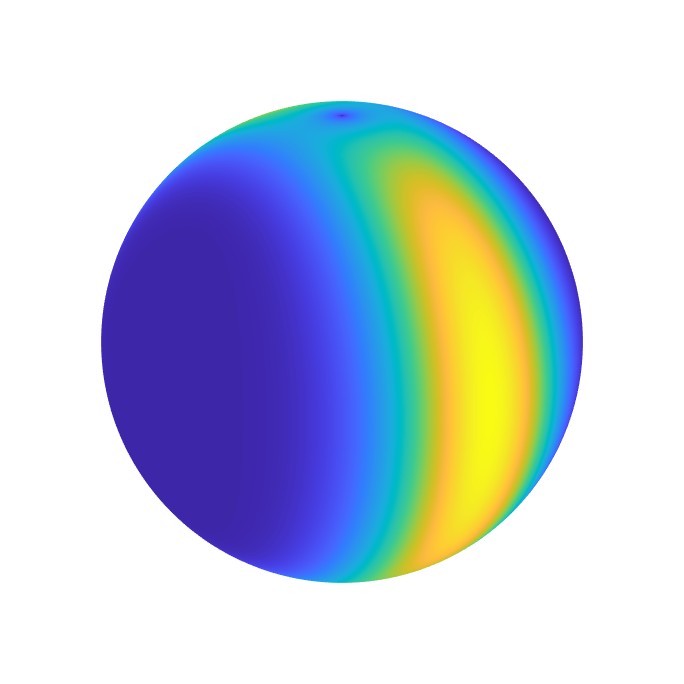} & \includegraphics[trim=30 30 30 30,clip,width=.118\linewidth,valign=m]{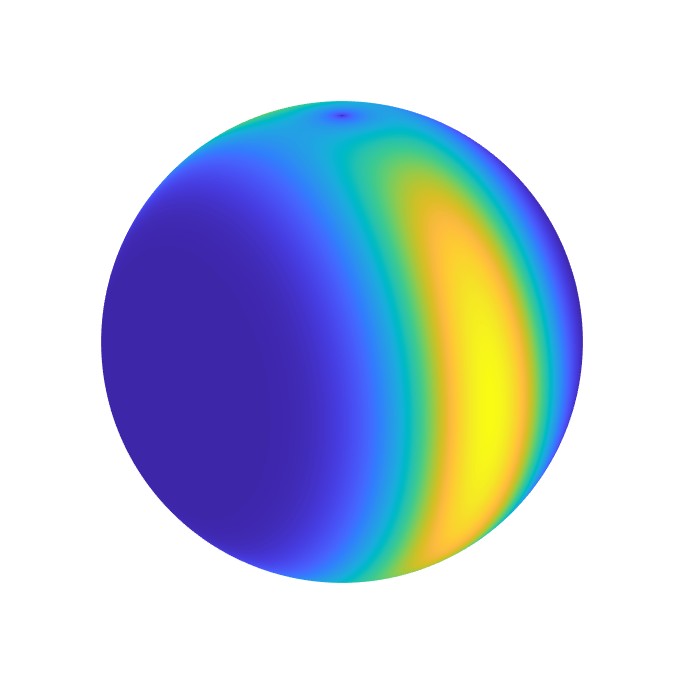}\\
$m=5$ & \includegraphics[trim=30 30 30 30,clip,width=.118\linewidth,valign=m]{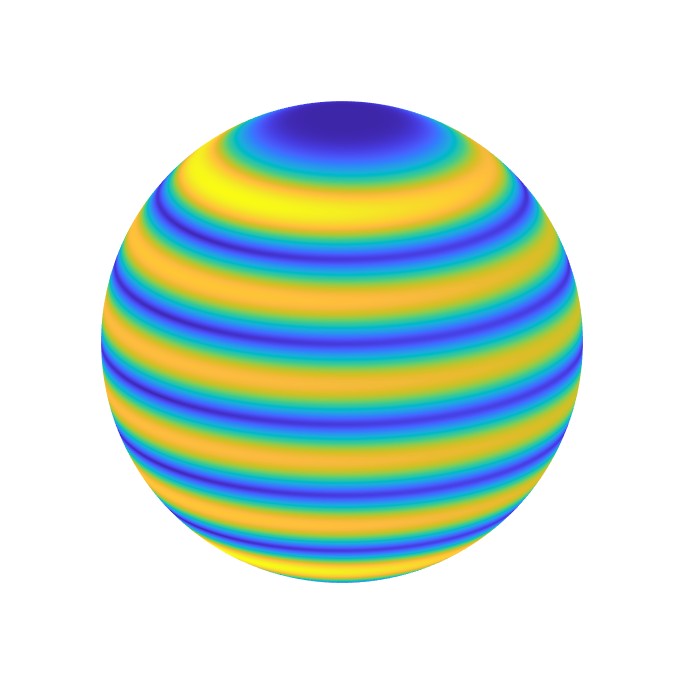} & \includegraphics[trim=30 30 30 30,clip,width=.118\linewidth,valign=m]{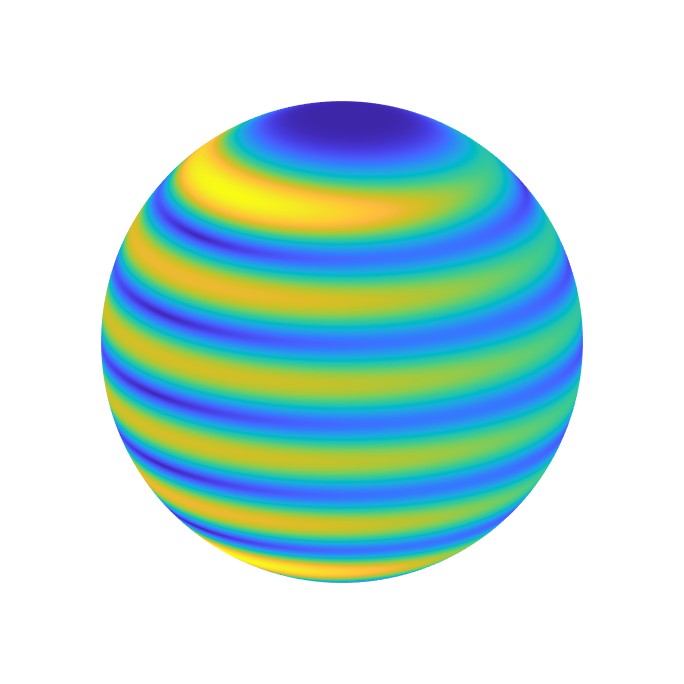} & \includegraphics[trim=30 30 30 30,clip,width=.118\linewidth,valign=m]{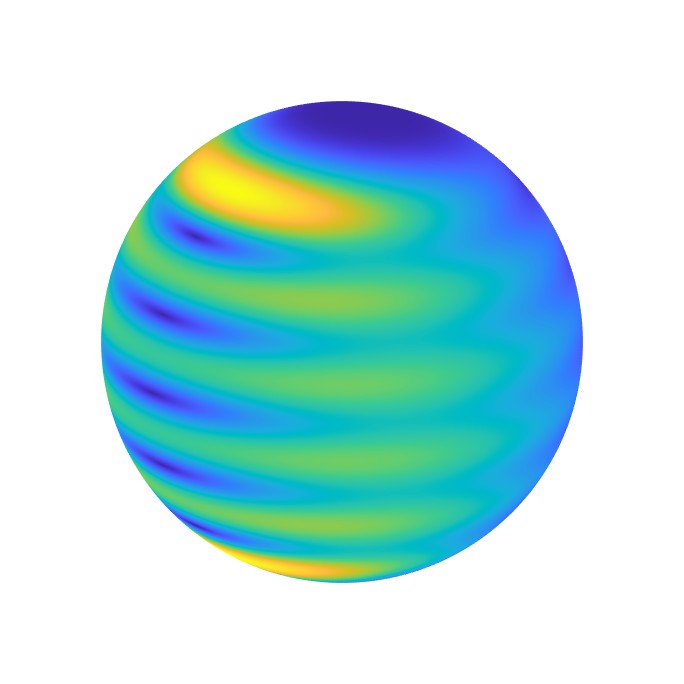} & \includegraphics[trim=30 30 30 30,clip,width=.118\linewidth,valign=m]{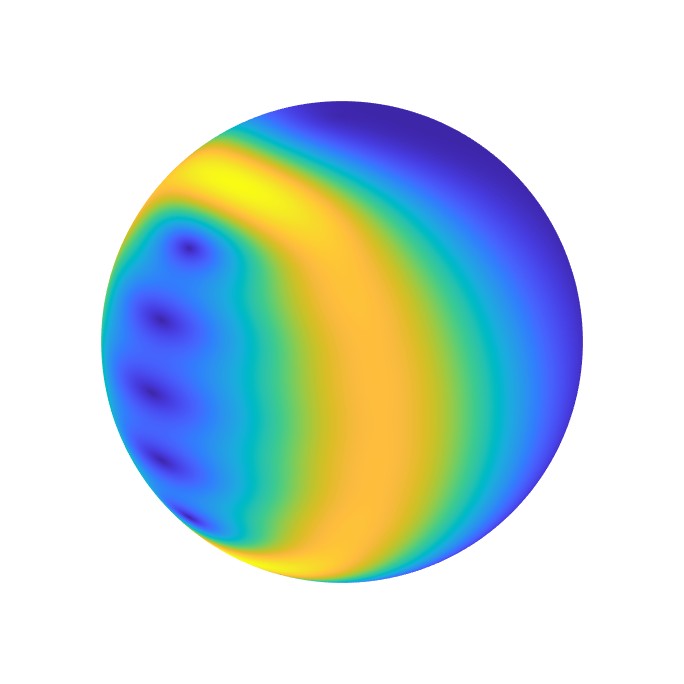} & \includegraphics[trim=30 30 30 30,clip,width=.118\linewidth,valign=m]{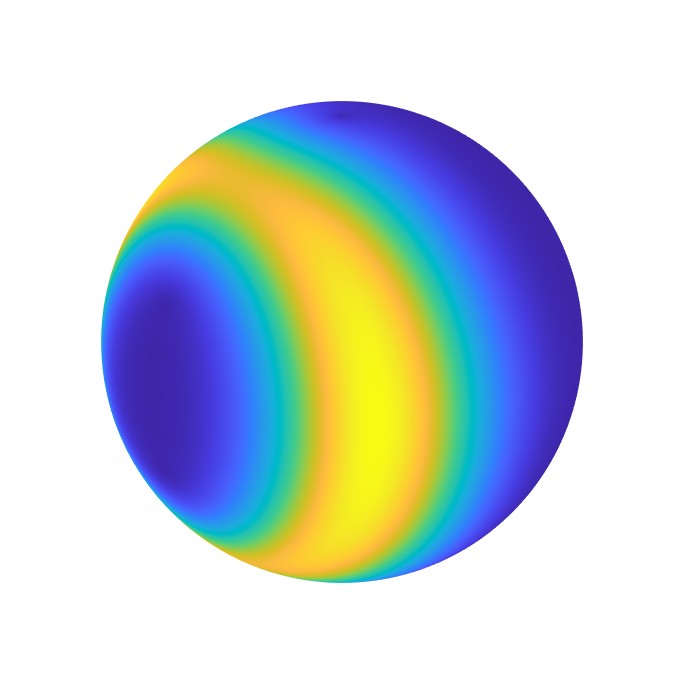} & \includegraphics[trim=30 30 30 30,clip,width=.118\linewidth,valign=m]{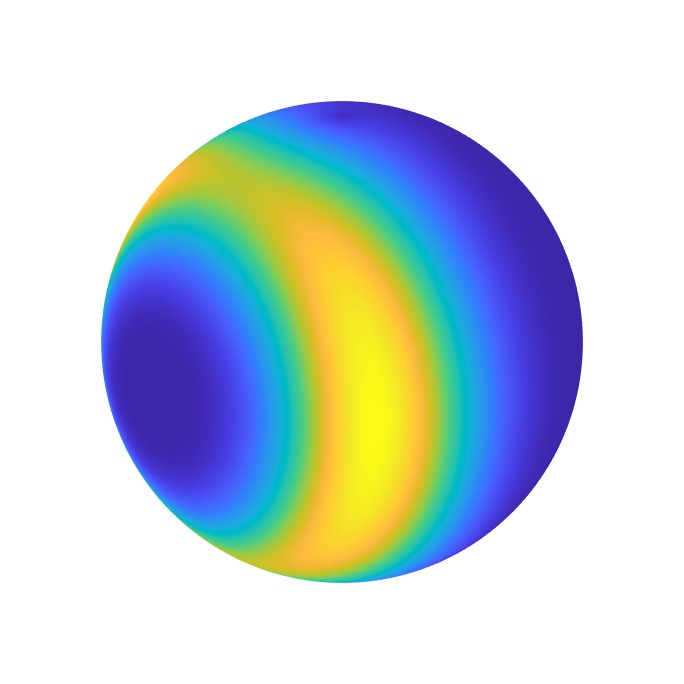}\\
$m=10$ & \includegraphics[trim=30 30 30 30,clip,width=.118\linewidth,valign=m]{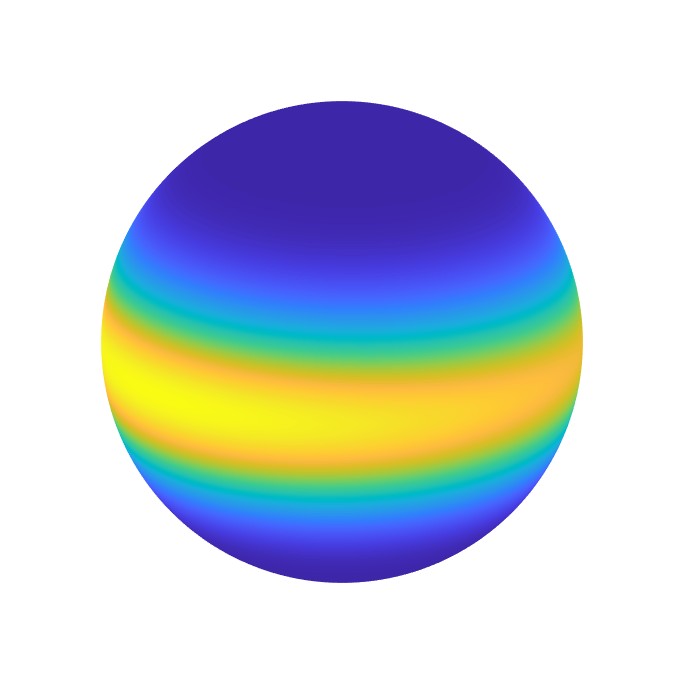} & \includegraphics[trim=30 30 30 30,clip,width=.118\linewidth,valign=m]{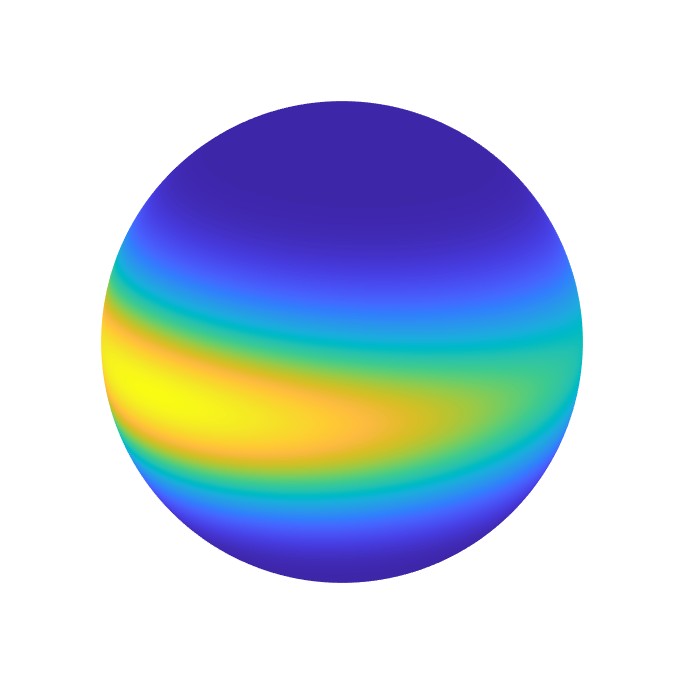} & \includegraphics[trim=30 30 30 30,clip,width=.118\linewidth,valign=m]{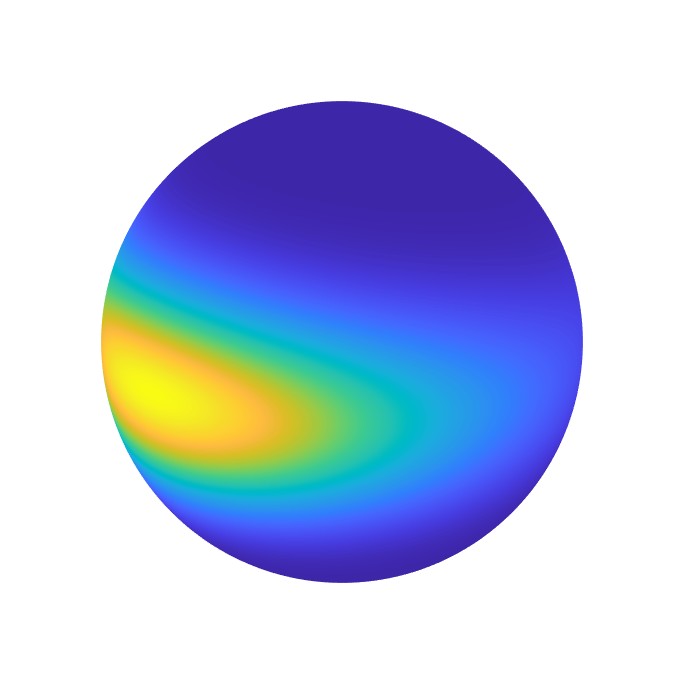} & \includegraphics[trim=30 30 30 30,clip,width=.118\linewidth,valign=m]{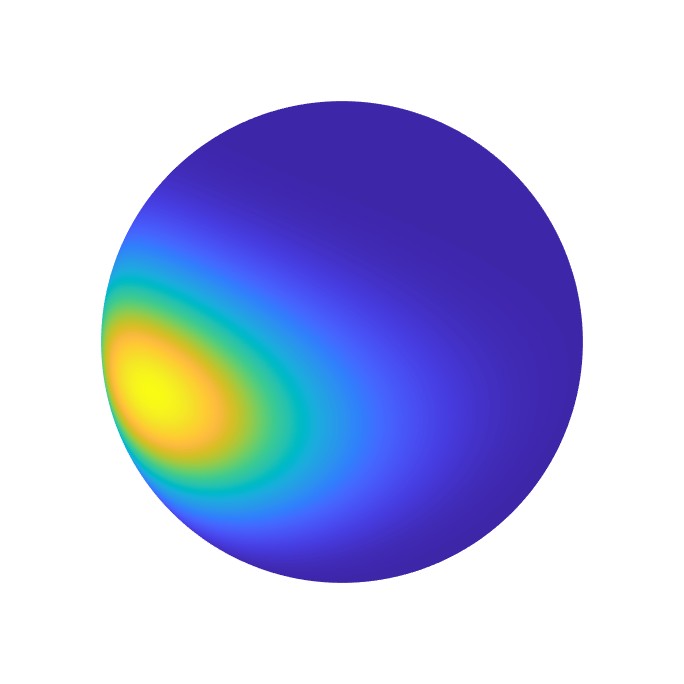} & \includegraphics[trim=30 30 30 30,clip,width=.118\linewidth,valign=m]{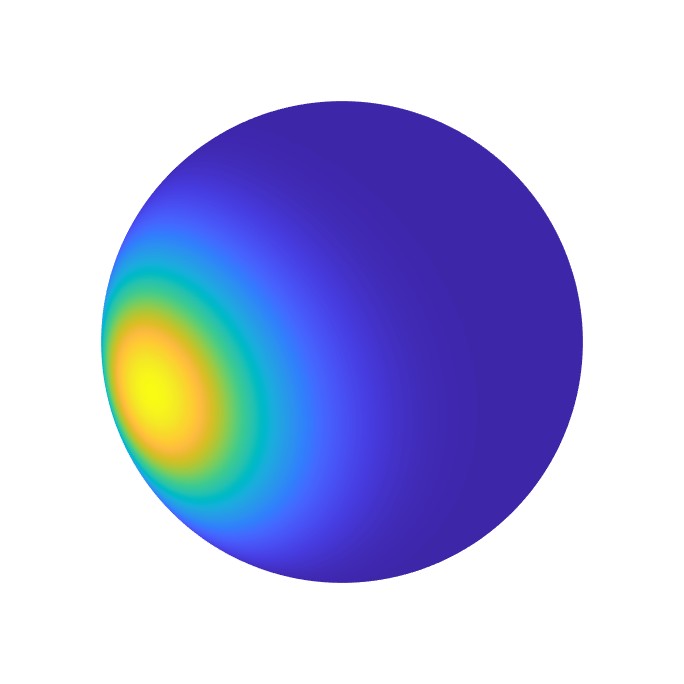} & \includegraphics[trim=30 30 30 30,clip,width=.118\linewidth,valign=m]{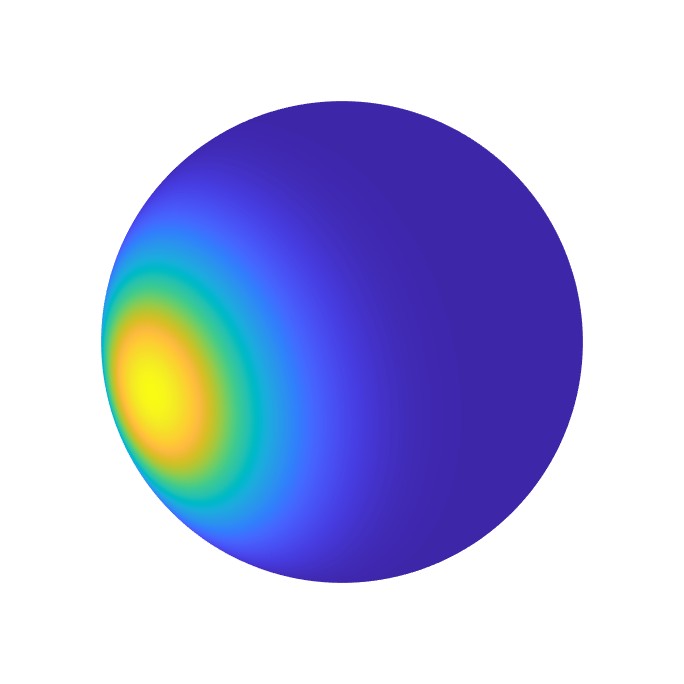}\\
\vspace{-3mm}\\
$\max$ & \footnotesize{$3\times10^{\text{-}6}$} & \footnotesize{$7.1\times10^{\text{-}6}$} & \footnotesize{$3\times10^{\text{-}5}$} & \footnotesize{$3.6\times10^{\text{-}4}$} & \footnotesize{$4.1\times10^{\text{-}2}$} & \footnotesize{$5\times10^{\text{-}14}$}\\
\vspace{-3mm}\\
\hline
\vspace{-3mm}\\
$m=0$ & \includegraphics[trim=30 30 30 30,clip,width=.118\linewidth,valign=m]{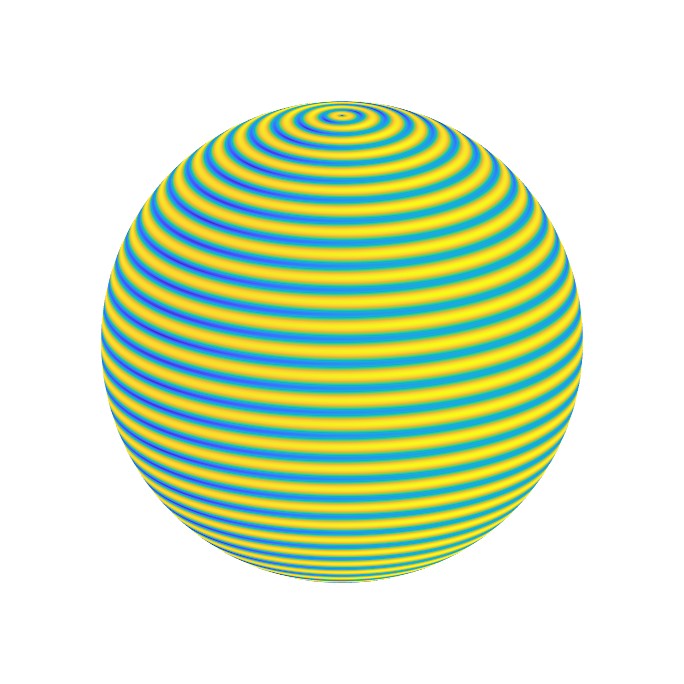} & \includegraphics[trim=30 30 30 30,clip,width=.118\linewidth,valign=m]{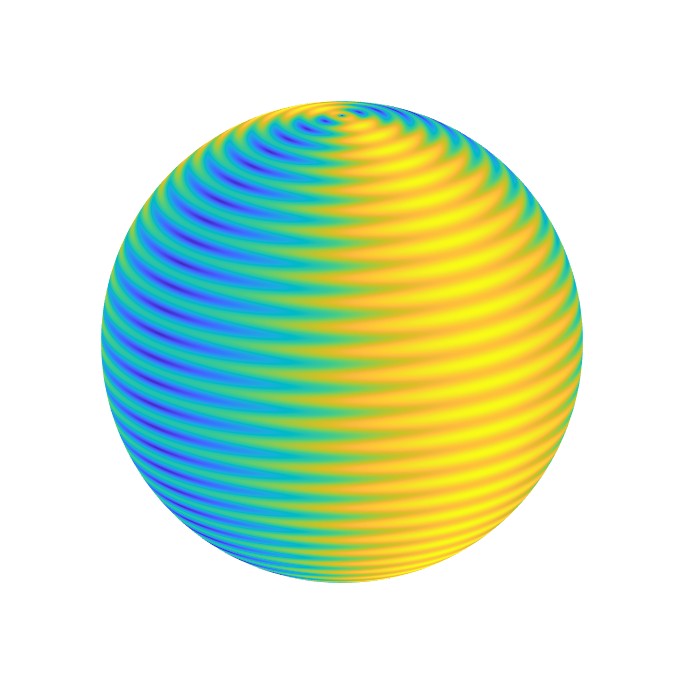} & \includegraphics[trim=30 30 30 30,clip,width=.118\linewidth,valign=m]{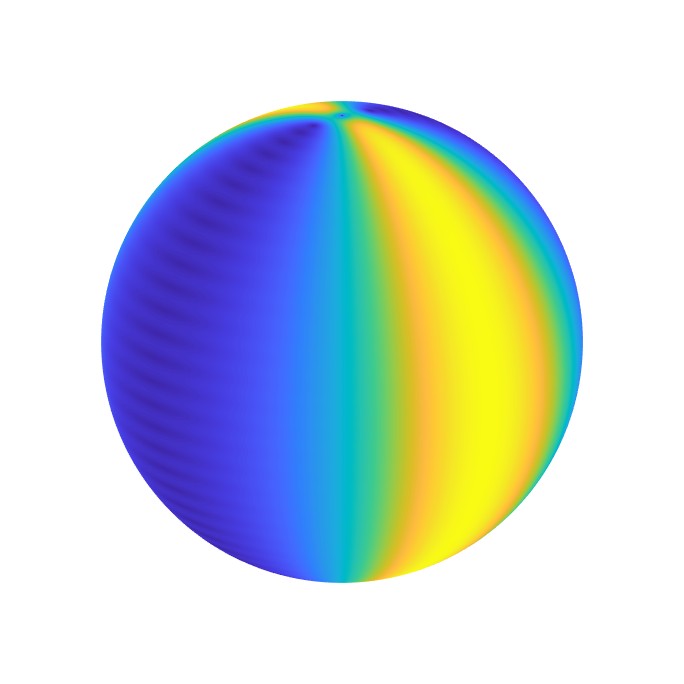} & \includegraphics[trim=30 30 30 30,clip,width=.118\linewidth,valign=m]{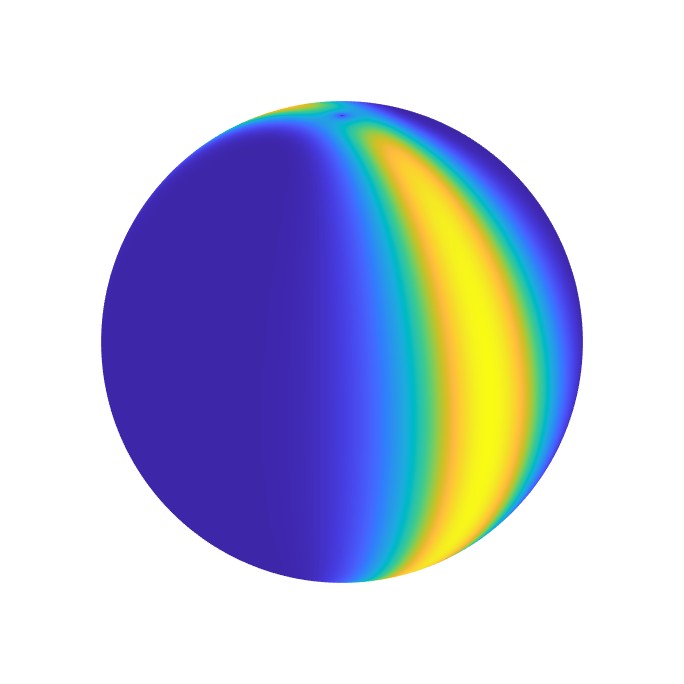} & \includegraphics[trim=30 30 30 30,clip,width=.118\linewidth,valign=m]{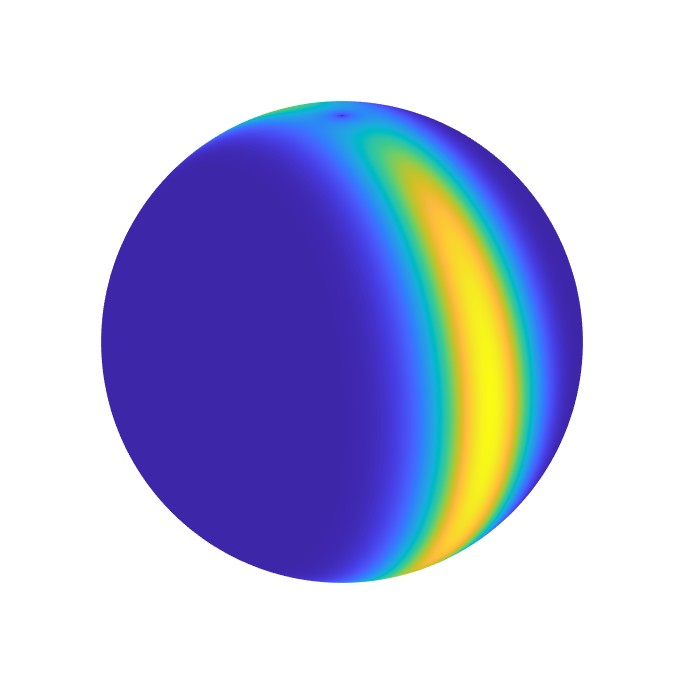} & \includegraphics[trim=30 30 30 30,clip,width=.118\linewidth,valign=m]{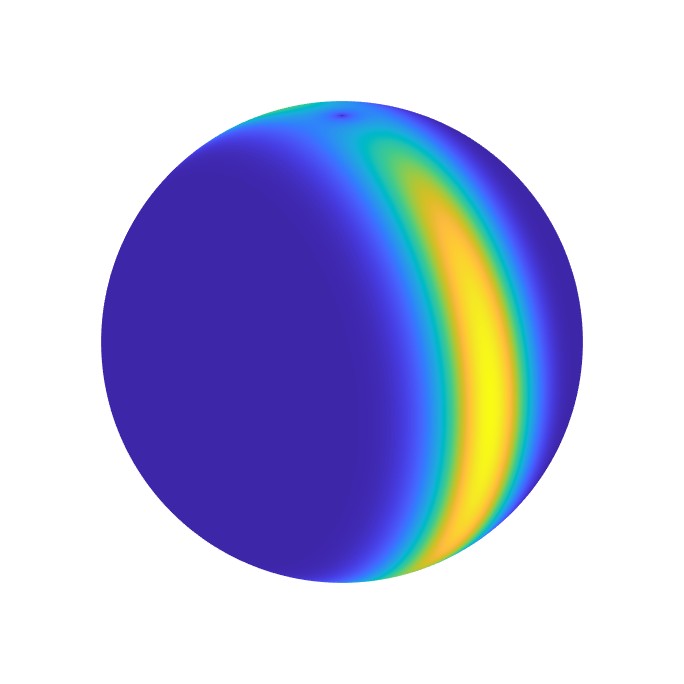}\\
$m=15$ & \includegraphics[trim=30 30 30 30,clip,width=.118\linewidth,valign=m]{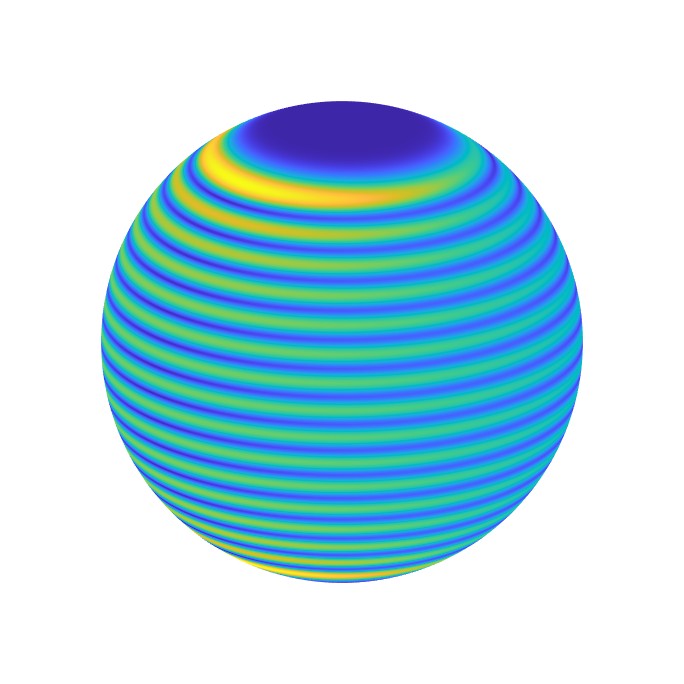} & \includegraphics[trim=30 30 30 30,clip,width=.118\linewidth,valign=m]{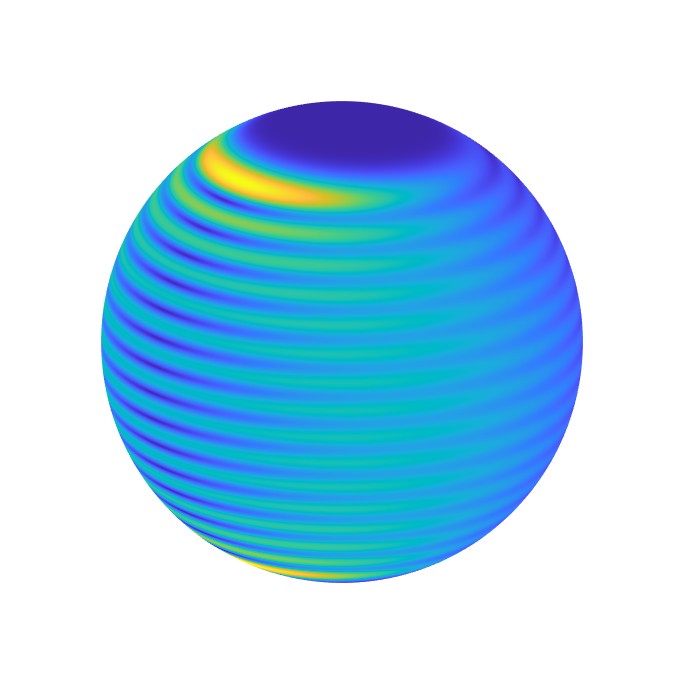} & \includegraphics[trim=30 30 30 30,clip,width=.118\linewidth,valign=m]{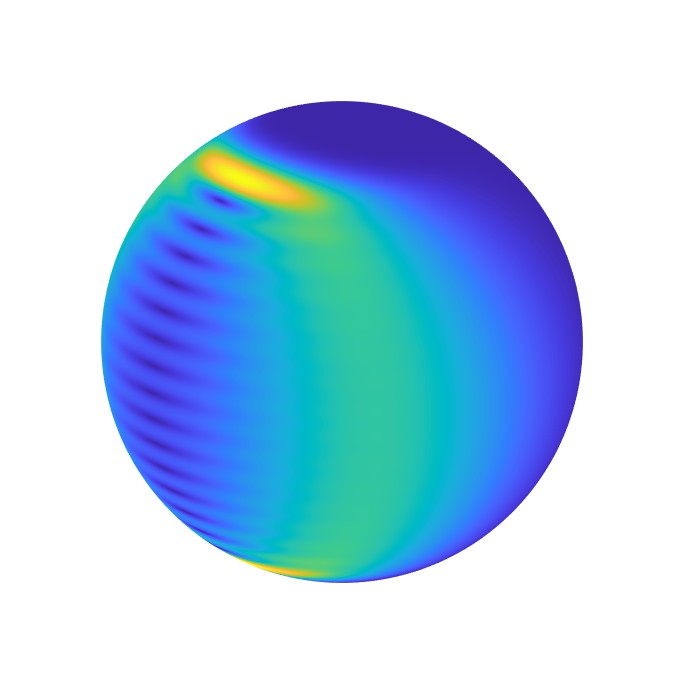} & \includegraphics[trim=30 30 30 30,clip,width=.118\linewidth,valign=m]{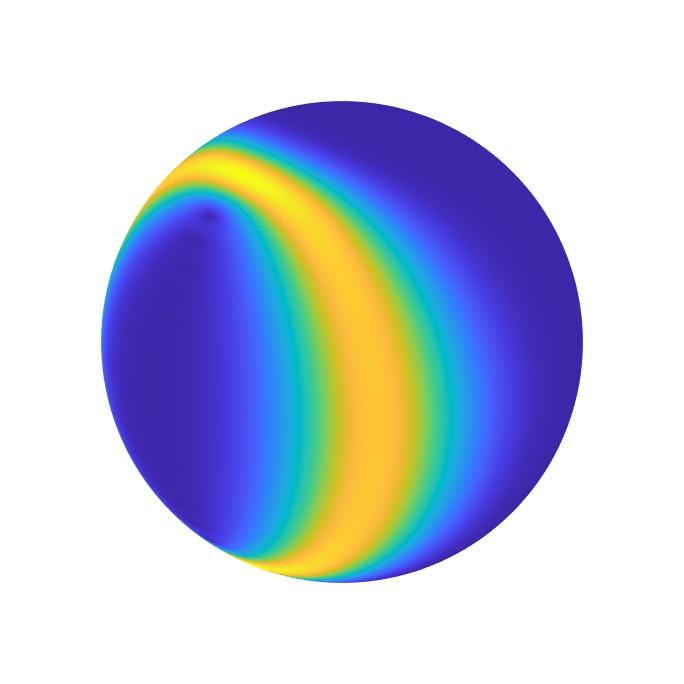} & \includegraphics[trim=30 30 30 30,clip,width=.118\linewidth,valign=m]{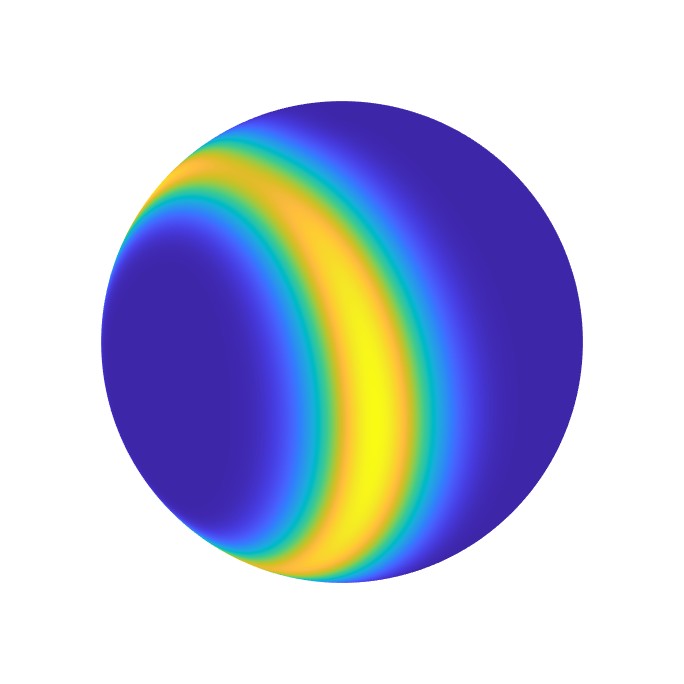} & \includegraphics[trim=30 30 30 30,clip,width=.118\linewidth,valign=m]{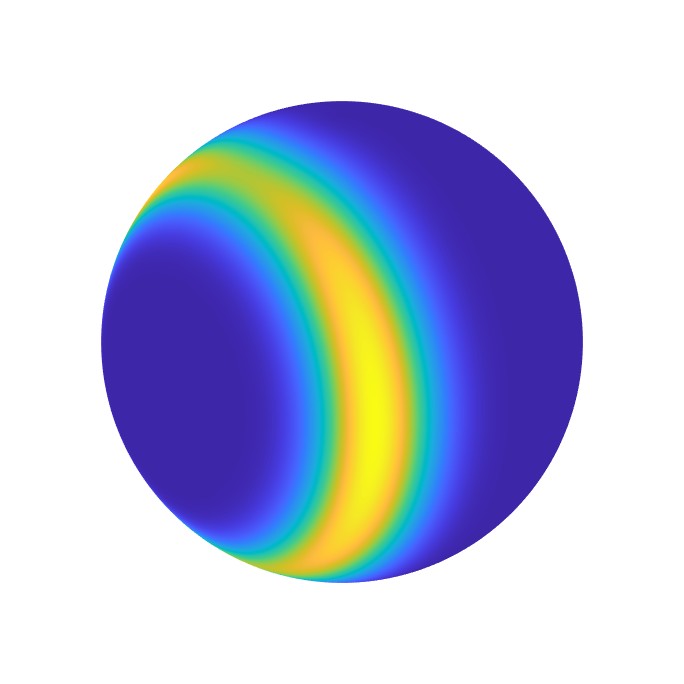}\\
$m=30$ & \includegraphics[trim=30 30 30 30,clip,width=.118\linewidth,valign=m]{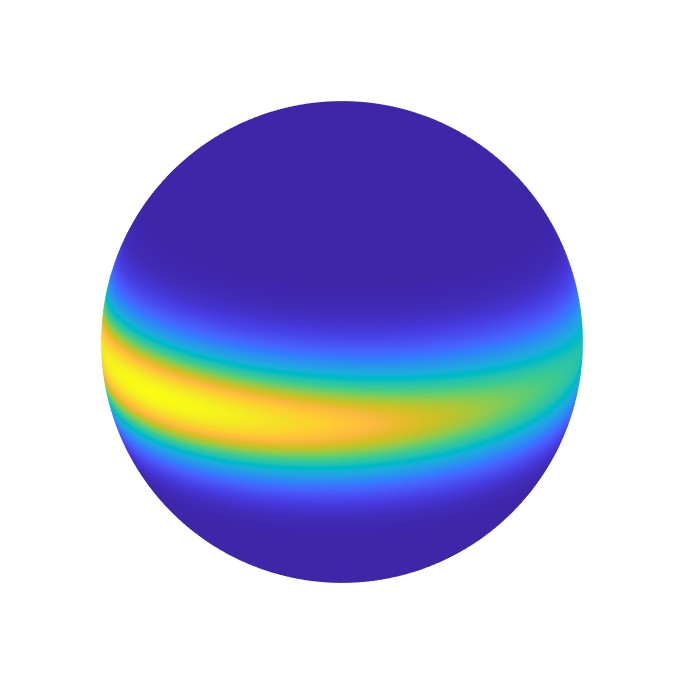} & \includegraphics[trim=30 30 30 30,clip,width=.118\linewidth,valign=m]{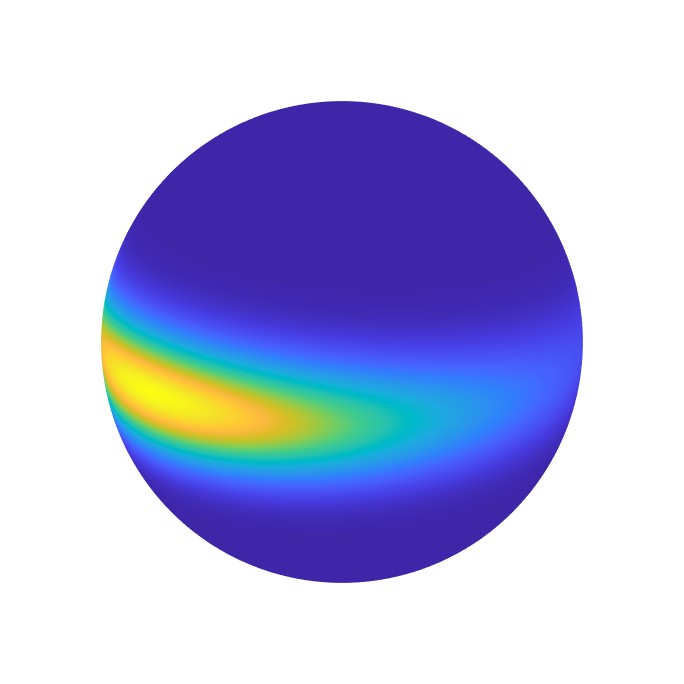} & \includegraphics[trim=30 30 30 30,clip,width=.118\linewidth,valign=m]{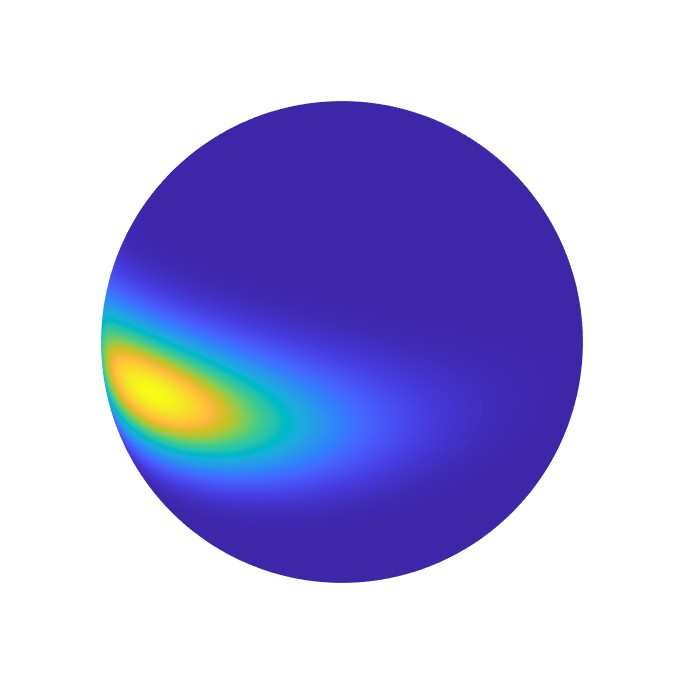} & \includegraphics[trim=30 30 30 30,clip,width=.118\linewidth,valign=m]{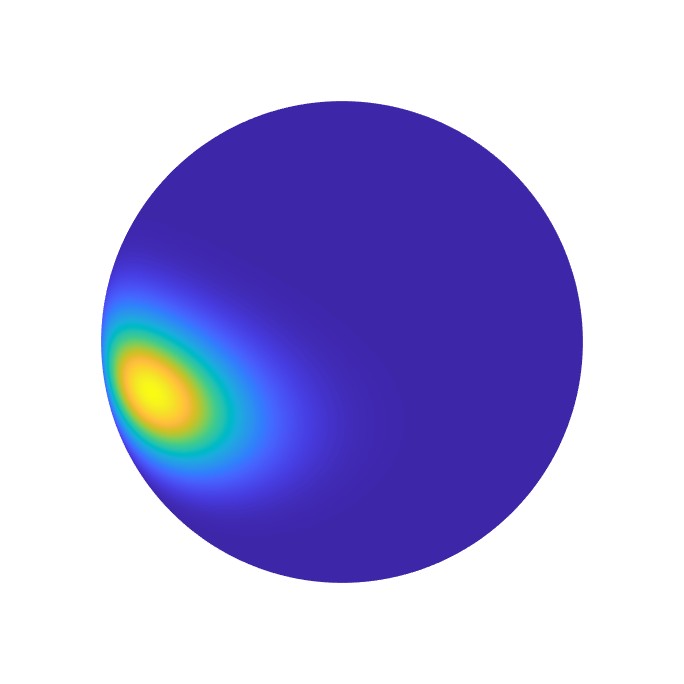} & \includegraphics[trim=30 30 30 30,clip,width=.118\linewidth,valign=m]{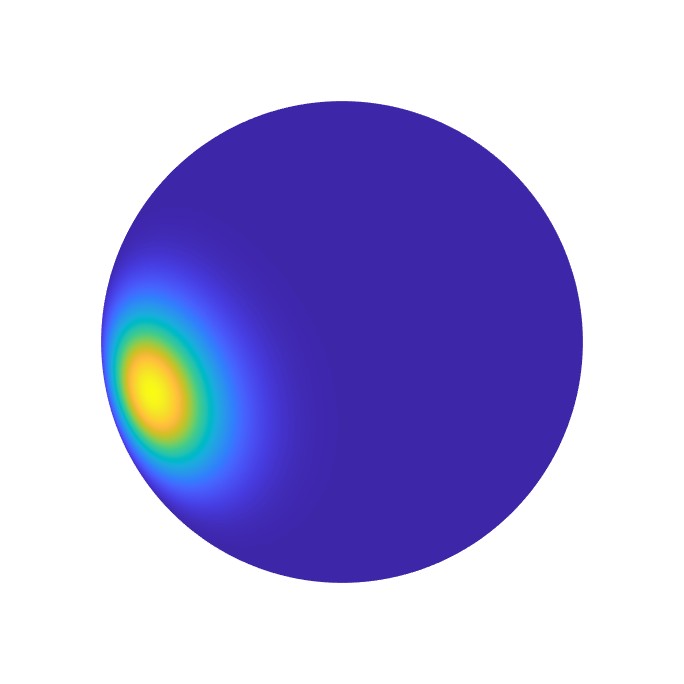} & \includegraphics[trim=30 30 30 30,clip,width=.118\linewidth,valign=m]{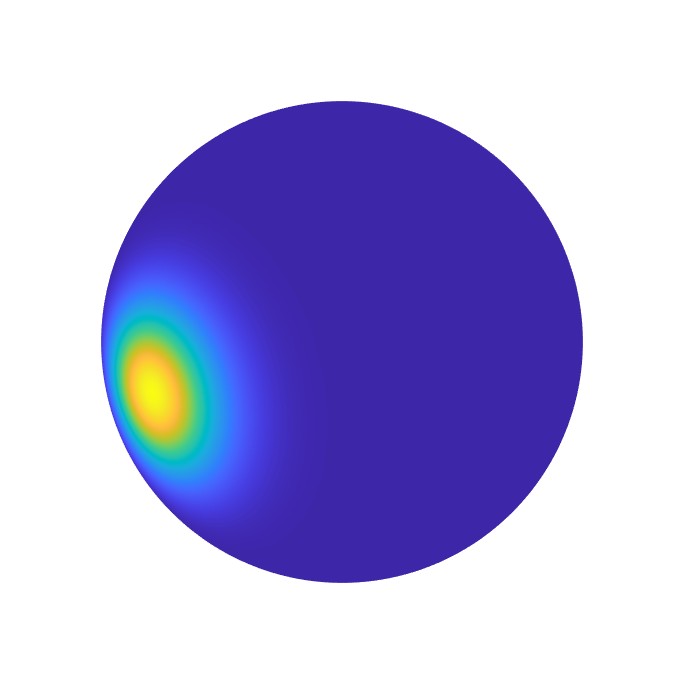}\\
\vspace{-3mm}\\
$\max$ & \footnotesize{$3\times10^{\text{-}22}$} & \footnotesize{$2.3\times10^{\text{-}21}$} & \footnotesize{$3.1\times10^{\text{-}20}$} & \footnotesize{$5.2\times10^{\text{-}17}$} & \footnotesize{$8.7\times10^{\text{-}8}$} & \footnotesize{$9.2\times10^{\text{-}5}$}\\
\vspace{-3mm}\\
\hline
\vspace{1.5mm}\\
\end{tabular}
\caption{Representations of the functions $(\zeta\sin{\theta_1} \cos{\theta_3},\zeta\sin{\theta_1} \sin{\theta_3},\zeta\cos{\theta_1}) \mapsto |w^{1/2}(\theta_1,\zeta)a_{\ell}^m(\theta_1,\cdot\,,\theta_3,\zeta)|$,
which are independent of the second argument of $a_{\ell}^m$,
for fixed $\ell=1$, $\ell=10$ and $\ell=30$ (from top to bottom) and various values of $m$ and $\zeta$.
For every degree $\ell$, the functions in each column share the same color scale, whose maximum value is placed at the bottom along the `max' row.
Wavenumber $\kappa=6$.}
\label{figure 5.1}
\end{figure}
\noindent Observe that $w$ does not depend on the Euler angles $\theta_2$ and $\theta_3$.
The $w$-weighted $L^2$ Hermitian product and the associated norm are defined by:
\begin{equation}
\begin{split} \label{A norm}
(u,v)_{\mathcal{A}}&:=\int_{Y}u(\mathbf{y})\overline{v(\mathbf{y})}w(\mathbf{y})\textup{d}\mathbf{y},\\
\|u\|^2_{\mathcal{A}}&:=(u,u)_{\mathcal{A}},
\end{split}
\,\,\,\,\,\,\,\,\,\,\,\,\,\,\forall u,v \in L^2(Y;w).
\end{equation}
We now define a proper subspace of $L^2(Y;w)$, that we denote by $\mathcal{A}$ and call \textit{space of Herglotz densities}.

\begin{definition}[\hypertarget{Definition 5.1}{Herglotz densities}]
We define, for any $(\ell,m) \in \mathcal{I}$
\begin{equation}
\begin{split} \label{a tilde definizione}
\tilde{a}_\ell^m(\mathbf{y})&:=\mathbf{D}^{m}_{\ell}(\boldsymbol{\theta}) \cdot \mathbf{P}_{\ell}(\zeta),\,\,\,\,\,\,\,\,\,\,\,\,\,\,\,\,\,\,\,\,\,\,\,\,  \forall \mathbf{y} \in Y,\\
a_\ell^m&:=\alpha_{\ell}\tilde{a}_\ell^m,\,\,\,\,\,\,\,\,\,\,\,\,\,\,\,\,\,\,\,\,\,\,\,\,\,\,\,\,\,\,\,\,\,\,\,\,\,\,\,\,\, \alpha_{\ell}:=\|\tilde{a}_\ell^m\|^{-1}_\mathcal{A}, 
\end{split}
\end{equation}
where $\mathbf{D}^{m}_{\ell}(\boldsymbol{\theta})$ and $\mathbf{P}_{\ell}(\zeta)$ are defined in \textup{(\ref{vector P})} and $\mathbf{y}=(\boldsymbol{\theta},\zeta)$. Furthermore, we introduce the space
\vspace{-3mm}
\begin{equation*}
\mathcal{A}:=\overline{\textup{span}\{a_\ell^m\}_{(\ell,m) \in \mathcal{I}}}^{\|\cdot\|_{\mathcal{A}}} \subsetneq L^2(Y;w).
\end{equation*}
\end{definition}

Similarly to the spherical waves (\ref{b tilde definizione}), also the Herglotz densities depend on two different parameters $(\ell,m) \in \mathcal{I}$, but their norm $\alpha_{\ell}$ is independent of the parameter $m$, as will be clear later on (see Lemma \hyperlink{Lemma 5.2}{5.3}). The wavenumber $\kappa$ appears explicitly within the definition (\ref{vector P}) of $\mathbf{P}_{\ell}(\zeta)$, therefore each $a_{\ell}^m$ for $(\ell,m) \in \mathcal{I}$ depends on it.

Some densities $a_{\ell}^m$, weighted by $w^{1/2}$, are represented in Figure \ref{figure 5.1}.
Note that, as the degree $\ell$ increases, the maximum values of the depicted functions are reached at ever larger values of the evanescence parameter $\zeta$, and this reflects the transition from propagative modes (e.g.\ $\ell=1$) to evanescent modes (e.g.\ $\ell=30$).
The sym-metries of these functions can be understood in light of the properties (\ref{sym1}) and (\ref{sym2}).
For convenience, we only report some cases where $m\geq 0$, since, thanks to (\ref{sym3}), for $m\leq0$ the functions are symmetric with respect to the plane $\{\mathbf{x} =(x,y,z)\in \R^3 : y=0\}$.
If $\zeta$ is large enough (e.g.\ $\zeta=10$), the supports create an annular structure, which collapses to a dot when $ |m|$ goes to $\ell$, while, if $|m|$ goes to $0$, it tends to form two vertical stripes that are symmetric with respect to the plane $\{\mathbf{x}=(x,y,z)\in \R^3 : x=0 \}$.
The smaller $\zeta$ gets, the less dependent the functions become on $\theta_3$ and the exhibited patterns match the zero distributions of the Ferrers functions $\mathsf{P}_{\ell}^m$ (see \cite[Sec.\ 4.16.2]{nist}).
In Figure \ref{figure 5.1}, for instance when $\zeta=10^{-3}$, we observe $\ell-|m|$ nearly uniform blue horizontal bands.

We present a lemma that will aid in the examination of the asymptotic behavior of the normalization coefficients $\alpha_{\ell}$ later on.

\begin{lemma}
We have for all $(\ell,m) \in \mathcal{I}$ and $z \geq 1$
\begin{equation}
(z-1)^{\ell}\leq \frac{\sqrt{\pi}(\ell-m)!P_{\ell}^m(z)}{2^{\ell}\Gamma\left(\ell+1/2\right)}\leq (z+1)^{\ell}.
\label{disuguaglianza lemma}
\end{equation}
\end{lemma}
\begin{proof}
Note that, due to (\ref{negative legendre2 polynomials}), $(\ell+m)!P_{\ell}^{-m}(z)=(\ell-m)!P_{\ell}^{m}(z)$, for every $(\ell,m) \in \mathcal{I}$. Therefore, we can assume $m \geq 0$ in the following. Thanks to (\ref{sum legendre expansion}) and since
\vspace{-1mm}
\begin{equation}
\Gamma\left(n+\frac{1}{2}\right)=\frac{\sqrt{\pi}(2n)!}{2^{2n}n!},\,\,\,\,\,\,\,\,\,\,\,\,\,\,\forall n \geq 0,
\label{gamma half}
\end{equation}
(see \cite[Eq.\ (5.5.5)]{nist}) it readily follows:
\vspace{-2mm}
\begin{equation*}
A_{\ell}^m(z):=\frac{\sqrt{\pi}(\ell-m)!P_{\ell}^m(z)}{2^{\ell}\Gamma\left(\ell+1/2\right)}=\binom{2\ell}{\ell+m}^{-1}\!\!\!\!\!\!\left(z-1\right)^{\ell}\sum_{k=0}^{\ell-m}\binom{\ell}{k}\binom{\ell}{m+k}\left(\frac{z+1}{z-1}\right)^{m/2+k}.
\end{equation*}
Thanks to the Vandermonde identity \cite[Eq.\ (1)]{Sokal}, we derive:
\vspace{-1mm}
\begin{align*}
A_{\ell}^m(z) &\geq \binom{2\ell}{\ell+m}^{-1}(z-1)^{\ell}\sum_{k=0}^{\ell-m}\binom{\ell}{k}\binom{\ell}{m+k}=(z-1)^{\ell},\\
A_{\ell}^m(z) &\leq \binom{2\ell}{\ell+m}^{-1}\left(\frac{z+1}{z-1}\right)^{\ell-m/2}(z-1)^{\ell}\sum_{k=0}^{\ell-m}\binom{\ell}{k}\binom{\ell}{m+k}\\
&=\left(\frac{z-1}{z+1}\right)^{m/2}(z+1)^{\ell}\leq (z+1)^{\ell}. \qedhere
\end{align*}
\end{proof}
\vspace{-3mm}
\begin{remark}
\hypertarget{Remark 5.3}{Numerical} experiments suggest the possibility of improving the upper bound in \textup{(\ref{disuguaglianza lemma})} with $z^{\ell}$. However, the lack of such a refinement does not affect the validity of the next result.
\end{remark}
\vspace{-1mm}

The coefficients $\alpha_{\ell}$ decay super-exponentially with $\ell$ after a pre-asymptotic regi-me up to
$\ell \approx \kappa$. The precise asymptotic behavior is given by the following lemma.

\begin{lemma}
\hypertarget{Lemma 5.2}{We} have that $\|\tilde{a}_{\ell}^m\|_{\mathcal{A}}$ is independent of the value of $m$ for all $(\ell,m) \in \mathcal{I}$ and furthermore
\vspace{-5mm}
\begin{equation}
\alpha_{\ell} \sim c(\kappa)\left(\frac{e\kappa}{2} \right)^{\ell}\ell^{-\left(\ell+\frac{1}{2}\right)},\,\,\,\,\,\,\,\,\,\,\,\,\,\,\,\text{as}\,\,\,\ell \rightarrow \infty,
\label{behaviour alpha_lm}
\end{equation}
where the constant $c(\kappa)$ only depends on $\kappa$.
\end{lemma}

\begin{proof}
Thanks to the orthogonality condition \cite[Sec.\ 4.10, Eq.\ (5)]{quantumtheory}, namely
\begin{equation}
\int_{\Theta}D_{\ell}^{m',m}(\boldsymbol{\theta})\overline{D_{q}^{n',n}(\boldsymbol{\theta})}\sin{(\theta_1)}\textup{d}\boldsymbol{\theta}=\frac{8\pi^2}{2\ell+1}\delta_{m,n}\delta_{m',n'}\delta_{\ell,q},
\label{wigner orthogonality}
\end{equation}
we have that
\begin{align*}
\|\tilde{a}_{\ell}^m\|^2_{\mathcal{A}}&=\int_{Y}|\mathbf{D}^{m}_{\ell}(\boldsymbol{\theta}) \cdot \mathbf{P}_{\ell}(\zeta)|^2w(\mathbf{y})\textup{d}\mathbf{y}\\
&=\sum_{m'=-\ell}^{\ell}\int_{\Theta}|D_{\ell}^{m',m}(\boldsymbol{\theta})|^2\sin{(\theta_1)}\textup{d}\boldsymbol{\theta}\int_0^{+\infty}\left[\gamma_{\ell}^{m'}P_{\ell}^{m'}\left(\frac{\zeta}{2\kappa}+1 \right)\right]^2\zeta^{1/2}e^{-\zeta}\textup{d}\zeta\\
&=\frac{8\pi^2}{2\ell+1}\sum_{m'=-\ell}^{\ell}\int_0^{+\infty}\left[\gamma_{\ell}^{m'}P_{\ell}^{m'}\left(\frac{\zeta}{2\kappa}+1 \right)\right]^2\zeta^{1/2}e^{-\zeta}\textup{d}\zeta. \numberthis \label{1 lemma 5.3}
\end{align*}
Observe that $\|\tilde{a}_{\ell}^m\|_{\mathcal{A}}$ is independent of the value of $m$. In what follows, we study the integral in (\ref{1 lemma 5.3}), which we denote henceforth by $B_{\ell}^{m'}$. Thanks to (\ref{disuguaglianza lemma}):
\begin{align*}
B_{\ell}^{m'}&\geq\left(\frac{2^{\ell}\gamma_{\ell}^{m'}\Gamma(\ell+1/2)}{\sqrt{\pi}(\ell-m')!}\right)^2\int_{0}^{+\infty}\left(\frac{\zeta}{2\kappa}\right)^{2\ell}\zeta^{1/2}e^{-\zeta}\textup{d}\zeta\\
&=\frac{1}{4\pi^2\kappa^{2\ell}}\frac{(2\ell+1)\Gamma^{\,2}(\ell+1/2)}{(\ell+m')!(\ell-m')!} \,\Gamma\left(2\ell+\frac{3}{2}\right)=e^{-2\kappa}C_{\ell}^{m'}, \numberthis \label{3 lemma 5.3}
\end{align*}
and analogously
\begin{align*}
B_{\ell}^{m'}&\leq\left(\frac{2^{\ell}\gamma_{\ell}^{m'}\Gamma(\ell+1/2)}{\sqrt{\pi}(\ell-m')!}\right)^2\int_{0}^{+\infty}\left(\frac{\zeta}{2\kappa}+2\right)^{2\ell}\zeta^{1/2}e^{-\zeta}\textup{d}\zeta\\
&=\left(\frac{2^{\ell}\gamma_{\ell}^{m'}\Gamma(\ell+1/2)}{\sqrt{\pi}(\ell-m')!}\right)^2\int_{4\kappa}^{+\infty}\left(\frac{\eta}{2\kappa}\right)^{2\ell}(\eta-4\kappa)^{1/2}e^{-(\eta-4\kappa)}\textup{d}\eta\\
&<\frac{2\ell+1}{4\pi^2}\frac{2^{2\ell}\Gamma^{\,2}(\ell+1/2)}{(\ell+m')!(\ell-m')!} \int_{0}^{+\infty}\left(\frac{\eta}{2\kappa}\right)^{2\ell}\eta^{1/2}e^{-\eta}e^{4\kappa}\textup{d}\eta\\
&=\frac{e^{4\kappa}}{4\pi^2\kappa^{2\ell}}\frac{(2\ell+1)\Gamma^{\,2}(\ell+1/2)}{(\ell+m')!(\ell-m')!} \,\Gamma\left(2\ell+\frac{3}{2}\right)=e^{2\kappa}C_{\ell}^{m'}, \numberthis \label{2 lemma 5.3}
\end{align*}
where we used the change of variable $\eta=\zeta+4\kappa$ and $C_{\ell}^{m'}$ is defined as
\begin{equation}
C_{\ell}^{m'}:=\frac{e^{2\kappa}}{4\pi^2\kappa^{2\ell}}\frac{(2\ell+1)\Gamma^{\,2}(\ell+1/2)}{(\ell+m')!(\ell-m')!} \Gamma\left(2\ell+\frac{3}{2}\right),\,\,\,\,\,\,\,\,\,\,\,\,\,\,\forall\,(\ell,m') \in \mathcal{I}.
\label{Clm constant}
\end{equation}
Using (\ref{ell asymptotics}) and \cite[Eq. (5.11.3)]{nist}, it is easily checked that as $\ell \rightarrow +\infty$:
\begin{align*}
\Gamma\left(2\ell+\frac{3}{2}\right) &\sim \sqrt{2\pi}e^{-\left(2\ell+\frac{3}{2}\right)}\left(2\ell+\frac{3}{2}\right)^{2\ell+1} \sim \sqrt{2\pi}e^{-2\ell}\left(2\ell\right)^{2\ell+1},\\
\Gamma^{\,2}\left(\ell+\frac{1}{2}\right) &\sim 2\pi e^{-(2\ell+1)}\left(\ell+\frac{1}{2}\right)^{2\ell} \sim 2\pi e^{-2\ell}\ell^{2\ell},\\
\frac{1}{(\ell+m')!(\ell-m')!}& \sim \frac{e^{2(\ell+1)}}{2\pi(\ell+m'+1)^{\ell+m'+1/2}(\ell-m'+1)^{\ell-m'+1/2}}\\
&\sim \frac{e^{2(\ell+1)}}{2\pi e^2 \ell^{2\ell+1}}=\frac{1}{2\pi}e^{2\ell}\ell^{-(2\ell+1)},\,\,\,\,\,\,\,\,\,\,\,\,\,\,\,\,\,\,\,\,\,\,\,\,\,\,\,\,\,\,\,\,\,\,\,\,\,\,\,\,\,\,\,\,\,\,|m'|\leq \ell\text{ fixed},
\end{align*}
and therefore
\begin{equation*}
\frac{(2\ell+1)\Gamma^{\,2}(\ell+1/2)}{(\ell+m')!(\ell-m')!}\Gamma\left(2\ell+\frac{3}{2}\right) \sim 2\sqrt{2\pi}e^{-2\ell}(2\ell)^{2\ell+1},\,\,\,\,\,\,\,\,\,\,\text{as}\,\,\, \ell \rightarrow +\infty.
\end{equation*}
\color{black} From (\ref{3 lemma 5.3}) and (\ref{2 lemma 5.3}), it follows that, as $\ell \rightarrow +\infty$, there exists a constant $c_1(\kappa)$, only dependent on the wavenumber $\kappa$, such that
\begin{equation*}
C_{\ell}^{m'} \sim \frac{e^{2\kappa}}{\pi\sqrt{2\pi}}\left(\frac{2}{e\kappa}\right)^{2\ell}\ell^{2\ell+1}\,\,\,\,\,\,\,\Rightarrow\,\,\,\,\,\,\,B_{\ell}^{m'} \sim c_1(\kappa)\left(\frac{2}{e\kappa}\right)^{2\ell}\ell^{2\ell+1}.
\end{equation*} \color{black}
Moreover, also $\|\tilde{a}_{\ell}^m\|^2_{\mathcal{A}}$ has the same behavior as $B_{\ell}^{m'}$ at infinity: in fact, thanks to (\ref{1 lemma 5.3}), we have
\begin{align*}
\|\tilde{a}_{\ell}^m\|^2_{\mathcal{A}} &\sim \frac{4\pi^2}{\ell}\sum_{m'=-\ell}^{\ell}c_1(\kappa)\left(\frac{2}{e\kappa}\right)^{2\ell}\ell^{2\ell+1} \sim c_2(\kappa)\left(\frac{2}{e\kappa}\right)^{2\ell}\ell^{2\ell+1},
\end{align*}
for some constant $c_2(\kappa)$ only dependent on $\kappa$; the claimed result (\ref{behaviour alpha_lm}) follows.
\end{proof}

\begin{lemma}
\hypertarget{Lemma 5.3}{The} space $(\mathcal{A},\|\cdot\|_{\mathcal{A}})$ is a Hilbert space and the family $\{a_{\ell}^m\}_{(\ell,m) \in \mathcal{I}}$ is a Hilbert basis (i.e an orthonormal basis):
\begin{equation*}
(a_{\ell}^m,a_q^n)_{\mathcal{A}}=\delta_{\ell, q}\delta_{m,n},\,\,\,\,\,\,\,\,\,\,\,\,\,\,\forall\, (\ell,m),(q,n) \in \mathcal{I},
\end{equation*}
and
\begin{equation*}
u=\sum_{\ell=0}^{\infty}\sum_{m=-\ell}^{\ell}(u,a_{\ell}^m)_{\mathcal{A}}\,a_{\ell}^m,\,\,\,\,\,\,\,\,\,\,\,\,\,\,\forall u \in \mathcal{A}.
\end{equation*}
\end{lemma}
\begin{proof}
Thanks to how we defined the Herglotz densities in (\ref{a tilde definizione}), it is enough to prove that the family $\{\tilde{a}_{\ell}^m\}_{(\ell,m) \in \mathcal{I}}$ is orthogonal, which can be readily observed by tracing back the steps in (\ref{1 lemma 5.3}).
\end{proof}
Using our definitions, the Jacobi–Anger expansion (\ref{complex expansion}) takes the simple form
\begin{equation}
\phi_{\mathbf{y}}(\mathbf{x})=\sum_{\ell=0}^{\infty}\sum_{{m}=-\ell}^{\ell}4 \pi i^{\ell}\,\overline{\tilde{a}_{\ell}^m(\mathbf{y})} \,\tilde{b}_{\ell}^m(\mathbf{x})=\sum_{\ell=0}^{\infty}\sum_{{m}=-\ell}^{\ell}\tau_{\ell}\,\overline{a_{\ell}^m(\mathbf{y})}\,b_{\ell}^m(\mathbf{x})
\label{tau jacobi-anger}
\end{equation}
where we have introduced
\begin{equation*}
\tau_{\ell}:=4\pi i^{\ell}(\alpha_{\ell}\beta_{\ell})^{-1},\,\,\,\,\,\,\,\,\,\,\,\,\,\,\forall \ell \geq 0.
\end{equation*}
The formula (\ref{tau jacobi-anger}) plays a central role in the following: it links the spherical waves basis (\ref{b tilde definizione}) of the Helmholtz solution space $\mathcal{B}$ to the Herglotz densities basis (\ref{a tilde definizione}) of the space $\mathcal{A}$ by means of the evanescent plane waves $\phi_{\mathbf{y}}$ in (\ref{evanescent wave}).

Thanks to the asymptotics presented in Lemma \hyperlink{Lemma 1.3}{1.4} and Lemma \hyperlink{Lemma 5.2}{5.4}, we can deduce the next result.
\begin{corollary}
\hypertarget{Corollary 5.1}{There} exist uniform bounds for $|\tau_{\ell}|$, namely
\begin{equation}
\tau_{-}:=\inf_{\ell \geq 0}|\tau_{\ell}|>0,\,\,\,\,\,\,\,\,\,\,\text{and}\,\,\,\,\,\,\,\,\,\,\tau_{+}:=\sup_{\ell \geq 0}|\tau_{\ell}|<\infty.
\label{uniform bounds tau}
\end{equation}
\end{corollary}

\begin{figure}
\centering
\begin{subfigure}{.44\textwidth}
  \centering
  \includegraphics[width=\linewidth]{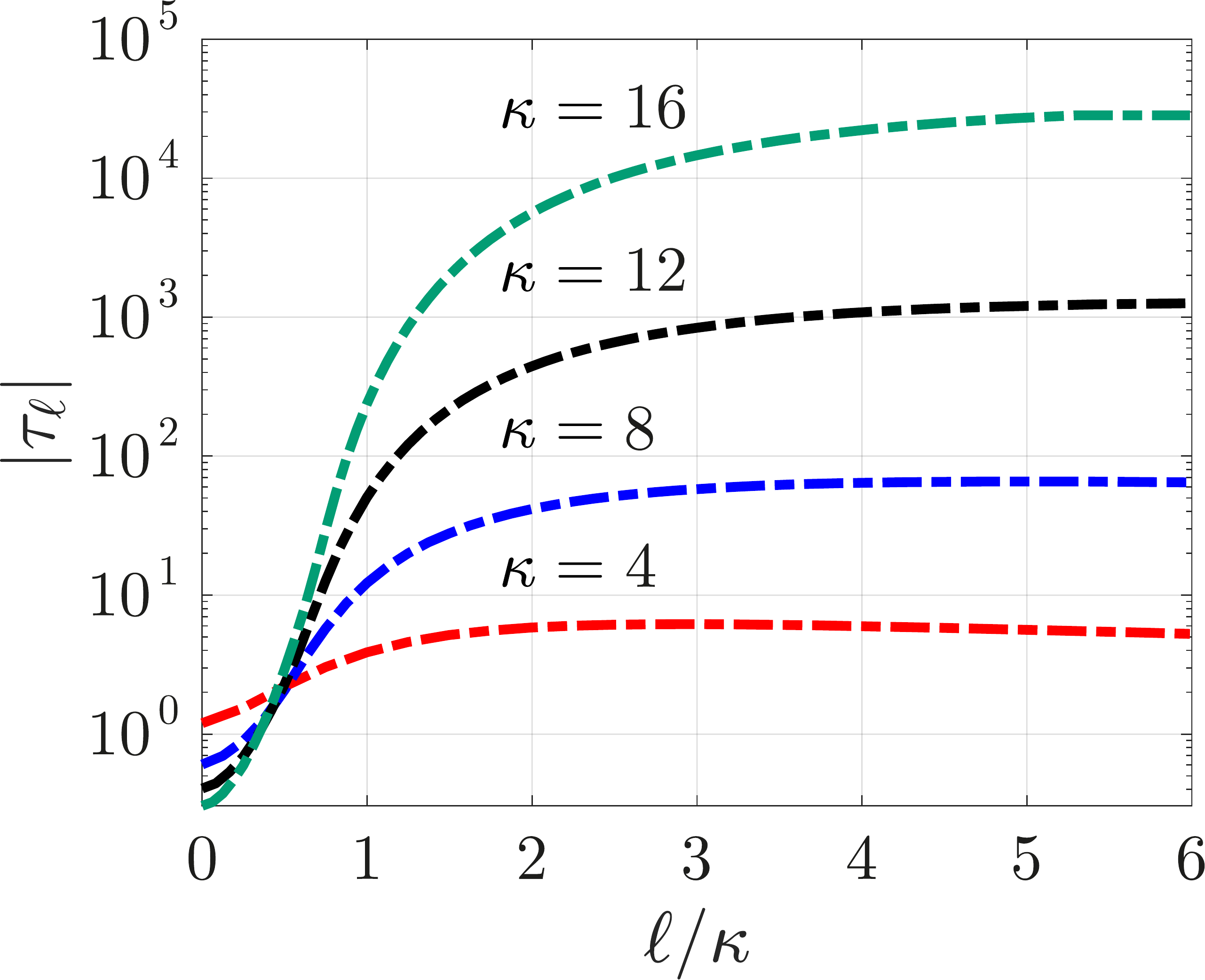}
  \caption{Dependence of $|\tau_{\ell}|$ on the mode number $\ell$ for various wavenumber $\kappa$.}
  \label{figure 5.2a}
\end{subfigure}
\hfill
\begin{subfigure}{.44\textwidth}
  \centering
  \includegraphics[width=0.97\linewidth]{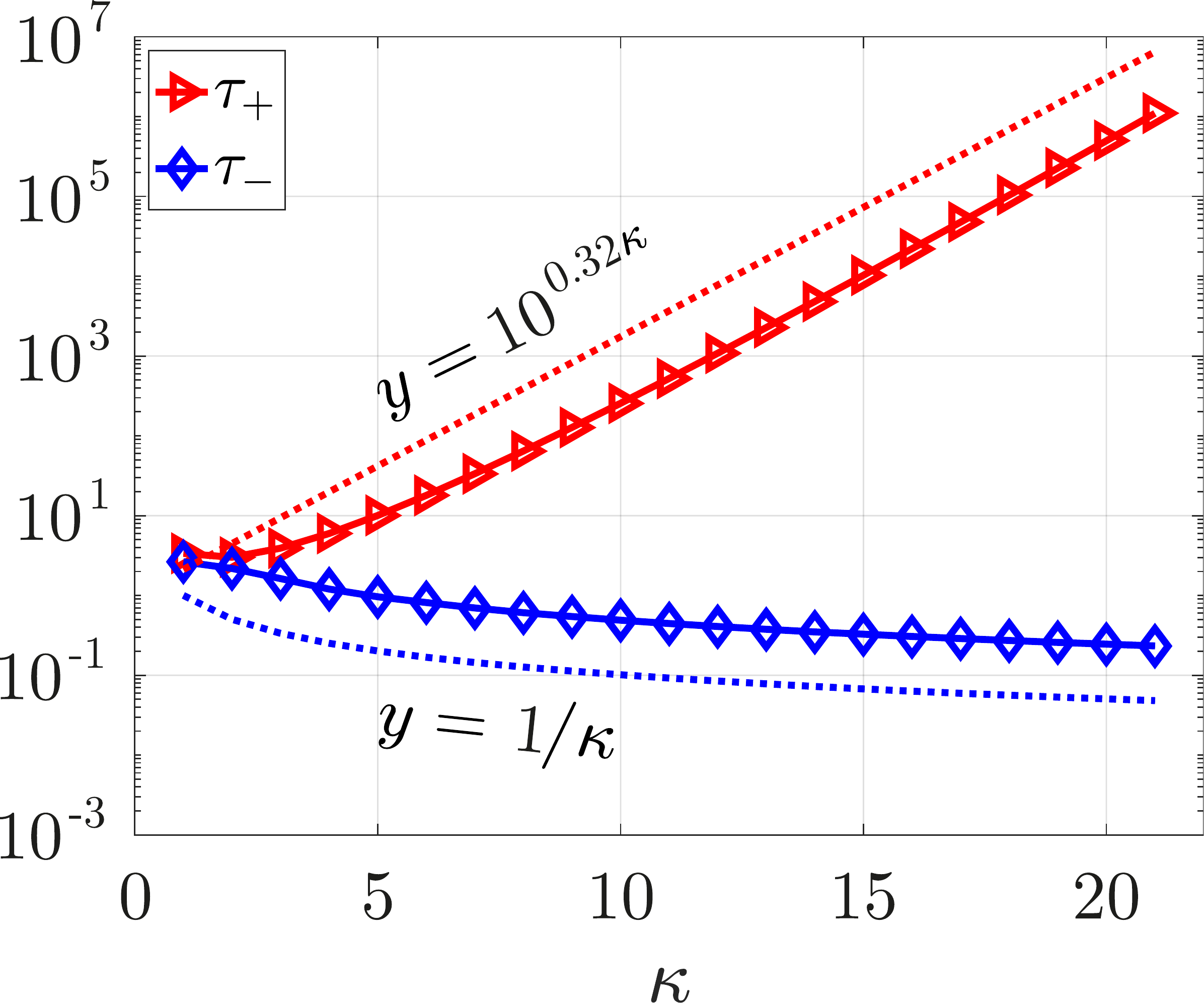}
  \caption{Dependence of $\tau_{\pm}$ defined in (\ref{uniform bounds tau}) on the wavenumber $\kappa$.}
  \label{figure 5.2b}
\end{subfigure}
\caption{}
\end{figure}

It is evident that specific pairs of norms for the Helmholtz solution space (\ref{B norm}) and the Herglotz density space (\ref{A norm}) are needed in order to establish the uniform bounds (\ref{uniform bounds tau}).
This justifies the choice of the weight (\ref{weight}).

The behavior of $|\tau_{\ell}|$ is crucial for the following analysis and is given in Figure \ref{figure 5.2a} for different wavenumbers $\kappa$. Observe that this plot aligns with the results stated in (\ref{uniform bounds tau}), since the curves display a flat asymptotic behavior for larger values of $\ell$.
Moreover, note that the values $\tau_{\pm}$ depend on the wavenumber $\kappa$, as shown in Figure \ref{figure 5.2b}.

\section{Herglotz integral representation}

\hypertarget{Section 5.2}{In} this section, we introduce the \textit{Herglotz transform} $T$. Through this integral operator, we are able to express any Helmholtz solution in $\mathcal{B}$ as a linear combination of evanescent plane waves, each weighted by an element of $\mathcal{A}$.
Borrowing the terminology from \textit{Frame Theory} (for a reference on this field see \cite{christensen}), we also describe the adjoint operator $T^*$, the corresponding frame and Gram operators $S$ and $G$ and prove some of their properties.

First we present the following lemma, useful for defining the Herglotz transform.
\begin{lemma}
\hypertarget{Lemma 5.4}{For} any $\mathbf{x} \in B_1$, $\mathbf{y} \mapsto \overline{\phi_{\mathbf{y}}(\mathbf{x})} \in \mathcal{A}$.
\end{lemma}
\begin{proof}
Let $\mathbf{x} \in B_1$ and define $v_{\mathbf{x}}:\mathbf{y} \mapsto \overline{\phi_{\mathbf{y}}(\mathbf{x})}$. We have, using the Jacobi--Anger identity (\ref{tau jacobi-anger}), that
\vspace{-5mm}
\begin{equation*}
v_{\mathbf{x}}(\mathbf{y})=\sum_{\ell=0}^{\infty}\sum_{{m}=-\ell}^{\ell}\overline{\tau_{\ell}}\,\overline{b_{\ell}^m(\mathbf{x})}a_{\ell}^m(\mathbf{y}),\,\,\,\,\,\,\,\,\,\,\forall \mathbf{y} \in Y.
\end{equation*}
Since $\{a_{\ell}^m\}_{(\ell,m) \in \mathcal{I}}$ is a Hilbert basis for $\mathcal{A}$, thanks to \cite[Eq.\ (2.4.105)]{nedelec} we get
\vspace{-3mm}
\begin{equation*}
\|v_{\mathbf{x}}\|^2_{\mathcal{A}}=\sum_{\ell=0}^{\infty}\sum_{m=-\ell}^{\ell}|\tau_{\ell}b_{\ell}^m|^2\leq \frac{\tau^2_{+}}{4\pi}\sum_{\ell=0}^{\infty}(2\ell+1)\beta^2_{\ell}j_{\ell}^2(\kappa|\mathbf{x}|).
\end{equation*}
Thanks to (\ref{beta_l asymptotic}) and (\ref{spherical bessel asymptotics}), it is readily seen that
\begin{equation*}
(2\ell+1)\beta^2_{\ell}j_{\ell}^2(\kappa|\mathbf{x}|) \sim 2 \kappa^2|\mathbf{x}|^{2\ell},
\end{equation*}
from which we conclude that $\|v_{\mathbf{x}}\|_{\mathcal{A}}<\infty$, since $|\mathbf{x}|<1$.
\end{proof}
Note that, if $\mathbf{x} \in \partial B_1$, then $\mathbf{y} \mapsto \overline{\phi_{\mathbf{y}}(\mathbf{x})}$ does not belong to $\mathcal{A}$, as can be easily seen from the previous proof.
We are ready to define the Herglotz transform.

\begin{definition}[\hypertarget{Definition 5.8}{Herglotz transform}]
For any $v \in \mathcal{A}$, we define the Herglotz transform, denoted by $T$, as the operator
\begin{equation}
\boxed{(Tv)(\mathbf{x}):=\int_{Y}v(\mathbf{y})\phi_{\mathbf{y}}(\mathbf{x})w(\mathbf{y})\textup{d}\mathbf{y},\,\,\,\,\,\,\,\,\,\,\,\,\forall \mathbf{x} \in B_1.}
\label{Herglotz transform}
\end{equation}
\end{definition}
This operator is well-defined on $\mathcal{A}$ thanks to Lemma \hyperlink{Lemma 5.4}{5.7}. In the setting of continuous-frame theory, see e.g.\ \cite[Eq.\ (5.27)]{christensen}, this operator is called \textit{synthesis operator}.
\begin{theorem}
\hypertarget{Theorem 5.1}{The} operator $T$ is bounded and invertible on $\mathcal{A}$:
\begin{align*}
T\,:\,\,&\mathcal{A} \rightarrow \mathcal{B},\\
&v \mapsto Tv=\sum_{\ell=0}^{\infty}\sum_{m=-\ell}^{\ell}\tau_{\ell}\,(v,a_{\ell}^m)_{\mathcal{A}}\,b_{\ell}^m, \numberthis \label{T explicit definition}
\end{align*}
and
\begin{equation}
\tau_{-}\|v\|_{\mathcal{A}}\leq\|Tv\|_{\mathcal{B}}\leq\tau_+\|v\|_{\mathcal{A}},\,\,\,\,\,\,\,\,\,\,\,\,\,\,\forall v \in \mathcal{A}.
\label{T operator bounds}
\end{equation}
In particular, $T$ is diagonal with respect to the bases $\{a_{\ell}^m\}_{(\ell,m) \in \mathcal{I}}$ and $\{b_{\ell}^m\}_{(\ell,m) \in \mathcal{I}}$, namely:
\begin{equation}
Ta_{\ell}^m=\tau_{\ell}\,b_{\ell}^m,\,\,\,\,\,\,\,\,\,\,\,\,\,\,\forall (\ell,m) \in \mathcal{I}.
\label{diagonalization}
\end{equation}
\end{theorem}
\begin{proof}
Thanks to the Jacobi--Anger identity (\ref{tau jacobi-anger}), for any $v \in \mathcal{A}$ and $\mathbf{x} \in B_1$ we have that
\begin{align*}
(Tv)(\mathbf{x})&=\int_{Y}\phi_{\mathbf{y}}(\mathbf{x})v(\mathbf{y})w(\mathbf{y})\textup{d}\mathbf{y}=\int_{Y}\left(\sum_{\ell=0}^{\infty}\sum_{m=-\ell}^{\ell}\tau_{\ell}\,b_{\ell}^m(\mathbf{x})\overline{a_{\ell}^m(\mathbf{y})} \right)v(\mathbf{y})w(\mathbf{y})\textup{d}\mathbf{y}\\
&=\sum_{\ell=0}^{\infty}\sum_{m=-\ell}^{\ell}\tau_{\ell}\int_{Y}\overline{a_{\ell}^m(\mathbf{y})}v(\mathbf{y})w(\mathbf{y})\textup{d}\mathbf{y}\,b_{\ell}^m(\mathbf{x})=\sum_{\ell=0}^{\infty}\sum_{m=-\ell}^{\ell}\tau_{\ell}\,(v,a_{\ell}^m)_{\mathcal{A}}\,b_{\ell}^m(\mathbf{x}),
\end{align*}
and so (\ref{T explicit definition}) holds. Hence, from Lemma \hyperlink{Lemma 1.1}{1.2}
\begin{equation*}
\|Tv\|^2_{\mathcal{B}}=\sum_{\ell=0}^{\infty}\sum_{m=-\ell}^{\ell}|\tau_{\ell}|^2|(v,a_{\ell}^m)_{\mathcal{A}}|^2,\,\,\,\,\,\,\,\,\,\,\,\,\,\,\forall v \in \mathcal{A},
\end{equation*}
and (\ref{T operator bounds}) is derived from the results of Lemma \hyperlink{Lemma 5.3}{5.5} and the uniform bounds in (\ref{uniform bounds tau}). It is easily verifiable that the inverse is provided by
\begin{equation}
T^{-1}u=\sum_{\ell=0}^{\infty}\sum_{m=-\ell}^{\ell}\tau_{\ell}^{-1}(u,b_{\ell}^m)_{\mathcal{B}}\,a_{\ell}^m,\,\,\,\,\,\,\,\,\,\,\,\,\,\,\forall u \in \mathcal{B}.
\label{T inverse definition}
\end{equation}
\end{proof}
It follows that the Herglotz transform $T$ is bounded and invertible between the space of
Herglotz densities $\mathcal{A}$ and the space of Helmholtz solutions $\mathcal{B}$. From (\ref{T inverse definition}), the inverse operator $T^{-1}$ can also be written as an integral operator: for any $u \in \mathcal{B}$,
\begin{equation*}
\left(T^{-1}u\right)(\mathbf{y})=\int_{B_1}u(\mathbf{x})\Psi(\mathbf{x},\mathbf{y})\textup{d}\mathbf{x}+\kappa^{-2}\int_{B_1}\nabla u(\mathbf{x})\cdot \nabla \Psi(\mathbf{x},\mathbf{y})\textup{d}\mathbf{x},\,\,\,\,\,\,\,\,\,\,\,\,\forall \mathbf{y} \in Y,
\end{equation*}
where the kernel $\Psi$ is defined, for any $\mathbf{y} \in Y$, as
\begin{equation*}
\Psi(\mathbf{x},\mathbf{y}):=\sum_{\ell=0}^{\infty}\sum_{m=-\ell}^{\ell}\tau_{\ell}^{-1}a_{\ell}^m(\mathbf{y})\overline{b_{\ell}^m(\mathbf{x})},\,\,\,\,\,\,\,\,\,\,\,\,\forall \mathbf{x} \in B_1.
\end{equation*}
The integral representation $Tv$ in (\ref{Herglotz transform}) resembles the Herglotz representation (\ref{herglotz}), but at the same time it increases its scope to all the Helmholtz solutions in $\mathcal{B}$.
In fact, the standard Herglotz representation (\ref{herglotz}) cannot represent all the Helmholtz solution as a continuous superposition of propagative plane waves (\ref{propagative wave}) with density $v \in L^2(\mathbb{S}^2)$, as previously explained in Section \hyperlink{Section 3.2}{3.2}.
However, by including evanescent waves (\ref{evanescent wave}), any Helmholtz solution can be represented by using the generalized Herglotz representation (\ref{Herglotz transform}). This is because the operator $T$ is an isomorphism between the spaces $\mathcal{A}$ and $\mathcal{B}$, meaning for any $u \in \mathcal{B}$, there is a unique corresponding $v \in \mathcal{A}$ such that $u=Tv$.
The price to pay for this result is the need for a 4D parameter domain (the Cartesian product $Y$) in place of a 2D one ($[0,\pi] \times [0,2\pi)$) and thus of a quadruple integral; the added dimensions correspond to the evanescence parameters $\theta_3$ and $\zeta$.

Moreover, for any $(\ell,m) \in \mathcal{I}$, the Herglotz density $\tau_{\ell}^{-1}a_{\ell}^m$ of the spherical wave $b_{\ell}^m$ is bounded in the $\mathcal{A}$ norm by $\tau_{-}^{-1}$ due to (\ref{diagonalization}), and thus uniformly with respect to the index $\ell$. This is in contrast to the standard Herglotz representation (\ref{herglotz spherical waves}) using only propagative plane waves, where the associated Herglotz densities can not be bounded uniformly in $L^2(\mathbb{S}^2)$ with respect to the index $\ell$.
In this sense, Theorem \hyperlink{Theorem 5.1}{5.9} can be considered a sort of stability result at the continuous level. Our goal is to derive a discrete version of this integral representation.

In the continuous-frame setting, see \cite[Eq.\ (5.28)]{christensen}, the adjoint operator of $T$, $T^{*}$, is referred to as the \textit{analysis operator}.
\begin{lemma}
\hypertarget{Lemma 5.5}{The} adjoint operator $T^*$ of $T$ is given for any $u \in \mathcal{B}$ by
\begin{equation*}
(T^*u)(\mathbf{y}):=(u,\phi_{\mathbf{y}})_{\mathcal{B}},\,\,\,\,\,\,\,\,\,\,\,\,\forall \mathbf{y} \in Y.
\end{equation*}
The operator $T^*$ is bounded and invertible on $\mathcal{B}$:
\begin{align*}
T^*\,:\,\,&\mathcal{B} \rightarrow \mathcal{A},\\
&u \mapsto T^*u=\sum_{\ell=0}^{\infty}\sum_{m=-\ell}^{\ell}\overline{\tau_{\ell}}\,(u,b_{\ell}^m)_{\mathcal{B}}\,a_{\ell}^m, \numberthis \label{T* explicit definition}
\end{align*}
and
\begin{equation}
\tau_{-}\|u\|_{\mathcal{B}}\leq\|T^*u\|_{\mathcal{A}}\leq\tau_+\|u\|_{\mathcal{B}},\,\,\,\,\,\,\,\,\,\,\,\,\forall u \in \mathcal{B}.
\label{T* operator bounds}
\end{equation}
\end{lemma}
\begin{proof}
For any $v \in \mathcal{A}$ and $u \in \mathcal{B}$, we have:
\begin{align*}
(Tv,u)_{\mathcal{B}}&=\left(\int_{Y}\phi_{\mathbf{y}}v(\mathbf{y})w(\mathbf{y})\textup{d}\mathbf{y},u \right)_{\mathcal{B}}=\int_{Y}v(\mathbf{y})(\phi_{\mathbf{y}},u)_{\mathcal{B}}\,w(\mathbf{y})\textup{d}\mathbf{y}\\
&=\left(v,\overline{\left(\phi_{\mathbf{y}},u\right)}_{\mathcal{B}}\right)_{\mathcal{A}}=\left(v,\left(u,\phi_{\mathbf{y}}\right)_{\mathcal{B}}\right)_{\mathcal{A}}.
\end{align*}
Furthermore, using the Jacobi-Anger identity (\ref{tau jacobi-anger}), for any $u \in \mathcal{B}$ and $\mathbf{y} \in Y$
\begin{equation*}
(T^*u)(\mathbf{y})=(u,\phi_{\mathbf{y}})_{\mathcal{B}}=\left(u,\sum_{\ell=0}^{\infty}\sum_{m=-\ell}^{\ell}\tau_{\ell}\,\overline{a_{\ell}^m(\mathbf{y})}\,b_{\ell}^m\right)_{\mathcal{B}}=\sum_{\ell=0}^{\infty}\sum_{m=-\ell}^{\ell}\overline{\tau_{\ell}}\,(u,b_{\ell}^m)_{\mathcal{B}}\,a_{\ell}^m(\mathbf{y}).
\end{equation*}
and so (\ref{T* explicit definition}) follows. Since $T$ is invertible, then $T^*$ is also invertible and therefore it remains to prove that (\ref{T* operator bounds}) holds. Hence, from Lemma \hyperlink{Lemma 5.3}{5.5}
\begin{equation*}
\|T^*u\|^2_{\mathcal{A}}=\sum_{\ell=0}^{\infty}\sum_{m=-\ell}^{\ell}|\tau_{\ell}|^2|(u,b_{\ell}^m)_{\mathcal{B}}|^2,
\end{equation*}
and (\ref{T* operator bounds}) is derived from Lemma \hyperlink{Lemma 1.2}{1.2} and the uniform bounds in (\ref{uniform bounds tau}).
\end{proof}

In the continuous frame terminology used in \cite{christensen}, we introduce the concepts of \textit{frame operator} and \textit{Gram operator} as follows:
\begin{equation*}
\begin{split}
S&:=TT^*\,:\,\mathcal{B} \rightarrow \mathcal{B},\\
G&:=T^*T\,:\,\mathcal{A} \rightarrow \mathcal{A}.
\end{split}
\end{equation*}
The frame operator formula can be made more explicit as follows: for any $u \in \mathcal{B}$,
\begin{equation*}
(Su)(\mathbf{x})=\int_Y(u,\phi_{\mathbf{y}})_{\mathcal{B}}\,\phi_{\mathbf{y}}(\mathbf{x})w(\mathbf{y})\textup{d}\mathbf{y},\,\,\,\,\,\,\,\,\,\,\,\,\forall \mathbf{x} \in B_1.
\end{equation*}
\begin{corollary}
\hypertarget{Corollary 5.3}{The} operators $S$ and $G$ are bounded, invertible, self-adjoint and positive. For any $v \in \mathcal{A}$ and $u \in \mathcal{B}$,
\begin{equation*}
\begin{split}
Su&=\sum_{\ell=0}^{\infty}\sum_{m=-\ell}^{\ell}|\tau_{\ell}|^2(u,b_{\ell}^m)_{\mathcal{B}}\,b_{\ell}^m,\\
Gv&=\sum_{\ell=0}^{\infty}\sum_{m=-\ell}^{\ell}|\tau_{\ell}|^2(v,a_{\ell}^m)_{\mathcal{A}}\,a_{\ell}^m,
\end{split}
\,\,\,\,\,\,\,\,\,\,\text{and}\,\,\,\,\,\,\,\,\,\,
\begin{split}
\tau^2_{-}\|u\|_{\mathcal{B}}&\leq\|Su\|_{\mathcal{B}}\leq\tau^2_+\|u\|_{\mathcal{B}},\\
\tau^2_{-}\|v\|_{\mathcal{A}}&\leq\|Gv\|_{\mathcal{A}}\leq\tau^2_+\|v\|_{\mathcal{A}}.
\end{split}
\end{equation*}
\end{corollary}
\begin{proof}
This result follows directly from Theorem \hyperlink{Theorem 5.1}{5.9} and Lemma \hyperlink{Lemma 5.5}{5.10}.
\end{proof}
Finally, we are now ready to prove that the evanescent plane waves (\ref{evanescent wave}) constitute a continuous frame for the space of Helmholtz solutions $\mathcal{B}$.
First, we outline the general definition in \cite[Def.\ 5.6.1]{christensen}.
\begin{definition}[\hypertarget{Definition 5.11}{Continuous frame}]
Let $\mathcal{H}$ be a complex Hilbert space and $X$ a measure space provided with a positive measure $\mu$. A continuous frame is a family of vectors $\{\phi_x\}_{x \in X} \subset \mathcal{H}$ for which:
\begin{itemize}
\item for all $u \in \mathcal{H}$, the function $x \mapsto \left(u,\phi_x\right)_{\mathcal{H}}$ is a measurable function on $X$;
\item there exist constants $A,B >0$ such that
\begin{equation*}
A\|u\|^2_{\mathcal{H}}\leq \int_X|(u,\phi_{x})_{\mathcal{H}}|^2\textup{d}\mu(x)\leq B\|u\|^2_{\mathcal{H}},\,\,\,\,\,\,\,\,\,\,\,\,\forall u \in \mathcal{H}.
\end{equation*}
\end{itemize}
\end{definition}

\begin{theorem}
\hypertarget{Theorem 5.13}{The} family $\{\phi_{\mathbf{y}}\}_{\mathbf{y} \in Y}$ is a continuous frame for $\mathcal{B}$. Besides, the optimal frame bounds are $A=\tau_{-}^2$ and $B=\tau_{+}^2$.
\end{theorem}
\begin{proof}
It is readily checked that the family of evanescent waves $\{\phi_{\mathbf{y}}\}_{\mathbf{y} \in Y}$ satisfies the two conditions in Definition \hyperlink{Definition 5.11}{5.12}.
In fact, for any $u \in \mathcal{B}$, the measurability of
\begin{equation*}
\mathbf{y} \mapsto (u,\phi_{\mathbf{y}})_{\mathcal{B}}=(T^*u)(\mathbf{y})
\end{equation*}
follows from $T^*u \in \mathcal{A}$, according to Lemma \hyperlink{Lemma 5.5}{5.10}, and $\mathcal{A} \subset L^2(Y;w)$.

The second condition, namely
\begin{equation*}
A\|u\|^2_{\mathcal{B}}\leq \int_Y|(u,\phi_{\mathbf{y}})_{\mathcal{B}}|^2w(\mathbf{y})\textup{d}\mathbf{y}\leq B\|u\|^2_{\mathcal{B}},\,\,\,\,\,\,\,\,\,\,\,\,\forall u \in \mathcal{B},
\end{equation*}
for some constants $A,B>0$, follows from the boundedness and positivity of the frame operator $S$, as proven in Corollary \hyperlink{Corollary 5.3}{5.11}. Indeed, for any $u \in \mathcal{B}$, we have
\begin{equation*}
\int_Y|(u,\phi_{\mathbf{y}})_{\mathcal{B}}|^2w(\mathbf{y})\textup{d}\mathbf{y}=(Su,u)_{\mathcal{B}}=\sum_{\ell=0}^{\infty}\sum_{m=-\ell}^{\ell}|\tau_{\ell}|^2|(u,b_{\ell}^m)_{\mathcal{B}}|^2,
\end{equation*}
which also shows that the frame bounds stated are optimal.
\end{proof}

\section{The reproducing kernel property}

\hypertarget{Section 5.3}{A} noteworthy consequence of the continuous frame result is presented in the following proposition, which we borrow from \cite[Sec.\ 6.3]{parolin-huybrechs-moiola} and to which we refer for the proof.
\begin{proposition}
\hypertarget{Proposition 5.1}{The} Herglotz density space $\mathcal{A}$ has the reproducing kernel pro-perty. The reproducing kernel is given by
\begin{equation}
K(\mathbf{z},\mathbf{y})=K_{\mathbf{y}}(\mathbf{z})=\left(K_{\mathbf{y}},K_{\mathbf{z}}\right)_{\mathcal{A}}=\sum_{\ell=0}^{\infty}\sum_{m=-\ell}^{\ell}\overline{a_{\ell}^m(\mathbf{y})}a_{\ell}^m(\mathbf{z}),\,\,\,\,\,\,\,\,\,\,\forall \mathbf{y},\mathbf{z} \in Y,
\label{reproducing kernel property}
\end{equation}
with pointwise convergence of the series and where $K_{\mathbf{y}} \in \mathcal{A}$ is the (unique) Riesz representation of the evaluation functional at $\mathbf{y} \in Y$, namely
\begin{equation}
v(\mathbf{y})=\left(v,K_{\mathbf{y}} \right)_{\mathcal{A}},\,\,\,\,\,\,\,\,\,\,\,\,\forall v \in \mathcal{A}.
\label{hccp}
\end{equation}
\end{proposition}

It is important to note that the reproducing kernel property (\ref{hccp}) implies that the evaluation of elements of $\mathcal{A}$ at any point in $Y$ is continuous \cite[Def.\ 1.2]{reproducing_kernels}.
Let $\mathbf{y} \in Y$, then for some constant $c_{\mathbf{y}}>0$:
\begin{equation*}
|v(\mathbf{y})|=\left|\left(v,K_{\mathbf{y}}\right)_{\mathcal{A}}\right|\leq c_{\mathbf{y}} \|v\|_{\mathcal{A}},\,\,\,\,\,\,\,\,\,\,\,\,\forall v \in \mathcal{A}.
\end{equation*}
The motivation for introducing the reproducing kernel property is highlighted in the following result, which is directly derived from Proposition \hyperlink{Proposition 5.1}{5.14}, Theorem \hyperlink{Theorem 5.1}{5.9} and the Jacobi–Anger identity (\ref{tau jacobi-anger}).
\begin{corollary}
\hypertarget{Corollary 5.14}{The} evanescent plane waves are the images under $T$ of the Riesz representation of the evaluation functionals, namely
\begin{equation}
\phi_{\mathbf{y}}=TK_{\mathbf{y}},\,\,\,\,\,\,\,\,\,\,\,\,\forall \mathbf{y} \in Y.
\label{corollary label}
\end{equation}
\end{corollary}
Therefore, approximating a Helmholtz solution $u \in \mathcal{B}$ using evanescent plane waves is, through the isomorphism $T$, equivalent to approximating its Herglotz density $v=T^{-1}u \in \mathcal{A}$ by an expansion of evaluation functionals, namely
\begin{equation}
v \approx \sum_{p=1}^P\mu_p K_{{\mathbf{y}}_p}\,\,\,\,\,\,\,\,  \rightleftarrowss{T}{T^{-1}}\,\,\,\,\,\,\,\,  u \approx \sum_{p=1}^P\mu_p \phi_{{\mathbf{y}}_p}
\label{isomorphism}
\end{equation}
for some set of coefficients $\boldsymbol{\mu}=\{\mu_{p}\}_{p=1}^P$.
The subsequent chapters provide numerical evidence that it is indeed possible to build such suitable approximations (up to some normalization of the families $\{K_{{\mathbf{y}}_p}\}_p$ and $\{\phi_{{\mathbf{y}}_p}\}_p$).

\chapter{Evanescent plane wave approximation sets}
\hypertarget{Chapter 6}{In} this chapter, we describe a method for the stable numerical approximation of a general Helmholtz solution in the unit ball $B_1$ by evanescent plane waves. 

The core of this procedure relies on the equivalence between this approximation problem and the one of the corresponding Herglotz density, as stated in (\ref{isomorphism}). In analogy with \cite[Sec.\ 7]{parolin-huybrechs-moiola}, the main idea is to adapt the sampling technique from \cite{Cohen_Migliorati,Hampton,Migliorati_Nobile} (referred to as \textit{coherence-optimal sampling}) to our situation, to generate a distribution of sampling nodes in $Y$ that will be used to reconstruct the Herglotz density.
Simple variants including the use of extremal point systems defined in (\ref{G matrix}) are also taken into account.
In particular, the last section delves deeper into the explanation of the numerical recipe and showcases a diverse array of sampling strategies.
The approach outlined has been found to be highly effective, as emerges from Chapter \hyperlink{Chapter 7}{7} and \cite[Sec.\ 8]{parolin-huybrechs-moiola}. However, there is still a lack of full proof of the accuracy and stability of the approximation of Helmholtz solutions using evanescent plane waves.

Let $u \in \mathcal{B}$ be the goal of our approximation problem, and $v:=T^{-1}u \in \mathcal{A}$ its corresponding Herglotz density. Additionally, let some tolerance $\eta > 0$ be given.

\section{Approximation sets}

\hypertarget{Section 6.1}{Since} both $\mathcal{A}$ and $\mathcal{B}$ are infinite-dimensional spaces, the strategy for constructing finite-dimensional approximation sets is to use the natural hierarchy of finite-dimensional subspaces created by truncating the Hilbert bases $\{a_{\ell}^m\}_{(\ell,m) \in \mathcal{I}}$ and $\{b_{\ell}^m\}_{(\ell,m) \in \mathcal{I}}$ respectively.

\begin{definition}[Truncated spaces]
For any $L \geq 0$, we define, respectively, the truncated Herglotz density space and the truncated Helmholtz solution space as
\begin{equation*}
\mathcal{A}_{L}:=\textup{span}\{a_\ell^m\}_{(\ell,m) \in \mathcal{I}\,:\,\ell\, \leq\, L} \subsetneq \mathcal{A},\,\,\,\,\,\,\,\,\,\,\,\,\,\,\mathcal{B}_{L}:=\textup{span}\{b_\ell^m\}_{(\ell,m) \in \mathcal{I}\,:\,\ell\, \leq\, L} \subsetneq \mathcal{B}.
\end{equation*}
We denote the dimension of both spaces $\mathcal{A}_L$ and $\mathcal{B}_L$ by
\begin{equation*}
N=N(L):=\dim \mathcal{A}_L=\dim \mathcal{B}_L=(L+1)^2 \in \N.
\end{equation*}
\end{definition}
We also introduce the orthogonal projection $\Pi_L$ onto $\mathcal{B}_L$ defined as
\begin{align*}
\Pi_L\,:\,\,&\mathcal{B} \rightarrow \mathcal{B},\\
&u \mapsto \sum_{\ell=0}^{L}\sum_{m=-\ell}^{\ell}\left( u,b_{\ell}^m\right)_{\mathcal{B}}b_{\ell}^m,
\end{align*}
and denote by $u_L:=\Pi_Lu \in \mathcal{B}_L$ the projection of $u$ onto $\mathcal{B}_L$ and by $v_L:=T^{-1}u_L \in \mathcal{A}_L$ the image under the inverse transform of $u_L$. Obviously, the sequence of projections $\{u_L\}_{L \geq 0}$ converges to $u$ in $\mathcal{B}$. In particular, we can define
\begin{equation}
L^*=L^*(u,\eta):=\min\{L \geq 0 : \|u-u_L\|_{\mathcal{B}} < \eta \|u\|_{\mathcal{B}}\}.
\label{L star}
\end{equation}
Furthermore, (\ref{T operator bounds}) implies that the sequence $\{v_L\}_{L \geq 0}$ converges to $v$ in $\mathcal{A}$: in fact, for any $L \geq L^*$, we have
\begin{equation*}
\|v-v_L\|_{\mathcal{A}} \leq \tau_{-}^{-1} \|u-u_L\|_{\mathcal{B}}  < \tau_{-}^{-1}\eta \|u\|_{\mathcal{B}}.
\end{equation*}

Our goal is to approximate with evanescent plane waves the projection $u_L=\Pi_L u \in \mathcal{B}_L$ (or equivalently $v_L=T^{-1}u_L \in \mathcal{A}_L$).
The main idea is to build approximations of elements of $\mathcal{A}_L$ by constructing a finite set of sampling nodes $\{\mathbf{y}_p\}_p$ in $Y$, according to the distribution outlined in \cite[Sec.\ 2.1]{Hampton}, \cite[Sec.\ 2.2]{Cohen_Migliorati} and \cite[Sec.\ 2]{Migliorati_Nobile}, whose related probability density is reported in (\ref{rho density}).
A variation of this strategy involves creating a finite set $\{\mathbf{y}_p\}_p$ in $Y$ by restricting the sampling based on this distribution to the \textit{evanescence domain} $[0,2\pi)\times[0,+\infty)$ and making use of the extremal point systems for the coordinates in $[0,\pi]\times[0,2\pi)$. This choice seems to be desirable due to the geometrical properties of these systems, which are able to provide well-distributed points. Note that this is equivalent to establishing a priori the direction of propagation of the evanescent plane waves and then selecting, through the sampling based on (\ref{rho density}), the waves intensities and the decay directions.
Even though the domain $[0,+\infty)$ is unbounded, the finite integrability of the weight in (\ref{weight}) allows for sampling in a bounded region only.
With an appropriate normalization factor, the associated set of sampling functionals $\{K_{\mathbf{y}_p}\}_p$ is expected to provide a good approximation of $v_L$. Therefore, up to some normalization factor, the approximation set for $u_L$ will be given by the evanescent plane waves $\{\phi_{\mathbf{y}_p}\}_p$.

The probability density function $\rho_N$ is defined (up to normalization) as the reciprocal of the $N$-term \textit{Christoffel function} following the approach in \cite[Eq.\ (2.6)]{Cohen_Migliorati}:
\begin{equation}
\rho_N:=\frac{w}{N\mu_N},\,\,\,\,\,\,\,\text{where}\,\,\,\,\,\,\,\mu_N(\mathbf{y}):=\left(\sum_{\ell=0}^L\sum_{m=-\ell}^{\ell}\left|a_{\ell}^m(\mathbf{y}) \right|^2 \right)^{-1},\,\,\,\,\,\,\,\forall \mathbf{y} \in Y.
\label{rho density}
\end{equation}
Note that $\rho_N$ and $\mu_N$ are well-defined since $0 < \mu_N \leq \mu_1 < \infty$, because $a_0^0 \neq 0$.
The function $\mu_N$ is actually independent of $\boldsymbol{\theta}$: in fact, thanks to the Wigner D-matrix unitarity condition \cite[Sec.\ 4.1, Eq.\
(6)]{quantumtheory}, it readily follows that
\begin{equation}
\mu_N^{-1}(\mathbf{y})=\sum_{\ell=0}^L\sum_{m=-\ell}^{\ell}\left|a_{\ell}^m(\mathbf{y}) \right|^2=\sum_{\ell=0}^L\alpha_{\ell}^2\left|\mathbf{P}_{\ell}(\zeta) \right|^2,\,\,\,\,\,\,\,\forall \mathbf{y} \in Y.
\label{muN independence}
\end{equation}

Hence, the density function $\rho_N$ is a bivariate function on $Y$, since it is independent of the Euler angles $\theta_2,\theta_3 \in [0,2\pi)$, and depends on $\theta_1$ only through the weight $w$ in (\ref{weight}).
The occurrence of this feature is solely attributed to the inevitable singularities that arise in the spherical parameterization when Euler angles are involved.
As a consequence, the sampling problem can be considered nearly one-dimensional, with the key parameter being $\zeta$. However, selecting an appropriate distribution for $\zeta$ poses a significant challenge.
Moreover, it is worth noting that $1/\mu_N$ corresponds to the truncated series expansion of the diagonal of the reproducing kernel $K$, which is obtained by taking $z=y$ and truncating the series in (\ref{reproducing kernel property}) at $L$.

The numerical recipe involves, for each $L \geq 0$, generating a sequence of node sets in the parametric domain $Y$
\begin{equation}
\numberset{Y}_L:=\{\numberset{Y}_{L,P}\}_{P \in \N},\,\,\,\,\,\,\,\text{where}\,\,\,\,\,\,\,\numberset{Y}_{L,P}:=\{\mathbf{y}_{L,P,\,p}\}_{p=1}^P,\,\,\,\,\,\,\,\forall P \in \N,
\label{sampling node set}
\end{equation}
using a sampling strategy such that $|\numberset{Y}_{L,\,P}|=P$, for all $P \in \N$, and the sequence $\numberset{Y}_L$ converges (in a suitable sense) to the density $\rho_{N(L)}$ defined in (\ref{rho density}) as $P$ tends to infinity. The sets are not assumed to be nested.

Two approximation sets can be constructed: one consisting of sampling functionals in $\mathcal{A}$ and the other of evanescent plane waves in $\mathcal{B}$.
Associated to the node sets (\ref{sampling node set}), we introduce a sequence of finite sets in $\mathcal{A}$ as follows:
\begin{equation}
\begin{split} \label{evaluation sets}
\boldsymbol{\Psi}_L&:=\{\boldsymbol{\Psi}_{L,P}\}_{P \in \N},\,\,\,\,\,\,\,\forall L\geq 0,\,\,\,\,\,\,\,\text{where}\\
\boldsymbol{\Psi}_{L,P}:=&\,\,\Biggl\{\sqrt{\frac{\mu_N(\mathbf{y}_{L,P,\,p})}{P} }K_{\mathbf{y}_{L,P,\,p}}\Biggr\}_{p=1}^P,\,\,\,\,\,\,\,\forall P \in \N.
\end{split}
\end{equation}

\noindent In the approximation sets, each evaluation functional $K_{\mathbf{y}_{L,P,\,p}}$ has been normalized by the real constant $\sqrt{\mu_N(\mathbf{y}_{L,P,\,p})/P} $ which is (numerically) close to $\|K_{\mathbf{y}_{L,P,\,p}}\|^{-1}_{\mathcal{A}}/\sqrt{P}$. More precisely, we have that
\begin{equation*}
\sqrt{\mu_{N}(\mathbf{y})}\|K_{\mathbf{y}}\|_{\mathcal{A}}=\left(\sum_{\ell=0}^L\sum_{m=-\ell}^{\ell}\left|a_{\ell}^m(\mathbf{y}) \right|^2 \right)^{-1/2}\left(\sum_{\ell=0}^{\infty}\sum_{m=-\ell}^{\ell}\left|a_{\ell}^m(\mathbf{y}) \right|^2 \right)^{1/2} \geq 1,\,\,\,\,\,\,\,\forall \mathbf{y} \in Y.
\end{equation*}
The normalization constant in (\ref{evaluation sets}) is crucial for the numerical stability of the scheme: in fact the stable approximation property (\ref{stable approximation}) of a set sequence depends on the normalization of its elements.

Associated to the node set sequences (\ref{sampling node set}) and the approximation set sequences (\ref{evaluation sets}) in $\mathcal{A}$, we define the sequence of approximation sets of normalized evanescent plane waves in $\mathcal{B}$ as follows
\begin{equation}
\begin{split} \label{evanescence sets}
&\,\,\,\boldsymbol{\Phi}:=\{\boldsymbol{\Phi}_{L,P}\}_{L\geq 0,P \in \N},\,\,\,\,\,\,\,\,\,\,\,\,\,\,\text{where}\\
\boldsymbol{\Phi}_{L,P}:=\Biggl\{&\sqrt{\frac{\mu_N(\mathbf{y}_{L,P,\,p})}{P} }\phi_{\mathbf{y}_{L,P,\,p}}\Biggr\}_{p=1}^P,\,\,\,\,\,\,\,\forall L \geq 0, \forall P \in \N.
\end{split}
\end{equation}
Due to (\ref{corollary label}), the sequence of sets (\ref{evanescence sets}) is the image of the sequence of sets (\ref{evaluation sets}) by the Herglotz transform operator $T$.

The numerical recipe for constructing the approximation sets $\boldsymbol{\Phi}_{L,P}$ is based on only two parameters, $L$ and $P$. The tuning of these parameters is straightforward:
\begin{itemize}

\item The first parameter to consider is $L$, which determines the Fourier truncation level. As $L$ increases, the accuracy of the approximation of $u$ (resp. $v=T^{-1}u$) by $u_L=\Pi_Lu$ (resp. $v_L=\Pi_Lv$) improves. The appropriate value for $L \geq L^*$ will depend on the regularity of the Helmholtz solution and thus on the decay rate of the modal expansion coefficients.
\item The second one is the dimension $P$ of the evanescent plane wave approximation space, which is also the number of sampling points in $Y$.
If $L$ is fixed, increasing $P$ should allow to control the accuracy of the approximation of $u_L$ (resp. $v_L=T^{-1}u_L$) by $\mathcal{T}_{\boldsymbol{\Phi}_{L,P}}\boldsymbol{\xi}$ (resp. $\mathcal{T}_{\boldsymbol{\Psi}_{L,P}}\boldsymbol{\xi}$) for some bounded coefficients $\boldsymbol{\xi} \in \C^P$.
The numerical results presented below corroborate this conjecture and show experimentally that $P$ should scale quadratically with $L$, and thus linearly with $N$, with a moderate proportionality constant (see Section \hyperlink{Section 7.2}{7.2}).
\end{itemize}
In our implementation (detailed in Section \hyperlink{Section 2.2}{2.2}), after selecting the approximation sets $\boldsymbol{\Phi}_{L,P}$, the computation of a specific set of coefficients $\boldsymbol{\xi}_{S,\epsilon}$ involves the use of two additional parameters, $S$ and $\epsilon$:
\begin{itemize}
\item The first parameter, $S$, refers to the number of sampling points on the boundary of the physical domain $B_1$. As stated in \cite{huybrechs1,huybrechs2}, it is recommended to use an adequate amount of oversampling.
In order to make use of the theoretical findings in Section \hyperlink{Section 2.3}{2.3}, it is necessary to select the sampling points $\{\mathbf{x}_s\}_{s=1}^S \subset \partial B_1$ and weights $\mathbf{w}_S \in \R^S$ appropriately to meet the requirements of (\ref{Riemann sum}).
In practice, we will use extremal systems of points introduced in Definition \hyperlink{Extremal points system}{2.2}, hence $S$ is chosen as a perfect square. In analogy with \cite{parolin-huybrechs-moiola}, we choose for simplicity an oversampling ratio of $2$, namely $S=\lceil \sqrt{2|\boldsymbol{\Phi}_{L,P}|} \rceil^2$.
It is possible that such a high degree of oversampling is not required and additional numerical experiments could be conducted to explore a reduction in the oversampling ratio $S/|\boldsymbol{\Phi}_{L,P}|$ in order to decrease the computational cost of the method.
\item The second parameter, $\epsilon$, is the regularization parameter used in the truncation of the singular values. To evaluate the method, we use a value of $\epsilon=10^{-14}$ in the numerical experiments that follow.
If less precise approximations are satisfactory, the parameter $\epsilon$ could be increased.
\end{itemize}

It is important to note that the selection of the reconstruction strategy does not affect the approximation sets $\boldsymbol{\Phi}_{L,P}$, together with their related accuracy and stability properties.
Although we presented the method of boundary sampling with regularized SVD as a simple example, other reconstruction strategies, such as sampling within the domain or using Galerkin or Petrov--Galerkin projections, can also be effective. Similarly, other regularization techniques, like Tikhonov regularization, can also be applied. Regardless of the specific strategy chosen, it is crucial to apply sufficient oversampling and regularization.

The construction used here, which is based on the ideas outlined in \cite[Sec.\ 7.3]{parolin-huybrechs-moiola}, builds upon similar concepts that have been previously explored in different contexts. Indeed, sampling node sets similar to the ones proposed here can be found in literature, such as in \cite{Cohen_Migliorati,Hampton,Migliorati_Nobile}.
The context of these works is the reconstruction of elements of finite-dimensional subspaces (with explicit orthonormal basis) in weighted $L^2$ spaces using sampling, as presented in \cite{Cohen_Migliorati}. This approach was later used to construct random cubature rules in \cite{Migliorati_Nobile}.
The idea behind these methods is that by sampling at specific nodes, it is possible to gather enough information to accurately reconstruct the function as an expansion in the (truncated) orthonormal basis.

Within this framework, the results from the literature state that to reconstruct an element $v_L=\Pi_L v \in \mathcal{A}_L$, it is sufficient to sample at the nodes $\boldsymbol{\Psi}_{L,P}$ for a sufficiently large value of $P$.
In contrast, the numerical method described above aims to construct an approximation of the element $v_L=\Pi_L v \in \mathcal{A}_L$ as an expansion in the set of evaluation functionals $\boldsymbol{\Psi}_{L,P}$ for some sufficiently large $P$. This means that the approximation we are seeking belongs to the span of the evaluation functionals, span $\boldsymbol{\Psi}_{L,P}$, which has trivial intersection with $\mathcal{A}_L$.
Thanks to (\ref{corollary label}), applying the Herglotz transform $T$ to this approximation in span $\boldsymbol{\Psi}_{L,P}$ yields an element in span $ \boldsymbol{\Phi}_{L,P}$ (i.e.\ a finite superposition of evanescent plane waves) that approximates $u_L=Tv_L \in \mathcal{B}_L$.
Despite the connections to related works, a full proof is currently missing and we lack a solid theoretical foundation to support this numerical method. However, the extensive numerical tests presented in Chapter \hyperlink{Chapter 7}{7} demonstrate the high level of accuracy and stability of the sets $\boldsymbol{\Phi}_{L,P}$.

\section{A conjectural stable approximation result}

\hypertarget{Section 6.2}{We} summarize below the speculations expressed in \cite[Sec.\ 7.4]{parolin-huybrechs-moiola}, which are suggested by the two-dimensional numerical experiments in \cite[Sec.\ 8]{parolin-huybrechs-moiola} and the ones given in the next chapter.
\begin{conjecture}
\hypertarget{Conjecture 6.1}{The} sequence of approximation sets $\boldsymbol{\Psi}_L$ defined in \textup{(\ref{evaluation sets})} is a stable approximation for $\mathcal{A}_L$, uniformly with respect to the truncation parameter $L$. Namely, there exist $\lambda^* \geq 0$ and $C^*\geq 0$ independent of $L$ such that $\forall L\geq 0$, $\exists P^*=P^*(L,\eta,\lambda^*,C^*) \in \N$ such that $\forall v_L \in \mathcal{A}_L$, $\exists \boldsymbol{\mu} \in \C^{P^*}$ such that
\begin{equation}
\|v_L-\mathcal{T}_{\boldsymbol{\Psi}_{L,P^*}}\boldsymbol{\mu}\|_{\mathcal{A}}\leq \eta \|v_L\|_{\mathcal{A}}\,\,\,\,\,\,\,\,\,\,\,\,\,\,\text{and}\,\,\,\,\,\,\,\,\,\,\,\,\,\,\|\boldsymbol{\mu}\|_{\ell^2}\leq C^*{P^*}^{\lambda^*}\|v_L\|_{\mathcal{A}}.
\label{inequalities}
\end{equation}
\end{conjecture}
For simplicity, we will assume in the following that any $P \geq P^*$ satisfies the two inequalities in (\ref{inequalities}). If this is not the case, the proofs can be easily adapted. Although this assumption is true when the sets are hierarchical, this is not a requirement.

If the conjecture stated earlier is valid, the stability of the approximation sets of evanescent plane waves (\ref{evanescence sets}) follows.

\begin{proposition}
Let $\delta > 0$. If Conjecture \hyperlink{Conjecture 6.1}{\textup{6.2}} holds, then the sequence of approximation sets \textup{(\ref{evanescence sets})} provides a stable approximation for $\mathcal{B}$. Moreover, assume to have a set of sampling points $\{\mathbf{x}_s\}_{s=1}^S \subset \partial B_1$ together with a positive weight vector $\mathbf{w}_S \in \R^S$ such that \textup{(\ref{Riemann sum})} is satisfied. If $\kappa^2$ is not a Dirichlet eigenvalue on $B_1$, then $\forall u \in \mathcal{B}\cap C^0(\overline{B_1})$, $\exists L \geq 0$, $P \in \N$, $S \in \N$ and $\epsilon \in (0,1]$ such that
\begin{equation*}
\|u-\mathcal{T}_{\boldsymbol{\Phi}_{L,P}}\boldsymbol{\xi}_{S,\epsilon}\|_{L^2(B_1)}\leq \delta \|u\|_{\mathcal{B}},
\end{equation*}
where $\boldsymbol{\xi}_{S,\epsilon} \in \C^P$ is computed with the regularization procedure in \textup{(\ref{xi Se solution})}. The SVD regularization parameter $\epsilon$ can be chosen as \textup{(\ref{epsilon0})}.
\end{proposition}
\begin{proof}
We need to establish the stability of the sequence of approximation sets, namely that for any $\tilde{\eta}>0$, there exists $\lambda \geq 0$ and $C \geq 0$ such that $\forall u \in \mathcal{B}\cap C^0(\overline{B_1})$, $\exists L \geq 0$, $P \in \N$ and $\boldsymbol{\mu} \in \C^P$ such that
\begin{equation}
\|u-\mathcal{T}_{\boldsymbol{\Phi}_{L,P}}\boldsymbol{\mu}\|_{\mathcal{B}}\leq \tilde{\eta} \|u\|_{\mathcal{B}}\,\,\,\,\,\,\,\,\,\,\,\,\,\,\text{and}\,\,\,\,\,\,\,\,\,\,\,\,\,\,\|\boldsymbol{\mu}\|_{\ell^2}\leq CP^{\lambda}\|u\|_{\mathcal{B}}.
\label{thesis proposition}
\end{equation}
Given that this holds, the stated result is a direct consequence of Corollary \hyperlink{Corollary 2.1}{2.5}.

Let $\eta >0$ and $u \in \mathcal{B}\cap C^0(\overline{B_1})$. For any $L \geq L^*(u,\eta)$ with $L^*$ defined in (\ref{L star}), if we let $u_L:=\Pi_L u$ we have
\begin{equation*}
\|u-u_L\|_{\mathcal{B}}\leq \eta \|u\|_{\mathcal{B}}.
\end{equation*}
Set $v:=T^{-1}u \in \mathcal{A}$ and $v_L:=T^{-1}u_L$.
Assuming the validity of Conjecture \hyperlink{Conjecture 6.1}{6.2}, there exist $\lambda^*\geq0$ and $C^*\geq0$, both independent of $L$, and $P^*=P^*(L,\eta,\lambda^*,C^*) \in \N$ such that, for any $P \geq P^*$, there exists a set of coefficients $\boldsymbol{\mu} \in \C^P$ such that the inequalities (\ref{inequalities}) hold.
Furthermore, thanks to (\ref{T operator bounds}) and (\ref{inequalities}), we have that
\begin{equation*}
\|u_L-\mathcal{T}_{\boldsymbol{\Phi}_{L,P}}\boldsymbol{\mu}\|_{\mathcal{B}} \leq \tau_{+} \eta \|v_L\|_{\mathcal{A}}\,\,\,\,\,\,\,\,\,\,\,\,\,\,\text{and}\,\,\,\,\,\,\,\,\,\,\,\,\,\,\|v_L\|_{\mathcal{A}} \leq \tau_{-}^{-1}\|u_L\|_{\mathcal{B}}\leq \tau_{-}^{-1}\|u\|_{\mathcal{B}}.
\end{equation*}
For any $L \geq L^*(u,\eta)$ and $P \geq P^*$, the total approximation error for the Helmholtz solution $u$ can be estimated, combining the previous bounds, as
\begin{equation}
\begin{split} \label{conclusion}
\|u-\mathcal{T}_{\boldsymbol{\Phi}_{L,P}}\boldsymbol{\mu}\|_{\mathcal{B}} &\leq \|u-u_L\|_{\mathcal{B}}+\|u_L-\mathcal{T}_{\boldsymbol{\Phi}_{L,P}}\boldsymbol{\mu}\|_{\mathcal{B}} \leq \left(1+\tau_{+}\tau_{-}^{-1} \right)\eta\|u\|_{\mathcal{B}}\\
&\text{and}\,\,\,\,\,\,\,\,\,\,\,\,\,\,\,\,\,\,\,\,\|\boldsymbol{\mu}\|_{\ell^2}\leq C^*P^{\lambda^*}\tau_{-}^{-1}\|u\|_{\mathcal{B}}.
\end{split}
\end{equation}
We conclude by choosing $\eta=\tilde{\eta}/\left(1+\tau_{+}\tau_{-}^{-1} \right)$ and noting that (\ref{conclusion}) is (\ref{thesis proposition}) with $\lambda=\lambda^*$ and $C=C^*\tau_{-}^{-1}$.
\end{proof}

The independence of the stability exponent $\lambda^*$ and the stability constant $C^*$ in Conjecture \hyperlink{Conjecture 6.1}{6.2} from the truncation parameter $L$ is crucial for the previous proof.
Without this uniform stability, Conjecture \hyperlink{Conjecture 6.1}{6.2} would not be strong enough to establish (\ref{thesis proposition}).

\section{Probability densities and samples}

\hypertarget{Section 6.3}{In} this section we describe the numerical recipe outlined in Section \hyperlink{Section 6.1}{6.1}.
In the left column of Figure \ref{figure 6.1}, we depict the probability density functions
\begin{equation}
\hat{\rho}_{N}(\zeta):=\int_{\Theta}\rho_{N}(\boldsymbol{\theta},\zeta)\,\textup{d}\boldsymbol{\theta},\,\,\,\,\,\,\,\,\,\,\,\,\,\,\forall \zeta \in [0,+\infty),
\label{rho zeta density}
\end{equation}
with respect to the ratio $\zeta/\kappa$, where $\rho_{N(L)}$ is defined in (\ref{rho density}).
The variable $L$ represents the truncation parameter, indicating that the sampling is done to approximate elements of $\mathcal{A}_L$, which has dimension $N(L)$.
It is worth noting that, although $\rho_{N}$ is a four-variables function in the parametric domain $Y$, it is sinusoidal with respect to $\theta_1$ (see (\ref{weight})) and constant with respect to the other Euler angles $\theta_2$ and $\theta_3$.

The main mode of the probability densities $\hat{\rho}_N$ is seen at $\zeta=0$, which represents purely propagative plane waves. As the wavenumber increases, both the peak in $\zeta=0$ and the density support get larger. This latter feature is different with respect to the two-dimensional case \cite{parolin-huybrechs-moiola}, due to the $\kappa$-dependent parametrization of the evanescence parameter $\zeta$ in (\ref{complex direction}); this is the reason why $\zeta$ is scaled by the wavenumber $\kappa$ in Figure \ref{figure 6.1}.
Eventually, the probability tends to zero exponentially as $\zeta$ gets large enough.
If $L \leq \kappa$, the densities are unimodal distributions, whereas, for $L \gg \kappa$, they are multimodal: in fact, for instance when $L=4\kappa$, we observe an additional mode for relatively large values of the evanescence parameter (roughly for $\zeta=5\kappa$), besides the main one at $\zeta=0$.

In analogy with \cite[Sec.\ 8.1]{parolin-huybrechs-moiola}, for any $L\geq 0$, one possible strategy is to generate $P=\nu N(L)$ samples in the Cartesian product $Y$ using the \textit{Inversion Transform Sampling} (ITS) technique suggested by \cite[Sec. 5.2]{Cohen_Migliorati}.
The process involves generating sampling sets in $[0,1]^4$ that converge (in a suitable sense) to the uniform distribution $\mathcal{U}_{[0,1]^4}$ when $P$ goes to infinity,
\begin{equation}
\{\mathbf{z}_p\}_{p=1}^P,\,\,\,\,\,\,\,\text{with}\,\,\,\,\,\,\,\mathbf{z}_p=(z_{p,\theta_1},z_{p,\theta_2},z_{p,\theta_3},z_{p,\zeta}) \in [0,1]^4,\,\,\,\,\,\,\,p=1,...,P,
\label{points in [0,1]}
\end{equation}
and then map them back to the parametric domain $Y$, to obtain sampling sets that converge to the probability density function $\rho_N$ as $P \rightarrow \infty$.
Hence, we need to invert each of the cumulative density functions associated with each of the parameters in $Y$. Given the nature of the dependence of $\rho_N$ on the angular variables, this result can be easily computed explicitly for every $\boldsymbol{\theta}=(\theta_1,\theta_2,\theta_3) \in \Theta$.
However, the numerical evaluation of the cumulative probability distribution related to the evanescence parameter $\zeta$, namely
\begin{equation}
\Upsilon_{N}(\zeta):=\int_{0}^{\zeta}\hat{\rho}_{N}(\eta)\,\textup{d}\eta,\,\,\,\,\,\,\,\,\,\,\,\,\,\,\forall \zeta \in [0,+\infty),
\label{zeta cumulative}
\end{equation}
is a bit difficult to implement, costly to run and numerically unstable. In fact, due to (\ref{vector P}), (\ref{weight}), (\ref{1 lemma 5.3}), (\ref{muN independence}), and definition (\ref{rho zeta density}), we should compute:
\begin{align*}
\Upsilon_{N}(\zeta)&=\int_{0}^{\zeta}\int_{\Theta}\rho_{N}(\boldsymbol{\theta},\eta)\,\textup{d}\boldsymbol{\theta}\textup{d}\eta=\int_{0}^{\zeta}\int_{\Theta}\frac{\eta^{1/2}e^{-\eta}\sin \theta_1}{N \mu_N(\boldsymbol{\theta},\eta)}\,\textup{d}\boldsymbol{\theta}\textup{d}\eta\\
&=\left(\int_{\Theta}\sin \theta_1\,\textup{d}\boldsymbol{\theta}\right)\left(\frac{1}{N}\sum_{\ell=0}^L\alpha_{\ell}^2\int_{0}^{\zeta}\left|\mathbf{P}_{\ell}(\eta)\right|^2\eta^{1/2}e^{-\eta}\,\textup{d}\eta\right)\\
&=\frac{1}{N}\sum_{\ell=0}^L(2\ell+1)\frac{\sum_{m=-\ell}^{\ell}\int_0^{\zeta}\left[\gamma_{\ell}^m P_{\ell}^m(\eta/2\kappa+1)\right]^2\eta^{1/2}e^{-\eta}\,\textup{d}\eta}{\sum_{m=-\ell}^{\ell}\int_0^{\infty}\left[\gamma_{\ell}^m P_{\ell}^m(\eta/2\kappa+1)\right]^2\eta^{1/2}e^{-\eta}\,\textup{d}\eta}. \numberthis \label{upsilon}
\end{align*}
Following the suggestion of Remark \hyperlink{Remark 5.3}{5.3} and the asymptotics \cite[Eq.\ (14.8.12)]{nist}, we propose to rely on the approximation
\begin{align*}
\int_0^{\zeta}\!\Big[\gamma_{\ell}^m P_{\ell}^m\!&\left(\frac{\eta}{2\kappa}+1\right)\Big]^2\!\eta^{1/2}e^{-\eta}\,\textup{d}\eta \approx \left(\frac{2^{\ell}\gamma_{\ell}^{m}\Gamma(\ell+1/2)}{\sqrt{\pi}(\ell-m)!}\right)^2\!\!\int_{0}^{\zeta}\left(\frac{\eta}{2\kappa}+1\right)^{2\ell}\!\!\!\eta^{1/2}e^{-\eta}\textup{d}\eta\\
&= \left(\frac{2^{\ell}\gamma_{\ell}^{m}\Gamma(\ell+1/2)}{\sqrt{\pi}(\ell-m)!}\right)^2\int_{2\kappa}^{2\kappa+\zeta}\left(\frac{\eta}{2\kappa}\right)^{2\ell}(\eta-2\kappa)^{1/2}e^{-(\eta-2\kappa)}\textup{d}\eta\\
&\leq \frac{e^{2\kappa}}{4\pi^2\kappa^{2\ell}} \frac{(2\ell+1) \Gamma^{\,2}(\ell+1/2)}{(l+m)!(l-m)!}\int_{2\kappa}^{2\kappa+\zeta}\eta^{2\ell+1/2}e^{-\eta}\textup{d}\eta\\
&= C_{\ell}^m \int_{2\kappa}^{2\kappa+\zeta}\eta^{2\ell+1/2}e^{-\eta}\textup{d}\eta, \numberthis \label{first approximation}
\end{align*}
where the constant $C_{\ell}^m$ was defined in (\ref{Clm constant}).
\begin{figure}
\centering
\begin{subfigure}{0.98\linewidth}
\includegraphics[width=0.97\linewidth]{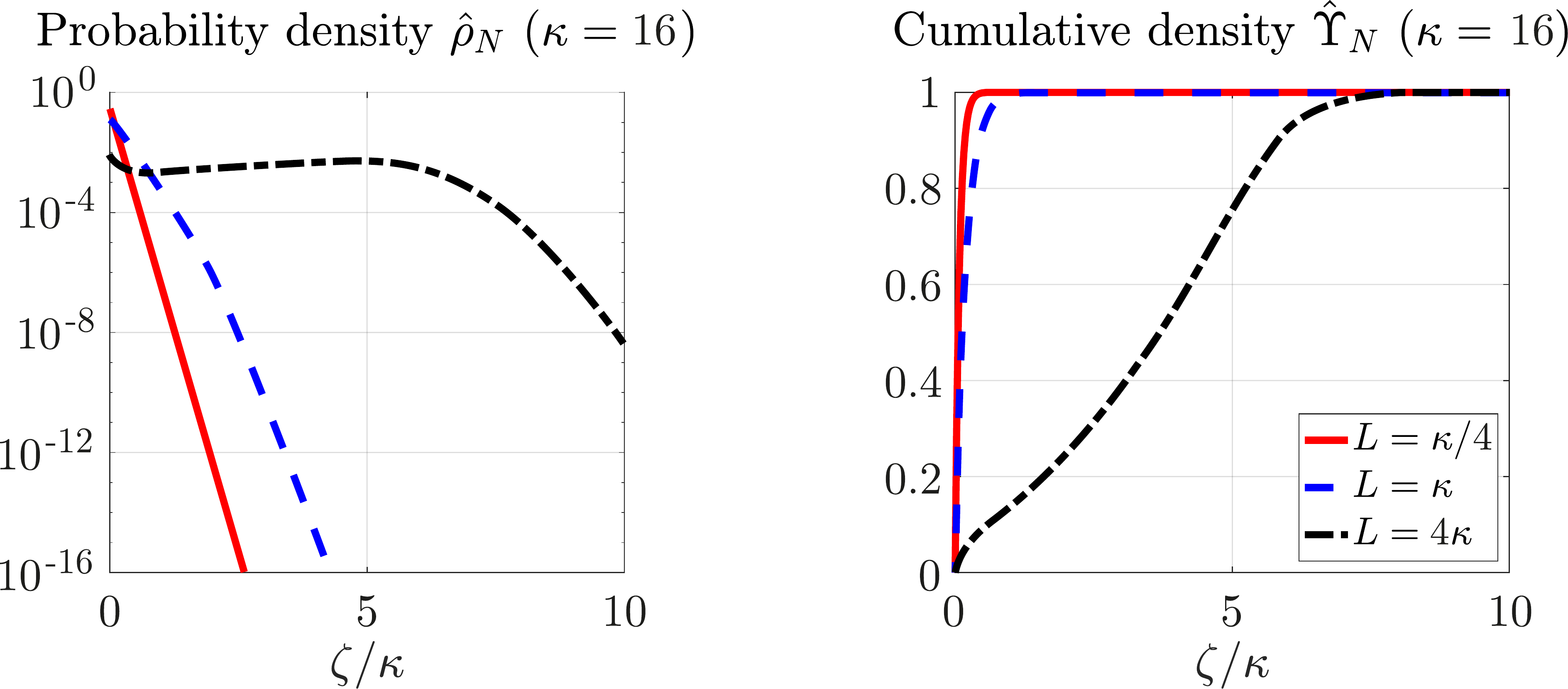}
\vspace{4mm}
\end{subfigure}
\begin{subfigure}{0.98\linewidth}
\includegraphics[width=0.97\linewidth]{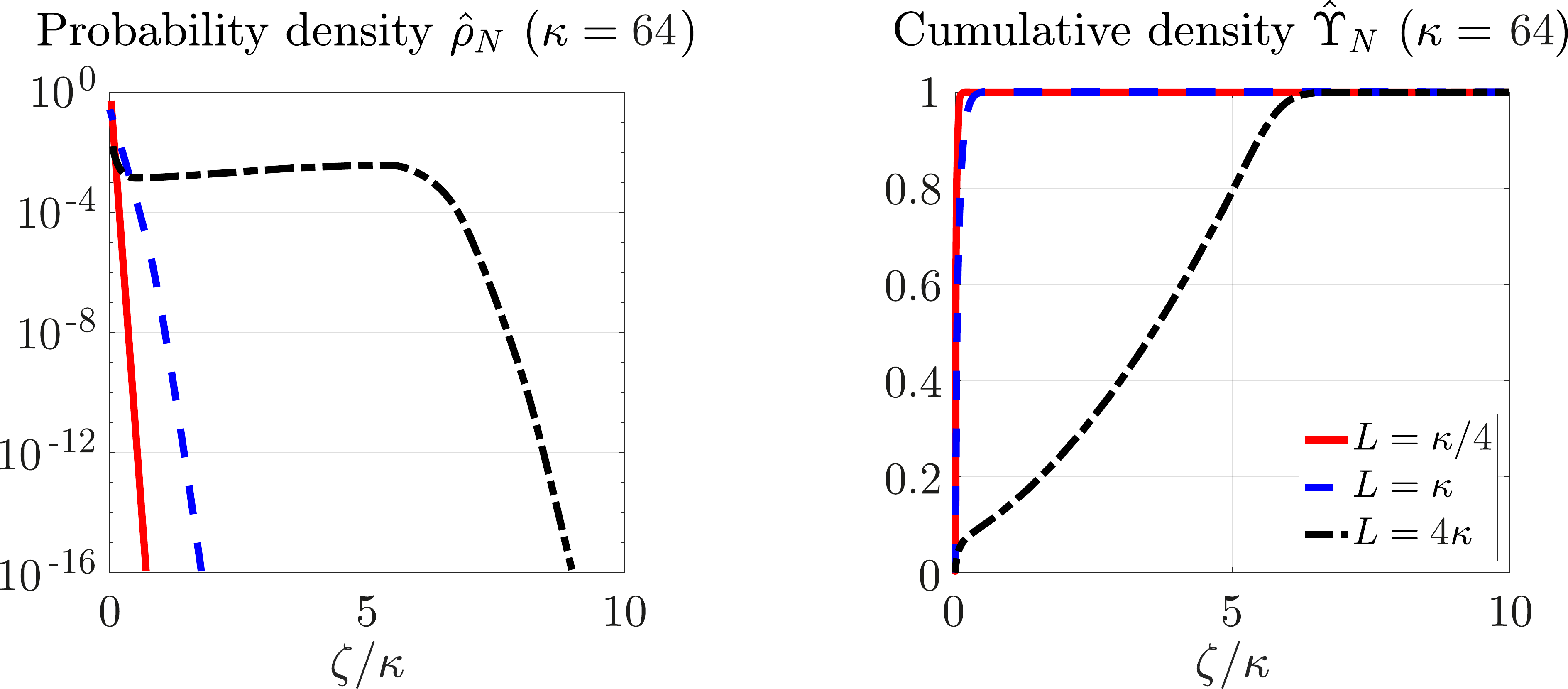}
\vspace{4mm}
\end{subfigure}
\begin{subfigure}{0.98\linewidth}
\includegraphics[width=0.97\linewidth]{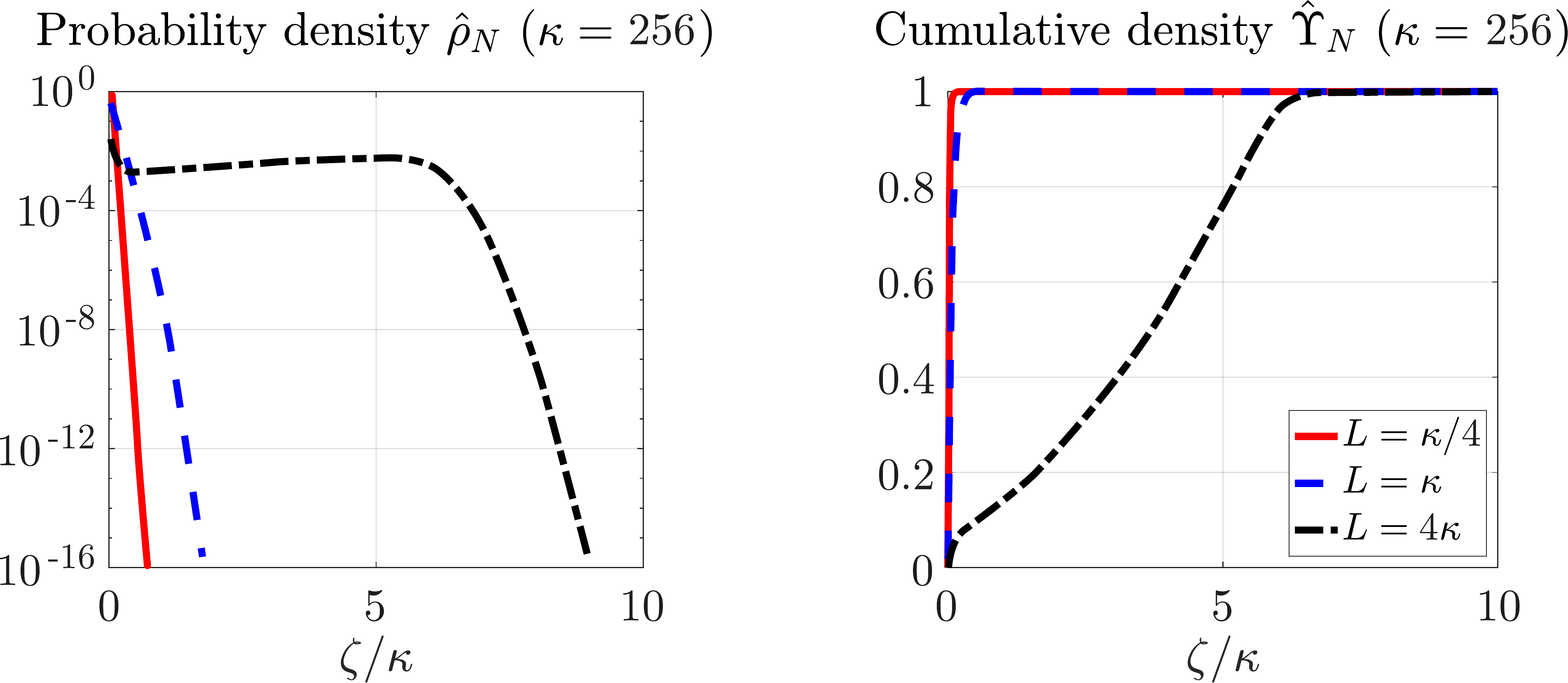}
\end{subfigure}
\caption{Sampling density functions $\hat{\rho}_N$ in (\ref{rho zeta density}) (left) and $\hat{\Upsilon}_N$ in (\ref{approx zeta cumulative}) (right) with respect to the $\kappa$-scaled evanescence parameter $\zeta$ constructed for the subspace $\mathcal{A}_L$. The wavenumber $\kappa$ varies in $\{16, 64, 256\}$ from top to bottom.}
\label{figure 6.1}
\end{figure}
Note that we proceeded by analogy with (\ref{2 lemma 5.3}), approximating the associated Legendre polynomials with a single monomial of degree $\ell$ and replacing the factor $(\eta-2\kappa)^{1/2}$ by $\eta^{1/2}$ in the integrand. Hence, we approximate both the numerator and the denominator in (\ref{upsilon}) with:
\begin{align*}
\int_0^{\zeta}\Big[\gamma_{\ell}^m P_{\ell}^m\!\left(\frac{\eta}{2\kappa}+1\right)\Big]^2\eta^{1/2}e^{-\eta}\,\textup{d}\eta &\approx C_{\ell}^m \left[ \Gamma\left(2\ell+\frac{3}{2},2\kappa\right)-\Gamma\left(2\ell+\frac{3}{2},2\kappa+ \zeta \right) \right],\\
\int_0^{\infty}\left[\gamma_{\ell}^m P_{\ell}^m\left(\frac{\eta}{2\kappa}+1\right)\right]^2\eta^{1/2}e^{-\eta}\textup{d}\eta &\approx C_{\ell}^m\,\Gamma\left(2\ell+\frac{3}{2},2\kappa\right), \numberthis \label{useful approximation}
\end{align*}
where we introduced the \textit{upper incomplete Gamma function} defined in \cite[Eq. (8.2.2)]{nist}.
Observe that this choice is a sort of interpolation between the lower and the upper bounds in (\ref{3 lemma 5.3}) and (\ref{2 lemma 5.3}) respectively.
Since the constant factors trivially simplify, due to this heuristic we get, for every $\zeta \in [0,+\infty)$,
\begin{align*}
\Upsilon_N(\zeta)&\approx \frac{1}{N}\sum_{\ell=0}^L(2\ell+1)\frac{\Gamma\left(2\ell+3/2,2\kappa\right)-\Gamma\left(2\ell+3/2,2\kappa+ \zeta \right)}{\Gamma\left(2\ell+3/2,2\kappa\right)}\\
&=1-\frac{1}{N}\sum_{\ell=0}^L(2\ell+1)\frac{\Gamma\left(2\ell+3/2,2\kappa+\zeta \right)}{\Gamma\left(2\ell+3/2,2\kappa\right)}. \numberthis \label{passo prima}
\end{align*}
Therefore, in order to approximate $\Upsilon_N$ in (\ref{zeta cumulative}), starting from (\ref{passo prima}), we define the cumulative distribution function
\begin{equation}
\hat{\Upsilon}_N(\zeta):=1-\frac{1}{N}\sum_{\ell=0}^L(2\ell+1)\frac{Q\left(2\ell+3/2,2\kappa+\zeta \right)}{Q\left(2\ell+3/2,2\kappa\right)},\,\,\,\,\,\,\,\,\,\,\,\,\,\,\forall \zeta \in [0,+\infty),
\label{approx zeta cumulative}
\end{equation}
where $Q$ is the \textit{normalized upper incomplete Gamma function} defined in \cite[Eq.\ (8.2.4)]{nist}. Introducing this function is necessary for the cumulative density function $\hat{\Upsilon}_N$ to be computed stably and without overflow issues. The function $\hat{\Upsilon}_N$ retains the following properties, which are crucial for our sampling purposes:
\begin{equation*}
    0 \leq \hat{\Upsilon}_{N}(\zeta) \leq 1,
    \qquad\,\,\,\,\,
    \hat{\Upsilon}_{N}(0) = 0,
    \qquad\,\,\,\,\,
    \lim_{\zeta \to \infty}\hat{\Upsilon}_{N}(\zeta) = 1.
\end{equation*}
The benefits of this very simple expression are quite clear compared to what needs to be computed otherwise in (\ref{upsilon}).
Some cumulative density functions $\hat{\Upsilon}_N$ are represented in the right column of Figure \ref{figure 6.1}.
When $\mathcal{A}_L$ only consists of elements related to the propagative regime ($L \leq \kappa$), the cumulative distributions $\hat{\Upsilon}_N$ are nearly step functions, especially for large wavenumbers. However, for $L > \kappa$, these functions are more complex (e.g.\ see the cases where $L=4\kappa$). Thus, for $L \leq \kappa$, it is safe to only choose propagative plane waves, as stated in Section \hyperlink{Section 3.4}{3.5}, but for $L > \kappa$, the selection of evanescent waves becomes a non-trivial task.

Therefore, after we generated the sampling sets (\ref{points in [0,1]}) in $[0,1]^4$, we map them to the parametric domain $Y$ obtaining:
\begin{equation}
\{\mathbf{y}_p\}_{p=1}^P,\,\,\,\,\,\,\,\text{with}\,\,\,\,\,\,\,\mathbf{y}_p=(\arccos{(1-2z_{p,\theta_1})},2\pi z_{p,\theta_2},2\pi z_{p,\theta_3},\hat{\Upsilon}_N^{-1}(z_{p,\zeta})) \in Y.
\label{points back in Y}
\end{equation}
The fact that the density function $\rho_N$ is sinusoidal in $\theta_1$ and constant in $\theta_2$ and $\theta_3$ simplifies the sample generation process, eliminating the need for the techniques in \cite[Sec. 5]{Cohen_Migliorati} that use tensor-product orthonormal bases. The inversion $\hat{\Upsilon}_{N}^{-1}$ can be computed using basic root-finding techniques. In our numerical experiments we rely on the bisection method, which is straightforward and reliable.

As anticipated previously, another possible strategy consists in using the extremal points spherical coordinates to replace the first two components of the sampling points in (\ref{points back in Y}). Thus, we need to generate the sampling point sets as in (\ref{points in [0,1]}) only in $[0,1]^2$ and then map them back to the evanescence domain $[0,2\pi)\times[0,+\infty)$.
Observe that using this approach requires $P \in \N$ to be a perfect square once again.

Once the samples $\{\mathbf{y}_p\}_{p=1}^P$ in the Cartesian product $Y$ have been generated, our next step is to construct the evanescent plane wave set (\ref{evanescence sets}). This process involves computing the $N$-term Christoffel function $\mu_N$, which, according to (\ref{muN independence}), depends on both the normalization coefficients $\alpha_{\ell}$ in (\ref{a tilde definizione}) and $\mathbf{P}_{\ell}(\zeta)$ in (\ref{vector P}). Although the latter can be simply derived through some recurrence formulae (see \cite[Eqs. (14.7.15) and (14.10.3)]{nist}), the former brings along some numerical difficulties due to the integral in (\ref{1 lemma 5.3}). However, the use of the approximation presented in (\ref{useful approximation}) can help to overcome this issue.
From (\ref{1 lemma 5.3}) and (\ref{Clm constant}), it follows:
\begin{align*}
\alpha_{\ell}^{-2}&=\frac{8\pi^2}{2\ell+1}\sum_{m=-\ell}^{\ell}\int_0^{\infty}\left[\gamma_{\ell}^m P_{\ell}^m\left(\frac{\eta}{2\kappa}+1 \right) \right]^2 \eta^{1/2}e^{-\eta}\,\textup{d}\eta\\
&\approx \frac{8\pi^2}{2\ell+1}\sum_{m=-\ell}^{\ell} \frac{e^{2\kappa}}{4\pi^2\kappa^{2\ell}} \frac{(2\ell+1) \Gamma^{\,2}(\ell+1/2)}{(l+m)!(l-m)!}\,\Gamma\left(2\ell+\frac{3}{2},2\kappa \right)\\
&=\frac{8\pi^2}{(2\ell+1)} \frac{e^{2\kappa}(2\ell+1)}{4\pi^2\kappa^{2\ell}}\frac{2^{2\ell}}{(2\ell)!}\frac{\sqrt{\pi}(2\ell)!}{2^{2\ell}\ell!}\,\Gamma\left(\ell+\frac{1}{2}\right)\Gamma\left(2\ell+\frac{3}{2},2\kappa \right)\\
&=\frac{2\sqrt{\pi}e^{2\kappa}}{\ell!\kappa^{2\ell}}\Gamma\left(\ell+\frac{1}{2}\right)\Gamma\left(2\ell+\frac{3}{2},2\kappa \right), \numberthis \label{alphal approximation}
\end{align*}
where we used (\ref{gamma half}) and
\begin{equation*}
\sum_{m=-\ell}^{\ell}\frac{1}{(\ell+m)!(\ell-m)!}=\frac{1}{(2\ell)!}\sum_{m=-\ell}^{\ell}\binom{2\ell}{\ell+m}=\frac{1}{(2\ell)!}\sum_{m=0}^{2\ell}\binom{2\ell}{m}=\frac{2^{2\ell}}{(2\ell)!}.
\end{equation*}
In our numerical experiments, we will adopt the approximations presented in (\ref{approx zeta cumulative}) and (\ref{alphal approximation}).

In Chapter \hyperlink{Chapter 7}{7} we test five methods of sampling. These strategies differ both in whether they incorporate extremal systems and in the way the initial sampling distribution is generated. More specifically, in the first three strategies, all the coordinates of the nodes in $Y$ are sampled according to the probability distribution (\ref{rho density}) and thus the initial samples are generated in $[0,1]^4$. Otherwise, in the last two strategies, the extremal points coordinates are involved in order to define the first two component of the sampling points in $Y$. In the latter case, the initial samples are generated in $[0,1]^2$.

\begin{definition}[\hypertarget{Definition 6.4}{Sampling strategies}]
We will consider the following strategies:
\begin{enumerate}[\normalfont(a)]
\item \textup{Deterministic sampling}: the initial samples in $[0,1]^4$ are a Cartesian product of four sets of equispaced points with equal number of points in every directions. The numerical results presented use the smallest $4$th-power integer greater than or equal to $P$ as the approximation set dimension.
\item \textup{Random sampling}: the initial samples in $[0,1]^4$ are generated randomly according to the product of four uniform distributions $\mathcal{U}_{[0,1]}$.
\item \textup{Sobol sampling}: the initial samples in $[0,1]^4$ correspond to Sobol sequences which are quasi-random low-discrepancy sequences.
\item \textup{Extremal--Random sampling}: the initial samples in $[0,1]^2$ are generated randomly according to the product of two uniform distributions $\mathcal{U}_{[0,1]}$. The numerical results presented use the smallest square integer greater than or equal to $P$ as the approximation set dimension.
\item \textup{Extremal--Sobol sampling}: the initial samples in $[0,1]^2$ correspond to Sobol sequences. The numerical results presented use the smallest square integer greater than or equal to $P$ as the approximation set dimension.
\end{enumerate}
\end{definition}

\begin{remark}
\hypertarget{Remark 6.5}{To} avoid overloading the notation while ensuring generality in the discussion, sometimes we will improperly use $P$ to denote the dimension of the evanescent plane wave approximation set $\boldsymbol{\Phi}_{L,P}$, even though it may not accurately reflect its actual size, which depends on the sampling strategy employed.
\end{remark}

Some examples of node sets resulting from the previous sampling strategies are depicted in Figure \ref{figure 6.2} and Figure \ref{figure 6.3}.
We only report the components related to the parameters $\theta_1 \in [0,\pi]$ and $\zeta \in [0,+\infty)$, since the probability density function $\rho_N$ is constant with respect to the other Euler angles $\theta_2$ and $\theta_3$.
As anticipated, for smaller values of $L$, the sampling points cluster near the line $\zeta=0$, which is the regime where propagative plane waves provide a sufficient approximation. However, it should be noted that there are no purely propagative plane waves at $\zeta=0$, as $\rho_N$ is a continuous distribution. For $L > \kappa$, the evanescence parameter $\zeta$ covers a wider range, with some concentration at the secondary peak of the distribution (roughly around $\zeta=5\kappa$), which aligns with the results of Figure \ref{figure 6.1}.

\begin{remark}
As we shall see, the previous sampling strategies lead to very similar results. In particular, these shows that employing extremal systems of points \textup{(\ref{G matrix})}, as within the sampling procedures \textup{(d)} and \textup{(e)} in Definition \textup{\hyperlink{Definition 6.4}{6.4}}, offers no further improvements.
It is worth noting that the problem of choosing the parameters $\theta_1 \in [0,\pi]$ and $\theta_2 \in [0,2\pi)$, here aimed at building the evanescent plane wave approximation sets \textup{(\ref{evanescence sets})}, is common not only to the construction of propagative plane wave direction sets \textup{(\ref{plane waves approximation set})}, but also to the definition of the sampling points $\{\mathbf{x}_s\}_{s=1}^S \subset \partial B_1$, which are used in the linear system \textup{(\ref{linear system})}.
This suggests that it is not necessary to exploit the geometric properties of extremal points, but rather that it is enough to rely on a weighted sampling strategy, be it deterministic, random or even quasi-random.
This allows us to achieve similar results with less effort, due to being able to dispense with extremal systems, and therefore with fewer constraints, given that the number of points no longer necessarily has to be a perfect square.
\end{remark}

\begin{figure}
\centering
\begin{subfigure}{0.98\linewidth}
\includegraphics[width=0.9\linewidth]{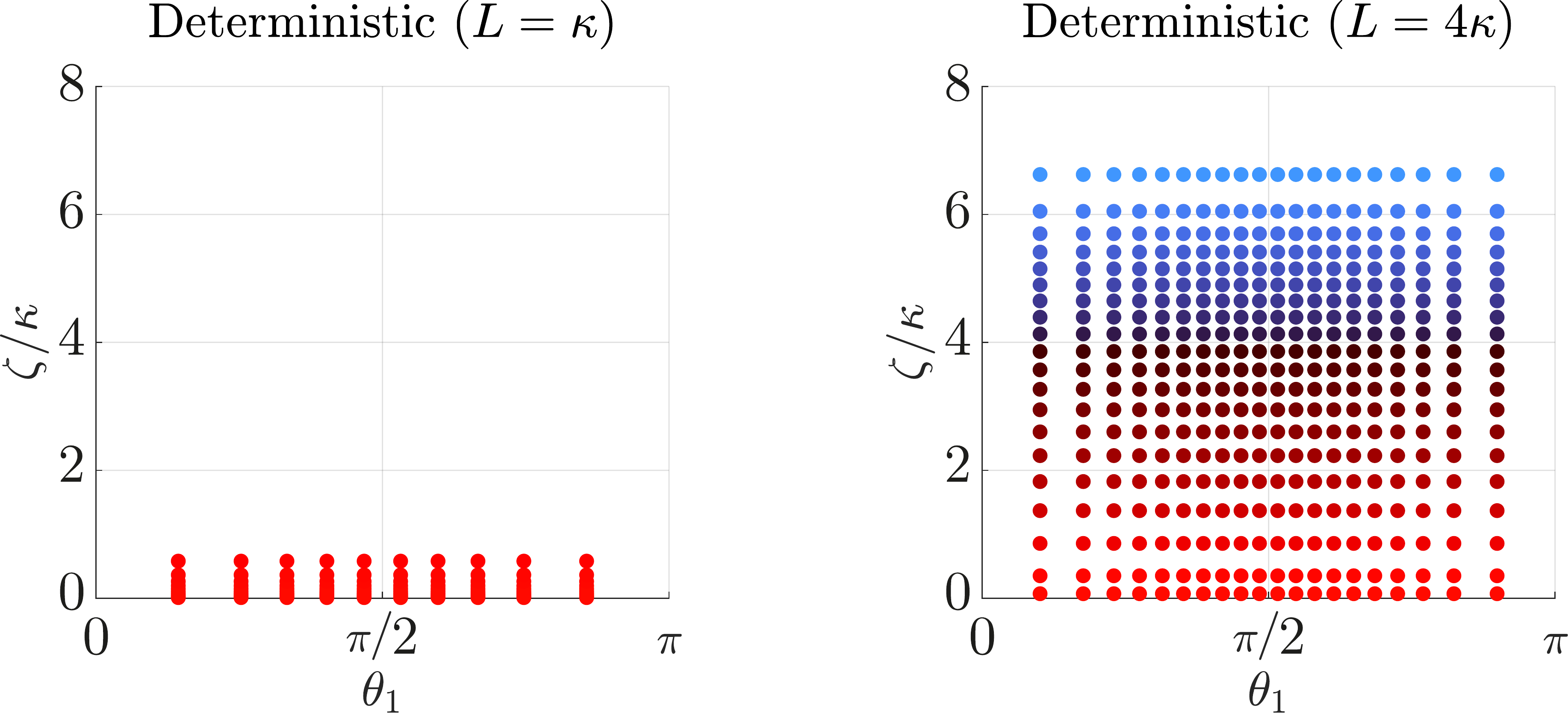}
\vspace{4mm}
\end{subfigure}
\begin{subfigure}{0.98\linewidth}
\includegraphics[width=0.9\linewidth]{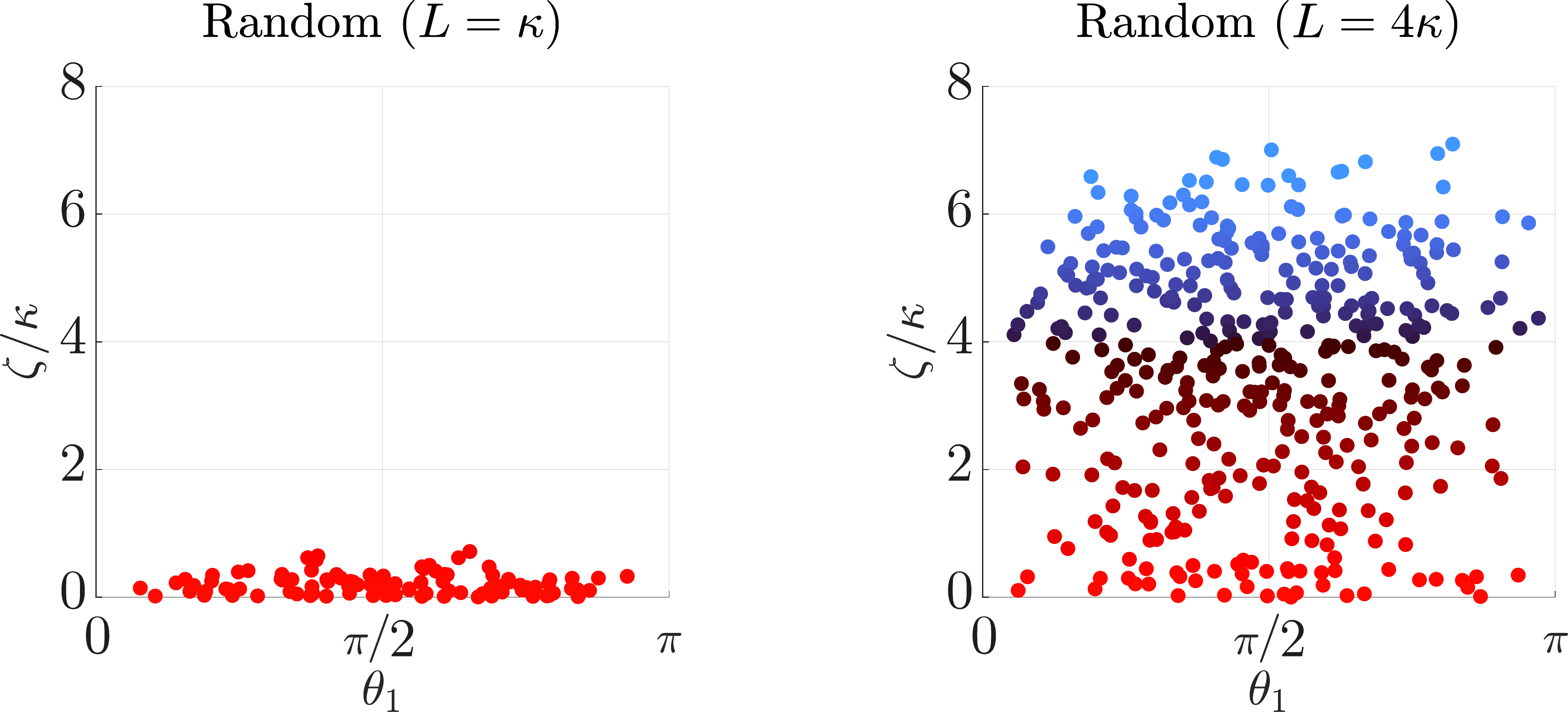}
\vspace{4mm}
\end{subfigure}
\begin{subfigure}{0.98\linewidth}
\includegraphics[width=0.9\linewidth]{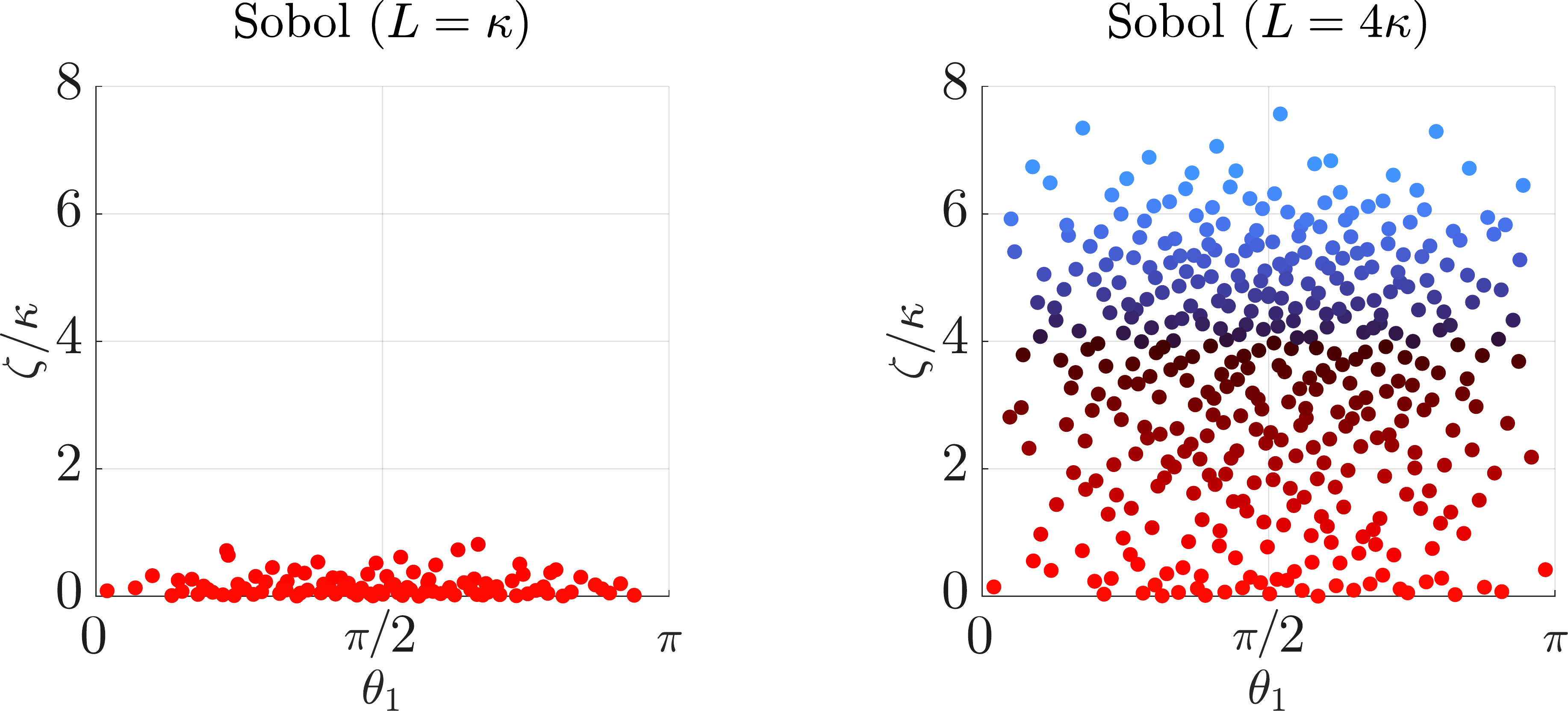}
\end{subfigure}
\caption{$P=6L$ samples in $[0,\pi]\times[0,+\infty)$ for $L$ equal to $\kappa$ (left) and $4\kappa$ (right) and for sampling strategies (a)--(c) presented in Definition \protect\hyperlink{Definition 6.4}{6.4} (from top to bottom). The points are colored according to the square root of $\mu_N$ in (\ref{rho density}), that is computed using the approximation in (\ref{alphal approximation}). Wavenumber $\kappa=16$.}
\label{figure 6.2}
\end{figure}

\begin{figure}
\centering
\begin{subfigure}{0.98\linewidth}
\includegraphics[width=0.9\linewidth]{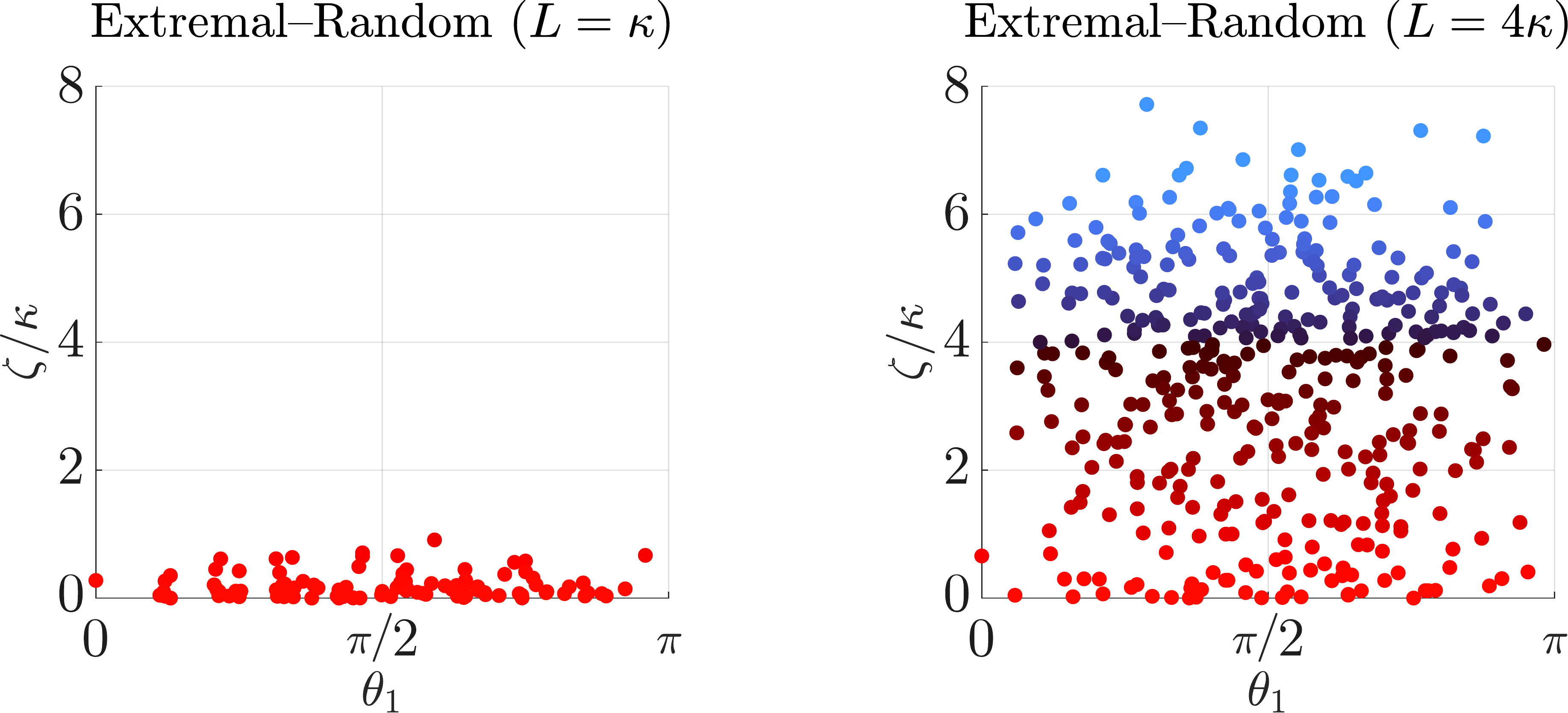}
\vspace{4mm}
\end{subfigure}
\rule{0pt}{4ex} 
\begin{subfigure}{0.98\linewidth}
\includegraphics[width=0.9\linewidth]{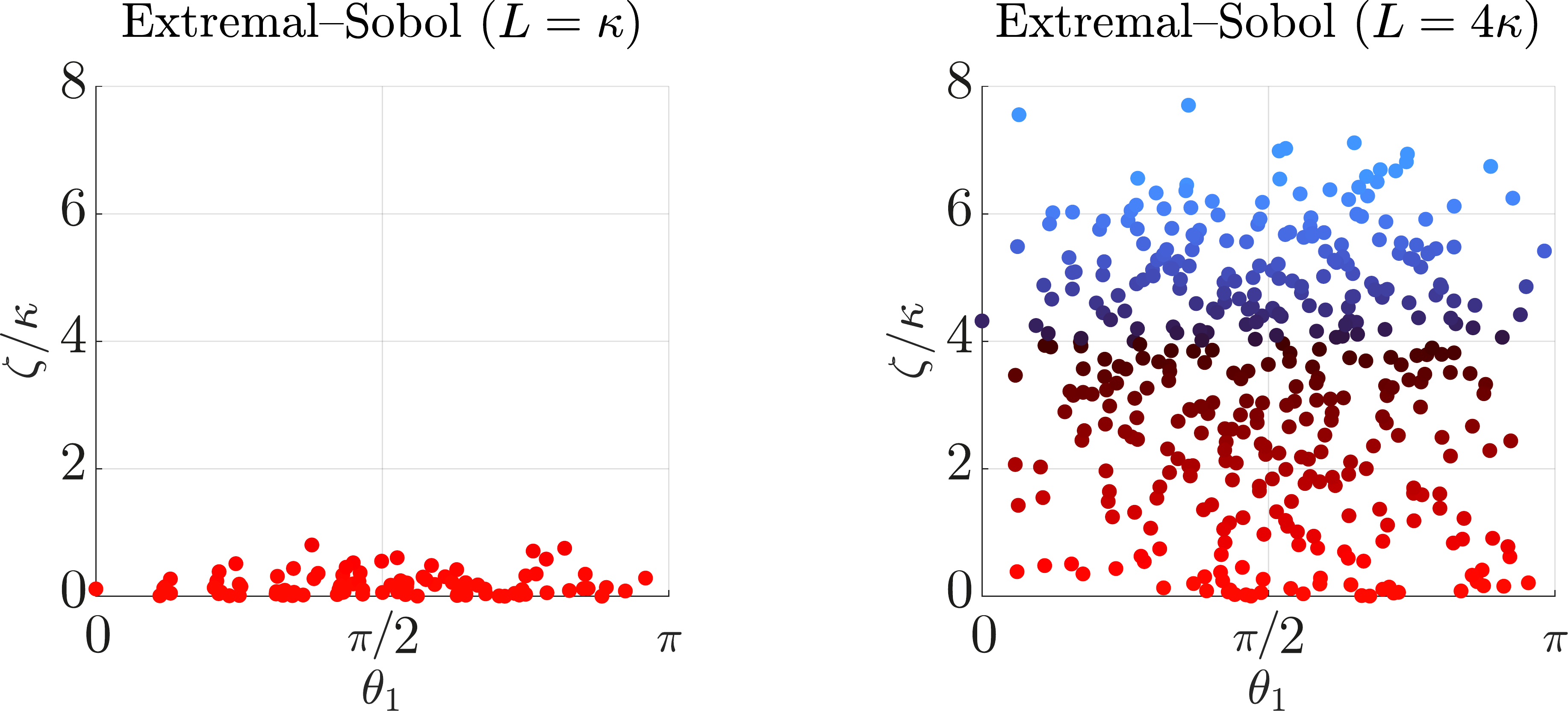}
\end{subfigure}
\caption{$P=\lceil\sqrt{6L}\rceil^2$ nodes in $[0,\pi]\times[0,+\infty)$ for $L$ equal to $\kappa$ (left) and $4\kappa$ (right) and for sampling strategies (d) and (e) presented in Definition \protect\hyperlink{Definition 6.4}{6.4} (from top to bottom). Since the nodes first components are defined thanks to the extremal systems, there is always a point for which $\theta_1=0$: in fact, as highlighted in Section \protect\hyperlink{Section 2.3}{2.3}, regardless of the cardinality of the system, the first extremal point is always fixed at the north pole. The points are colored according to the square root of $\mu_N$ in (\ref{rho density}), that is computed using the approximation in (\ref{alphal approximation}). Wavenumber $\kappa=16$.}
\label{figure 6.3}
\end{figure}

\chapter{Numerical results}
\hypertarget{Chapter 7}{We} present numerical evidence that the described recipe can be used to obtain stable and accurate approximations of Helmholtz solutions inside the unit ball $B_1$.
First, we repeat the numerical experiments of Section \hyperlink{Section 3.4}{3.5} considering approximation sets consisting of evanescent plane waves, built according to the sampling strategies outlined in Definition \hyperlink{Definition 6.4}{6.4}.
Then, we investigate the validity of Conjecture \hyperlink{Conjecture 6.1}{6.2}, by reconstructing some solution surrogates and studying the convergence of the error. Moreover, we analyse the optimal size of the approximation set $P$ in relation to the truncation parameter $L$.
Finally, we conclude this chapter with some numerical results involving the approximation of the fundamental solution of the Helmholtz equation in 3D, both in the unit ball $B_1$ and in different geometries.

\section{Evanescent plane waves stability}

\hypertarget{Section 7.1}{We} consider again the numerical test from Section \hyperlink{Section 3.4}{3.5}, which showed that any approximation using propagative plane waves is unstable. We aim to determine if our proposed method using evanescent plane waves improves stability without sacrificing accuracy.
The context remains unchanged: we calculate approximations of the spherical waves $b_{\ell}^0$ for several degrees $\ell$, since, similarly to Figure \ref{figure triangle}, the accuracy and stability properties of the numerical results do not differ significantly varying the order $|m|\leq \ell$, as shown in Figure \ref{figure 7.1} for the particular case where $P=4L^2=64\kappa^2$.
However, this time we use the approximation sets $\boldsymbol{\Phi}_{L,P}$ defined in (\ref{evanescence sets}), whose evanescent plane waves are characterized by the parameters $\{\mathbf{y}_{L,P,\,p}\}_{p}$, obtained according to the sampling strategies presented in Definition \hyperlink{Definition 6.4}{6.4}.
The normalization of the evanescent waves is described in (\ref{evanescence sets}), where the parameter $L$ is fixed at $4\kappa$.

The numerical results are displayed in Figure \ref{figure 7.2} and Figure \ref{figure 7.4}. On one hand, the left panel shows the relative residual $\mathcal{E}$ defined in (\ref{relative residual}) as an indicator of the approximation's accuracy. On the other hand, the right panel displays the magnitude of the coefficients, $\|\boldsymbol{\xi}_{S,\epsilon}\|_{\ell^2}$, to indicate the stability of the approximation. These results should be compared to those in Figure \ref{figure 3.4}, which were obtained when only propagative plane waves were used in the approximation set defined in (\ref{plane waves approximation set}).
\begin{figure}
\centering
\includegraphics[width=.82\linewidth]{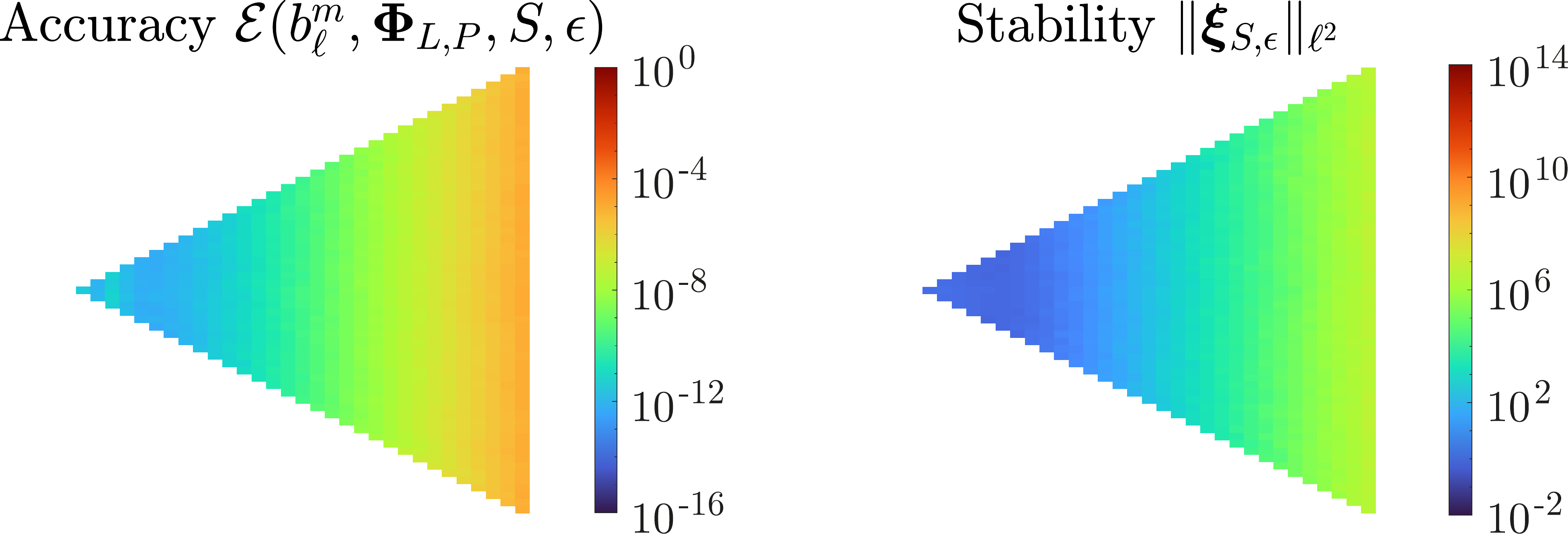}
\caption{Accuracy $\mathcal{E}$ as defined in (\ref{relative residual}) (left) and stability $\|\boldsymbol{\xi}_{S,\epsilon}\|_{\ell^2}$ (right) of the approximation of spherical waves $b_{\ell}^m$ by evanescent plane waves, whose parameters in $Y$ are choosen according to the Sobol sampling (c) in Definition \protect\hyperlink{Definition 6.4}{6.4}. The index $\ell$ varies along the abscissa within the range $0 \leq \ell \leq 5\kappa$, while $m$ varies along the ordinate within the range $0 \leq |m| \leq \ell$ forming a triangle.
Truncation at $L=4\kappa$, DOF budget $P=4L^2$, wavenumber $\kappa=6$ and regularization parameter $\epsilon=10^{-14}$.} \label{figure 7.1}
\end{figure}
The key result is that by using a sufficient number of waves, i.e.\ making $P$ large enough, we can approximate all the modes $\ell \leq L=4\kappa$ to near machine precision. This encompasses the propagative modes $\ell \leq \kappa$, which were already well-approximated by only using propagative plane waves, but, more importantly, also includes the evanescent modes $\kappa < \ell \leq L=4\kappa$ for which purely propagative plane waves provided poor or no approximation.
Additionally, even modes with an higher degree, i.e.\ $L=4\kappa < \ell \leq 5\kappa$, are accurately approximated.
The norms of the coefficients $\|\boldsymbol{\xi}_{S,\epsilon}\|_{\ell^2}$ in the approximate expansions are moderate and this is in stark contrast with the results of Section \hyperlink{Section 3.4}{3.5}.
Moreover, it can be seen that for small values of $P$, such as $P=L^2=16\kappa^2$ and $P=2L^2=32\kappa^2$, purely propagative plane waves provide a better approximation of propagative modes than evanescent plane waves. This is because the approximation spaces made of propagative plane waves are tuned for propagative modes, which span a space of dimension $(\kappa+1)^2$. On the other hand, the approximation spaces created using evanescent plane waves target a larger number of modes, including some evanescent modes, which span a space of dimension $N=(L+1)^2$, where $L=4\kappa$ in this numerical experiment.

Comparing Figure \ref{figure 7.3} and Figure \ref{figure 7.5} with Figure \ref{figure 3.3}, it can be seen that when $P$ is large enough, the condition number of the matrix $A$ is comparable for both propagative and evanescent plane waves. The improved accuracy for evanescent modes is not due to a better conditioning of the linear system, but rather to an increase of the $\epsilon$-rank (i.e.\ the number of singular values greater than $\epsilon$) of the matrix (from less than $10^3$ for propagative plane waves to around $5 \times 10^3$ for evanescent plane waves in the case $P=16L^2=256\kappa^2$). Raising the truncation parameter $L$ allows to increase the $\epsilon$-rank.

\begin{figure}
\centering
\begin{subfigure}{0.98\linewidth}
\includegraphics[width=\linewidth]{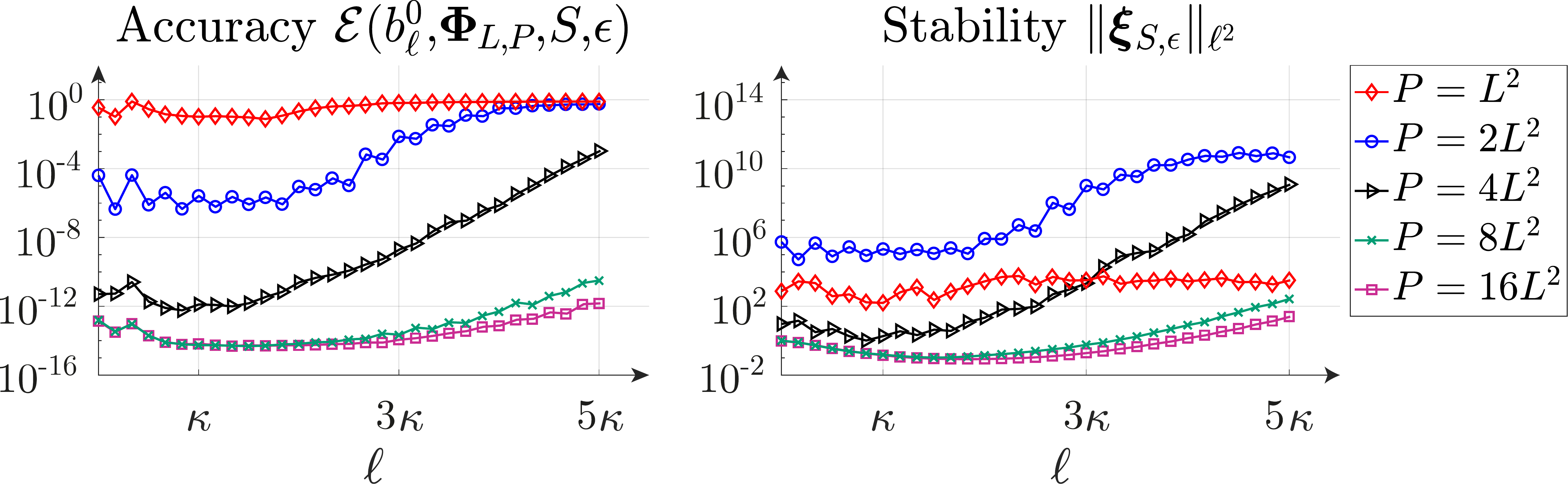}
\subcaption{Deterministic sampling.}
\end{subfigure}
\begin{subfigure}{0.98\linewidth}
\vspace{5mm}
\includegraphics[width=\linewidth]{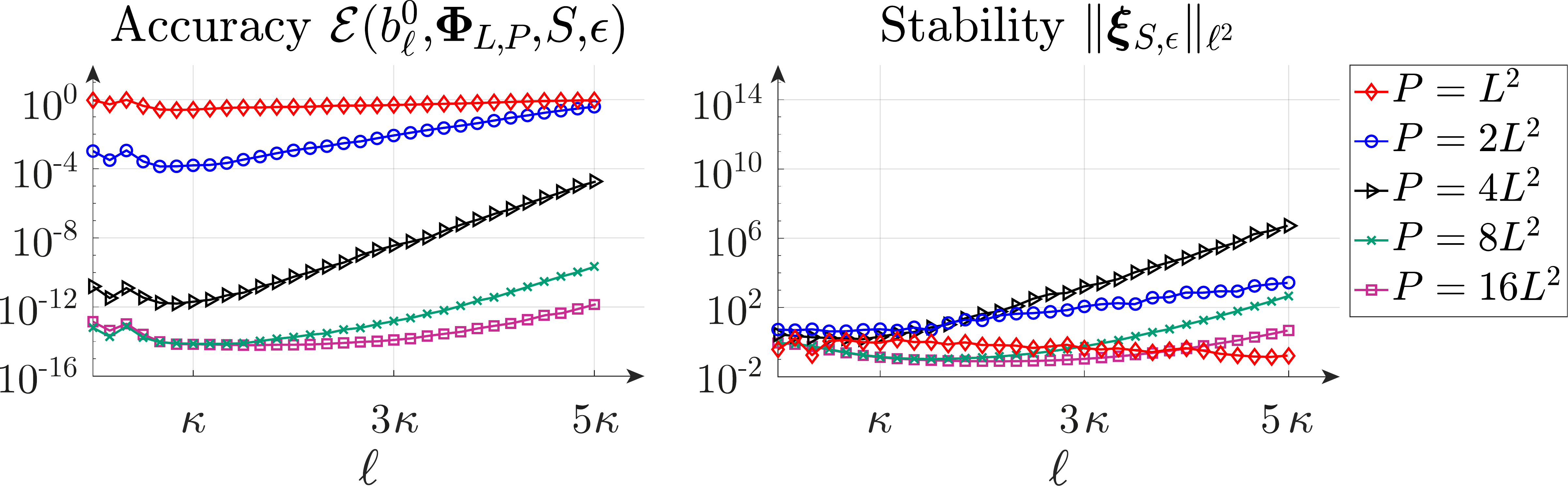}
\subcaption{Random sampling.}
\end{subfigure}
\begin{subfigure}{0.98\linewidth}
\vspace{5mm}
\includegraphics[width=\linewidth]{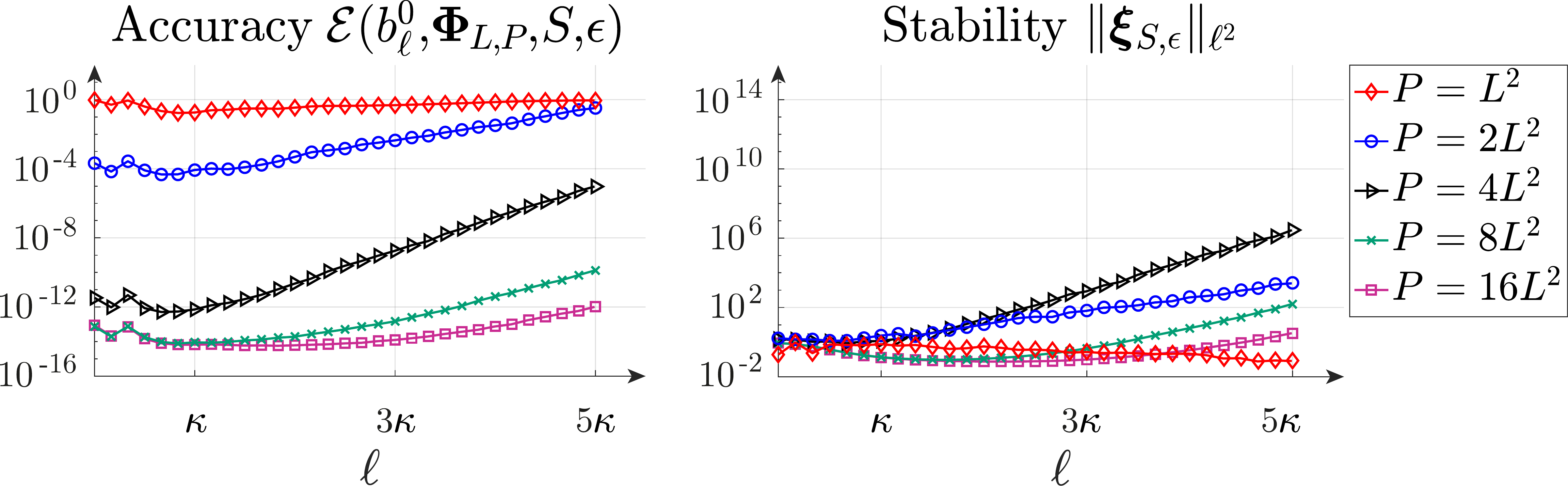}
\subcaption{Sobol sampling.}
\end{subfigure}
\caption{Accuracy $\mathcal{E}$ as defined in (\ref{relative residual}) (left) and stability $\|\boldsymbol{\xi}_{S,\epsilon}\|_{\ell^2}$ (right) of the approximation of spherical waves $b_{\ell}^0$ by evanescent plane waves, whose parameters in $Y$ are chosen according to the sampling strategies (a)--(c) presented in Definition \protect\hyperlink{Definition 6.4}{6.4}.
Compare these results with those presented in Figure \ref{figure 3.4}.
Truncation at $L=4\kappa$, wavenumber $\kappa=6$ and regularization parameter $\epsilon=10^{-14}$.}
\label{figure 7.2}
\end{figure}

\begin{figure}
\centering
\begin{subfigure}{\linewidth}
\centering
\includegraphics[width=8.7cm]{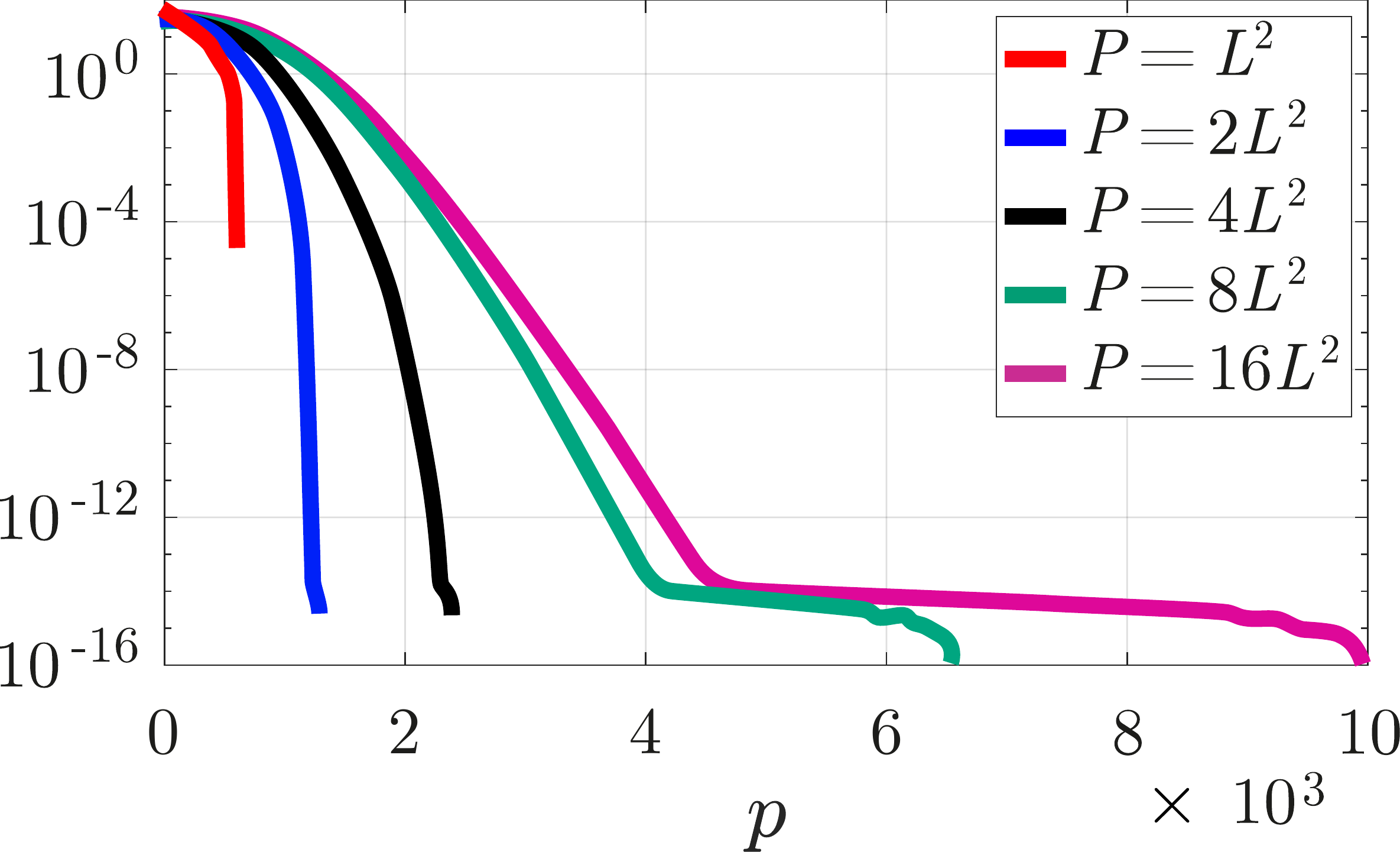}
\subcaption{Deterministic sampling.}
\end{subfigure}
\begin{subfigure}{\linewidth}
\centering
\vspace{5mm}
\includegraphics[width=8.7cm]{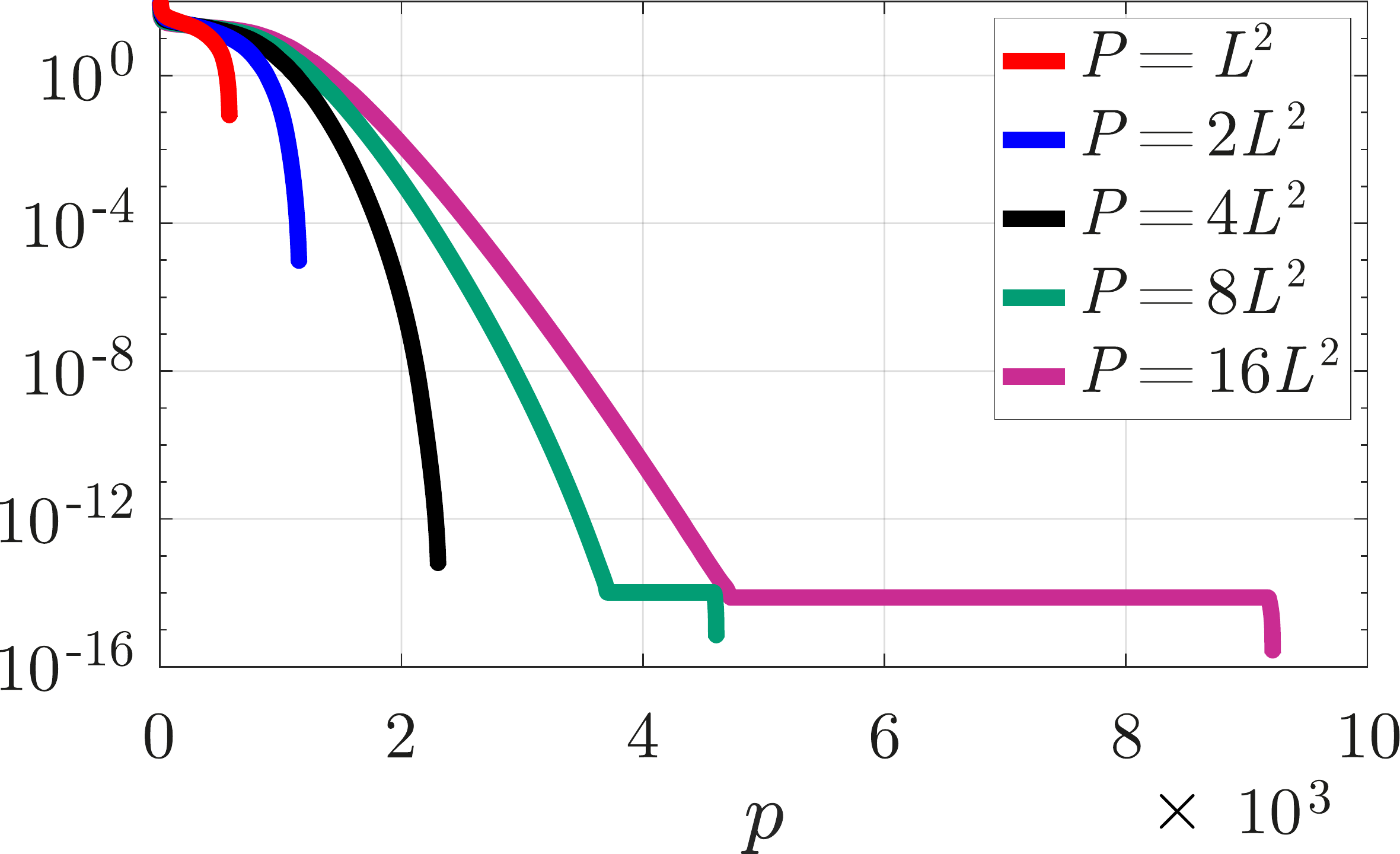}
\subcaption{Random sampling.}
\end{subfigure}
\begin{subfigure}{\linewidth}
\centering
\vspace{5mm}
\includegraphics[width=8.7cm]{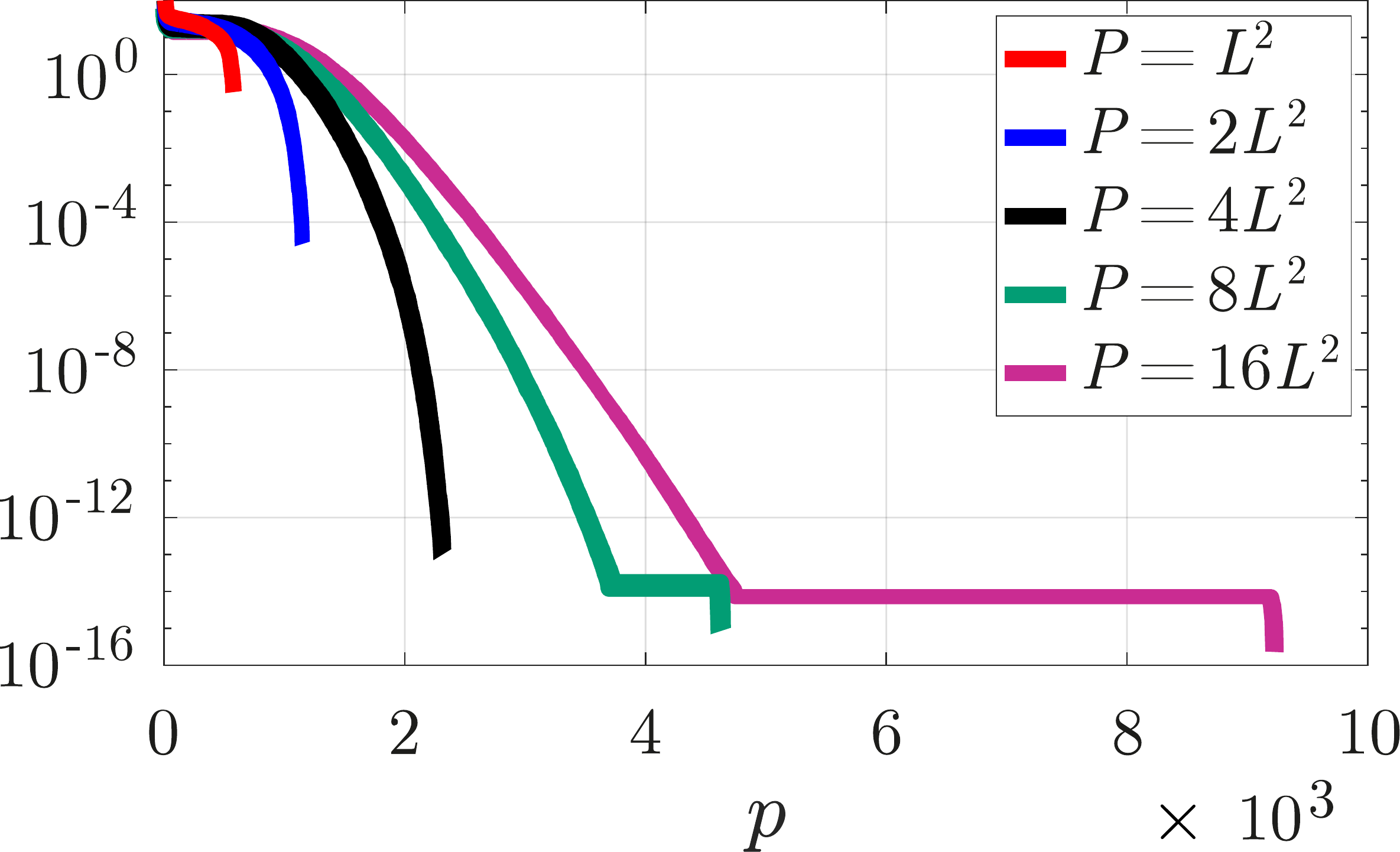}
\subcaption{Sobol sampling.}
\end{subfigure}
\caption{Singular values $\{\sigma_p\}_{p}$ of the matrix $A$ using evanescent plane wave approximation sets (\ref{evanescence sets}), whose parameters in $Y$ are chosen according to the sampling strategies (a)--(c) presented in Definition \protect\hyperlink{Definition 6.4}{6.4}.
Compare these results with those presented in Figure \ref{figure 3.2}.
Truncation at $L=4\kappa$, wavenumber $\kappa=6$.}
\label{figure 7.3}
\end{figure}

\begin{figure}
\centering
\begin{subfigure}{0.98\linewidth}
\phantomcaption
\phantomcaption
\phantomcaption
\includegraphics[width=\linewidth]{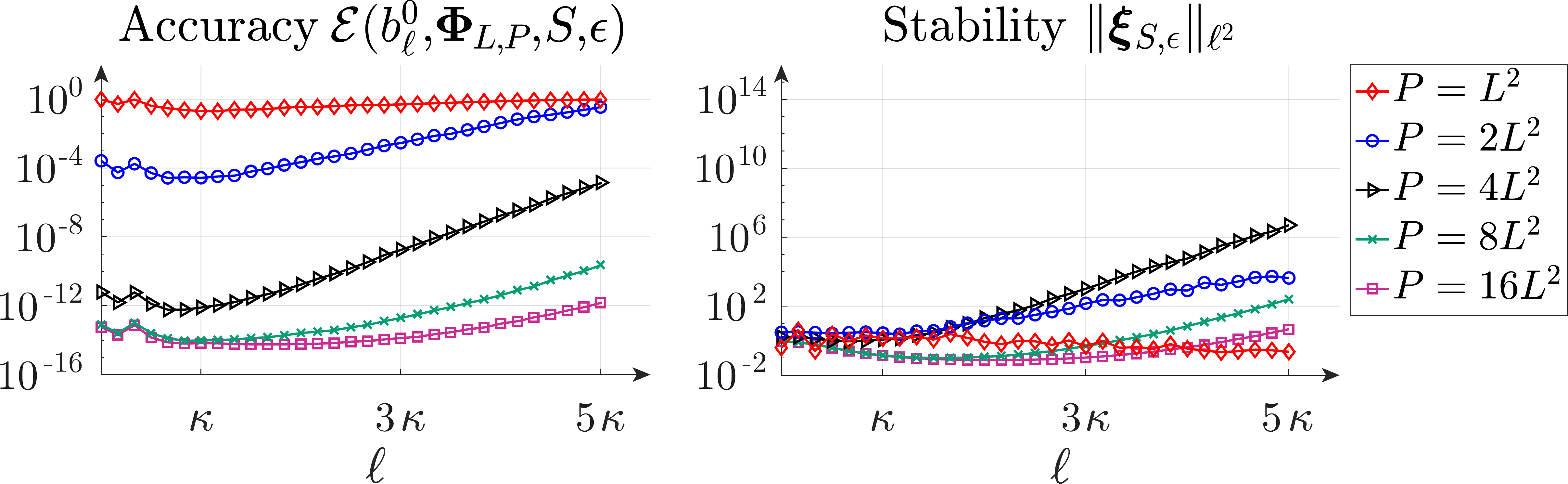}
\subcaption{Extremal--Random sampling.}
\end{subfigure}
\begin{subfigure}{0.98\linewidth}
\vspace{5mm}
\includegraphics[width=\linewidth]{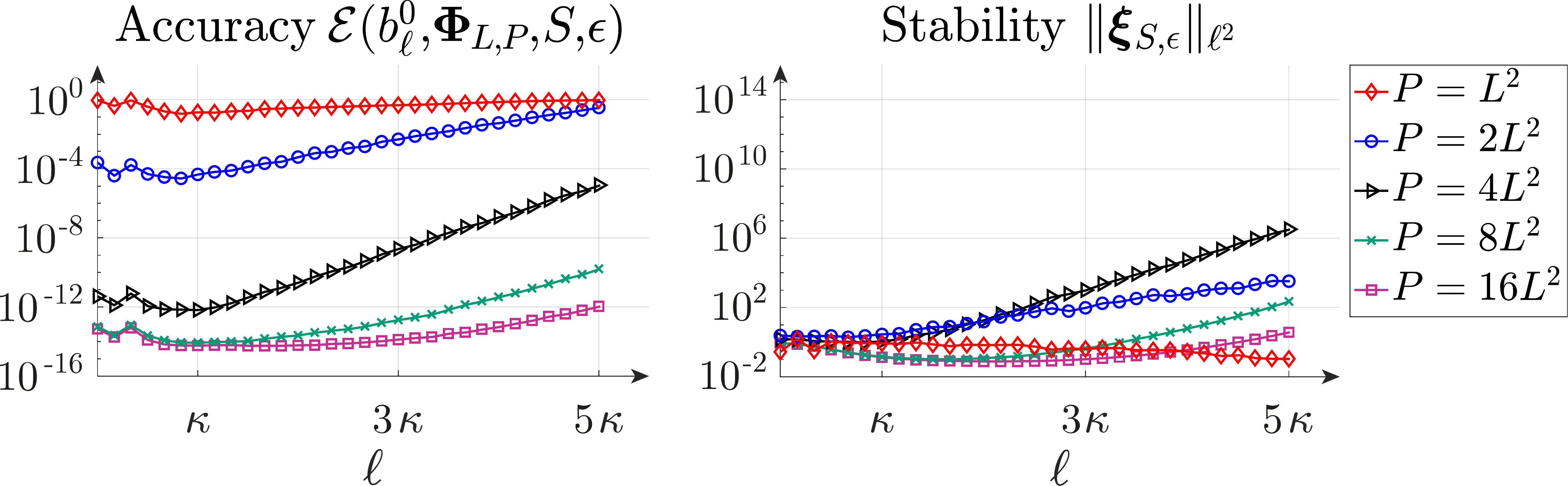}
\subcaption{Extremal--Sobol sampling.}
\end{subfigure}
\caption{ The caption of Figure \ref{figure 7.2} applies here as well, with the only difference that the parameters in $Y$ are now chosen according to the sampling strategies (d) and (e) in Definition \protect\hyperlink{Definition 6.4}{6.4}.}
\label{figure 7.4}
\end{figure}

\begin{figure}
\centering

\begin{subfigure}{0.49\textwidth}
\phantomcaption
\phantomcaption
\phantomcaption
\includegraphics[width=0.97\linewidth]{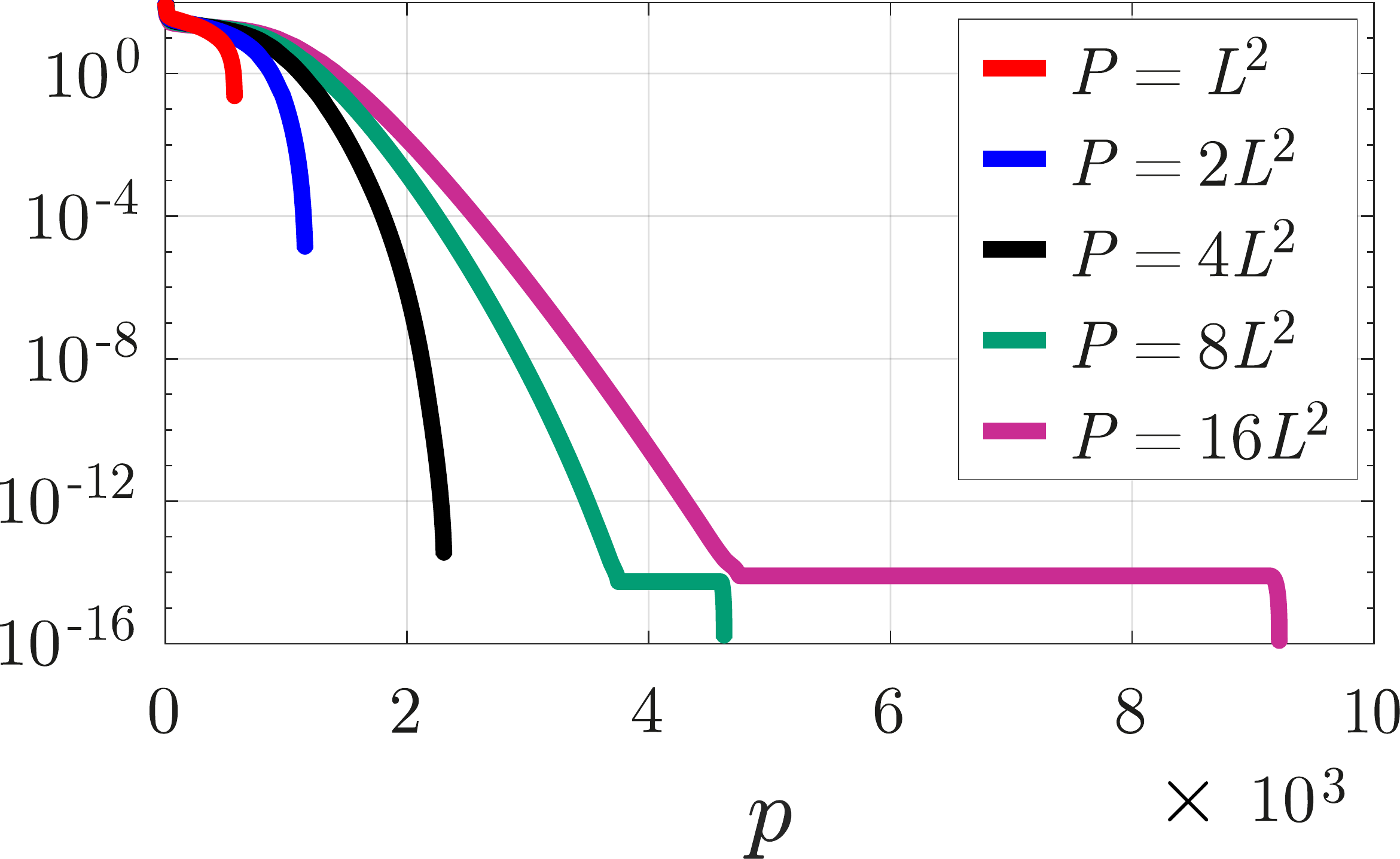}
\caption{Extremal--Random sampling.}
\end{subfigure}
\begin{subfigure}{0.49\textwidth}
\includegraphics[width=0.97\linewidth]{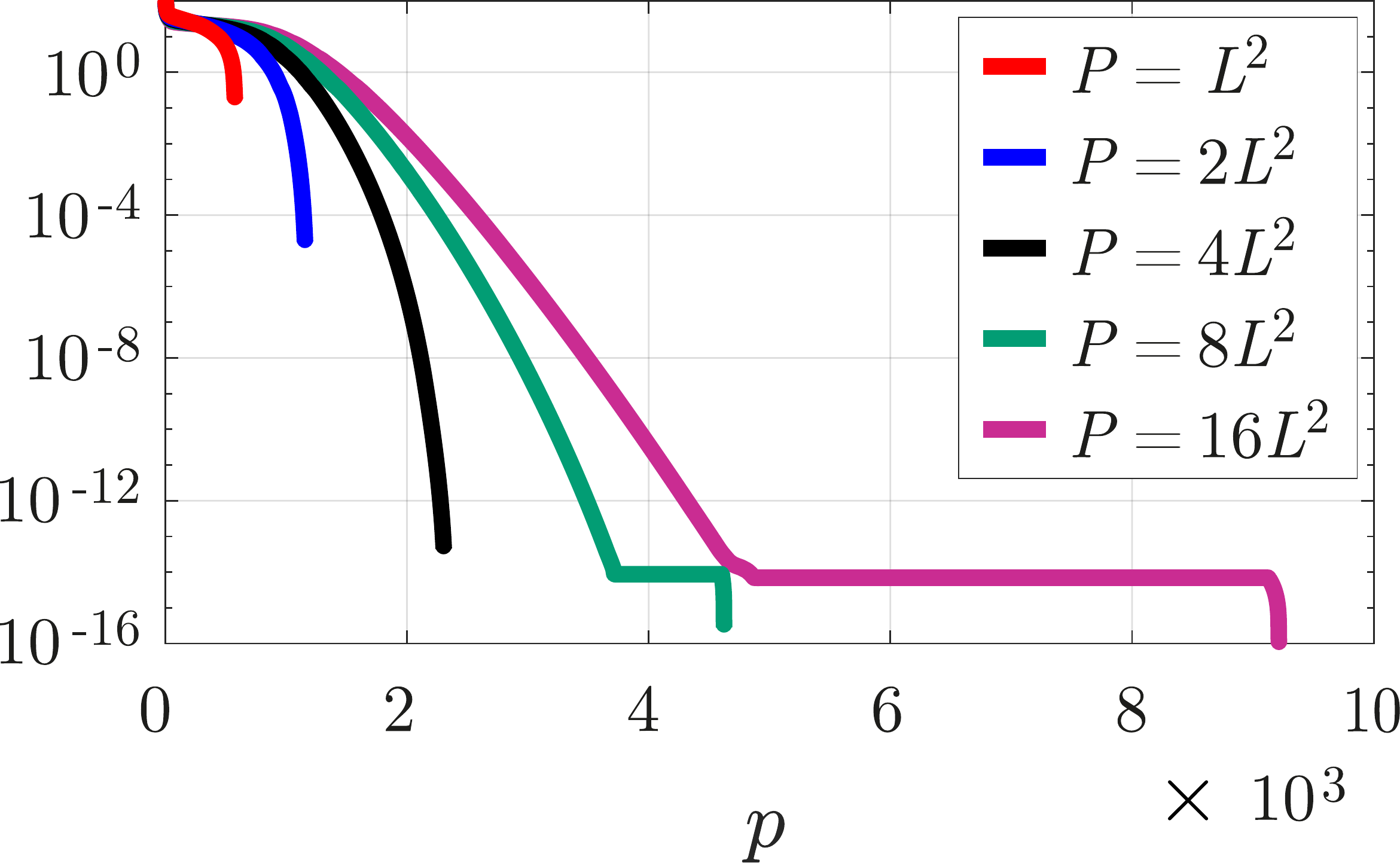}
\caption{Extremal--Sobol sampling.}
\end{subfigure}
\caption{The caption of Figure \ref{figure 7.3} applies here as well, with the only difference that the parameters in $Y$ are now chosen according to the sampling strategies (d) and (e) in Definition \protect\hyperlink{Definition 6.4}{6.4}.}
\label{figure 7.5}
\end{figure}

\section{Approximation of solution surrogates}

\hypertarget{Section 7.2}{We} test the previously described procedure by reconstructing a solution surrogate of the following form:
\begin{equation}
u:=\sum_{\ell=0}^L\sum_{m=-\ell}^{\ell}\hat{u}_{\ell}^mb_{\ell}^m \in \mathcal{B}_L.
\label{surrogate solution}
\end{equation}
The coefficients $\{\hat{u}_{\ell}^m\}_{(\ell,m) \in \mathcal{I}}$ of the expansion (\ref{surrogate solution}) are products of normally-dist-ributed random numbers (with mean $0$ and standard deviation $1$) and the scaling factors $[\max\{1,\ell-\kappa\}]^{-1}$. The coefficients of any element of $\mathcal{B}$ decay in modulus as $o(\ell^{-1})$ for $\ell \rightarrow \infty$, therefore, this scenario is quite challenging.

We perform the previously outlined procedure for all the sampling strategies presented in Definition \hyperlink{Definition 6.4}{6.4}.
The primary objective is to examine the validity of Conjecture \hyperlink{Conjecture 6.1}{6.2} focusing on how the error behaves as the dimension $P$ of the approximation space increases (see Remark \hyperlink{Remark 6.5}{6.5} for the exact approximation space dimension). In analogy with what was previously done, we choose $S=\lceil \sqrt{2|\boldsymbol{\Phi}_{L,P}|} \rceil^2$ extremal points, presented in Section \hyperlink{Section 2.3}{2.3}, for sampling on the sphere.
The numerical results are displayed in Figure \ref{figure 7.6} and Figure \ref{figure 7.7}. The left panel shows the relative residual $\mathcal{E}$, defined in (\ref{relative residual}), as a measure of the accuracy of the approximation. The right panel depicts the magnitude of the coefficients, namely $\|\boldsymbol{\xi}_{S,\epsilon}\|_{\ell^2}/\|u\|_{\mathcal{B}}$, as a measure of the stability of the approximation.

The error decreases quite rapidly with respect to the ratio $P/N(L)=P/(L+1)^2$, which is the dimension of the approximation set divided by the dimension of the space of the possible solution surrogates.
The numerical results depicted in Figure \ref{figure 7.6} and Figure \ref{figure 7.7} suggest that the size of the approximation set $P$ should vary quadratically with respect to the truncation parameter $L$.
In fact, when $L$ is large enough (e.g.\ $L \geq 2\kappa$), the decay is largely independent of $L$: this is consistent with Conjecture \hyperlink{Conjecture 6.1}{6.2}.
This is an important question regarding the efficiency of the proposed method: the approximation spaces $\boldsymbol{\Phi}_{L,P}$ defined in (\ref{evanescence sets}) are quasi-optimal, i.e.\ we only need $P=\mathcal{O}(N)$ DOFs with a moderate proportionality constant to approximate $N$ spherical modes with reasonable accuracy.

Furthermore, note that the magnitude of the coefficients $\|\boldsymbol{\xi}_{S,\epsilon}\|_{\ell^2}/\|u\|_{\mathcal{B}}$ in the expansions decreases as the dimension of the approximation space grows. This confirms that accurate and stable approximations can be obtained. The reported values of $\|\boldsymbol{\xi}_{S,\epsilon}\|_{\ell^2}/\|u\|_{\mathcal{B}}$ for small values of $P/N(L)$ (especially the initial increase) are not significant, as they correspond to imprecise approximations.

In Figure \ref{figure 7.8}, we present the plots of a solution surrogate (\ref{surrogate solution}), with wavenumber $\kappa=5$ and truncation parameter $L=5\kappa=25$. Additionally, in Figure \ref{figure 7.9} and Figure \ref{figure 7.10} we provide the absolute error using first $P=4(L+1)^2=2704$ plane waves and then $P=9(L+1)^2=6084$ plane waves, whether they are propagative or evanescent (where the nodes in $Y$ are selected using the Extremal--Sobol strategy (e) in Definition \hyperlink{Definition 6.4}{6.4}).
In both cases the same regularized oversampling technique discussed in Section \hyperlink{Section 2.2}{2.2} is used.

\begin{figure}
\centering
\begin{subfigure}{0.98\linewidth}
\includegraphics[width=0.98\linewidth]{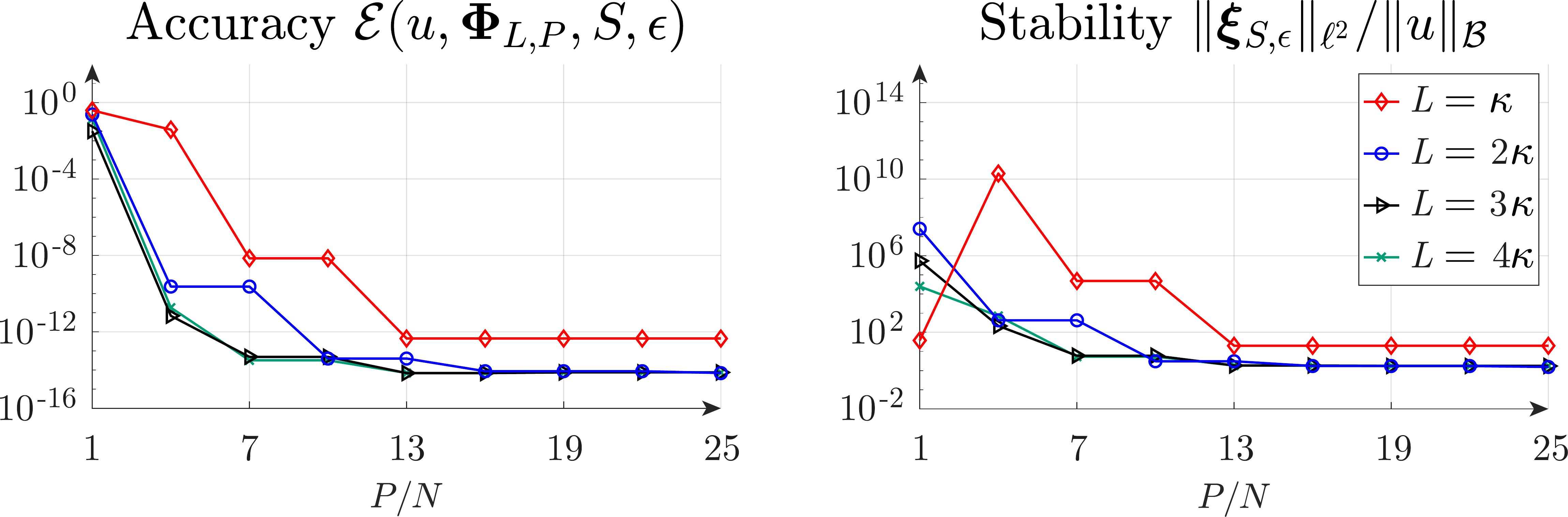}
\subcaption{Deterministic sampling.}
\end{subfigure}
\begin{subfigure}{0.98\linewidth}
\vspace{5mm}
\includegraphics[width=0.98\linewidth]{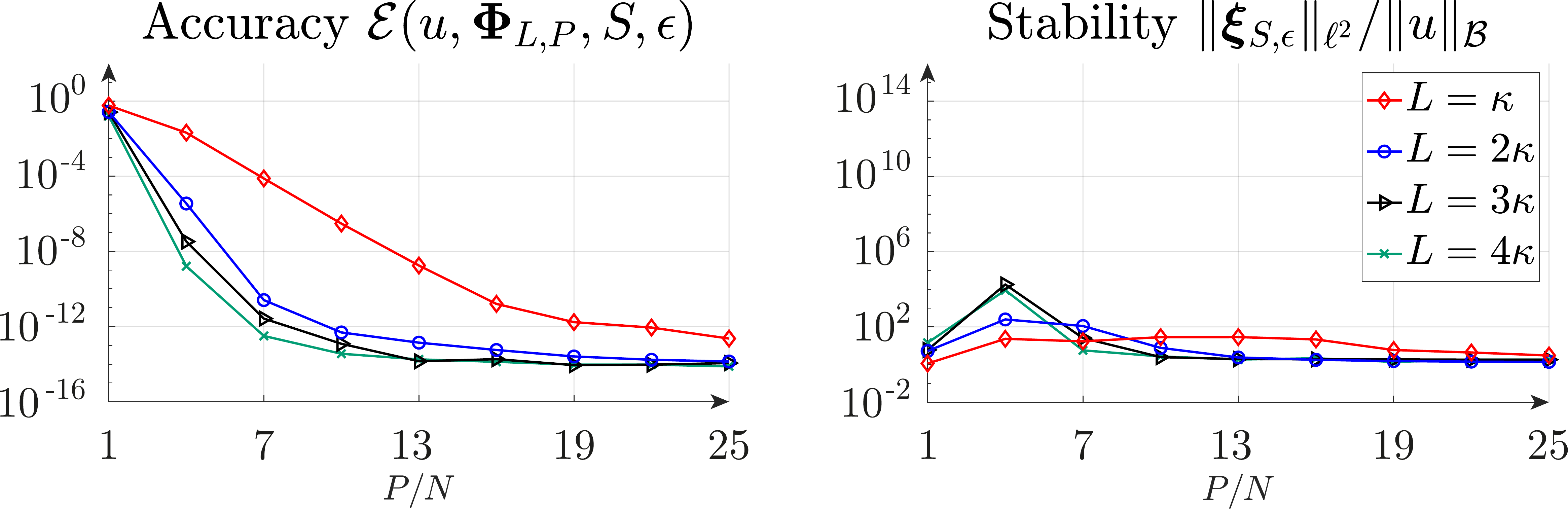}
\subcaption{Random sampling.}
\end{subfigure}
\begin{subfigure}{0.98\linewidth}
\vspace{5mm}
\includegraphics[width=0.98\linewidth]{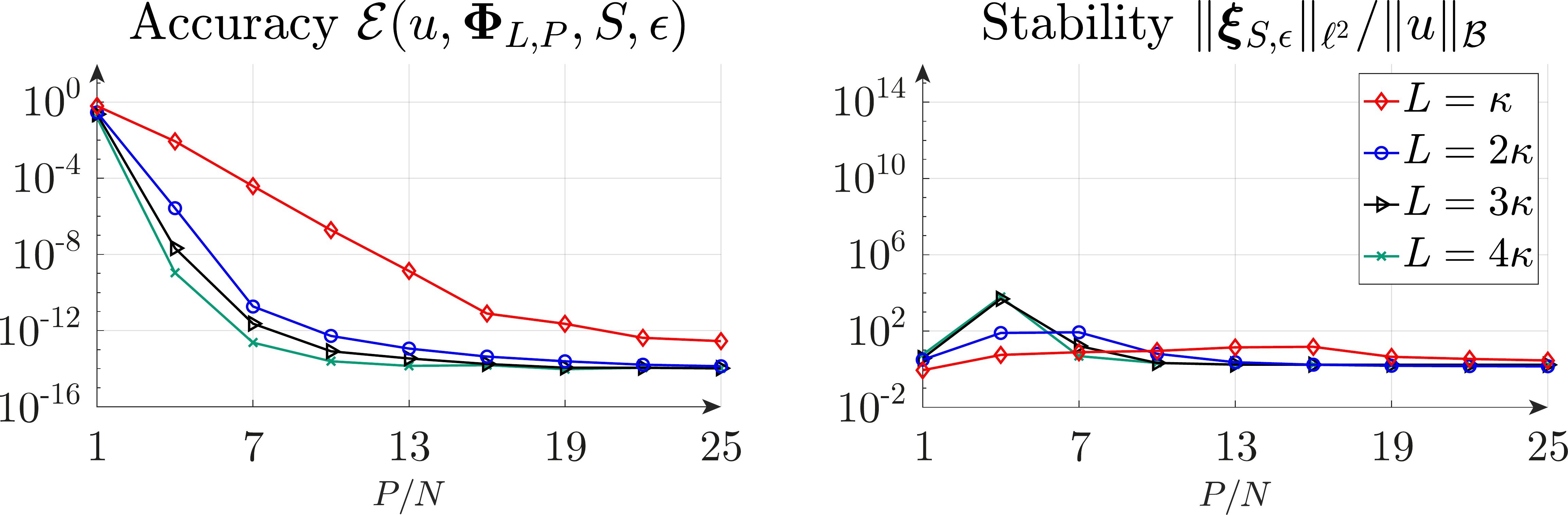}
\subcaption{Sobol sampling.}
\end{subfigure}
\caption{Accuracy $\mathcal{E}$ as defined in (\ref{relative residual}) (left) and stability $\|\boldsymbol{\xi}_{S,\epsilon}\|_{\ell^2}/\|u\|_{\mathcal{B}}$ (right) of the approximation of solution surrogates $u$ in the form (\ref{surrogate solution}) by $P$ evanescent plane waves, whose parameters in $Y$ are chosen according to the sampling strategies (a)--(c) presented in Definition \protect\hyperlink{Definition 6.4}{6.4}. The horizontal axis represents the ratio $P/N(L)$, where $N(L)=(L+1)^2$ is the dimension of the space $\mathcal{B}_L$, to which $u$ belongs. Wavenumber $\kappa=6$ and regularization parameter $\epsilon=10^{-14}$.}
\label{figure 7.6}
\end{figure}
\begin{figure}[H]
\centering
\begin{subfigure}{0.98\linewidth}
\phantomcaption
\phantomcaption
\phantomcaption
\includegraphics[width=0.98\linewidth]{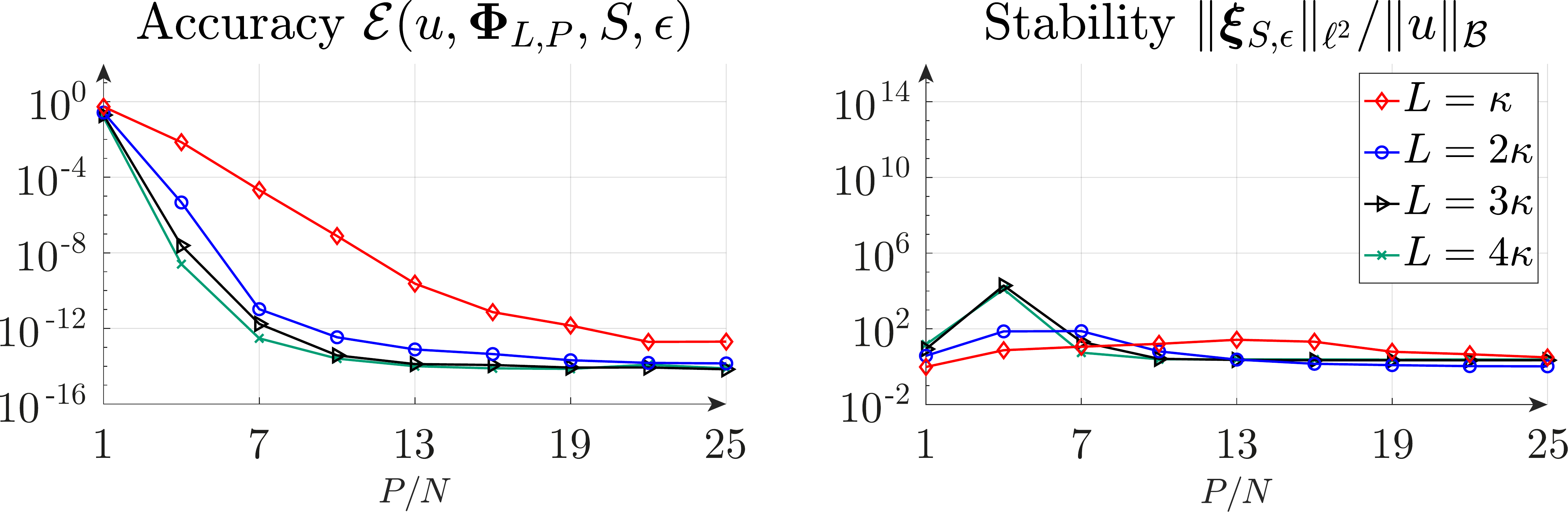}
\subcaption{Extremal--Random sampling.}
\end{subfigure}
\begin{subfigure}{0.98\linewidth}
\vspace{5mm}
\includegraphics[width=0.98\linewidth]{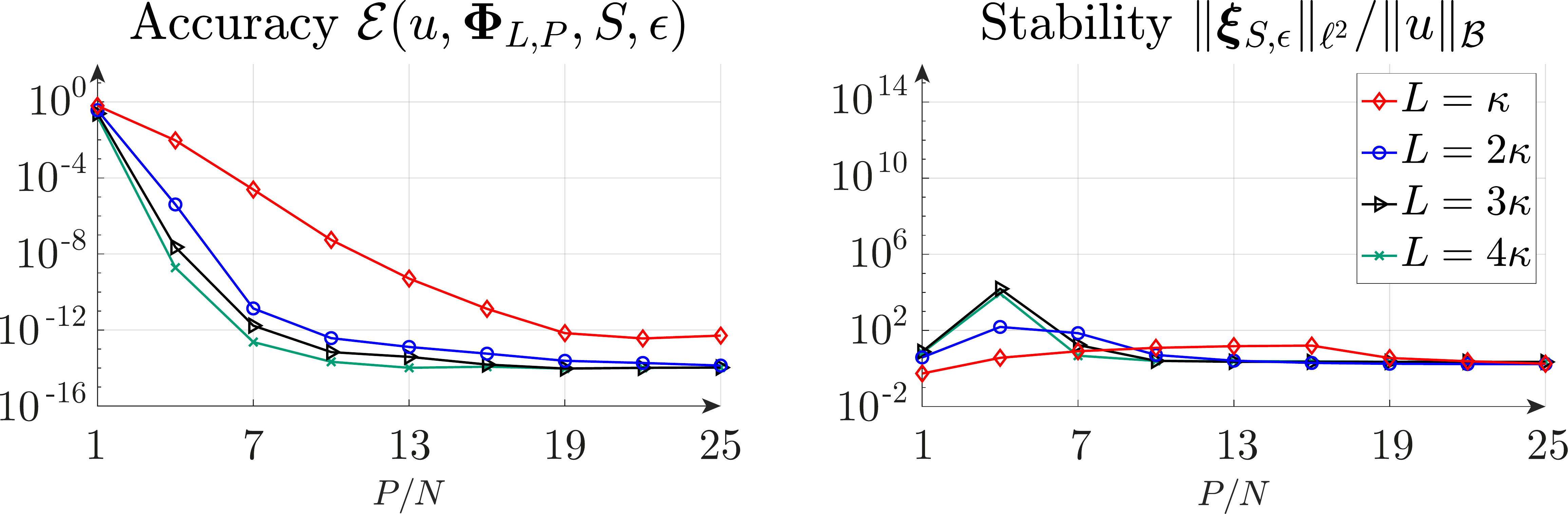}
\subcaption{Extremal--Sobol sampling.}
\end{subfigure}
\caption{The caption of Figure \ref{figure 7.6} applies here as well, with the only difference that the parameters in $Y$ are now chosen according to the sampling strategies (d) and (e) in Definition \protect\hyperlink{Definition 6.4}{6.4}.}
\label{figure 7.7}
\end{figure}
\vspace{-1mm}

The error resulting from the use of propagative plane waves $\boldsymbol{\Phi}_P$ is considerably higher compared to the one obtained by using evanescent plane waves $\boldsymbol{\Phi}_{L,P}$. As depicted in Figure \ref{figure 7.9} and Figure \ref{figure 7.10}, the difference is approximately $8$ and $12$ orders of magnitude larger, respectively, when measured using the $L^{\infty}$ norm.
In both scenarios, the error is concentrated near the boundary. However, while increasing the degrees of freedom does not lead to improved accuracy in the case of propagative waves, particularly on $\partial B_1$, using evanescent waves can reduce the error.
This is due to the fact that evanescent plane waves can effectively capture the higher Fourier modes of Helmholtz solutions, which is not possible with propagative plane waves.

Regarding the approximation by evanescent plane waves, we can estimate the number of DOFs per wavelength, denoted by $\lambda=2\pi/\kappa$, used in each direction with $\lambda \sqrt[3]{3|\boldsymbol{\Phi}_{L,P}|/4\pi}$, which is approximately $10.9$ in Figure \ref{figure 7.9} and $14.2$ in Figure \ref{figure 7.10}.
In low order methods, a commonly employed rule of thumb is to use roughly $6 \sim 10$ DOFs per wavelength to achieve $1$ or $2$ digits of accuracy.
However, observe that we can obtain $8 \sim 12$ digits of accuracy with just a fraction more of this amount.

In summary, the numerical results are fully compatible with Conjecture \hyperlink{Conjecture 6.1}{6.2}.
\begin{figure}
\centering
\begin{tabular}{cc}
\includegraphics[trim=50 50 50 50,clip,width=.45\textwidth,valign=m]{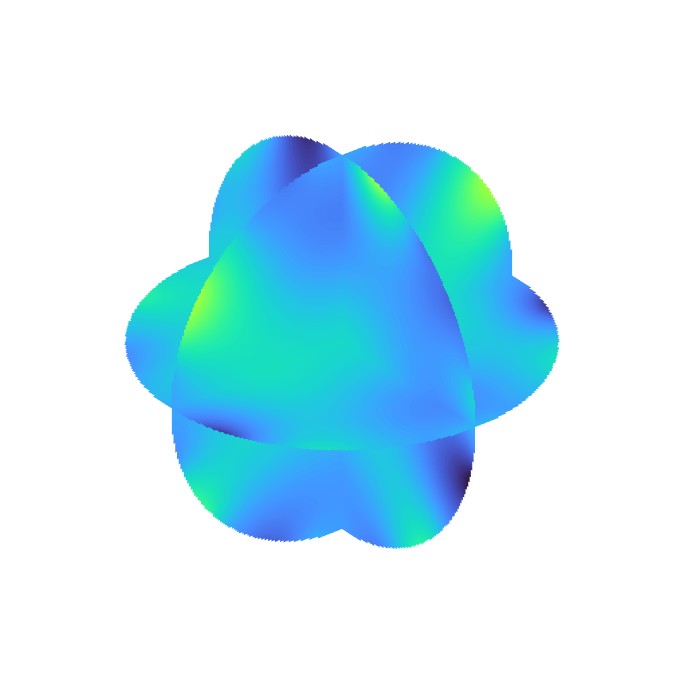} &  \includegraphics[trim=50 50 50 50,clip,width=.45\textwidth,valign=m]{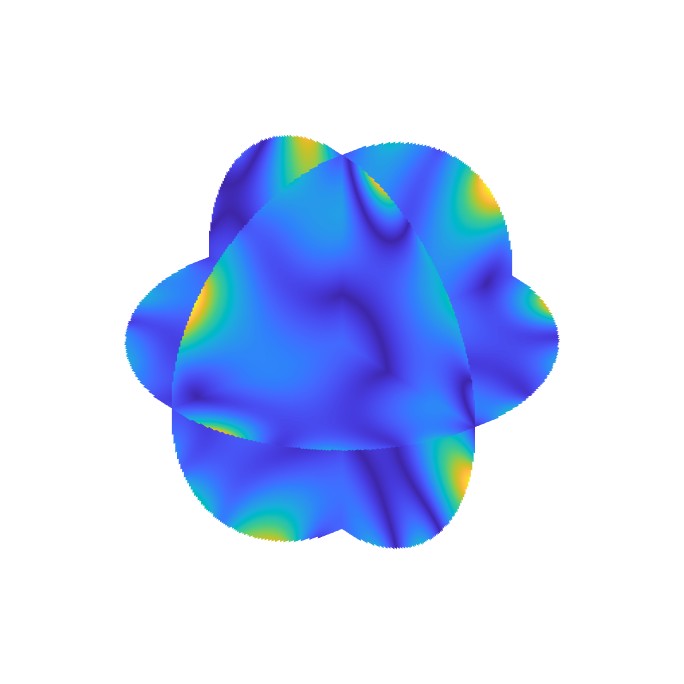}\\
\includegraphics[width=.4\textwidth]{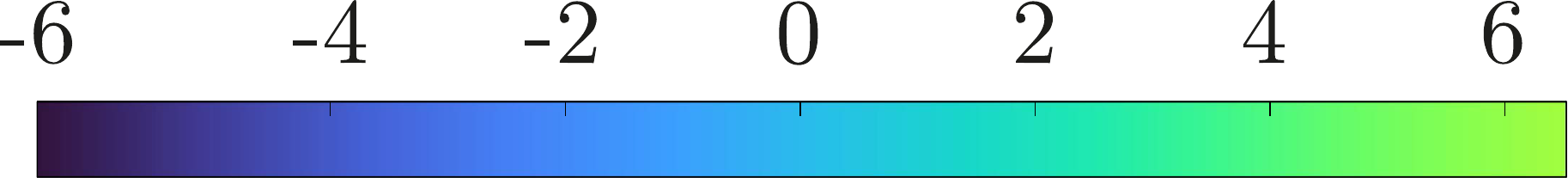} & \includegraphics[width=.4\textwidth]{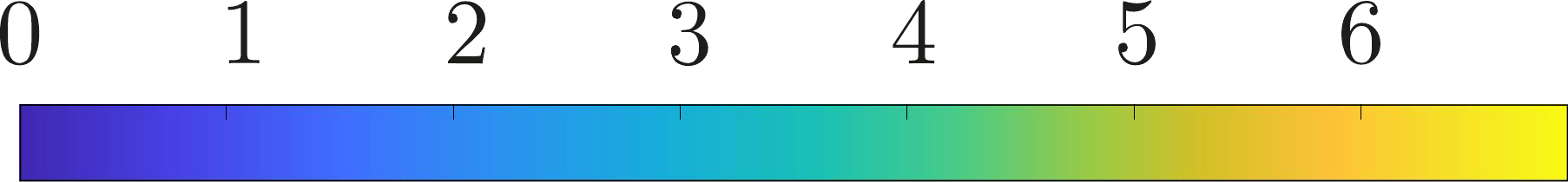} \\
Real part of target solution $\Re{u}$
& Modulus of target solution $|u|$\\
\rule{0pt}{3.3ex}
\includegraphics[trim=50 50 50 50,clip,width=.45\textwidth,valign=m]{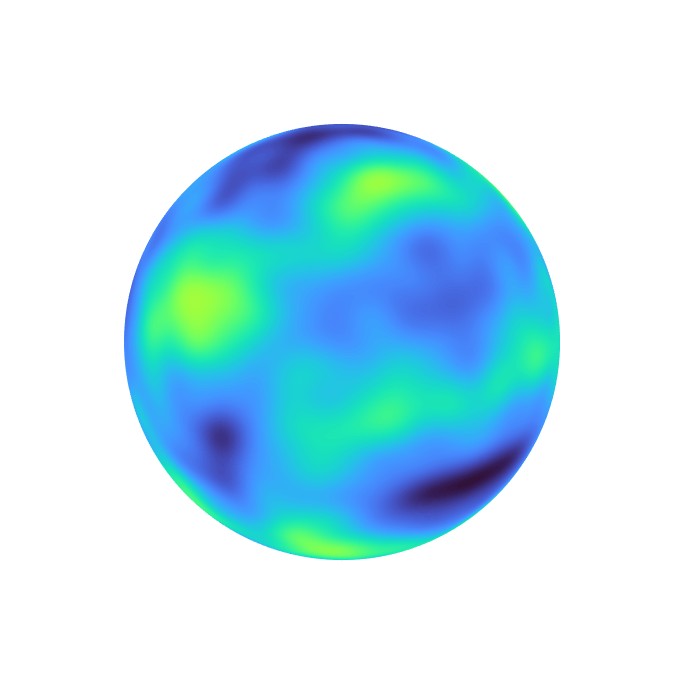} & \includegraphics[trim=50 50 50 50,clip,width=.45\textwidth,valign=m]{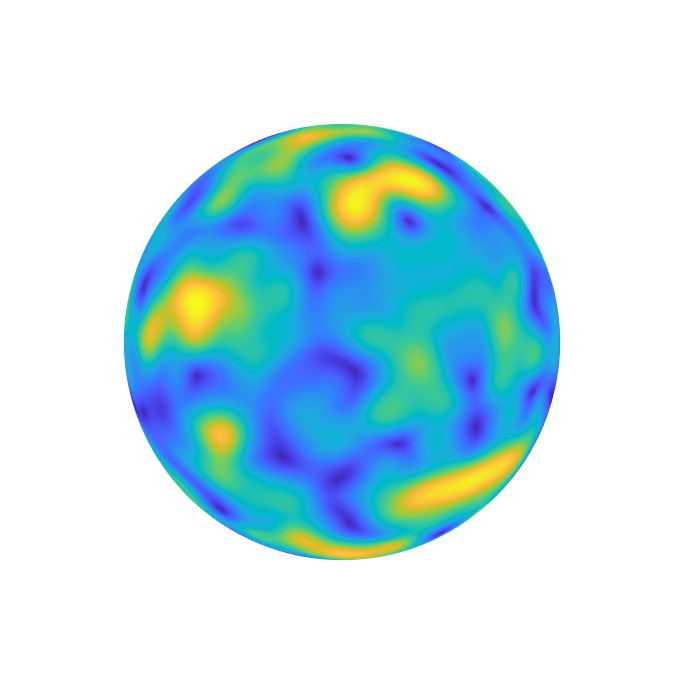}\\
\includegraphics[width=.4\textwidth]{imagespdf/bar_real.pdf} & \includegraphics[width=.4\textwidth]{imagespdf/bar_abs.pdf} \\
Real part of target solution $\Re{u}$
& Modulus of target solution $|u|$\\
\end{tabular}
\caption{Solution surrogate $u$, target of the approximation, defined in (\ref{surrogate solution}) with wavenumber $\kappa=5$ and $L=5\kappa=25$. Both the real part $\Re{u}$ and the modulus $|u|$ of the target solution are plotted on $B_1 \cap \{\mathbf{x}=(x,y,z) : xyz=0\}$ (top) and on the unit sphere $\partial B_1$ (bottom).}
\label{figure 7.8}
\end{figure}

\begin{figure}
\centering
\begin{tabular}{cc}
\includegraphics[trim=50 50 50 50,clip,width=.45\textwidth,valign=m]{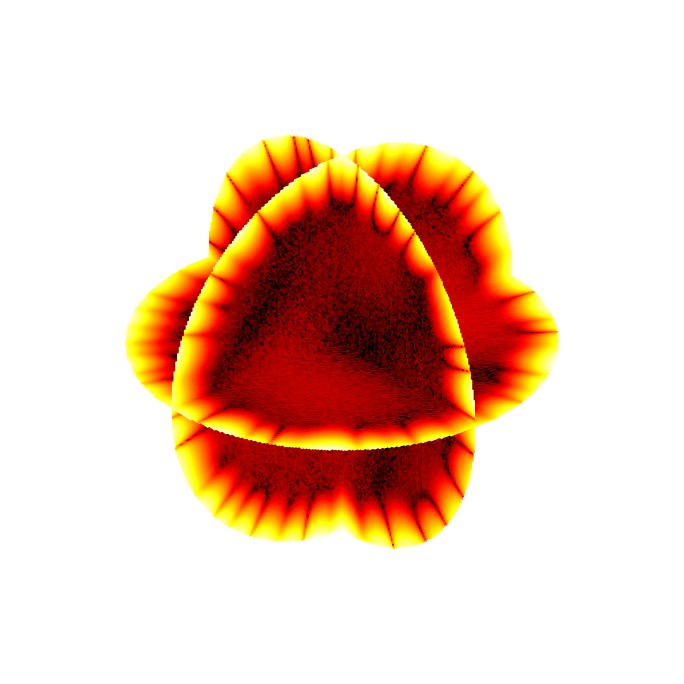} &  \includegraphics[trim=50 50 50 50,clip,width=.45\textwidth,valign=m]{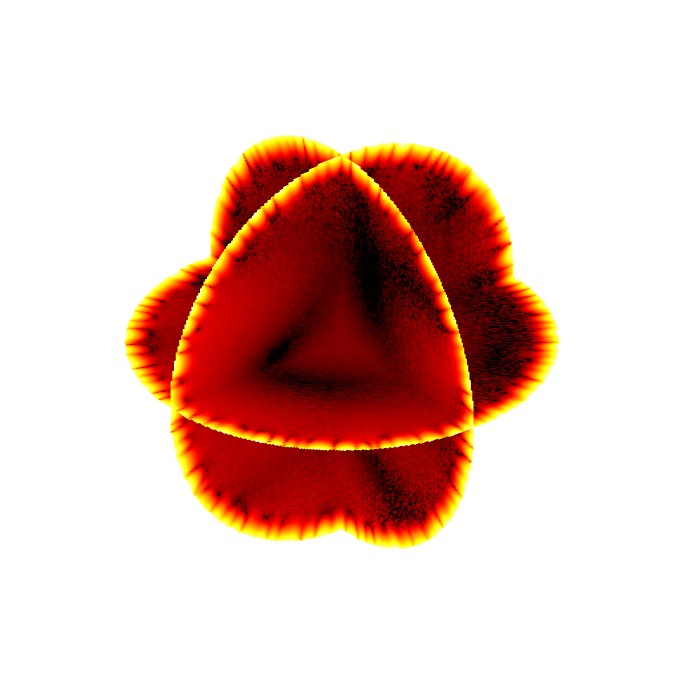}\\
\includegraphics[width=.4\textwidth]{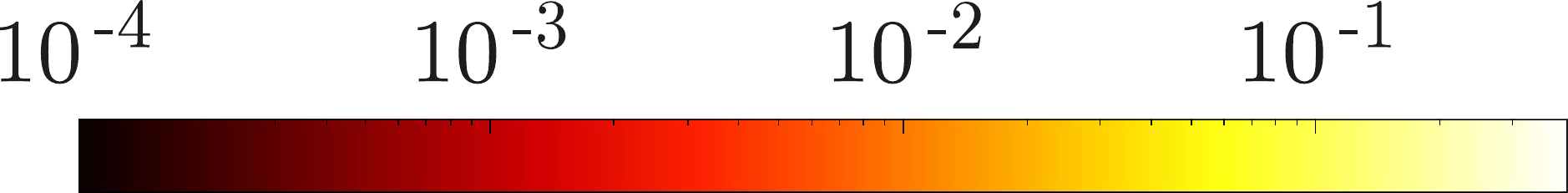} &
\includegraphics[width=.4\textwidth]{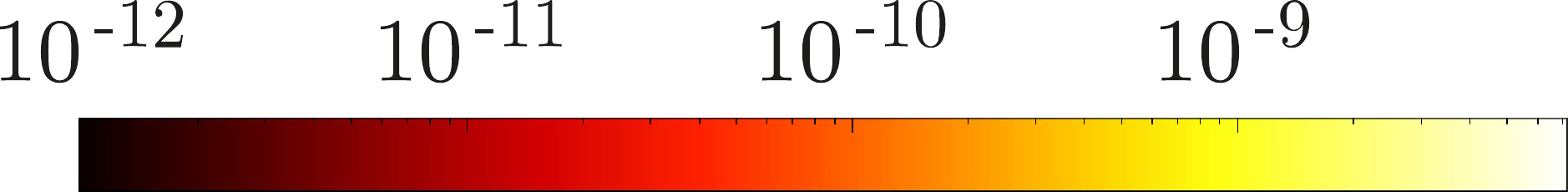} \\
Absolute error using PPWs
& Absolute error using EPWs\\
\rule{0pt}{3.3ex}
\includegraphics[trim=50 50 50 50,clip,width=.45\textwidth,valign=m]{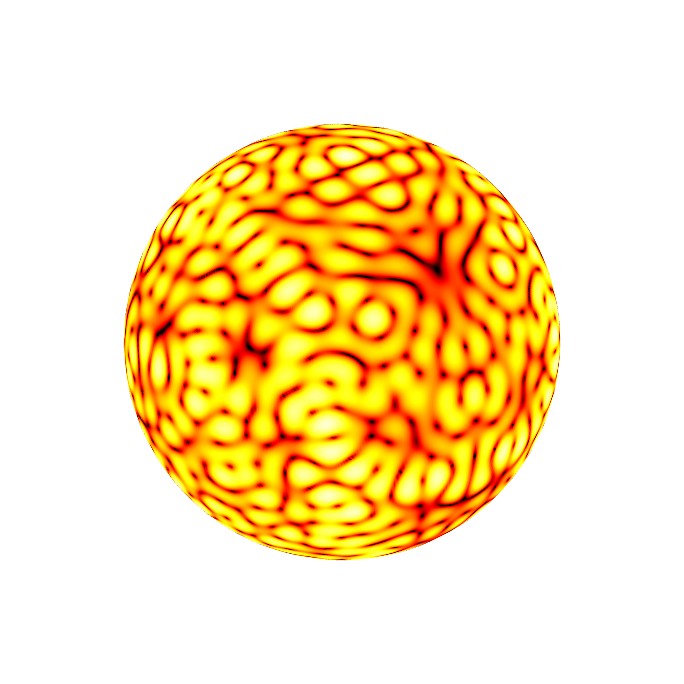} & \includegraphics[trim=50 50 50 50,clip,width=.45\textwidth,valign=m]{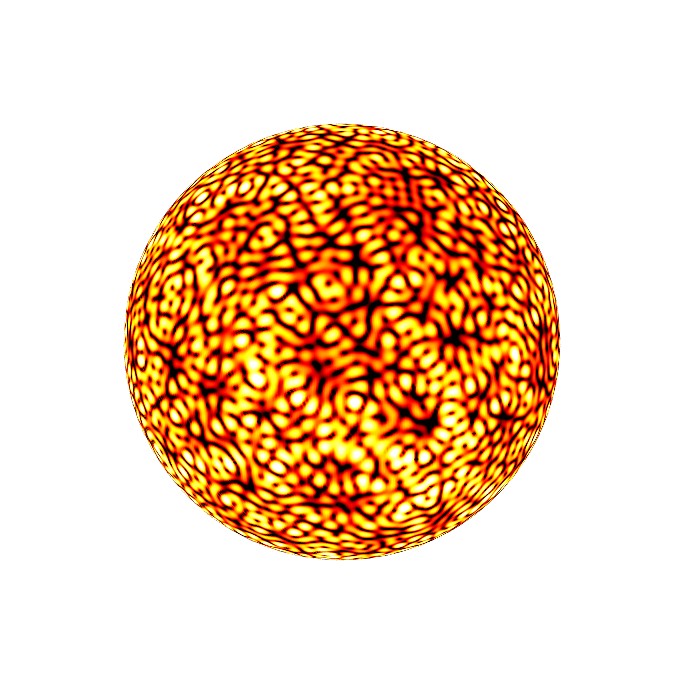}\\
\includegraphics[width=.4\textwidth]{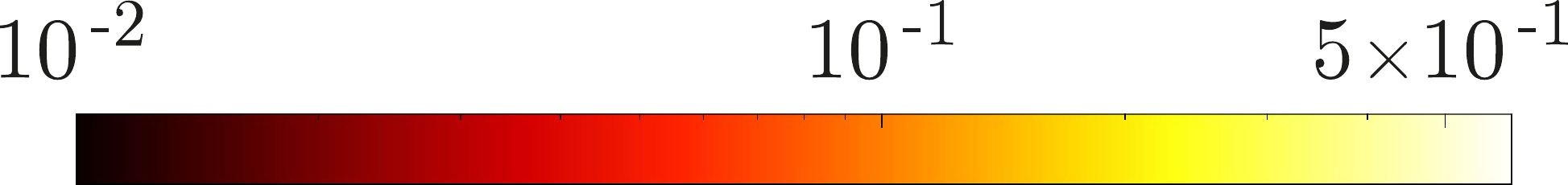} &
\includegraphics[width=.4\textwidth]{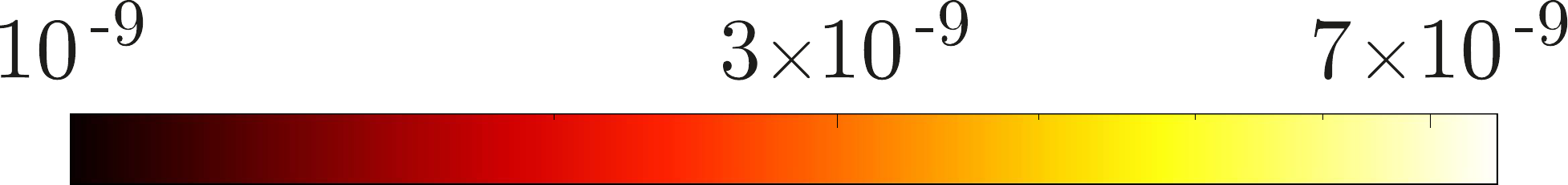}\\
Absolute error using PPWs
& Absolute error using EPWs\\
\end{tabular}
\caption{Absolute errors of the approximation of the solution surrogate $u$, defined in (\ref{surrogate solution}) with $L=5\kappa=25$. The error is provided using $P=4(L+1)^2=2704$ plane waves, either propagative ones $\boldsymbol{\Phi}_P$ from (\ref{plane waves approximation set}) (left) or evanescent ones $\boldsymbol{\Phi}_{L,P}$ from (\ref{evanescence sets}), whose parameters are constructed using the Extremal--Sobol strategy (e) presented in Definition \protect\hyperlink{Definition 6.4}{6.4} (right). The absolute errors are plotted both on $B_1 \cap \{\mathbf{x}=(x,y,z) : xyz=0\}$ (top) and on the unit sphere $\partial B_1$ (bottom). Wavenumber $\kappa=5$ and regularization parameter $\epsilon=10^{-14}$.}
\label{figure 7.9}
\end{figure}

\begin{figure}
\centering
\begin{tabular}{cc}
\includegraphics[trim=50 50 50 50,clip,width=.45\textwidth,valign=m]{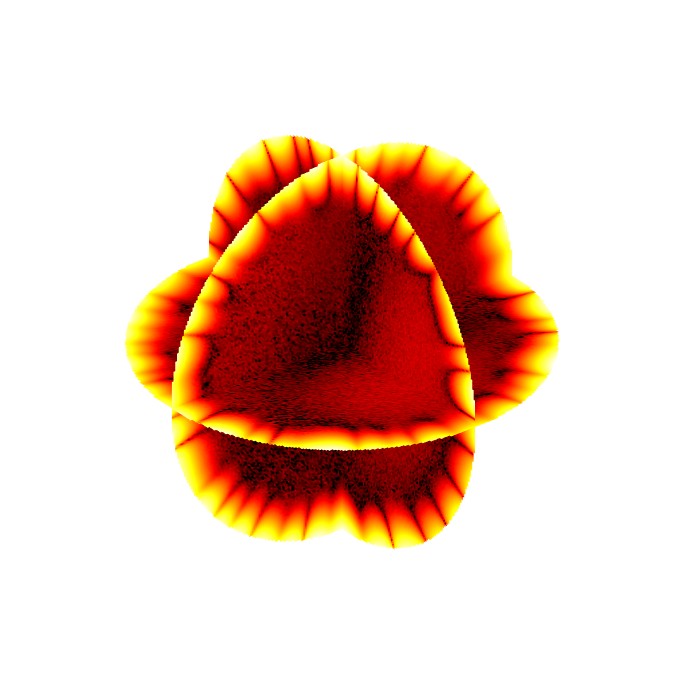} &  \includegraphics[trim=50 50 50 50,clip,width=.45\textwidth,valign=m]{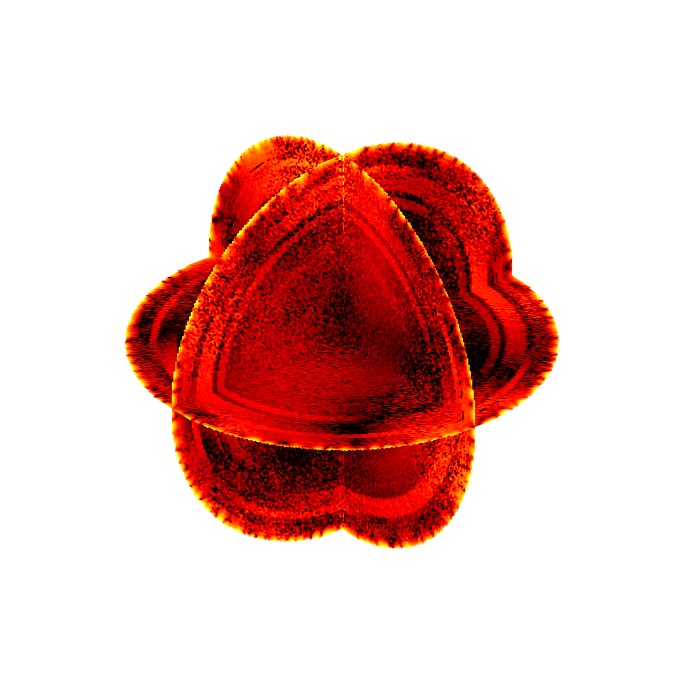}\\
\includegraphics[width=.4\textwidth]{imagespdf/bar_propagative.pdf} &
\includegraphics[width=.4\textwidth]{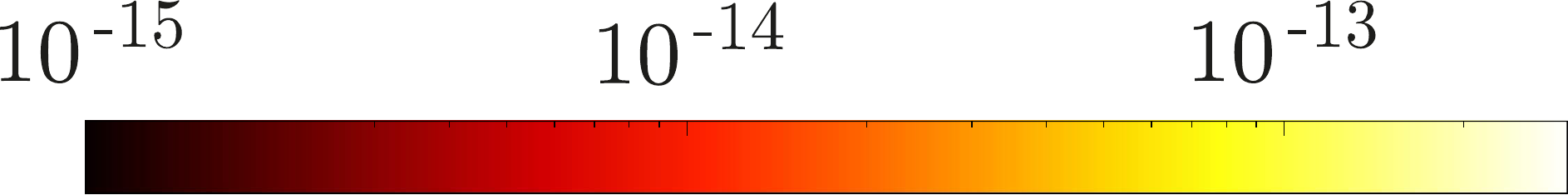}\\
Absolute error using PPWs
& Absolute error using EPWs\\
\rule{0pt}{3.3ex}
\includegraphics[trim=50 50 50 50,clip,width=.45\textwidth,valign=m]{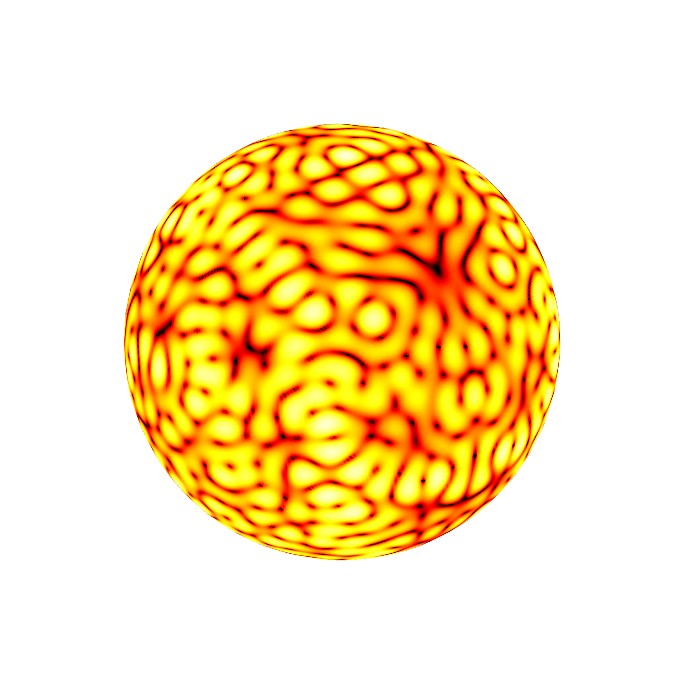} & \includegraphics[trim=50 50 50 50,clip,width=.45\textwidth,valign=m]{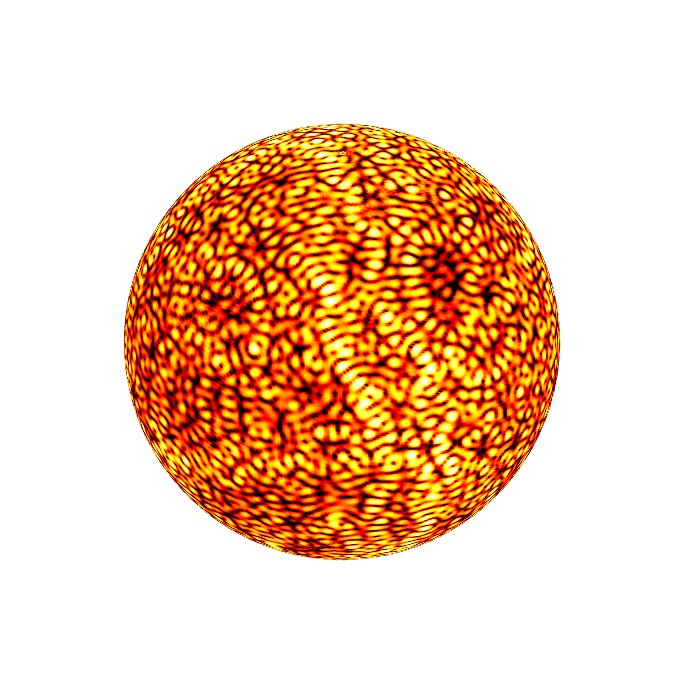}\\
\includegraphics[width=.4\textwidth]{imagespdf/bar_propagative_sphere.pdf} &
\includegraphics[width=.4\textwidth]{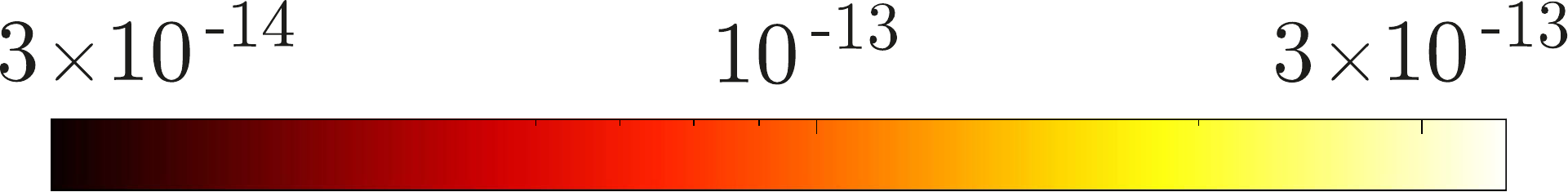}\\
Absolute error using PPWs
& Absolute error using EPWs\\
\end{tabular}
\caption{The caption of Figure \ref{figure 7.9} applies here as well, with the only difference that now the absolute error is provided using $P=9(L+1)^2=6084$ plane waves.}
\label{figure 7.10}
\end{figure}

\section{Enhanced accuracy near singularities}

\begin{figure}
\centering
\includegraphics[width=\linewidth]{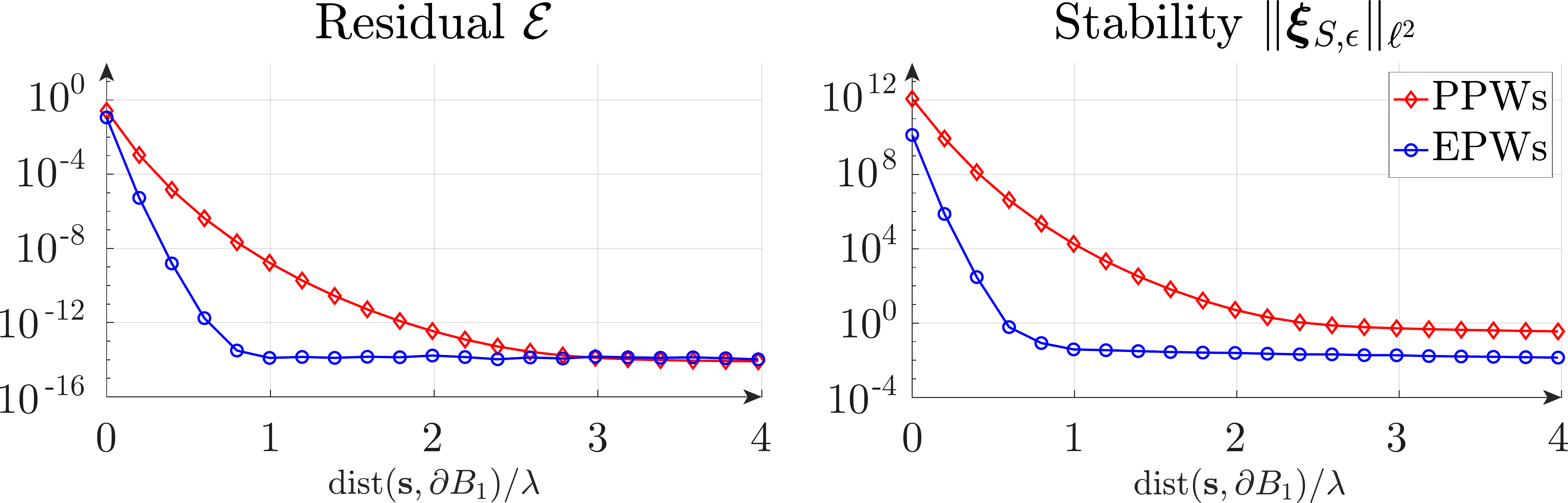}
\caption{Residual $\mathcal{E}$ as defined in (\ref{relative residual}) (left) and stability $\|\boldsymbol{\xi}_{S,\epsilon}\|_{\ell^2}$ (right) of the approximation of the fundamental solution $\Phi_{\mathbf{s}}$ presented in (\ref{fundamental solution}) by $P=2704$ plane waves, either propagative or evanescent (where the nodes in $Y$ are selected using the Extremal–Sobol strategy (e) presented in Definition \protect\hyperlink{Definition 6.4}{6.4}). The values are plotted with respect to the distance of the singularities $\mathbf{s}$ from the surface $\partial B_1$, which has been scaled by the wavelength $\lambda=2\pi/\kappa$. Wavenumber $\kappa=5$ and regularization parameter $\epsilon=10^{-14}$.}
\label{figure 7.11}
\end{figure}

\hypertarget{Section 7.3}{In} the next sections we will consider some numerical experiments involving the fundamental solution of the Helmholtz equation, namely:
\begin{equation}
\Phi_{\mathbf{s}}(\mathbf{x}):=\frac{1}{4\pi}\frac{e^{i\kappa|\mathbf{x}-\mathbf{s}|}}{|\mathbf{x}-\mathbf{s}|},\,\,\,\,\,\,\,\,\,\,\,\,\,\,\forall \mathbf{x} \in \Omega,
\label{fundamental solution}
\end{equation}
where $\Omega \subset \R^3$ is a convex domain and $\mathbf{s} \in \R^3 \setminus \overline{\Omega}$. For now, let us assume $\Omega=B_1$; in the next section we will consider different geometries in order to show that the approximation sets that we built based on the previous analysis for the unit ball $B_1$ also possess excellent approximation properties on other shapes.

In Figure \ref{figure 7.11}, we study the accuracy $\mathcal{E}$, as defined in (\ref{relative residual}), and the stability $\|\boldsymbol{\xi}_{S,\epsilon}\|_{\ell^2}$ of the approximation of the fundamental solution $\Phi_{\mathbf{s}}$ defined in (\ref{fundamental solution}) by plane waves, either propagative or evanescent.
The abscissa shows the distance of the singularity $\mathbf{s}$ from the surface $\partial B_1$, which has been scaled by the wavelength $\lambda=2\pi/\kappa$.
Due to the quasi-optimality of the approximation spaces $\boldsymbol{\Phi}_{L,P}$, which was hinted by the numerical results of Section \hyperlink{Section 7.2}{7.2}, the truncation parameter $L$ is computed from $P$ as $L:=\max\{\lceil \kappa\rceil,\lfloor\sqrt{P/10}\rfloor\}$. Better proportionality constants in the relation $P=\nu L^2$ can be investigated by analogy with \cite[Sec.\ 8.4]{parolin-huybrechs-moiola}.

These results imply that evanescent plane waves are effective in capturing the higher Fourier modes of the fundamental solution $\Phi_{\mathbf{s}}$ when the singularity is close to the boundary $\partial B_1$. Specifically, if $\mathbf{s} \in \partial B_1$, then both the plane wave approximation sets are inadequate for approximating the fundamental solution $\Phi_{\mathbf{s}}$, due to the large coefficients: in fact, if $\mathbf{s} \in \partial B_1$, then $\Phi_{\mathbf{s}} \not \in \mathcal{B}$.
As we move away from the singularity, the evanescent waves become more suitable for approximating the higher Fourier modes that arise due to the presence of the nearby singularity.
An example is given in Figure \ref{figure 7.11}, where it can be observed that the accuracy provided by the evanescent wave sets already reaches $10^{-14}$ when $\textup{dist}(\mathbf{s},\partial B_1) = \lambda$.
If the singularity $\mathbf{s}$ is located sufficiently far away (approximately $3$ wavelengths, i.e.\ $\textup{dist}(\mathbf{s},\partial B_1)=3\lambda$), both types of plane wave sets provide good approximations of the fundamental solution $\Phi_{\mathbf{s}}$, since they only need to account for its propagative modes.

In Figure \ref{figure 7.12} we report the convergence of the approximation by plane waves for increasing size of the approximation set $P$.
We consider the fundamental solution $\Phi_{\mathbf{s}}$ defined in (\ref{fundamental solution}) with wavenumber $\kappa=5$ and singularity $\mathbf{s}=(1+2\lambda/3,0,0)$.
It is worth noting that when propagative plane waves are employed, the residual of the linear system initially reduces swiftly with increasing $P$. However, it eventually plateaus before reaching machine precision due to the rapid growth of the coefficients. Conversely, when using evanescent plane wave approximation sets, the residual converges to machine precision and the coefficients magnitude remains reasonable upon achieving the final accuracy.
In fact, by using evanescent plane waves, the truncation parameter $L$, and consequently the number of approximated modes, grows concurrently with the size of the approximation set $P$, providing an increasingly accurate approximation. In contrast, when using propagative waves, increasing the discrete space only enhances the approximation of propagative modes, while ignoring the higher Fourier modal contents.

From Figure \ref{figure 7.13} to Figure \ref{figure 7.14 bis} we showcase two instances of the fundamental solution $\Phi_{\mathbf{s}}$ with wavenumber $\kappa=5$ and different choices of the singularity $\mathbf{s} \in \R^3\setminus\overline{B_1}$, along with the errors in approximation using plane waves, whether they are propagative or evanescent. The results are consistent with Figure \ref{figure 7.11}.

\begin{figure}
\centering
\includegraphics[width=\linewidth]{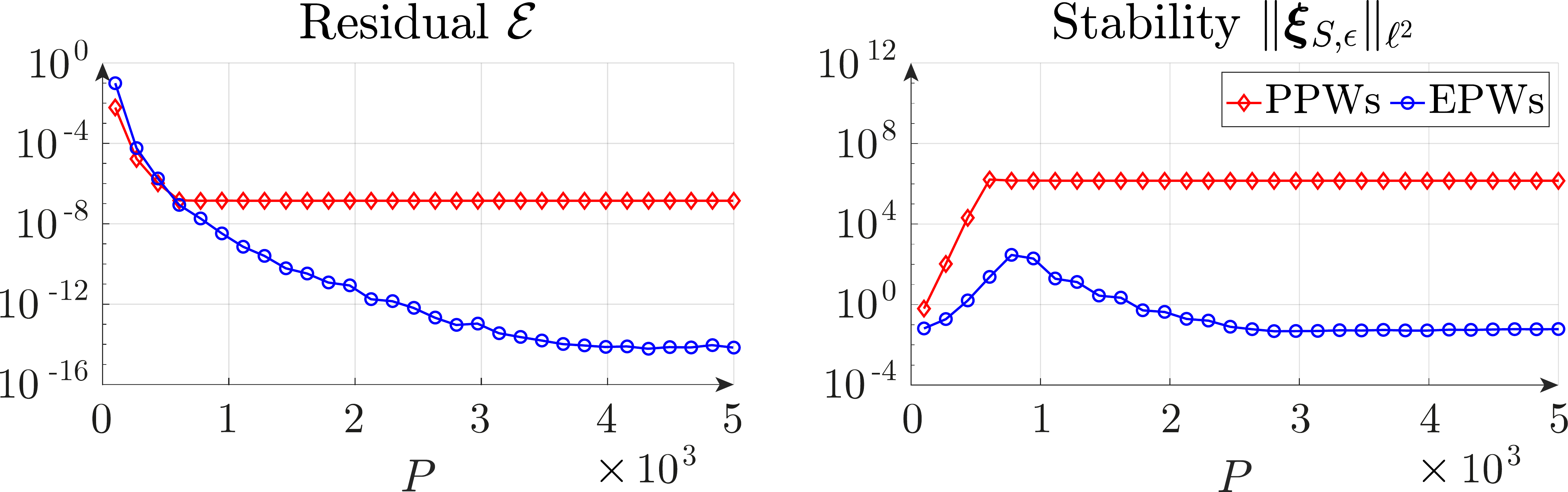}
\caption{Residual $\mathcal{E}$ as defined in (\ref{relative residual}) (left) and stability $\|\boldsymbol{\xi}_{S,\epsilon}\|_{\ell^2}$ (right) of the approximation of the fundamental solution $\Phi_{\mathbf{s}}$ presented in (\ref{fundamental solution}) with $\mathbf{s}=(2\lambda/3,0,0)\in \R^3\setminus\overline{B_1}$ by plane waves, either propagative or evanescent (where the nodes in $Y$ are selected using the Extremal–Sobol strategy (e) presented in Definition \protect\hyperlink{Definition 6.4}{6.4}). We report the convergence of the approximation for increasing size of the approximation set $P$. Wavenumber $\kappa=5$ and regularization parameter $\epsilon=10^{-14}$.}
\label{figure 7.12}
\end{figure}

\begin{figure}
\centering
\begin{tabular}{cc}
\includegraphics[trim=50 50 50 50,clip,width=.45\textwidth,valign=m]{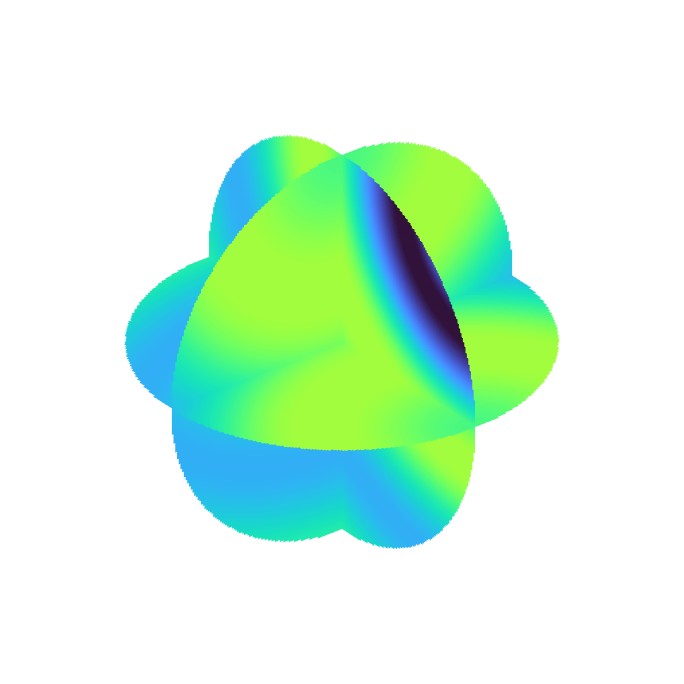} &  \includegraphics[trim=50 50 50 50,clip,width=.45\textwidth,valign=m]{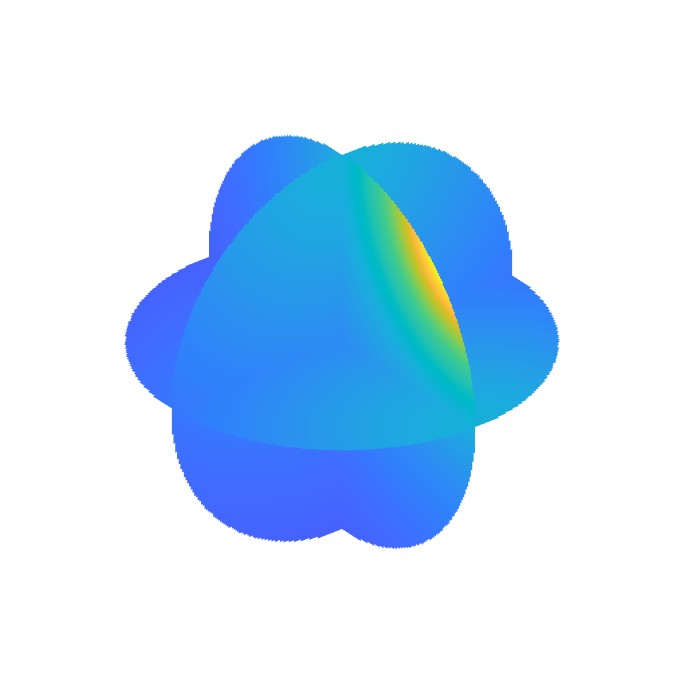}\\
\includegraphics[width=.4\textwidth]{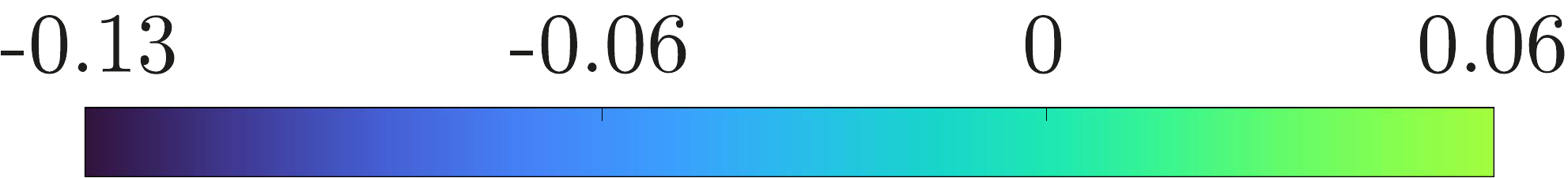} &
\includegraphics[width=.36\textwidth]{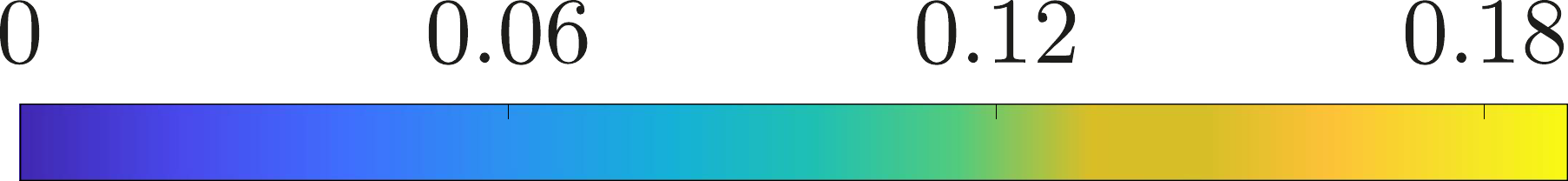}\\
Real part $\Re\{\Phi_{\mathbf{s}}\}$
& Modulus $|\Phi_{\mathbf{s}}|$\\
\rule{0pt}{3.3ex}
\includegraphics[trim=50 50 50 50,clip,width=.45\textwidth,valign=m]{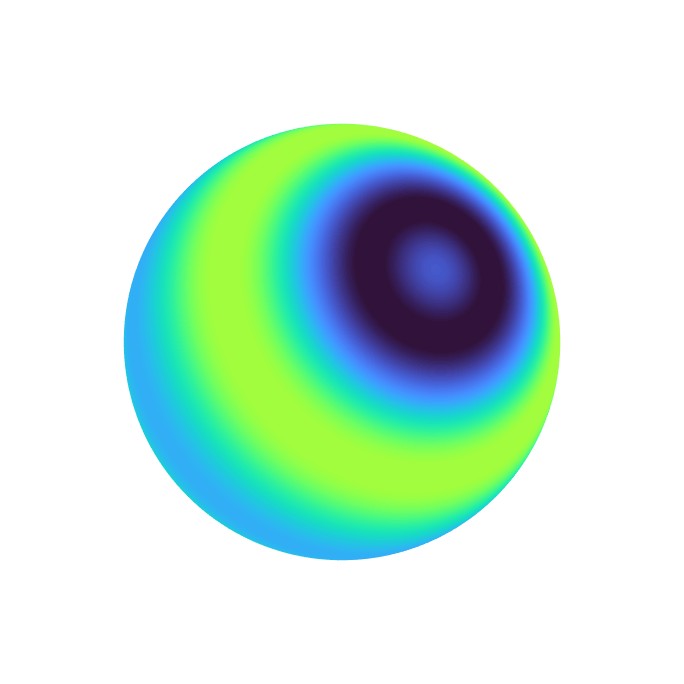} & \includegraphics[trim=50 50 50 50,clip,width=.45\textwidth,valign=m]{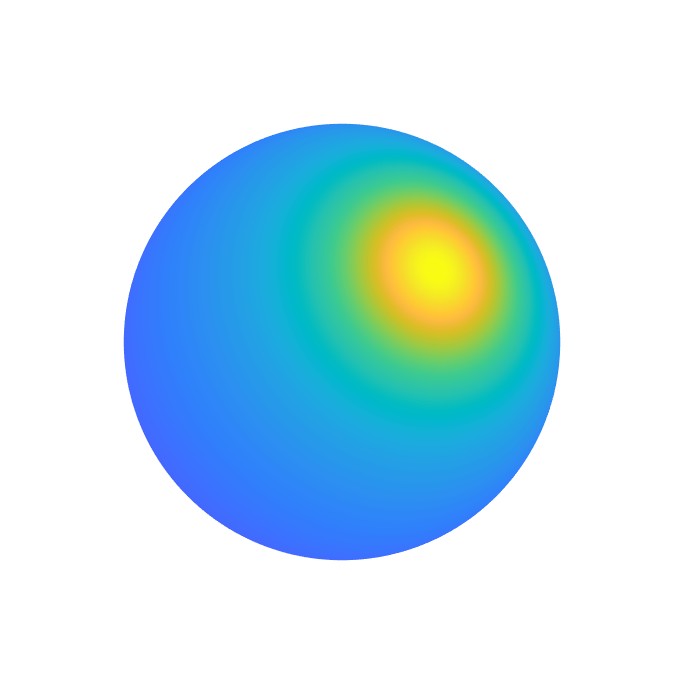}\\
\includegraphics[width=.4\textwidth]{imagespdf/bar_fundamental_real.pdf} &
\includegraphics[width=.36\textwidth]{imagespdf/bar_fundamental_abs.pdf}\\
Real part $\Re\{\Phi_{\mathbf{s}}\}$
& Modulus $|\Phi_{\mathbf{s}}|$\\
\end{tabular}
\caption{Fundamental solution $\Phi_{\mathbf{s}}$ defined in (\ref{fundamental solution}) with wavenumber $\kappa=5$ and $\mathbf{s} \in \R^3\setminus\overline{B_1}$ so that $\textup{dist}(\mathbf{s},\partial B_1)=\lambda/3$. Both the real part $\Re\{\Phi_{\mathbf{s}}\}$ and the modulus $|\Phi_{\mathbf{s}}|$ of the fundamental solution are plotted on $B_1 \cap \{\mathbf{x}=(x,y,z) : xyz=0\}$ (top) and on the unit sphere $\partial B_1$ (bottom).}
\label{figure 7.13}
\end{figure}
\begin{figure}
\centering
\begin{tabular}{cc}
\includegraphics[trim=50 50 50 50,clip,width=.45\textwidth,valign=m]{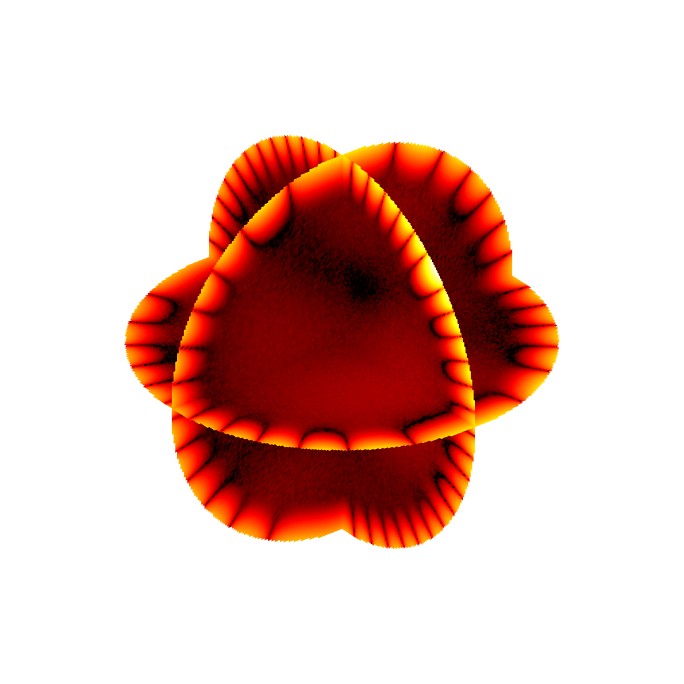} &  \includegraphics[trim=50 50 50 50,clip,width=.45\textwidth,valign=m]{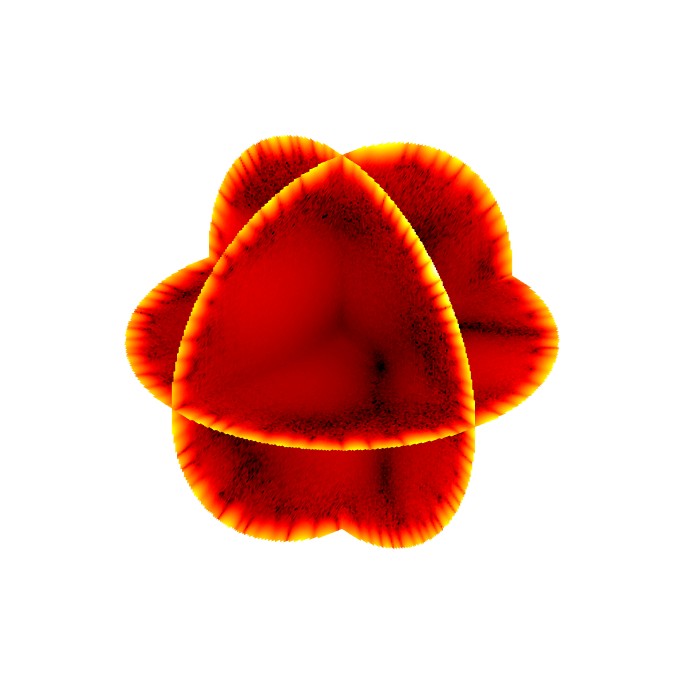}\\
\includegraphics[width=.4\textwidth]{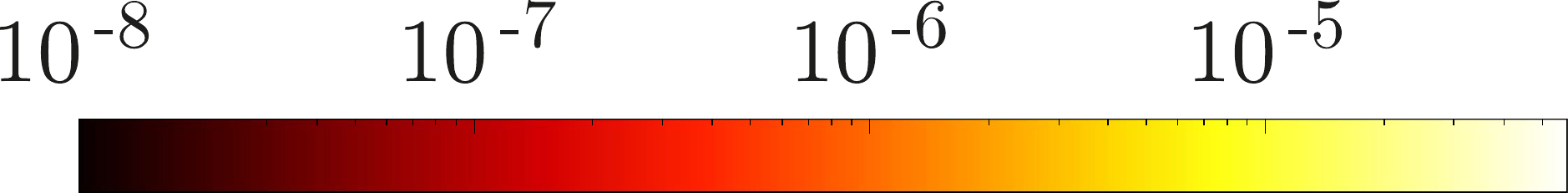} &
\includegraphics[width=.4\textwidth]{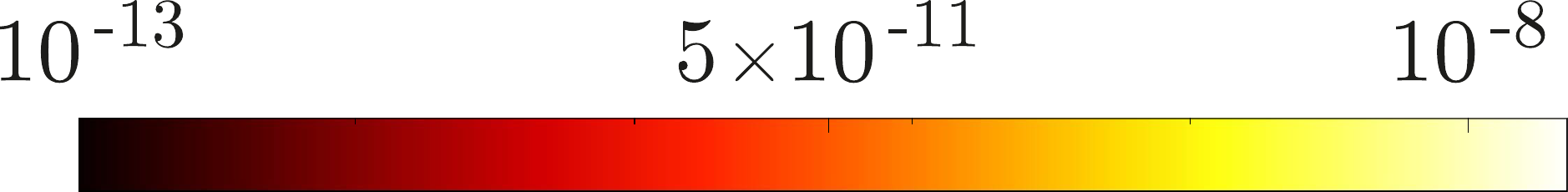}\\
Absolute error using PPWs
& Absolute error using EPWs\\
\rule{0pt}{3.3ex}
\includegraphics[trim=50 50 50 50,clip,width=.45\textwidth,valign=m]{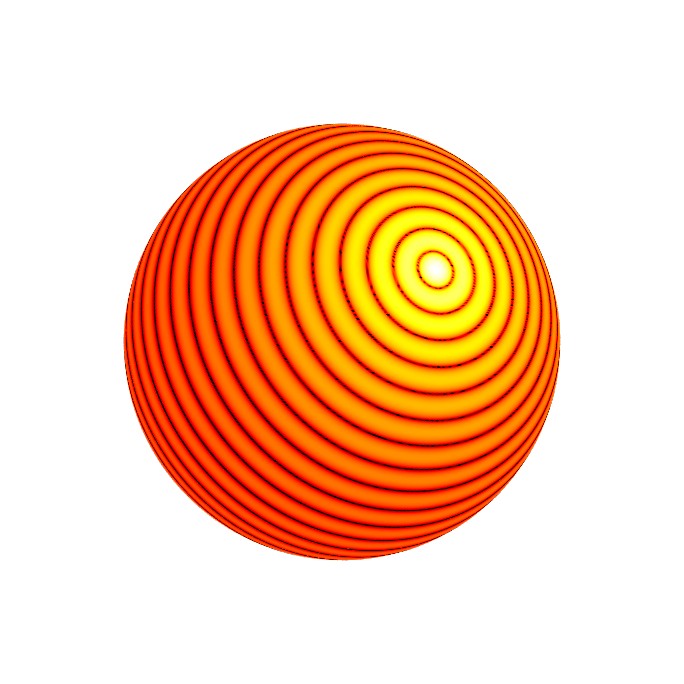} & \includegraphics[trim=50 50 50 50,clip,width=.45\textwidth,valign=m]{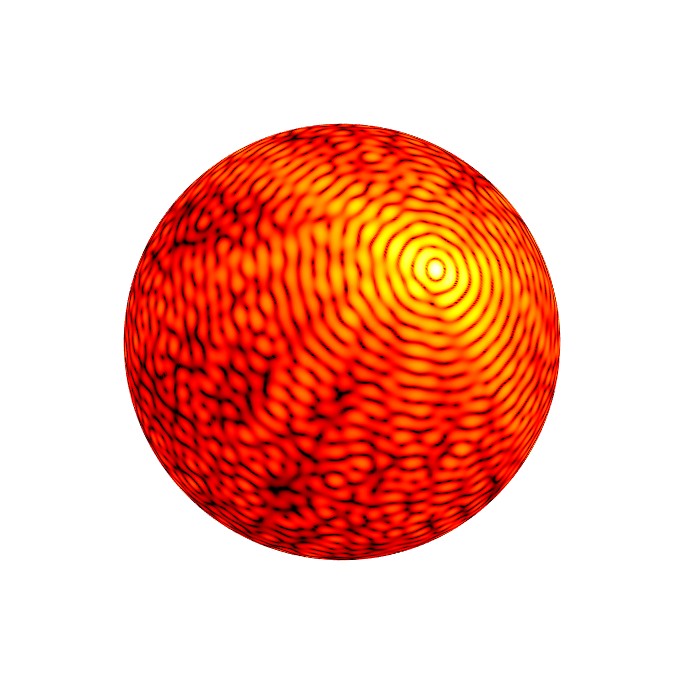}\\
\includegraphics[width=.4\textwidth]{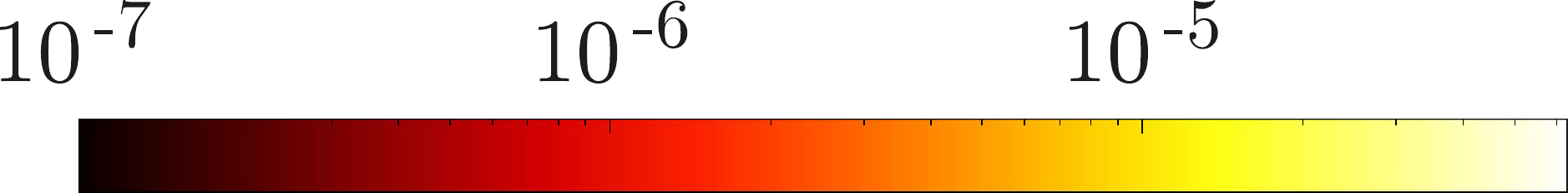} &
\includegraphics[width=.4\textwidth]{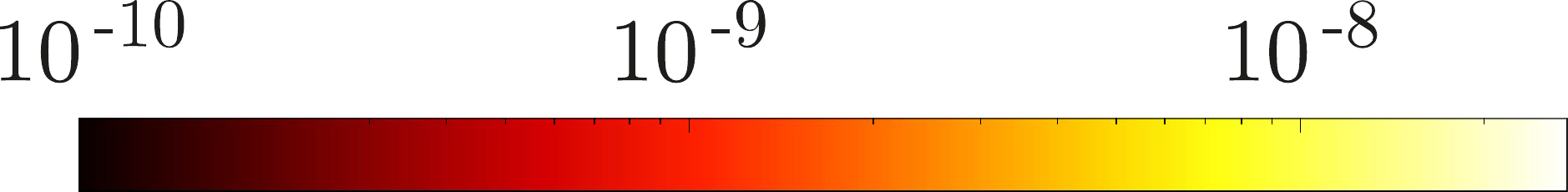}\\
Absolute error using PPWs
& Absolute error using EPWs\\
\end{tabular}
\caption{Absolute errors of the approximation of the fundamental solution $\Phi_{\mathbf{s}}$ defined in (\ref{fundamental solution}) with $\mathbf{s} \in \R^3\setminus\overline{B_1}$ so that $\textup{dist}(\mathbf{s},\partial B_1)=\lambda/3$. The error is provided using $P=2704$ plane waves, either propagative ones $\boldsymbol{\Phi}_P$ from (\ref{plane waves approximation set}) (left) or evanescent ones $\boldsymbol{\Phi}_{L,P}$ from (\ref{evanescence sets}), whose parameters are constructed using the Extremal--Sobol strategy (e) presented in Definition \protect\hyperlink{Definition 6.4}{6.4} (right). The absolute errors are plotted both on $B_1 \cap \{\mathbf{x}=(x,y,z) : xyz=0\}$ (top) and on the unit sphere $\partial B_1$ (bottom). Wavenumber $\kappa=5$ and regularization parameter $\epsilon=10^{-14}$. The results agree with the ones reported in \ref{figure 7.11}.}
\label{figure 7.14}
\end{figure}

\begin{figure}
\centering
\begin{tabular}{cc}
\includegraphics[trim=50 50 50 50,clip,width=.45\textwidth,valign=m]{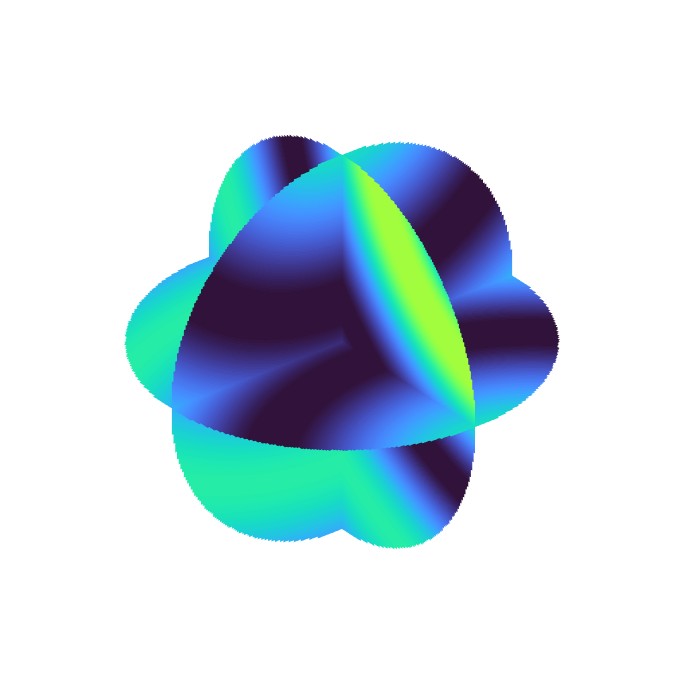} &  \includegraphics[trim=50 50 50 50,clip,width=.45\textwidth,valign=m]{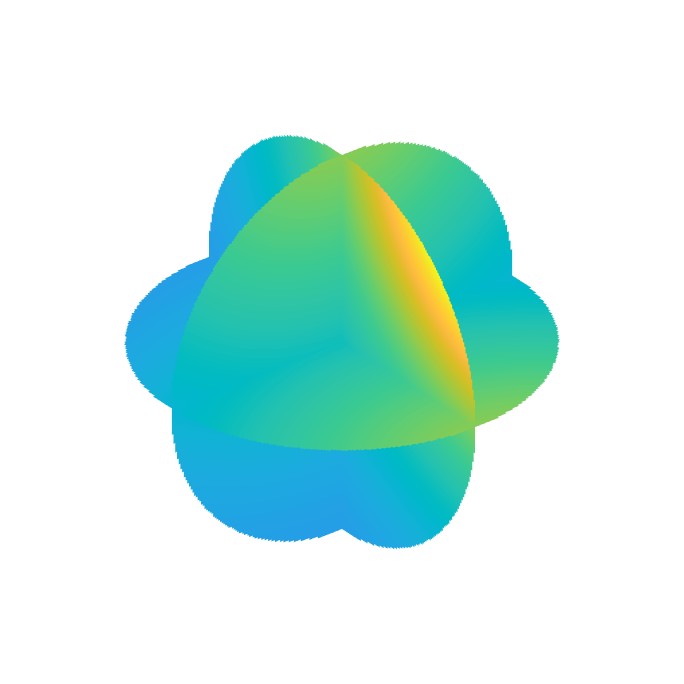}\\
\includegraphics[width=.4\textwidth]{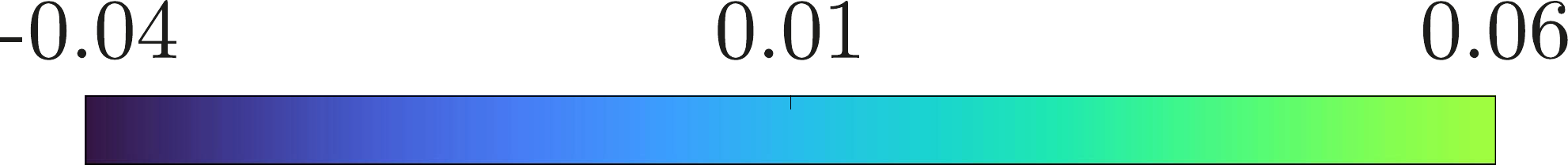} &
\includegraphics[width=.385\textwidth]{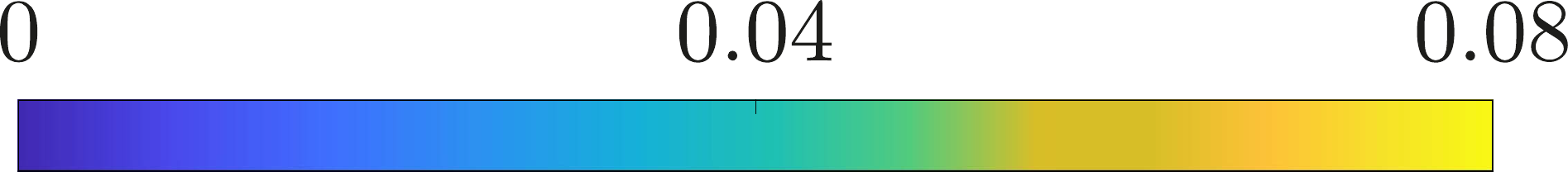}\\
Real part $\Re\{\Phi_{\mathbf{s}}\}$
& Modulus $|\Phi_{\mathbf{s}}|$\\
\rule{0pt}{3.3ex}
\includegraphics[trim=50 50 50 50,clip,width=.45\textwidth,valign=m]{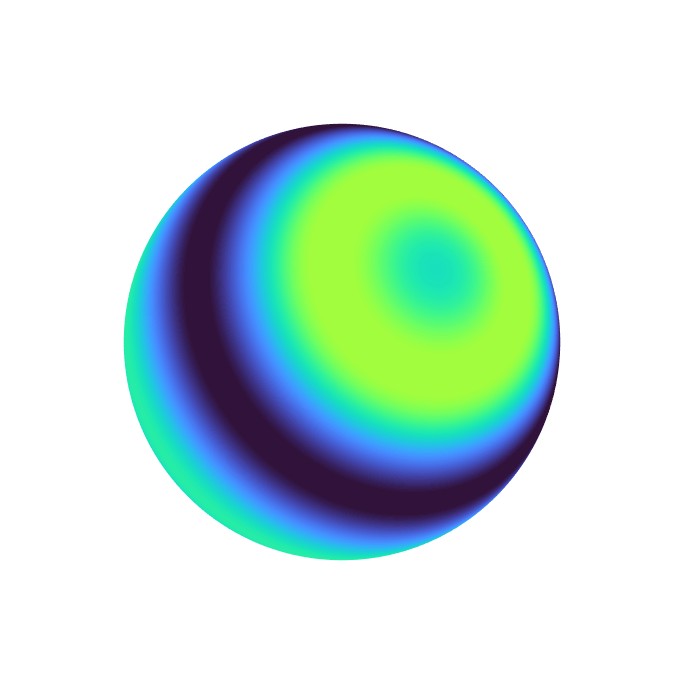} & \includegraphics[trim=50 50 50 50,clip,width=.45\textwidth,valign=m]{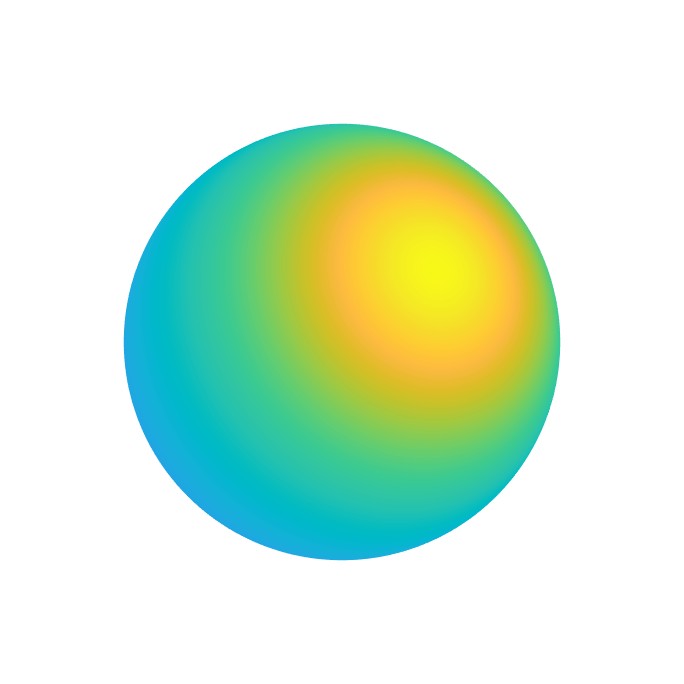}\\
\includegraphics[width=.4\textwidth]{imagespdf/bar_re.pdf} &
\includegraphics[width=.385\textwidth]{imagespdf/bar_ab.pdf}\\
Real part $\Re\{\Phi_{\mathbf{s}}\}$
& Modulus $|\Phi_{\mathbf{s}}|$\\
\end{tabular}
\caption{Fundamental solution $\Phi_{\mathbf{s}}$ defined in (\ref{fundamental solution}) with wavenumber $\kappa=5$ and $\mathbf{s} \in \R^3\setminus\overline{B_1}$ so that $\textup{dist}(\mathbf{s},\partial B_1)=4\lambda/5$. Both the real part $\Re\{\Phi_{\mathbf{s}}\}$ and the modulus $|\Phi_{\mathbf{s}}|$ of the fundamental solution are plotted on $B_1 \cap \{\mathbf{x}=(x,y,z) : xyz=0\}$ (top) and on the unit sphere $\partial B_1$ (bottom).}
\label{figure 7.13 bis}
\end{figure}
\begin{figure}
\centering
\begin{tabular}{cc}
\includegraphics[trim=50 50 50 50,clip,width=.45\textwidth,valign=m]{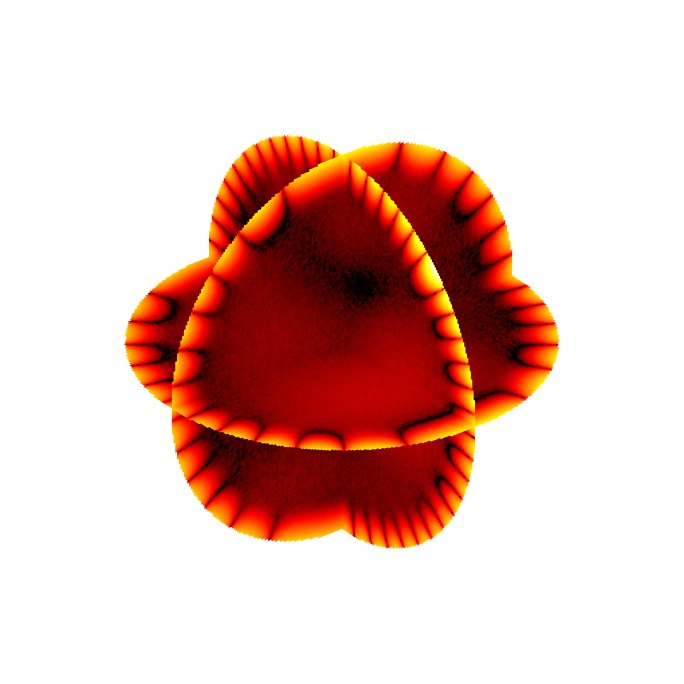} &  \includegraphics[trim=50 50 50 50,clip,width=.45\textwidth,valign=m]{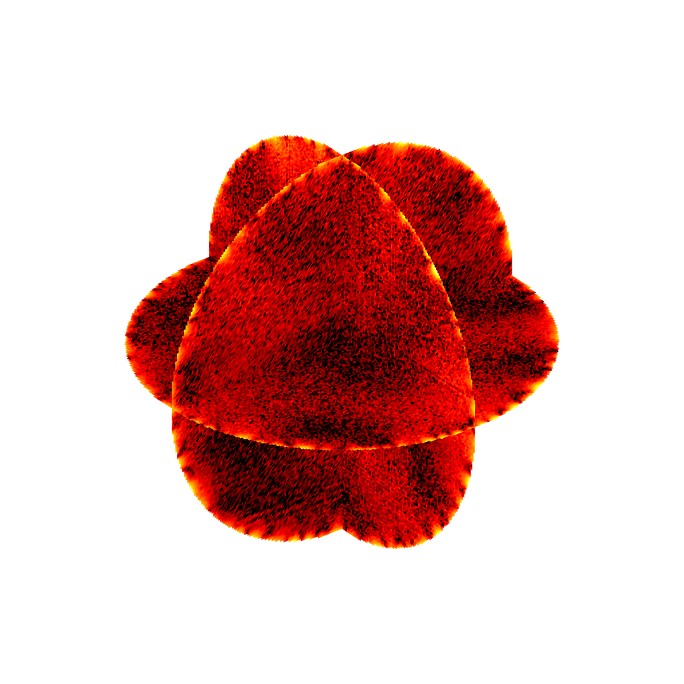}\\
\includegraphics[width=.4\textwidth]{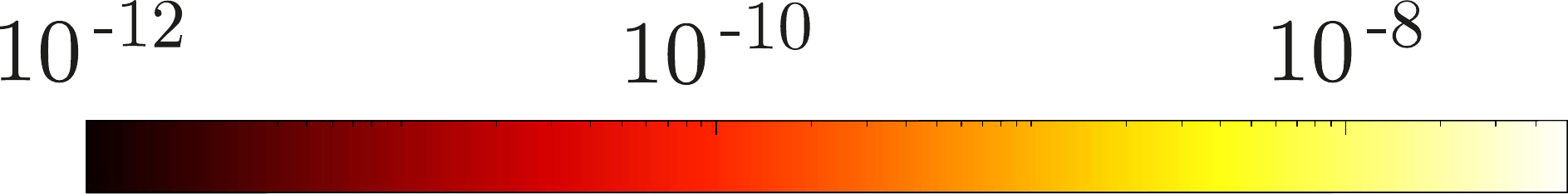} &
\includegraphics[width=.4\textwidth]{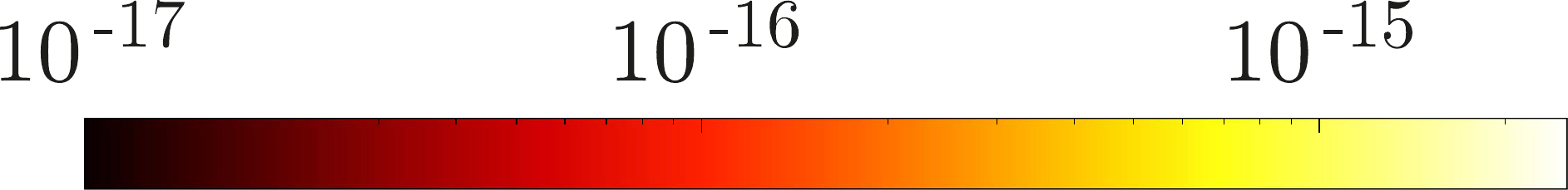}\\
Absolute error using PPWs
& Absolute error using EPWs\\
\rule{0pt}{3.3ex}
\includegraphics[trim=50 50 50 50,clip,width=.45\textwidth,valign=m]{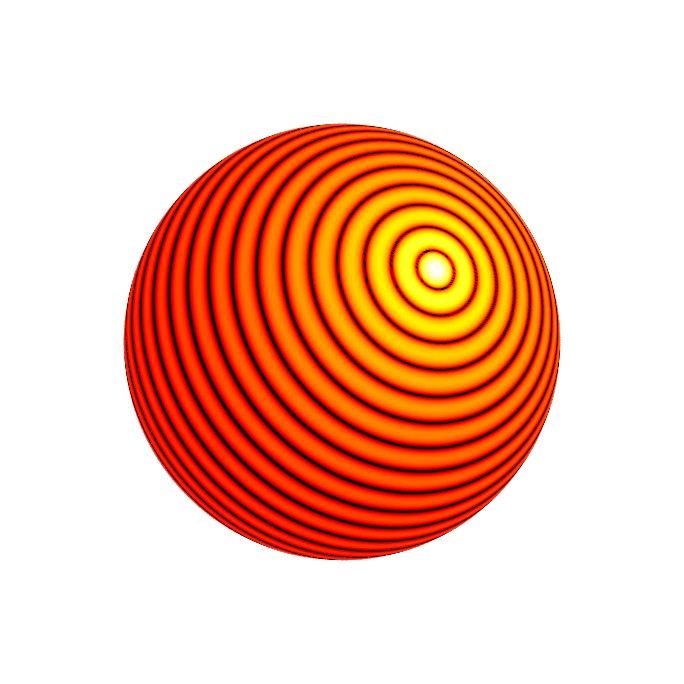} & \includegraphics[trim=50 50 50 50,clip,width=.45\textwidth,valign=m]{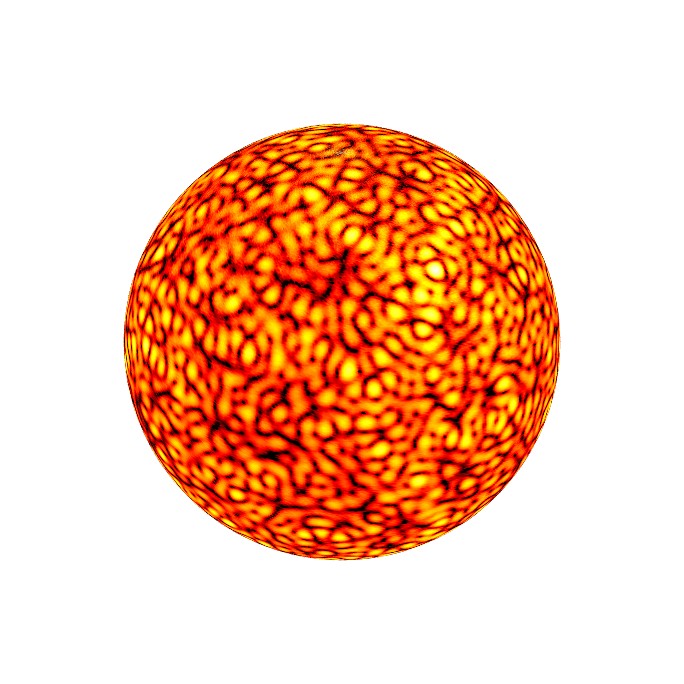}\\
\includegraphics[width=.4\textwidth]{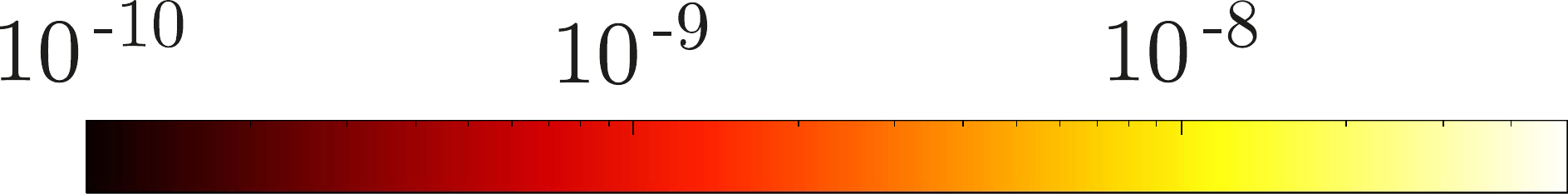} &
\includegraphics[width=.4\textwidth]{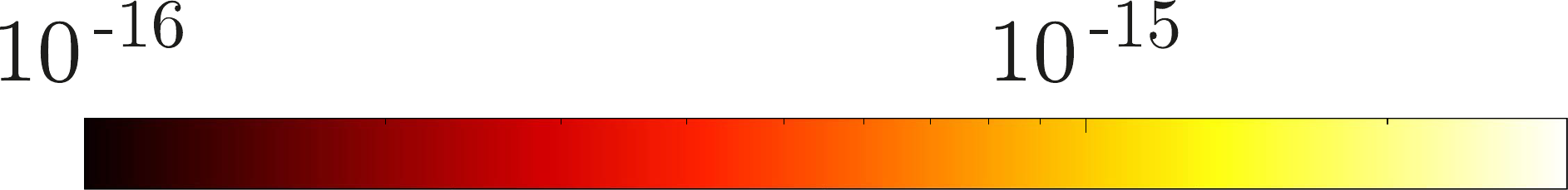}\\
Absolute error using PPWs
& Absolute error using EPWs\\
\end{tabular}
\caption{The caption of Figure \ref{figure 7.14} applies here as well, with the only difference that now the singularity $\mathbf{s} \in \R^3\setminus\overline{B_1}$ of the fundamental solution $\Phi_{\mathbf{s}}$ is chosen so that $\textup{dist}(\mathbf{s},\partial B_1)=4\lambda/5$. The results agree with the ones reported in \ref{figure 7.11}.}
\label{figure 7.14 bis}
\end{figure}

\section{Different geometries}

\begin{figure}
\centering
\begin{tabular}{cc}
\includegraphics[trim=20 20 20 20,clip,width=.4\textwidth,valign=m]{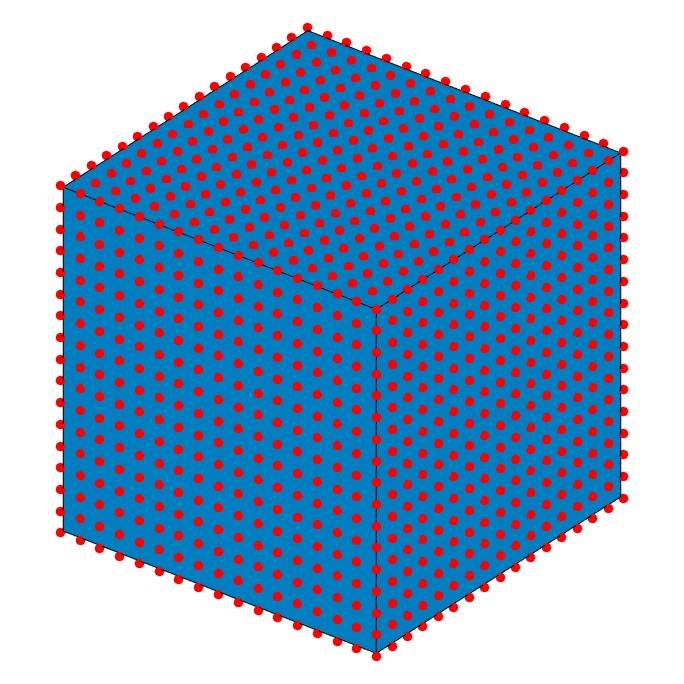} &  \includegraphics[trim=60 60 60 60,clip,width=.45\textwidth,valign=m]{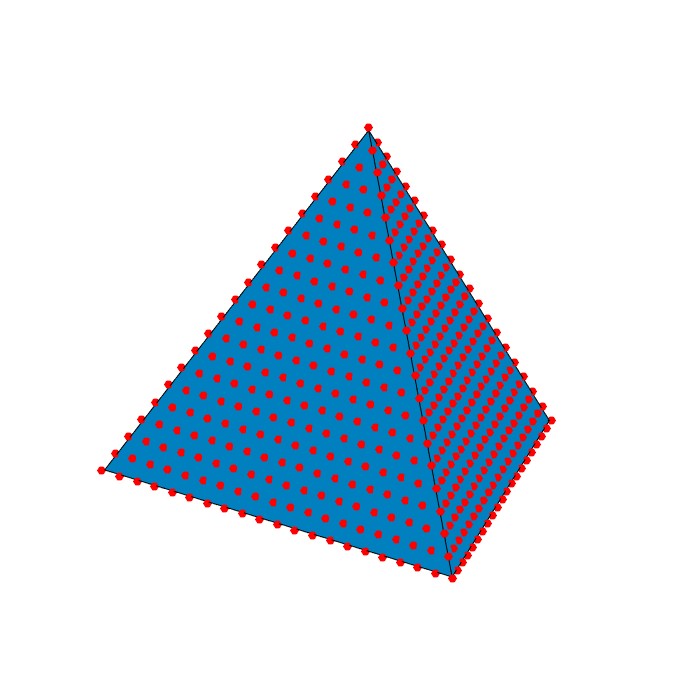}\\
\end{tabular}
\caption{Equispaced sampling points $\{\mathbf{x}_s\}_{s=1}^S \subset \partial \Omega$. Equispaced nodes are initially placed along the edges. Following this, for every face, a sequence of parallel equispaced node sets is generated starting from an edge. In case $\Omega=Q_1$ (left), the number of nodes in each sequence is the same. If $\Omega=T_1$ (right), the number of nodes decreases gradually by one until the opposite vertex is reached.}
\label{figure 7.15}
\end{figure}

\hypertarget{Section 7.4}{In} conclusion of this chapter, we present some numerical results on various shapes to show that the approximation sets we developed, based on the analysis of the unit ball $B_1$, perform well on other geometries too.
Once again, the objective of the approximation problem is to determine the fundamental solution of the Helmholtz equation $\Phi_{\mathbf{s}}$, as defined in (\ref{fundamental solution}), but this time for a convex polyhedron $\Omega \subset \R^3$ inscribed in the unit sphere and with the singularity $\mathbf{s}$ belonging to $\R^3\setminus\overline{\Omega}$.
We will consider both $\Omega=Q_1$, where $Q_1$ is the cube with the edges parallel to the Cartesian axes, and $\Omega=T_1$, where $T_1$ is the regular tetrahedron of vertices

\footnotesize 
\begin{align*}
\mathbf{v}_1&=\left(-\sqrt{\frac{2}{9}},\sqrt{\frac{2}{3}},-\frac{1}{3}\right), & \mathbf{v}_2&=\left(-\sqrt{\frac{2}{9}},-\sqrt{\frac{2}{3}},-\frac{1}{3}\right), &
\mathbf{v}_3&=\left(\sqrt{\frac{8}{9}},0,-\frac{1}{3}\right), & \mathbf{v}_4&=\left(0,0,1\right).
\end{align*}
\medskip

\normalsize
We re-examine the convergence of the plane wave approximation as the size of the approximation set $P$ increases.
The approximation recipe is borrowed from Section \hyperlink{Section 2.2}{2.2}: it involves using equispaced Dirichlet data points $\{\mathbf{x}_s\}_{s=1}^S$ on the boundary of $\Omega$ (see Figure \ref{figure 7.15}) and solving over-sampled linear systems using a regularized SVD. Note that, since evenly spaced sampling points are employed, we can choose $w_s=|\partial \Omega|/S$ for every $s=1,...,S$ in (\ref{A matrix definition}). The approximation sets consist of two types of plane waves: propagative, with directions determined by the extremal point systems (\ref{G matrix}), and evanescent, as described in (\ref{evanescence sets}). The construction of the evanescent plane wave set follows the Extremal--Sobol strategy (e) which is outlined in Definition \protect\hyperlink{Definition 6.4}{6.4}. The truncation parameter $L$ is calculated based on the dimension $P$ of the approximation set, as $L:=\max\{\lceil \kappa\rceil,\lfloor\sqrt{P/10}\rfloor\}$. Finally, the evanescent plane waves are normalized to have a unit $L^{\infty}$ norm on the boundary $\partial \Omega$, which is the only variation from the sets used for the spherical geometry.

Figure \ref{figure 7.16} illustrates the convergence of the plane wave approximation as the size of the approximation set $P$ increases. We consider the fundamental solution $\Phi_{\mathbf{s}}$ defined in (\ref{fundamental solution}) with wavenumber $\kappa=5$ and singularity $\mathbf{s}=(1/\sqrt{3}+2\lambda/3,0,0)$, if $\Omega=Q_1$, and $\mathbf{s}=(0,0,-1/3-2\lambda/3)$, if $\Omega=T_1$, so that in both cases $\text{dist}(\mathbf{s},\partial \Omega)=2\lambda/3$ as in Figure \ref{figure 7.12}.
The results are consistent with those presented in Section \hyperlink{Section 7.3}{7.3}: evanescent plane waves are able to approximate more modes providing better accuracy, on the contrary of propagative plane waves which instead stall before reaching machine precision due to the rapidly growing coefficients.
However, it should be noted that the use of evanescent plane waves results in a slower convergence rate compared to the scenario depicted in Figure \ref{figure 7.12}, where a spherical geometry is considered.

\begin{figure}
\centering
\includegraphics[width=\linewidth]{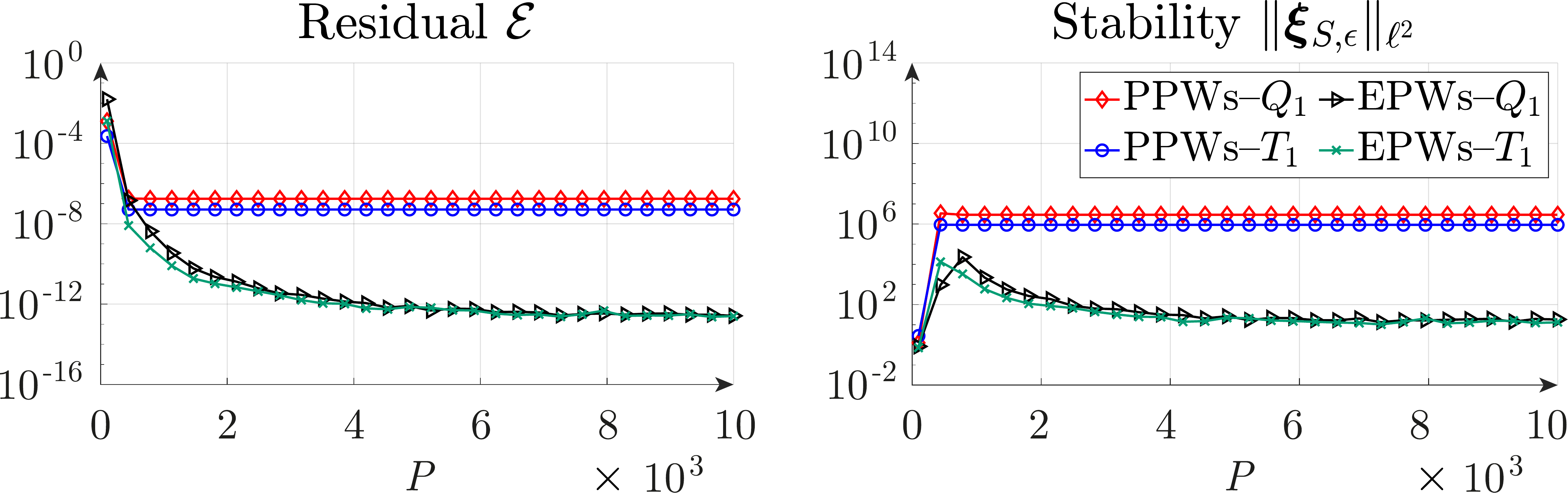}
\caption{Residual $\mathcal{E}$ as defined in (\ref{relative residual}) (left) and stability $\|\boldsymbol{\xi}_{S,\epsilon}\|_{\ell^2}$ (right) of the approximation of the fundamental solution $\Phi_{\mathbf{s}}$ presented in (\ref{fundamental solution}), both in $Q_1$, with $\mathbf{s}=(1/\sqrt{3}+2\lambda/3,0,0)$, and in $T_1$, with $\mathbf{s}=(0,0,-1/3-2\lambda/3)$. Both plane wave approximation sets are employed, either propagative or evanescent (where the nodes in $Y$ are selected using the Extremal–Sobol strategy (e) presented in Definition \protect\hyperlink{Definition 6.4}{6.4}). We report the convergence of the approximation for increasing size of the approximation set $P$. Wavenumber $\kappa=5$ and regularization parameter $\epsilon=10^{-14}$.}
\label{figure 7.16}
\end{figure}

Lastly, we show some examples of the fundamental solution $\Phi_{\mathbf{s}}$ with wavenumber $\kappa=5$ and several choices of the singularity $\mathbf{s}$, both in $Q_1$ and in $T_1$, along with the errors in approximation using plane waves, whether they are propagative or evanescent (see from Figure \ref{figure 7.18} to Figure \ref{figure 7.23}).

It is worth noting that the previously outlined numerical recipe, in which equispaced Dirichlet data points $\{\mathbf{x}_s\}_{s=1}^S \subset \partial \Omega$ are considered, seems to yield some inaccuracies near the corners. For this reason, in Figure \ref{figure 7.20} and in Figure \ref{figure 7.23} we report some numerical experiments where we choose locally-refined sampling points. In this case, we set the weights $\mathbf{w}_S \in \R^S$ in the linear system (\ref{linear system}) to constant.
To construct the points grid, Chebyshev nodes are first placed along the edges. Then, for each face, a sequence of parallel
Chebyshev node set is generated in parallel starting from an edge. If $\Omega=Q_1$, the number of nodes in each sequence remains the same, whereas, if $\Omega=T_1$, the number of nodes is gradually reduced by one until the opposite vertex is reached.
Observe that one possible approach to obtain a set of equispaced sampling points on $\partial \Omega$ (such as those depicted in Figure \ref{figure 7.15}) is to use equispaced nodes, rather than Chebyshev nodes, at each step. Some examples of locally-refined grids are depicted in Figure \ref{figure 7.17}.

The accuracy and stability analysis of this method (not plotted) is totally analogous to the one reported in Figure \ref{figure 7.16}, but both Figure \ref{figure 7.20} and Figure \ref{figure 7.23} show greater accuracy near the corners if compared to Figure \ref{figure 7.19} and Figure \ref{figure 7.22}, where instead equispaced nodes are considered.

These results show the promising prospects of the suggested numerical approach for plane wave approximations and Trefftz methods. Notably, the results are quite impressive, considering that the numerical method used to develop the approximations is not fine-tuned for these particular geometries, apart from the $L^{\infty}$ re-normalization on the boundary and, possibly, the choice of Dirichlet sampling points $\{\mathbf{x}_s\}_{s=1}^S \subset \partial \Omega$. We are confident that the outlined recipe could be improved defining better rules that are specifically tailored to the underlying geometries and hence lead to even more effective approximation schemes.

\begin{figure}
\centering
\begin{tabular}{cc}
\includegraphics[trim=20 20 20 20,clip,width=.4\textwidth,valign=m]{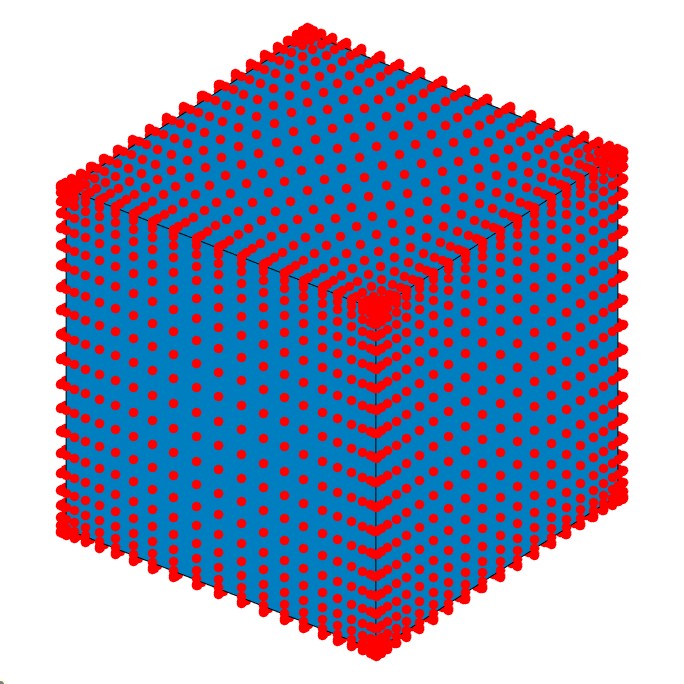} &  \includegraphics[trim=60 60 60 60,clip,width=.45\textwidth,valign=m]{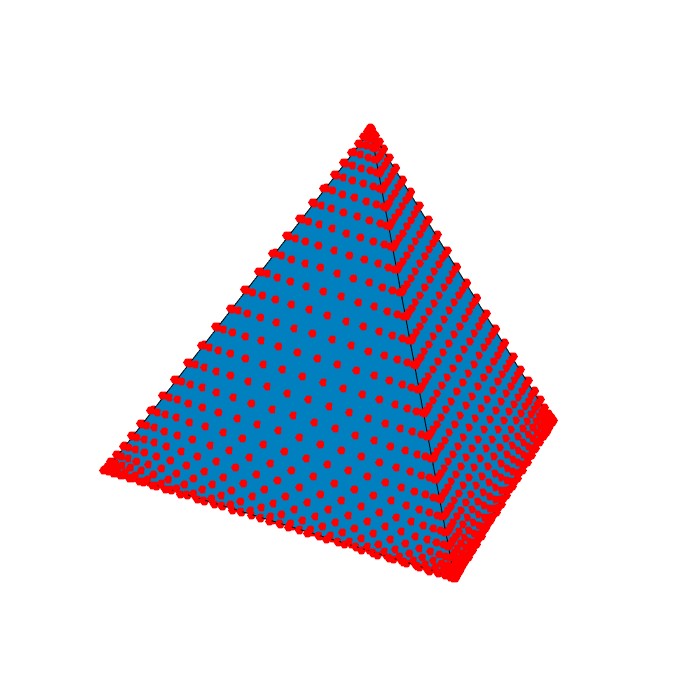}\\
\end{tabular}
\caption{Locally-refined sampling points $\{\mathbf{x}_s\}_{s=1}^S \subset \partial \Omega$. Chebyshev nodes are initially placed along the edges. Following this, for every face, a sequence of parallel Chebyshev node sets is generated starting from an edge. In case $\Omega=Q_1$ (left), the number of nodes in each sequence is the same. If $\Omega=T_1$ (right), the number of nodes decreases gradually by one until the opposite vertex is reached.}
\label{figure 7.17}
\end{figure}

\begin{figure}
\centering
\begin{tabular}{cc}
\includegraphics[trim=20 20 20 20,clip,width=.43\textwidth,valign=m]{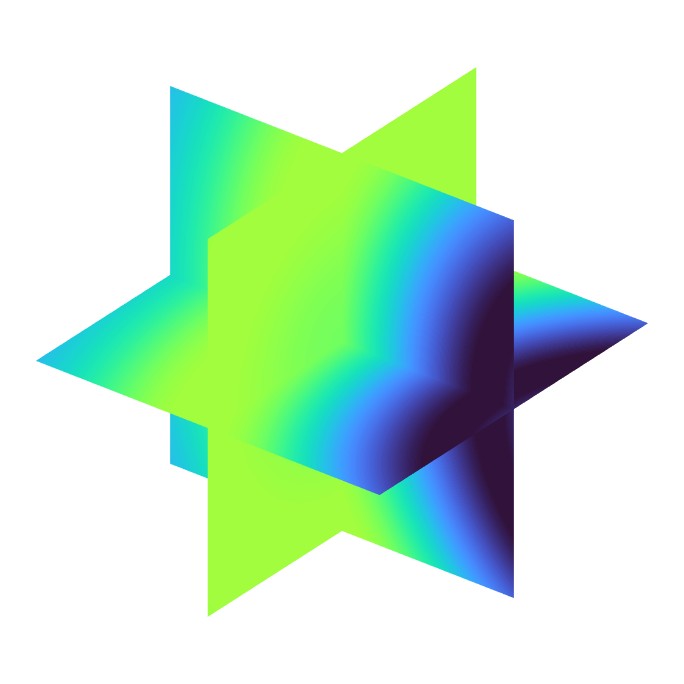} &  \includegraphics[trim=20 20 20 20,clip,width=.43\textwidth,valign=m]{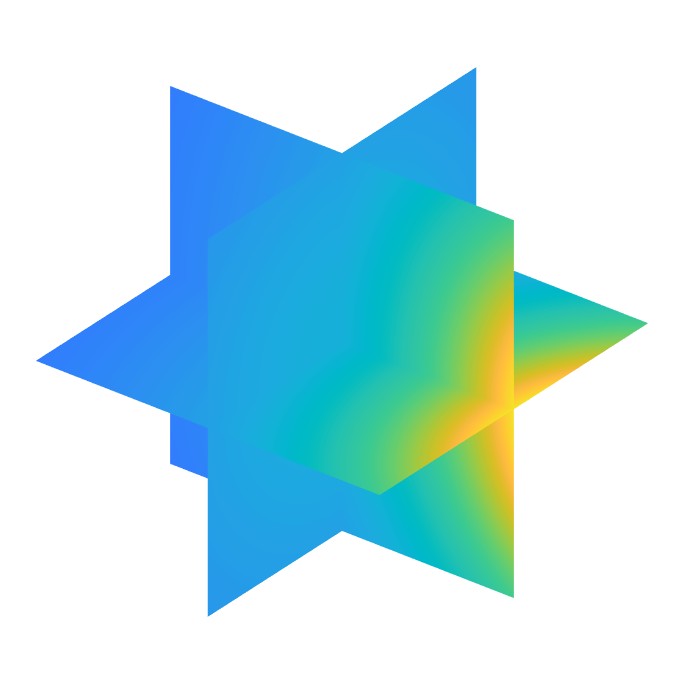}\\
\includegraphics[width=.4\textwidth]{imagespdf/bar_fundamental_real.pdf} &
\includegraphics[width=.36\textwidth]{imagespdf/bar_fundamental_abs.pdf}\\
\vspace{1cm}
Real part $\Re\{\Phi_{\mathbf{s}}\}$
& Modulus $|\Phi_{\mathbf{s}}|$\\
\includegraphics[trim=0 0 0 0,clip,width=.38\textwidth,valign=m]{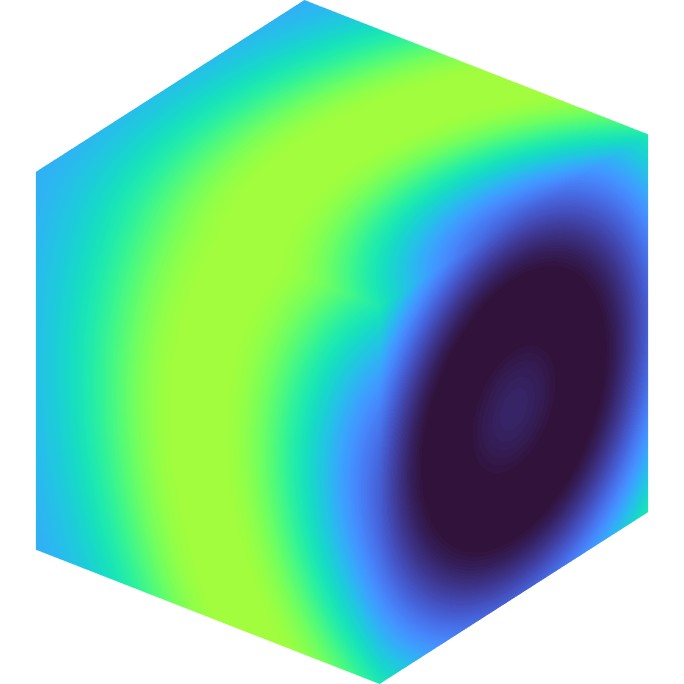} & \includegraphics[trim=0 0 0 0,clip,width=.38\textwidth,valign=m]{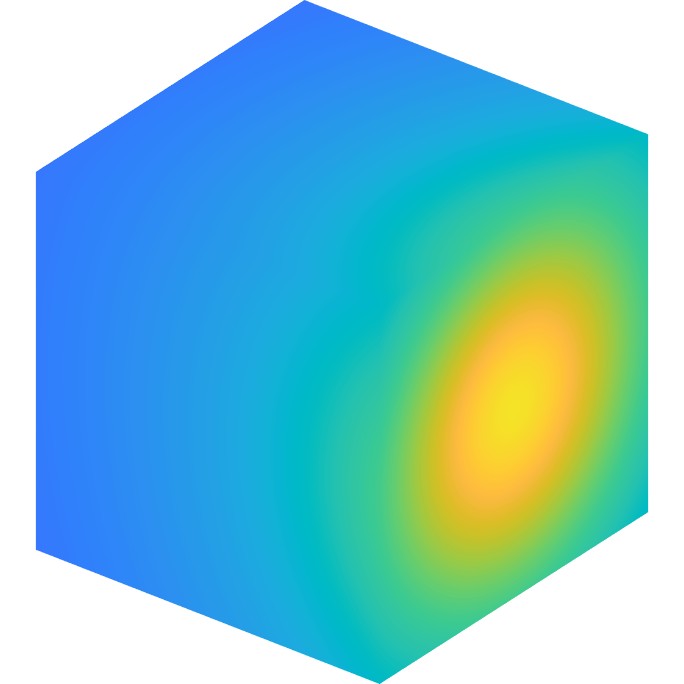}\\
\vspace{2mm}\\
\includegraphics[width=.4\textwidth]{imagespdf/bar_fundamental_real.pdf} &
\includegraphics[width=.36\textwidth]{imagespdf/bar_fundamental_abs.pdf}\\
Real part $\Re\{\Phi_{\mathbf{s}}\}$
& Modulus $|\Phi_{\mathbf{s}}|$\\
\end{tabular}
\caption{Fundamental solution $\Phi_{\mathbf{s}}$ defined in (\ref{fundamental solution}) with wavenumber $\kappa=5$ and $\mathbf{s} \in \R^3\setminus\overline{Q_1}$ so that $\textup{dist}(\mathbf{s},\partial Q_1)=\lambda/3$. Both the real part $\Re\{\Phi_{\mathbf{s}}\}$ and the modulus $|\Phi_{\mathbf{s}}|$ of the fundamental solution are plotted on $Q_1 \cap \{\mathbf{x}=(x,y,z) : xyz=0\}$ (top) and on the boundary $\partial Q_1$ (bottom).}
\label{figure 7.18}
\end{figure}
\begin{figure}
\centering
\begin{tabular}{cc}
\includegraphics[trim=20 20 20 20,clip,width=.43\textwidth,valign=m]{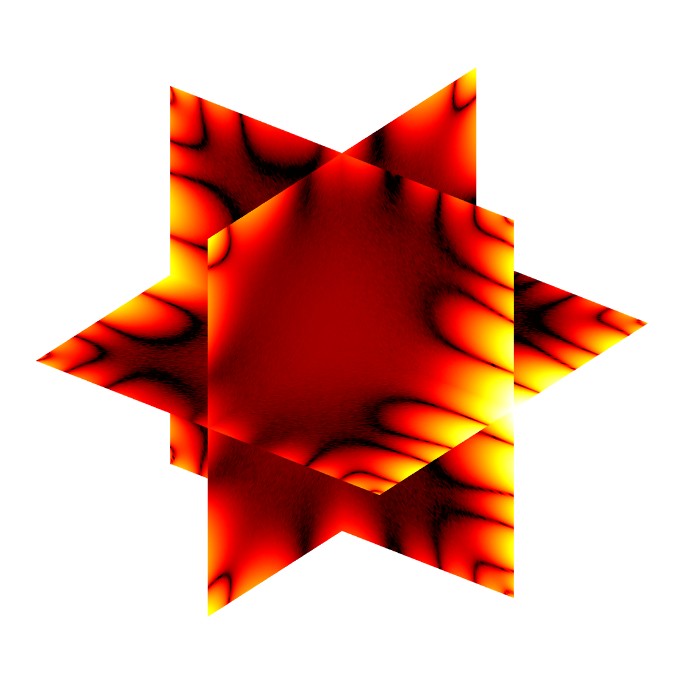} &  \includegraphics[trim=20 20 20 20,clip,width=.43\textwidth,valign=m]{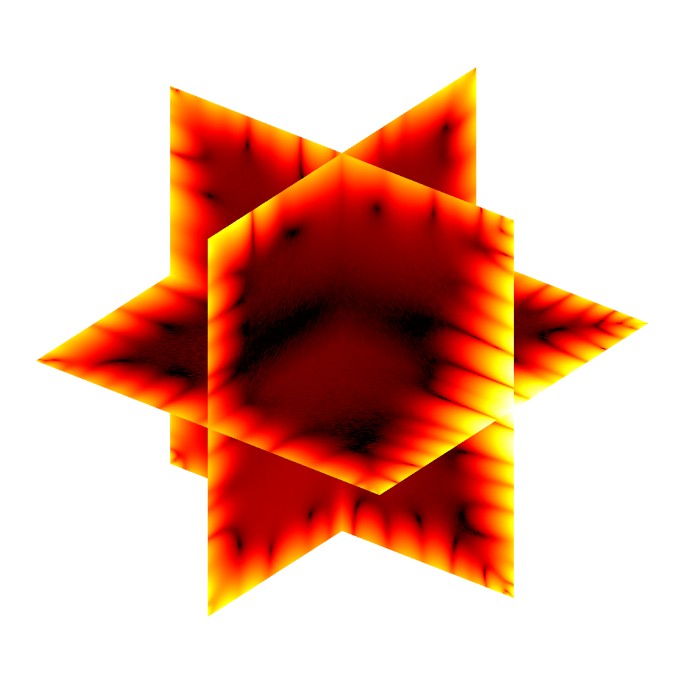}\\
\includegraphics[width=.4\textwidth]{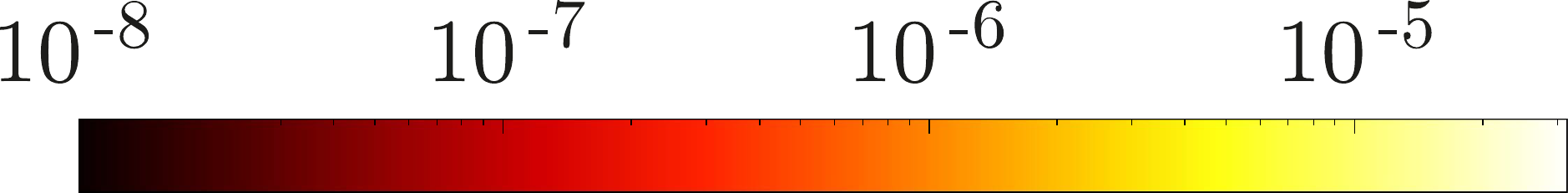} &
\includegraphics[width=.4\textwidth]{imagespdf//bar_evanescent_cube.pdf}\\
\vspace{1cm}
Absolute error using PPWs
& Absolute error using EPWs\\
\includegraphics[trim=0 0 0 0,clip,width=.38\textwidth,valign=m]{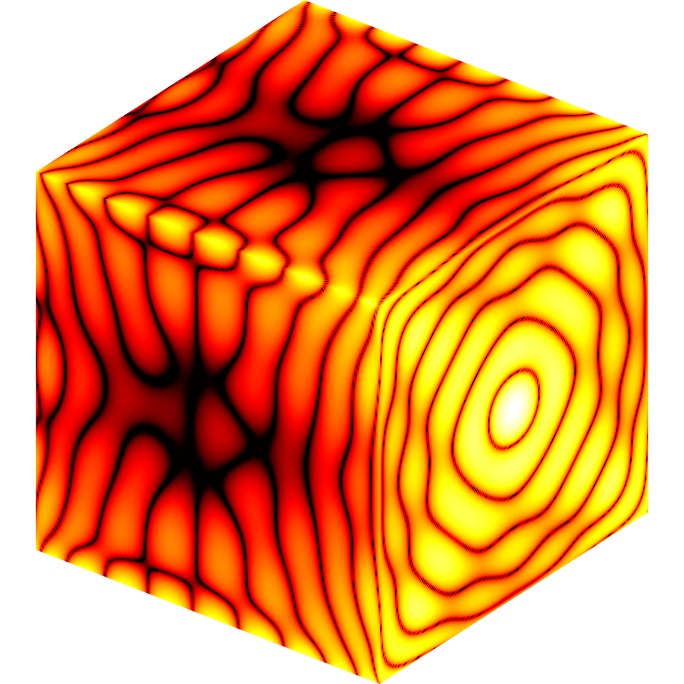} & \includegraphics[trim=0 0 0 0,clip,width=.38\textwidth,valign=m]{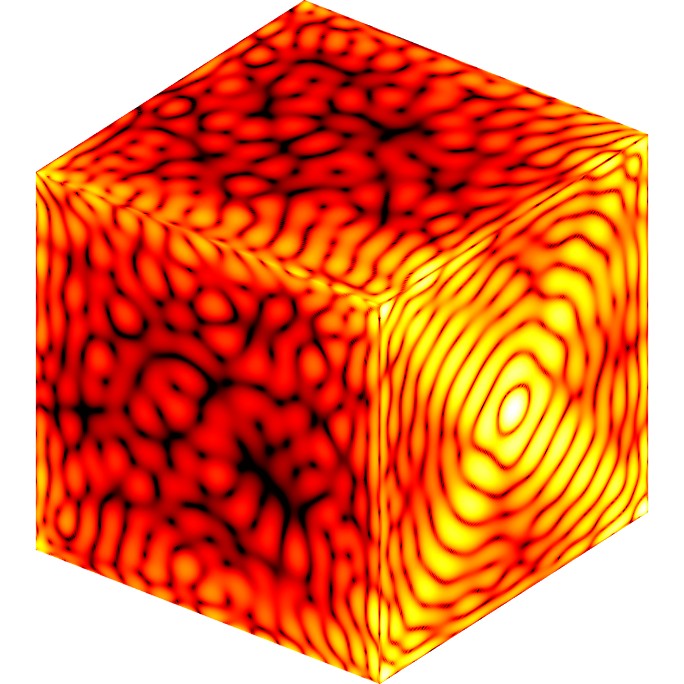}\\
\vspace{2mm}\\
\includegraphics[width=.4\textwidth]{imagespdf//bar_propagative_cube.pdf} &
\includegraphics[width=.4\textwidth]{imagespdf//bar_evanescent_cube2.pdf}\\
Absolute error using PPWs
& Absolute error using EPWs\\
\end{tabular}
\caption{Absolute errors of the approximation of the fundamental solution $\Phi_{\mathbf{s}}$ defined in (\ref{fundamental solution}) with $\mathbf{s} \in \R^3\setminus\overline{Q_1}$ so that $\textup{dist}(\mathbf{s},\partial Q_1)=\lambda/3$. The error is provided using $P=2704$ plane waves, either propagative ones $\boldsymbol{\Phi}_P$ from (\ref{plane waves approximation set}) (left) or evanescent ones $\boldsymbol{\Phi}_{L,P}$ from (\ref{evanescence sets}), whose parameters are constructed using the Extremal--Sobol strategy (e) presented in Definition \protect\hyperlink{Definition 6.4}{6.4} (right).
In both cases, the Dirichlet data sampling nodes on the boundary $\partial Q_1$ are equispaced.
The absolute errors are plotted both on $Q_1 \cap \{\mathbf{x}=(x,y,z) : xyz=0\}$ (top) and on the boundary $\partial Q_1$ (bottom). Wavenumber $\kappa=5$ and regularization parameter $\epsilon=10^{-14}$.}
\label{figure 7.19}
\end{figure}
\begin{figure}
\centering
\begin{tabular}{cc}
\includegraphics[trim=20 20 20 20,clip,width=.43\textwidth,valign=m]{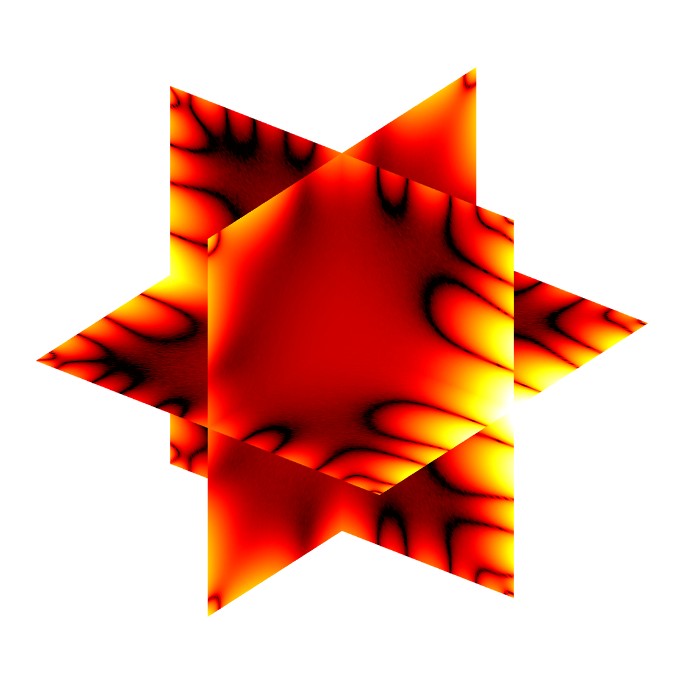} &  \includegraphics[trim=20 20 20 20,clip,width=.43\textwidth,valign=m]{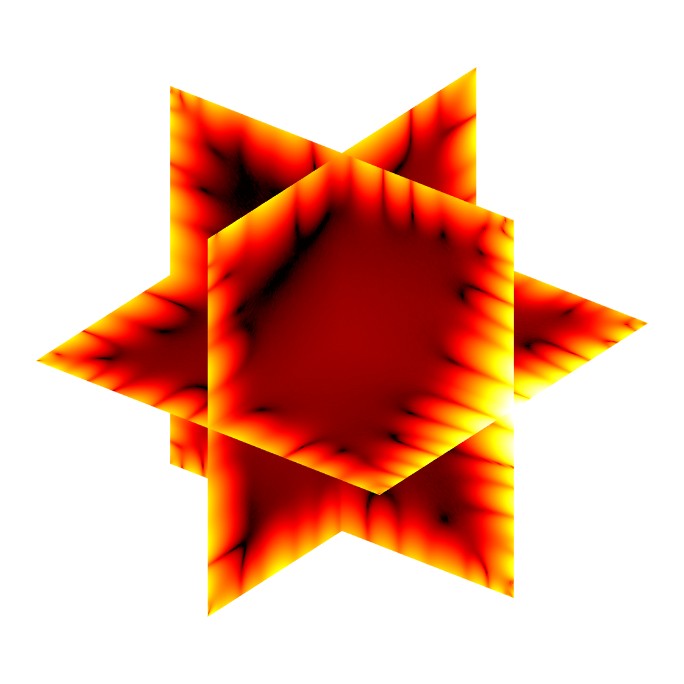}\\
\includegraphics[width=.4\textwidth]{imagespdf//bar_propagative_cube.pdf} &
\includegraphics[width=.4\textwidth]{imagespdf//bar_evanescent_cube.pdf}\\
\vspace{1cm}
Absolute error using PPWs
& Absolute error using EPWs\\
\includegraphics[trim=0 0 0 0,clip,width=.38\textwidth,valign=m]{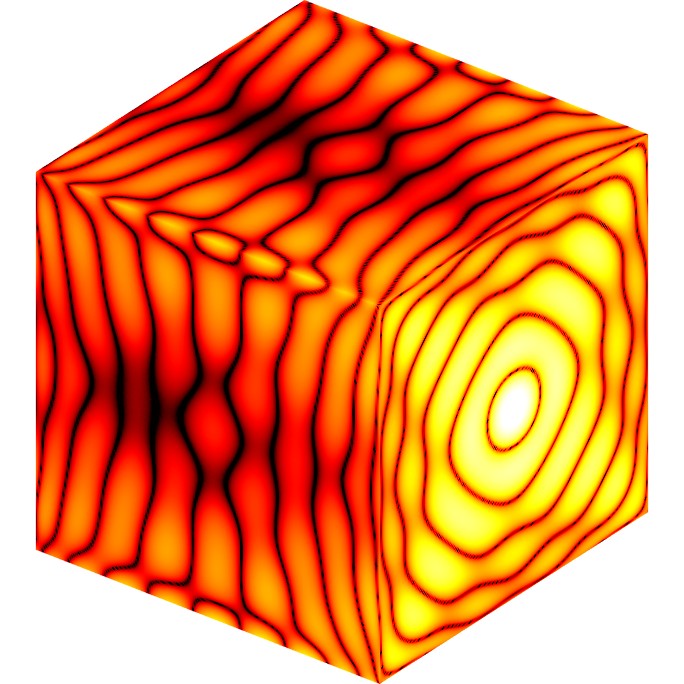} & \includegraphics[trim=0 0 0 0,clip,width=.38\textwidth,valign=m]{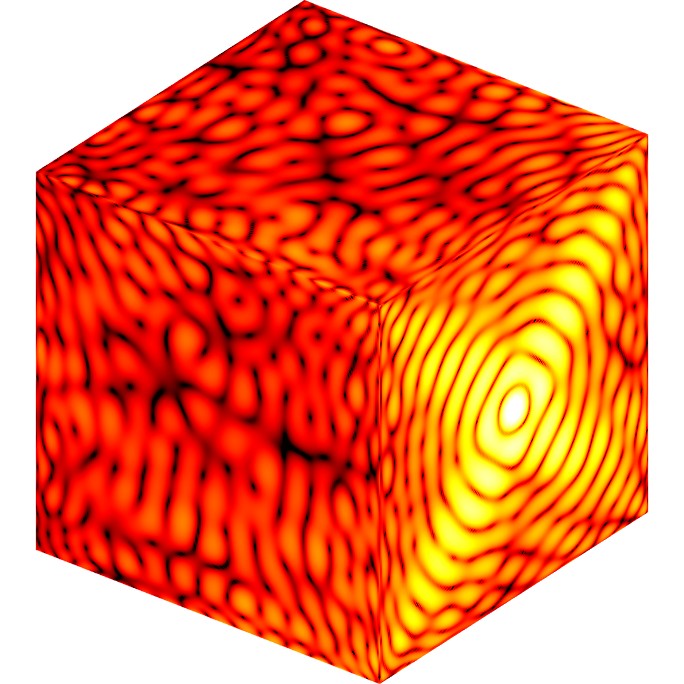}\\
\vspace{2mm}\\
\includegraphics[width=.4\textwidth]{imagespdf//bar_propagative_cube.pdf} &
\includegraphics[width=.4\textwidth]{imagespdf//bar_evanescent_cube2.pdf}\\
Absolute error using PPWs
& Absolute error using EPWs\\
\end{tabular}
\caption{The caption of Figure \ref{figure 7.19} applies here as well, with the only difference that now the Dirichlet data sampling nodes on the boundary $\partial Q_1$ are locally-refined near the corners.}
\label{figure 7.20}
\end{figure}

\begin{figure}
\centering
\begin{tabular}{cc}
\includegraphics[trim=75 75 75 85,clip,width=.45\textwidth,valign=m]{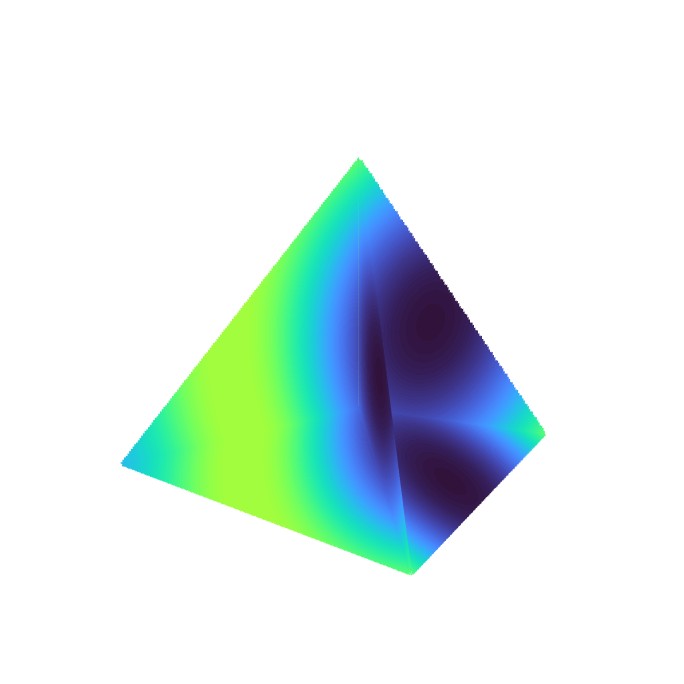} &  \includegraphics[trim=75 75 75 85,clip,width=.45\textwidth,valign=m]{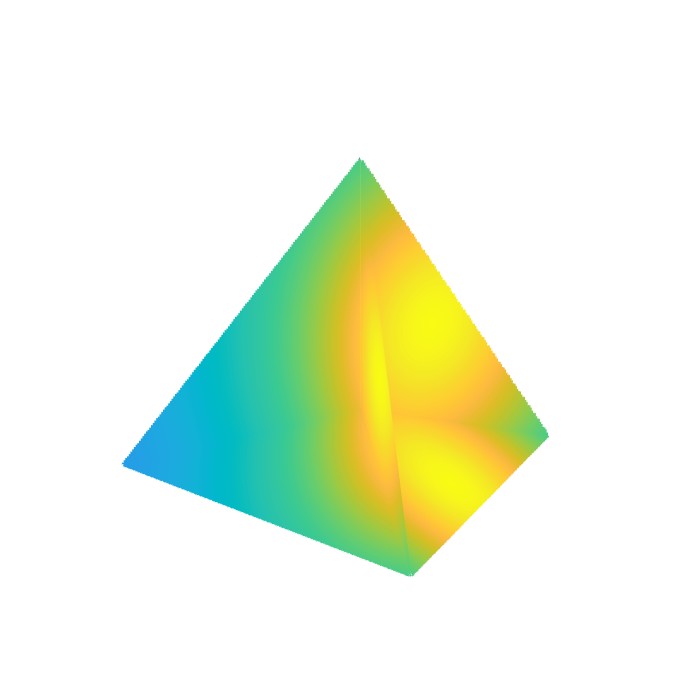}\\
\vspace{-3mm}\\
\includegraphics[width=.4\textwidth]{imagespdf//bar_fundamental_real.pdf} &
\includegraphics[width=.36\textwidth]{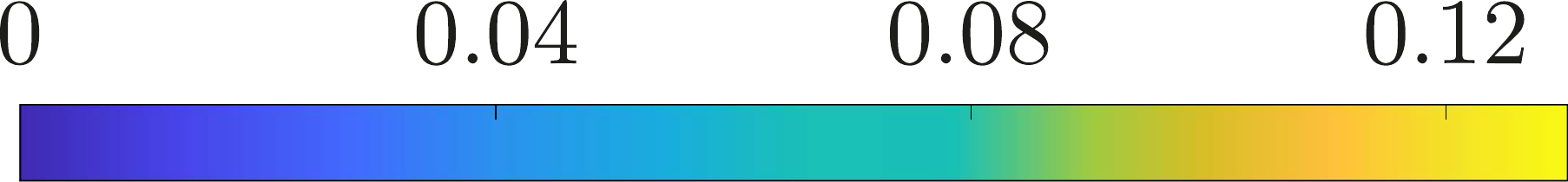}\\
Real part $\Re\{\Phi_{\mathbf{s}}\}$
& Modulus $|\Phi_{\mathbf{s}}|$\\
\includegraphics[trim=70 70 70 70,clip,width=.45\textwidth,valign=m]{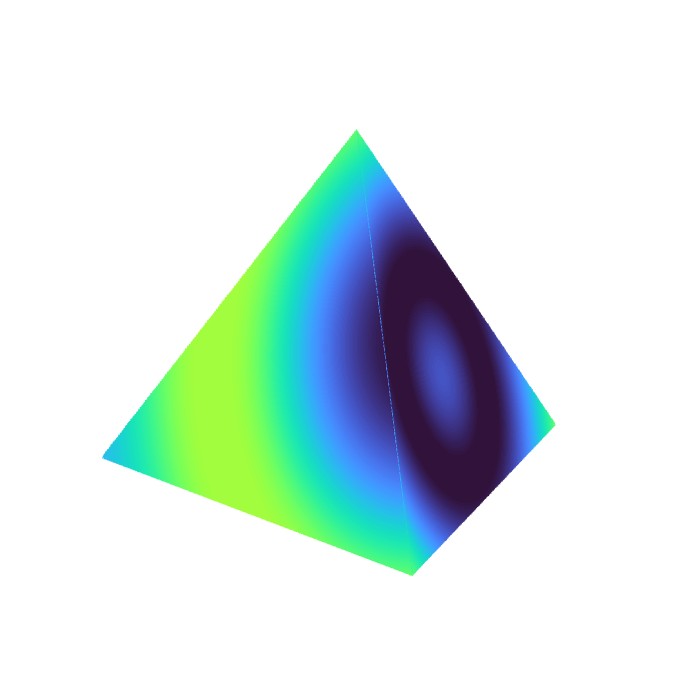} & \includegraphics[trim=70 70 70 70,clip,width=.45\textwidth,valign=m]{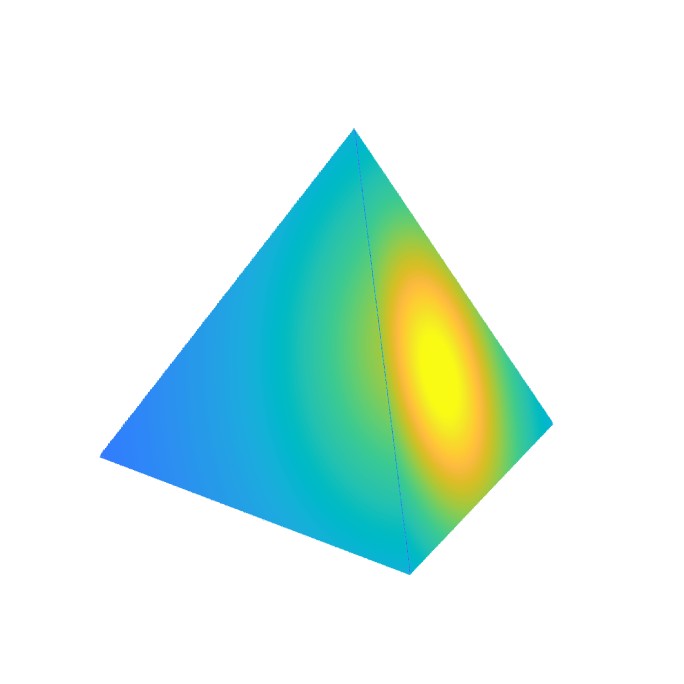}\\
\vspace{-3mm}\\
\includegraphics[width=.4\textwidth]{imagespdf//bar_fundamental_real.pdf} &
\includegraphics[width=.36\textwidth]{imagespdf//bar_fundamental_abs.pdf}\\
Real part $\Re\{\Phi_{\mathbf{s}}\}$
& Modulus $|\Phi_{\mathbf{s}}|$\\
\end{tabular}
\caption{Fundamental solution $\Phi_{\mathbf{s}}$ defined in (\ref{fundamental solution}) with wavenumber $\kappa=5$ and $\mathbf{s}\in \R^3\setminus\overline{T_1}$ so that $\textup{dist}(\mathbf{s},\partial T_1)=\lambda/3$. Both the real part $\Re\{\Phi_{\mathbf{s}}\}$ and the modulus $|\Phi_{\mathbf{s}}|$ of the fundamental solution are plotted on the six triangles with one vertex at the barycenter of $T_1$ and the other two at two vertices of $T_1$ (top) and on the boundary $\partial T_1$ (bottom).}
\label{figure 7.21}
\end{figure}
\begin{figure}
\centering
\begin{tabular}{cc}
\includegraphics[trim=75 75 75 85,clip,width=.45\textwidth,valign=m]{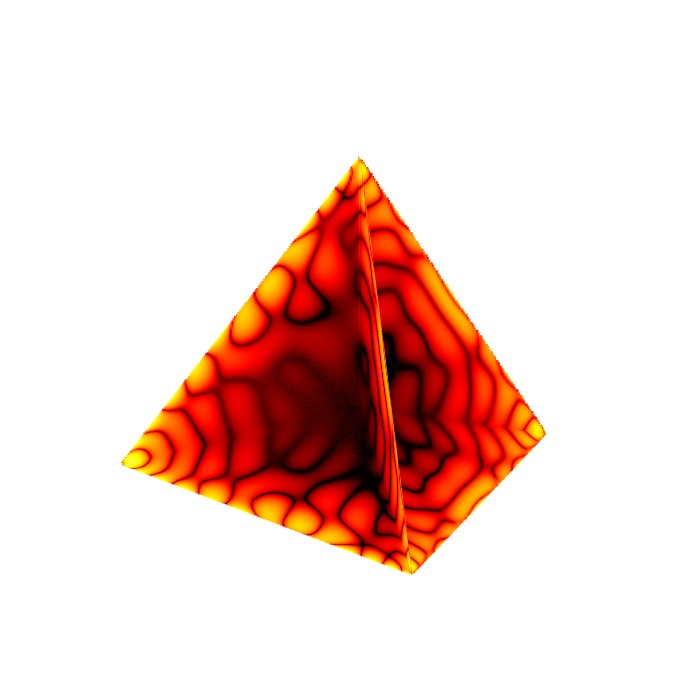} &  \includegraphics[trim=75 75 75 85,clip,width=.45\textwidth,valign=m]{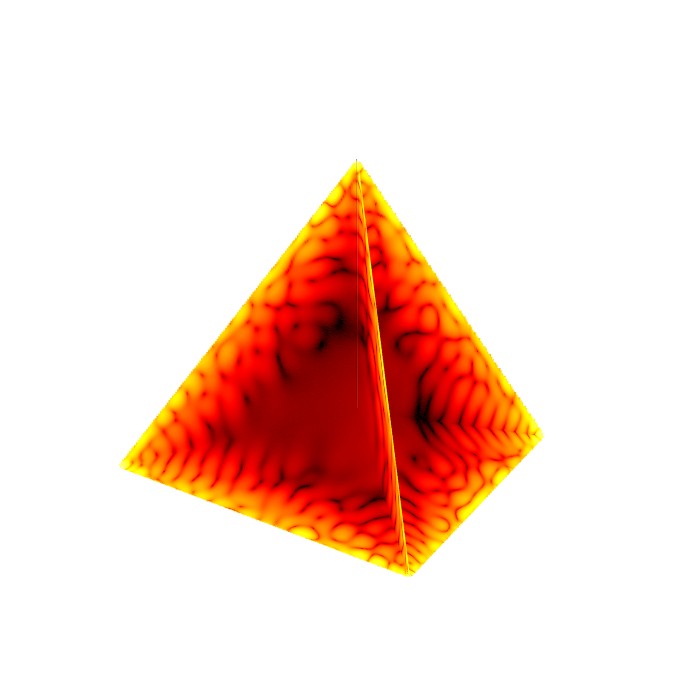}\\
\vspace{-3mm}\\
\includegraphics[width=.4\textwidth]{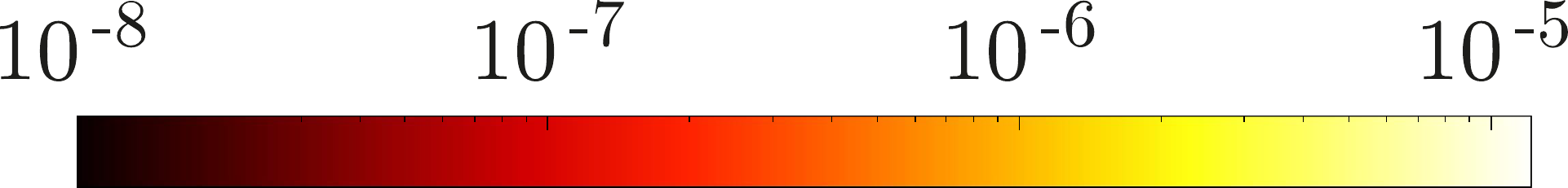} &
\includegraphics[width=.4\textwidth]{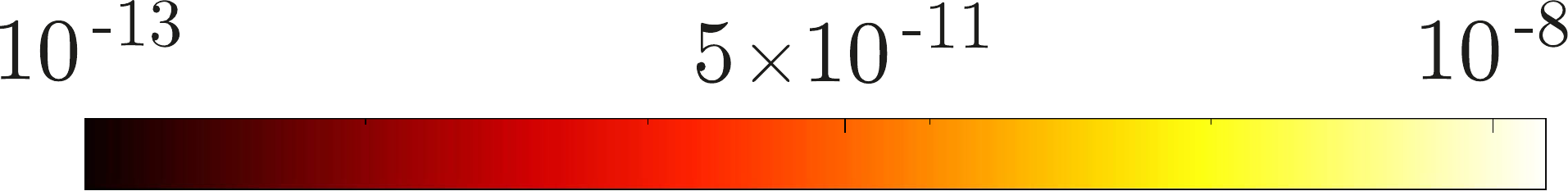}\\
Absolute error using PPWs
& Absolute error using EPWs\\
\includegraphics[trim=70 70 70 70,clip,width=.45\textwidth,valign=m]{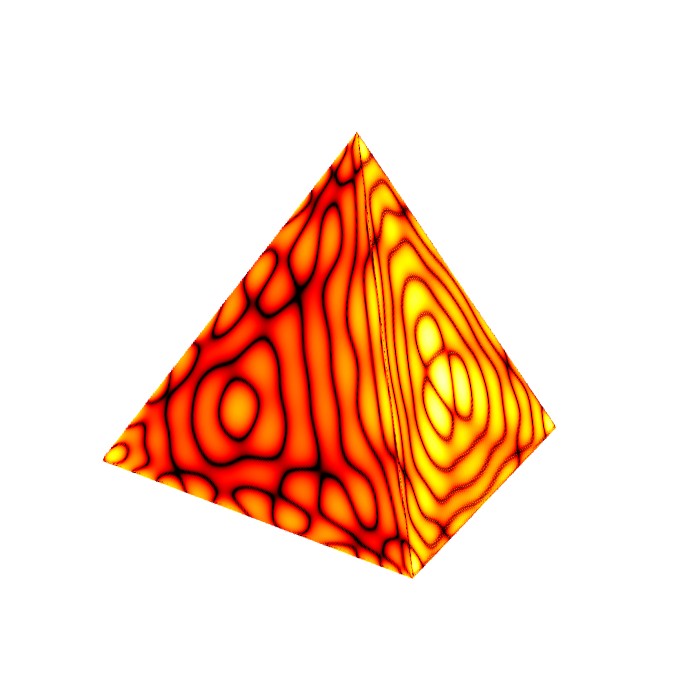} & \includegraphics[trim=70 70 70 70,clip,width=.45\textwidth,valign=m]{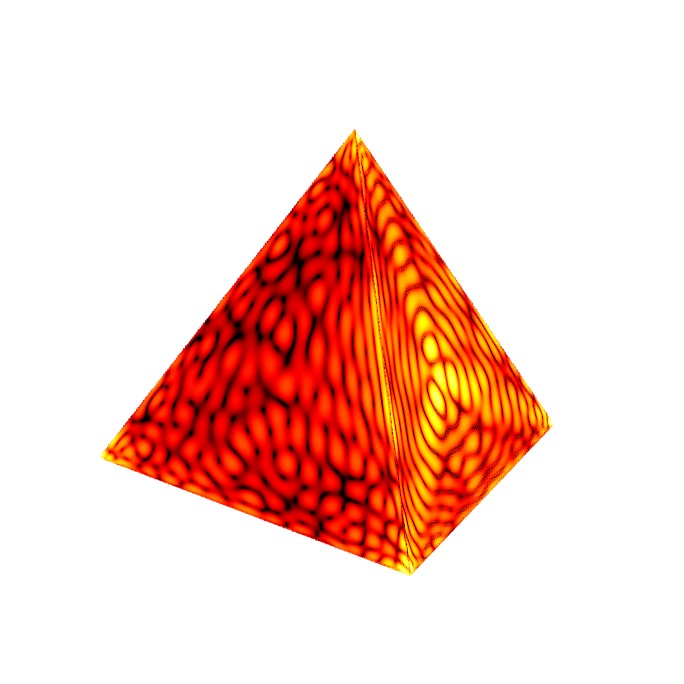}\\
\vspace{-3mm}\\
\includegraphics[width=.4\textwidth]{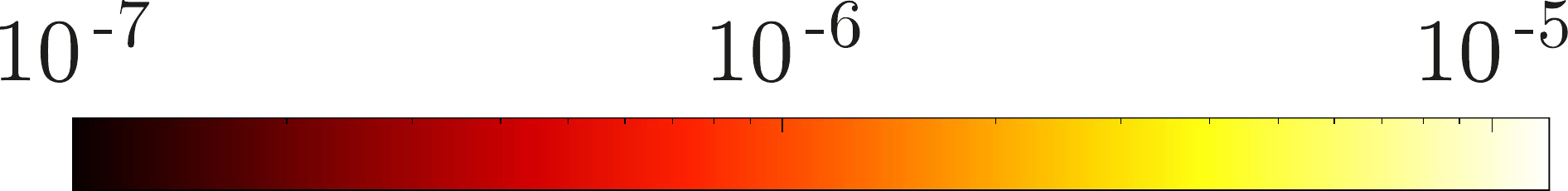} &
\includegraphics[width=.4\textwidth]{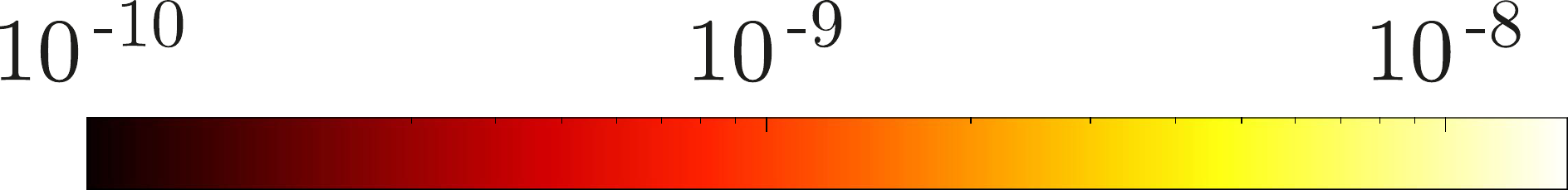}\\
Absolute error using PPWs
& Absolute error using EPWs\\
\end{tabular}
\caption{Absolute errors of the approximation of the fundamental solution $\Phi_{\mathbf{s}}$ defined in (\ref{fundamental solution}) with $\mathbf{s}\in \R^3\setminus\overline{T_1}$ so that $\textup{dist}(\mathbf{s},\partial T_1)=\lambda/3$. The error is provided using $P=2704$ plane waves, either propagative ones $\boldsymbol{\Phi}_P$ from (\ref{plane waves approximation set}) (left) or evanescent ones $\boldsymbol{\Phi}_{L,P}$ from (\ref{evanescence sets}), whose parameters are constructed using the Extremal--Sobol strategy (e) presented in Definition \protect\hyperlink{Definition 6.4}{6.4} (right).
In both cases, the Dirichlet data sampling nodes on the boundary $\partial T_1$ are equispaced.
The absolute errors are plotted both on the six triangles with one vertex at the barycenter of $T_1$ and the other two at two vertices of $T_1$ (top) and on the boundary $\partial T_1$ (bottom). Wavenumber $\kappa=5$ and regularization parameter $\epsilon=10^{-14}$.}
\label{figure 7.22}
\end{figure}
\begin{figure}
\centering
\begin{tabular}{cc}
\includegraphics[trim=75 75 75 85,clip,width=.45\textwidth,valign=m]{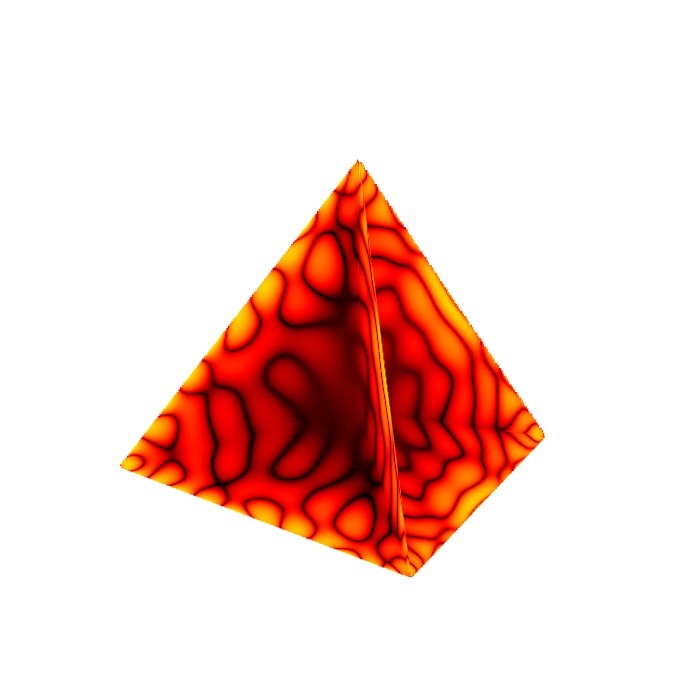} &  \includegraphics[trim=75 75 75 85,clip,width=.45\textwidth,valign=m]{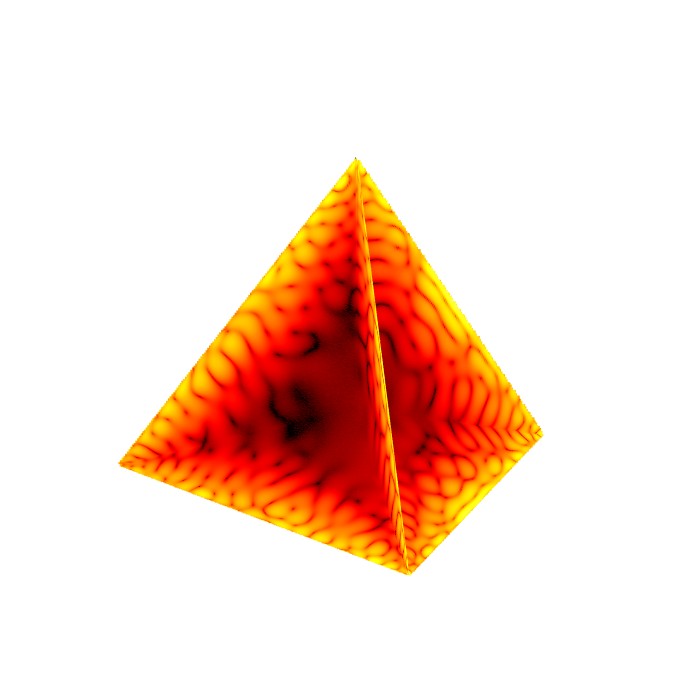}\\
\vspace{-3mm}\\
\includegraphics[width=.4\textwidth]{imagespdf//bar_propagative_tetra.pdf} &
\includegraphics[width=.4\textwidth]{imagespdf//bar_evanescent_tetra.pdf}\\
Absolute error using PPWs
& Absolute error using EPWs\\
\includegraphics[trim=70 70 70 70,clip,width=.45\textwidth,valign=m]{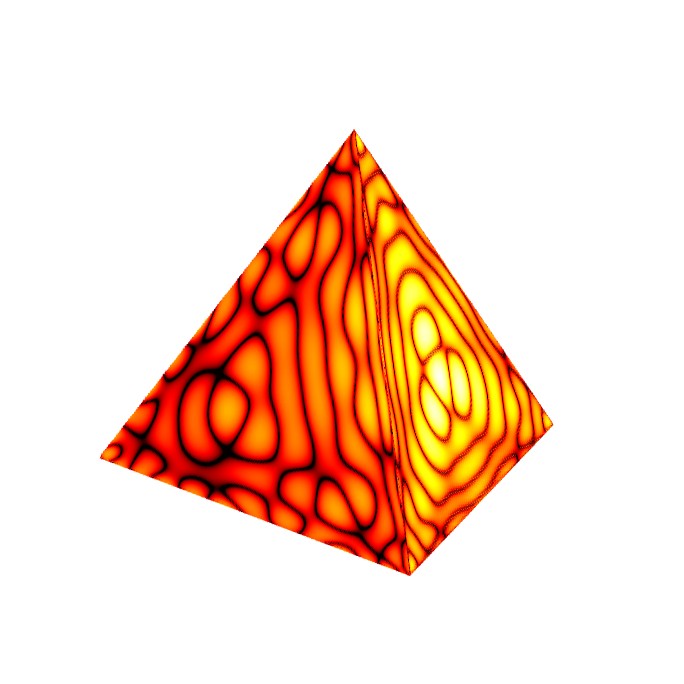} & \includegraphics[trim=70 70 70 70,clip,width=.45\textwidth,valign=m]{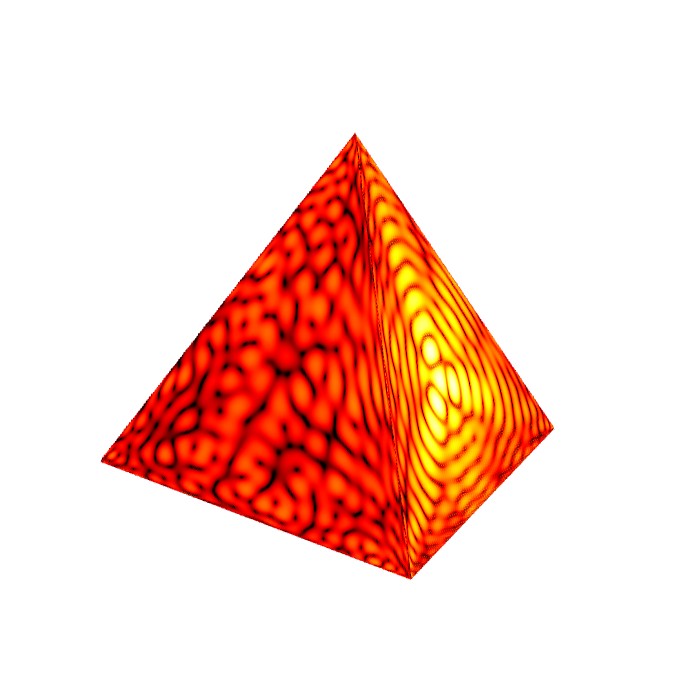}\\
\vspace{-3mm}\\
\includegraphics[width=.4\textwidth]{imagespdf//bar_propagative_tetra2.pdf} &
\includegraphics[width=.4\textwidth]{imagespdf//bar_evanescent_tetra2.pdf}\\
Absolute error using PPWs
& Absolute error using EPWs\\
\end{tabular}
\caption{The caption of Figure \ref{figure 7.22} applies here as well, with the only difference that now the Dirichlet data sampling nodes on the boundary $\partial T_1$ are locally-refined near the corners.}
\label{figure 7.23}
\end{figure}

\chapter*{Conclusions}
Propagative plane wave-based Trefftz methods are known to suffer from numerical instability, which can significantly reduce the accuracy and efficiency of approximation schemes.
This issue is due to the ill-conditioning of the linear system that necessarily arises when the size of the approximation space increases, as propagative plane waves with similar directions of propagation are almost-linear dependent.
However, the error analysis discussed in \cite[Sec.\ 3.4]{parolin-huybrechs-moiola}, here extended to the three-dimensional case (see Section \hyperlink{Section 2.4}{2.4}), shows that evanescent plane waves, providing accurate approximations in the form of expansions with bounded coefficients, are able to overcome this problem.

The presented analysis broadens the findings of \cite[Th.\ 4.3]{parolin-huybrechs-moiola} by demonstrating that propagative plane waves lack high Fourier modal content also in 3D. As a result, their expansion coefficients blow up as soon as less regular Helmholtz solutions are considered. Theorem \hyperlink{Theorem 3.1}{3.4} states this result, which is verified numerically.
On the contrary, the modal analysis made possible by generalizing the Jacobi--Anger identity to complex-valued directions (see Theorem \hyperlink{Theorem 4.6}{4.7}), shows that evanescent plane waves are able to cover the entire modal spectrum.
Moreover, Theorem \hyperlink{Theorem 5.1}{5.9} extends the key results presented in \cite[Th.\ 6.7]{parolin-huybrechs-moiola} to three dimensions by stating that any Helmholtz solution within the unit ball can be uniquely expressed as a continuous superposition of evanescent plane waves.
The derived numerical recipes show the accurate and stable behavior of evanescent plane wave approximation sets, suggesting their potential usefulness in Trefftz methods.

In conclusion, this paper contributes to the ongoing research on stable and accurate approximation schemes based on evanescent plane waves, building upon the previous work presented in \cite{parolin-huybrechs-moiola}. While Conjecture \hyperlink{Conjecture 6.1}{6.2} remains unsolved, further investigation is necessary to provide a rigorous justification of the approximation properties of the evanescent plane wave sets.

Future studies will focus on broadening the scope in order to encompass more general geometries, allowing for the application of Trefftz schemes, and exploring more complex boundary value problems, such as those involving time-harmonic Maxwell and elastic wave equations. We are confident that this will provide valuable insights into the potential applications of evanescent plane wave approximation schemes in a broader range of contexts.
\phantomsection
\addcontentsline{toc}{chapter}{Conclusions}

\clearpage
\phantomsection
\addcontentsline{toc}{chapter}{References}

\printbibliography[title={References}]

\end{document}